\documentclass[12pt,a4paper]{article}
\usepackage{lmodern}
\usepackage[T1]{fontenc}
\usepackage{hyperref}
\usepackage{amsfonts,amsmath,amssymb,amsopn,amsthm}
\usepackage[all,2cell]{xy}\UseAllTwocells
\usepackage{vmargin}
\usepackage{makeidx}
\usepackage{graphicx}
\usepackage{tikz}
\usepackage{comment}
\usepackage{mathabx}
\usepackage{bbm}
\usepackage[style=alphabetic,citestyle=alphabetic,sorting=nyt,sortcites=true,autopunct=true,babel=hyphen,hyperref=true,abbreviate=false,backref=false]{biblatex}
\AtEveryBibitem{ \clearfield{url} \clearfield{urldate} \clearfield{review} \clearfield{series} \clearfield{doi} \clearfield{isbn} \clearfield{issn} } 
\defbibheading{bibempty}{}
\DeclareNolabel{\nolabel{\regexp{[\p{Z}]+}}}
\usepackage{scalerel,stackengine}
\stackMath
\newcommand\reallywidehat[1]{%
\savestack{\tmpbox}{\stretchto{%
  \scaleto{%
    \scalerel*[\widthof{\ensuremath{#1}}]{\kern-.6pt\bigwedge\kern-.6pt}%
    {\rule[-\textheight/2]{1ex}{\textheight}}
  }{\textheight}%
}{0.5ex}}%
\stackon[1pt]{#1}{\tmpbox}%
}
\title{Rigid cohomology of locally noetherian schemes \\ Part 2: Crystals}
\author{Bernard Le Stum}
\date{Version of \today}
\newtheorem{thm}{Theorem}[section]
\newtheorem{prop}[thm] {Proposition}
\newtheorem{cor}[thm] {Corollary}
\newtheorem{lem}[thm] {Lemma}
\theoremstyle{definition}
\newtheorem{dfn}[thm] {Definition}
\newenvironment{xmp}[1][Example]{\begin{trivlist} \item[\hskip \labelsep {\bfseries #1}]}{\end{trivlist}}
\newcommand{\Addresses}{{
 \bigskip
 \footnotesize

Bernard Le Stum, \textsc{IRMAR, Université de Rennes,
Campus de Beaulieu, 35042 Rennes cedex, France}\par\nopagebreak
\texttt{bernard.le-stum@univ-rennes1.fr}

}}

\begin{document}

\maketitle
\bigskip

\begin{abstract}
We introduce the general notions of an overconvergent site and a constructible crystal on an overconvergent site.
We show that if $V$ is a geometric materialization of a locally noetherian formal scheme $X$ over an analytic space $O$ defined over $\mathbb Q$, then the category of constructible crystals on $X/O$ is equivalent to the category of constructible modules endowed with an overconvergent connection on the tube $\,]X[_V$ of $X$ in $V$.
We also show that the cohomology of a constructible crystal is then isomorphic to the de Rham cohomology of its realization on the tube $\,]X[_V$.
This is a generalization of rigid cohomology.
Finally, we prove universal cohomological descent and universal effective descent with respect to  constructible crystals with respect to the $h$-topology.
This encompass flat and proper descent and generalizes all previous descent results in rigid cohomology.
\end{abstract}

\tableofcontents

\addcontentsline{toc}{section}{Introduction}
\section*{Introduction}

\subsection*{History}

The guideline for this work goes back to the milestone \cite{Grothendieck68} of Alexander Grothendieck and the thesis \cite{Berthelot74} of his student Pierre Berthelot (see also \cite{BerthelotOgus78}) who, incidentally, happened to be my thesis advisor.
The strategy consists in developing a sufficiently general theory of crystals and their cohomology.
One then shows that this corresponds locally to modules endowed with an integrable connection and their de Rham cohomology.

The objects we are interested in (overconvergent isocrystals and rigid cohomology) were introduced by Berthelot in the early 80's (\cite{Berthelot96c*} and \cite{Berthelot83*} - see also \cite{LeStum07}).
In \cite{Ogus90}, Arthur Ogus gave a crystalline interpretation of convergent cohomology, which coincides with rigid cohomology on proper (possibly singular) varieties.
Recently, Dingxin Zhang showed in \cite{Zhang19} that Ogus'method still works with weakly formal schemes and dagger rigid spaces, allowing him to treat open varieties as well and recover rigid cohomology in general.
In the meantime, in \cite{LeStum11}, I already explained how one can describe rigid cohomology in the spirit of crystalline cohomology.
The present article is mostly an enhancement of \cite{LeStum11} (we use adic spaces instead of Berkovich analytic spaces).

I also introduced in \cite{LeStum16} the notion of a constructible isocrystal that generalizes the notion of an overconvergent isocrystal (see also \cite{LeStum14} for the one-dimensional case).
With respect to overconvergent isocrystals, they play the role that constructible sheaves play with respect to lisse sheaves.
Our purpose in this article is to give a treatment as general as possible of the theory of constructible isocrystals and their cohomology, show that one recovers locally a de Rham theory and prove in the end both cohomological and effective descent with respect to the $h$-topology.

\subsection*{Overview}

The basic notion here is that of an \emph{overconvergent space}. This is couple made of a formal scheme $X$ and an adic space $V$ but they are not unrelated.
More precisely, there should exist another formal scheme $P$ such that $X$ embeds in $P$ and $V$ maps to the adic space $P^{\mathrm{ad}}$ associated to $P$.
We turn them into a category called the \emph{absolute overconvergent site}.
The topology on this site is simply the topology inherited from the topology of the (various) adic spaces $V$.

We will make some noetherian hypothesis that are necessary for the very definition of $P^{\mathrm{ad}}$.
However, giving a map from $V$ to $P^{\mathrm{ad}}$ is equivalent to giving a map from the integral version $V^+$ of $V$ to $P$.
It is therefore possible to totally remove the noetherian hypothesis in the definition of the overconvergent site.
In particular, we would then be able to treat the classical situation when $V$ is an analytic variety defined over a general non-archimedean field $K$ (for a non discrete valuation).
I decided not to do that mostly because the geometry of the formal scheme $P$ reflects beautifully into the associated adic space $P^{\mathrm{ad}}$ and this gives a very nice bridge from the algebraic to the analytic world.

Let us give ourselves an (analytic) overconvergent space $(C,O)$ and a formal scheme $X$ (locally formally of finite type) over $C$.
We are interested in defining the notion of a \emph{constructible crystal} $E$ on $X$ over $O$ and compute its cohomology.
This is related to the classical situation as follows.
We would fix a non archimedean field $K$ with valuation ring (assumed to be discrete here) $\mathcal V$ and residue field $k$.
We would choose a formal scheme (locally finitely presented) $S$ over $\mathcal V$ and let $C := S_k$ be its special fiber (which is an algebraic variety over $k$) and $O := S_K$ be its generic fiber (which is an analytic variety over $K$).
We would then let $X$ be an algebraic variety over $C$.
In this situation, there exists a notion of overconvergent isocrystal on $X/O$ and this corresponds in our theory to a constructible crystal which is finitely presented.
One may also consider rigid cohomology for this overconvergent isocrystal and this will be identical to the cohomology of our constructible crystal.

In order to define our (constructible) crystals on $X/O$ ($C$ being understood from the context) and their cohomology, we proceed exactly as in Berthelot's thesis.
We define a ringed site $(X/O)^\dagger$ and a \emph{crystal} will be a module on this ringed site which is rigidified by requiring some transition maps to be bijective.
It is said to be \emph{constructible} if there exists a locally finite locally closed covering on which it becomes finitely presented.
There exists a morphism from $(X/O)^\dagger$ to the \emph{tube} of $C$ in $O$ (which is identical to $O$ in the classical situation) and the cohomology of $E$ is simply derived pushforward along this map.

For more flexibility and in order to allow some generalization, I will introduce the general notion of \emph{an overconvergent site}: this is a fibered category over the absolute overconvergent site.
It is endowed with the inherited topology.
As an elementary example, we may consider the absolute overconvergent site itself, or the absolute \emph{analytic} overconvergent site (whose objects are overconvergent spaces $(X,V)$ with $V$ analytic) or the overconvergent site represented by some overconvergent space $(X,V)$ (whose objects are simply morphisms $(Y,W) \to (X,V)$).
More interesting, with $(C,O)$ and $X$ as above, the objects of the overconvergent site $(X/O)^\dagger$ are morphisms $(U,V) \to (C,O)$ endowed with a factorization of the first map through $X$.

Following Berthelot, I proved in \cite{LeStum17*} a \emph{strong fibration theorem} that has the following immediate consequence: if $V$ is a \emph{geometric materialization} of $X$, then the corresponding morphism $(X,V) \to (X/O)^\dagger$ is a local epimorphism, or a covering if you prefer (theorem \ref{strfib}).
\emph{Geometric materialization} means that $V$ comes from a morphism of formal schemes which is formally smooth and partially proper around $X$.
It is then possible to describe the crystals on $X/O$ and their cohomology by working directly on $(X,V)$.
More precisely, assuming $O$ is analytic and defined over $\mathbb Q$, one shows that the category of constructible crystals on $X/O$ is equivalent to the category of modules endowed with an overconvergent integrable connection on the tube of $X$ (theorem \ref{crismic}).
Moreover, the cohomology of the crystal coincides with the de Rham cohomology of the module (corollary \ref{cohdR}).

We prove in the end that constructible crystals satisfy \emph{total descent} 
(we mean universally effective descent and universally cohomological descent) for the $h$-topology.
The \emph{$h$-topology} is generated by zariski open coverings and partially proper surjective maps.
This applies in particular to faithfully flat or proper dominant maps.
As a consequence, we recover classical effective descent (due to Christopher Lazda) an cohomological descent (due to Bruno Chiarellotto and Nobuo Tsuzuki) for overconvergent isocrystals.

\subsection*{Content}

In the first section, we give the definition of an overconvergent site and many examples.
In particular, we show that even if we mostly concentrate afterwards on the case of an overconvergent site of the form $(X/O)^\dagger$ (as explained above), the theory applies to algebraic stacks as well and we can also consider the convergent or even the partially overconvergent situation.
We then derive from the strong fibration theorem that a geometric materialization provides a local epimorphism as explained above.

In the second section, we study (abelian) sheaves on an overconvergent site $T$.
We start with the case of a representable site $T := (X,V)$ and consider its relation with the site of the tube of $X$ in $V$ (for the induced topology).
In general, one can always identify an object of $T$ with a morphism of overconvergent site $(X,V) \to T$ and any sheaf $E$ on $T$ will provide by pullback a family of realizations $E_V$ on the tubes.
As one may expect for such a theory, a sheaf on $T$ is completely understood from its realizations $E_V$ (and the transition maps).
Moreover, if $T' \to T$ is a morphism of overconvergent sites, derived pushforward may also be recovered from the cohomology on the various tubes of the overconvergent spaces $(X,V)$ that are defined over $T$.
We end this section with a fine study of the behavior of sheaves along a formal embedding $\gamma : Y \hookrightarrow X$.
In particular, in the analytic case, we introduce the notion of an \emph{overconvergent direct image} that replaces usual direct image (which does not preserve crystals in general).

In section three, we first introduce the structural ring $\mathcal O^\dagger_T$ of an overconvergent site $T$.
Its realization on the tube of an overconvergent space $(X,V)$ is the restriction $\mathcal O^\dagger_V$ to the tube of $X$ of the structural sheaf $\mathcal O_V$ of the adic space $V$.
This allows us to define as usual a crystal as an $\mathcal O^\dagger_T$-module with bijective linear transition maps.
As a first elementary example, the category of crystals on a representable site $(X,V)$ is equivalent to the category of $\mathcal O_V^\dagger$-modules.
An important issue in the theory concerns preservation of crystals under derived pushforward.
We show that the question reduces to a base change condition which should be easier to investigate.
Finally, we show that overconvergent direct image under a formal embedding preserves crystals and this provides an adjoint for usual inverse image (on crystals).

We focus in section four on the case of a formal scheme $X$ defined over some overconvergent space $(C,O)$ as above.
At some point, we will consider a geometric materialization $V$ but we start with the case of a general morphism $(X,V) \to (C,O)$.
Such a morphism always factors through $(X/O)^\dagger$ and we may consider the sieve $(X,V/O)^\dagger \subset (X/O)^\dagger$ generated by the image of $(X,V)$.
When $V$ is a geometric materialization, this is a covering sieve so that the properties of $(X,V/O)^\dagger$ will transfer directly to $(X/O)^\dagger$.
In general, assuming that $V$ is flat and locally of finite type over $O$, one can define the notions of an \emph{(overconvergent) stratification} or an \emph{(overconvergent) integrable connection} on an $\mathcal O_V^\dagger$-module.
They fit into a diagram
\[
\xymatrix{
\mathrm{Cris}(X/O)^\dagger \ar[r] \ar[d]^{(\star)} &  \mathrm{MIC}(X,V/O)^\dagger \ar@{^{(}->}[r] &\mathrm{MIC}(X,V/O)
\\ \mathrm{Cris}(X,V/O)^\dagger \ar[r]^-\simeq &  \mathrm{Strat}(X,V/O)^\dagger \ar[r]^{(1)} \ar@{->>}[u] &\mathrm{Strat}(X,V/O) \ar[u]^{(\star)}.}
\]
The $\star$ indicates that this is an equivalence when $V$ is a geometric materialization and $O$ is defined over $\mathbb Q$.
In order to go further, we also introduce the notion of a \emph{derived linearization} that may be seen as a sort of an inverse for realization of crystals.
More precisely, to any $\mathcal O_V^\dagger$-module $\mathcal F$ endowed with an integrable connection, one associates a complex $\mathrm RL_{\mathrm{dR}}\mathcal F$ on $(X,V/O)^\dagger$.
When $E$ is a crystal on $(X,V/O)^\dagger$, there exists an adjunction map $E \to \mathrm RL_{\mathrm{dR}}E_V$ and $E$ is called a \emph{de Rham crystal} when this is an isomorphism.
In this case, the cohomology of $E$ coincides with the de Rham cohomology of $E_V$.

In section five, we finally introduce the finiteness conditions that are necessary to go further.
We first study the cohomology of an (abelian) sheaf on the tube of a formal scheme $X$ in an analytic adic space $V$.
The problem is that, even when $V$ is quasi-compact, the tube need not be.
We show however that the tube is always paracompact and that cohomology can be computed on small neighborhoods.
As a first consequence, we prove that $\mathcal O_V^\dagger$ is always a coherent ring.
We then introduce the notion of a constructible crystal.
This is an enhanced version of the notion of a finitely presented crystal since we also allow any overconvergent direct image of a finitely presented crystal along a formal embedding.
There also exists a notion of constructible module on a tube and we prove the fundamental result that the map $(1)$ above is fully faithful on constructible modules endowed with an overconvergent stratification (in the case of a geometric materialization).

We prove our main statements in section six.
As a direct consequence of the previous results, we show that the category of constructible crystals on $(X/O)^\dagger$ is equivalent to the category of constructible modules on the tube of $X$ endowed with an overconvergent integrable connection in the case of a geometric materialization $V$ defined over $\mathbb Q$.
Next, we prove the overconvergent Poincaré lemma and deduce that a constructible crystal $E$ on $X/O$ is de Rham on $V$.
In particular, the cohomology of $E$ coincides with the de Rham cohomology of its realization $E_V$.
We end the section with a description of the situation in the Monsky-Washnitzer setting.

In section seven, we prove total descent with respect to  constructible crystals for the $h$-topology.
We follow Tsuzuki's strategy.
We first show that a morphism $(Y,W) \to (X,V)$ of analytic overconvergent spaces satisfies total descent with respect to  constructible crystals when it comes from a finite faithfully flat morphism of formal schemes.
For morphisms of analytic overconvergent sites, we proceed by induction on the dimension.
We first do the case of a proper birational morphism.
Then, we treat the case of a finite surjective map.
One can then deduce the general case from these.

In the appendix, we first review the theory of fibered categories and fix some terminology and notations. In particular, we show that we may always see an object of a fibered category $T$ as a morphism $U \to T$ from a representable fibered category.
Then, we briefly recall how effective descent works.
We spend more time on cohomological descent because the classical theory is not sufficient for us.
The point is that the category of crystals is not a thick subcategory of the category of all modules.
However, it was pointed out by David Zureick-Brown that this condition is not necessary if one follows Brian Conrad's strategy. 
Due to the lack of precise reference, we work out the details. 

\subsection*{Many thanks}

Several parts of this paper were influenced by the conversations that I had with many mathematicians and I want to thank in particular Ahmed Abbes, Tomoyuki Abe, Richard Crew, Veronika Ertl, Kazuhiro Fujiwara, Michel Gros, Fumiharu Kato,  Christopher Lazda, Vincent Mineo-Kleiner, Laurent Moret-Bailly, Matthieu Romagny and Alberto Vezzani.

\subsection*{Notations/conventions:}

\begin{enumerate}
\item We keep all notations/conventions from \cite{LeStum17*}.
We will assume throughout the paper that formal schemes are \emph{locally noetherian} and that adic spaces are \emph{locally of noetherian type} (locally of finite type over a noetherian ring of definition).
\item
We will not worry about our universe which is fixed but may be increased when necessary.
Also, we will not bother much either about infinite complexes that can always be truncated in order to satisfy our needs.
\item
If $\mathcal C$ is a category (resp.\ a site), we denote by $\widehat {\mathcal C}$ (resp.\ $\widetilde {\mathcal C}$) the topos of all presheaves (resp.\ sheaves) on $\mathcal C$.
When $f : \mathcal C' \to \mathcal C$ is a functor (resp.\ a morphism of sites), we should write $\widehat f$ (resp.\ $\widetilde f$) for the corresponding morphism of topoi but we will generally still use the letter $f$ when we believe that there is no ambiguity.
\item When $\mathcal A$ is a sheaf of rings, we will denote by $\mathrm{Mod}(\mathcal A)$ the category of $\mathcal A$-modules.
For a ringed site $T$, we may also write $\mathrm{Mod}(T)$ instead of $\mathrm{Mod}(\mathcal O_T^\dagger)$.

\end{enumerate}

\section{Overconvergent sites}

In order to treat the case of algebraic stacks as well (see Zureick-Brown's Ph.\ D. thesis \cite{ZureickBrown10}), it will be necessary to work with fibered categories.
Actually, the strategy in \cite{LeStum11} was to work over discrete fibered categories even if we were only interested in algebraic varieties in the end.
An overconvergent site will therefore be defined as a fibered category over the absolute overconvergent site introduced in \cite{LeStum17*}.

\subsection{Definition}

We describe here the categories on which our crystals will live.
This is done in a very general way, allowing further generalization, even if most results afterwards will only hold in somehow classical situations.
We start by recalling the context from \cite{LeStum17*}.

\emph{Formal schemes} are always assumed to be locally noetherian.
This is a technical condition that insures that everything will work smoothly.
We define a \emph{formal embedding} $X \hookrightarrow P$ as a locally closed embedding of formal schemes.
They form a category in the obvious way.
When $X$ is closed in $P$, we will consider the \emph{completion} $P^{/X}$ of $P$ along $X$ which is obtained by refining locally the topology on the defining rings.
\emph{Adic spaces} are always assumed to be locally of noetherian type: it means that they are locally defined by Huber (i.e.\ $f$-adic)  rings that are finitely generated over some noetherian ring of definition.
Again, this is a technical condition that could be refined a bit (although general adic spaces are quite hard to deal with in general).
To any (locally noetherian) formal scheme $P$, one can associate an adic space $P^{\mathrm{ad}}$ by locally considering continuous valuations instead of open primes, and there exists a specialization map $\mathrm{sp} : P^{\mathrm{ad}} \to P$.
If $X \hookrightarrow P$ is a formal embedding, then one defines the \emph{tube} of $X$ in $P$ as follows: if $X$ is closed, then we set $\,]X[_{P} := P^{/X,\mathrm{ad}}$ and in general, we use the standard boolean argument.
Note that this is \emph{not} the same thing as the inverse image under specialization.
More precisely, in the analytic case, the tube of an open (resp.\ closed) subset is the closure (resp.\ the interior) of its inverse image under specialization.

An \emph{overconvergent space} is a triple $(X \hookrightarrow P \leftarrow V)$ made of a formal embedding $X \hookrightarrow P$ and a morphism $V \to P^{\mathrm{ad}}$ of adic spaces.
The tube $]X[_V$ of $X$ in $V$ is then the inverse image in $V$ of the tube of $X$ in $P$.
The overconvergent space is said to be \emph{analytic} if $V$ is an analytic (adic) space.
Again overconvergent spaces form a category in the obvious way and a morphism in this category is called a \emph{formal morphism} of overconvergent spaces.
The category $\mathbf{Ad}^\dagger$ of overconvergent spaces is obtained by making invertible \emph{strict neighborhoods}: formal morphisms
\[
(f,v,u) : (X' \hookrightarrow P' \leftarrow V') \to (X \hookrightarrow P \leftarrow V)
\]
such that $f$ is an isomorphism, $v$ is locally noetherian, $u$ is an open immersion and the morphism $]f[_u$ on the tubes is surjective (in which case it is automatically a homeomorphism).
We will then denote by $(X,V)$ the corresponding object because the formal scheme $P$ now plays a secondary role.
We endow this category with the topology coming from the adic side and call the corresponding site (resp.\ topos) the \emph{absolute overconvergent site} (resp.\ \emph{topos}).

We send the reader to the beginning of appendix \ref{fibcat} for a brief review of the theory of fibered categories.

\begin{dfn}
An \emph{overconvergent site} is a fibered category $T$ over the absolute overconvergent site $\mathbf{Ad}^\dagger$.
A morphism of overconvergent sites is a morphism of fibered categories between two overconvergent sites.
\end{dfn}

There also exists $2$-morphisms and overconvergent sites actually form a $2$-category $\mathbb Fib(\mathbf{Ad}^\dagger)$.
An overconvergent site $T$ is always implicitly endowed with the inherited topology (the coarsest topology making the projection cocontinuous).
A morphism of overconvergent sites $f : T' \to T$ in the above sense is then automatically a morphism of sites.
When $E$ is a sheaf on $T$, we will often simply denote by $E_{|T'} := f^{-1}E$ its inverse image on $T'$.

We will call an overconvergent site $T$ \emph{classic}\footnote{We do not want to say \emph{discrete} because it has too many meanings in mathematics.} if it is fibered in sets (or more generally in equivalence relations - this is equivalent).
Classic overconvergent sites form a genuine subcategory which is equivalent to the category $\widehat {\mathbf{Ad}^\dagger}$ of presheaves on the absolute overconvergent site $\mathbf{Ad}^\dagger$.
In particular, this is a topos.
We are mostly interested in classic overconvergent sites.
I believe however that, even in this case, the language of fibered categories is more natural.

As explained in appendix \ref{fibcat}, we can always consider an overconvergent space $(X,V)$ as a (classic) \emph{representable} overconvergent site.
We will say a little more about this in the examples below.
It is then equivalent to give an object of an overconvergent site $T$ lying over the overconvergent \emph{space} $(X,V)$ or a morphism of overconvergent \emph{sites} $s : (X,V) \to T$.
Moreover, giving a morphism in the overconvergent site $T$ is equivalent to giving a morphism $(f,u) : (X',V') \to (X,V)$ of overconvergent spaces together with a natural transformation
\begin{align*}
\eta : \xymatrix{
 (X',V') \xtwocell[0,2]{}\omit{<2>} \ar[rd]_-{(f,u)} \ar@/^.2cm/[rr]^{s'}
&& T \\
& (X,V). \ar[ru]_-s &}
\end{align*}
We will usually ignore the transformation $\eta$ which really plays a secondary role in the theory.
This is also harmless in the sense that we are mostly interested in classic overconvergent sites in which case $\eta$ does not appear at all (the diagram is simply assumed to be commutative).

\subsection{Examples} \label{fundxmp}


\begin{enumerate}
\item \label{repcas}
We may - and will - consider any object $(X, V)$ of the absolute overconvergent site $\mathbf{Ad}^\dagger$ (resp.\ any morphism in $\mathbf{Ad}^\dagger$) as an overconvergent site (resp.\ a morphism of overconvergent sites).
We may write $\mathbf{Ad}^\dagger_{/(X,V)}$ instead of $(X, V)$ in order to avoid confusion between the site and the space.
An object of the overconvergent site $(X,V)$ is then an overconvergent space $(Y, W)$ endowed with a (structural) morphism $(Y,W) \to (X,V)$ of overconvergent spaces and a morphism in the overconvergent site $(X,V)$ is a morphism of overconvergent spaces which is compatible with the structural morphisms.
This provides a first example of a classic overconvergent site.

In the same way, we will always consider a formal scheme $X$ (resp.\ a morphism of formal schemes) as a fibered category (resp.\ a morphism of fibered categories) over the site $\mathbf{FS}$ of (locally noetherian) formal schemes (with the \emph{coarse} topology).

\item
Let us consider the forgetful functor
\[
\mathbf{Ad}^\dagger \to \mathbf{FS}, \quad (U, V) \mapsto U.
\]
Any fibered category $T$ over $\mathbf{FS}$ will then give rise to an overconvergent site
\[
T^\dagger := \mathbf{Ad}^\dagger \times_{\mathbf{FS}} T.
\]
An object of $T^\dagger$ is a couple made of an overconvergent space $(U, V)$ and a morphism $s \colon U \to T$ of fibered categories (up to an automorphism of $U$).
Note that giving $s$ is equivalent to giving an object of $T$ lying over $U$.

Since we can always consider a formal scheme $X$ as a fibered category over $\mathbf{FS}$, we obtain, as a particular case, the overconvergent site $X^\dagger$.
This is another example of a classic overconvergent site.
An object of $X^\dagger$ is a couple made of an overconvergent space $(U, V)$ and a morphism of formal schemes $s \colon U \to X$.

\item
If $(X, V)$ is an overconvergent space, then there exists a natural morphism of overconvergent sites $(X,V) \to X^\dagger$.
It sends an overconvergent space $(Y, W)$ over $(X,V)$ to the couple made of the overconvergent space $(Y, W)$ and the morphism $Y \to X$.
We may denote the \emph{full image} of this morphism by $(X,V)^\dagger$ (which should not be confused with $(X,V)$ itself).
In other words, $(X,V)^\dagger$ is the sieve generated by $(X,V)$ inside $X^\dagger$.
This is a classic overconvergent site again.
An object of $(X,V)^\dagger$ is an overconvergent space $(Y,W)$ together with a \emph{specific} morphism $Y \to X$ that extends to \emph{some} morphism $(Y,W) \to (X,V)$ (only existence is required).

\item \label{mainob}
Assume now that we are given an overconvergent (base) space $(C, O)$ and a morphism of fibered categories $T \to C$ over $\mathbf {FS}$.
Note that $T$ could be an algebraic stack for example.
By functoriality, there exists a morphism of overconvergent sites   $T^\dagger \to C^\dagger$  and we also have at our disposal the morphism $(C,O) \to C^\dagger$.
Then, we set
\[
(T/O)^\dagger := (C,O) \times_{C^\dagger} T^\dagger
\]
(we do not mention $C$ in the notations because it plays an accessory role).
An object of $(T/O)^\dagger$ is an overconvergent space $(U,V)$ over $(C, O)$ together with a \emph{fixed} factorization $U \to T \to C$.
A morphism in $(T/O)^\dagger$ is a morphism of overconvergent spaces which is compatible with the given factorizations.
This $(T/O)^\dagger$ is actually the main object of interest for us (especially when $T$ is representable) and it is a classic overconvergent site when the category $T$ is fibered in sets (for example representable).
When $O = S^{\mathrm{an}}$, we may also write $(T/S)^\dagger$ and when $O = \mathrm{Spa}(R,R^+)$, we may write $(T/R)^\dagger$.
The same convention applies to the examples derived below.

Note that, if $(X,V)$ is an overconvergent space, then $(X,V) = (X/V)^\dagger$ so that our example here is a generalization of example \eqref{repcas}.

\item \label{compbs}
We can mix the last two examples if we are given a morphism of overconvergent spaces $(X,V) \to (C, O)$.
More precisely, in this situation, there exists a canonical morphism $(X, V) \to (X/O)^\dagger$.
We will denote its \emph{full image} by $(X,V/O)^\dagger$.
This classic overconvergent site is the right place to do the computations.
An object of $(X,V/O)^\dagger$ is an overconvergent space $(Y,W)$ over $(C, O)$ together with a \emph{specific} factorization $Y \to X \to C$ that extends to \emph{some} factorization $(Y,W) \to (X,V) \to (C,O)$ (only existence is required).
Beware that the inclusion
\[
(X,V/O)^\dagger \hookrightarrow (C,O) \times_{C^\dagger} (X,V)^\dagger
\]
is strict in general and $(X,V)^\dagger$ should actually be considered as a toy example that will not play any role in the future.

\item There exists a way to include the convergent case into the picture.
An overconvergent space $(X \hookrightarrow P \leftarrow V)$ is said to be \emph{convergent} if $X$ is closed in $P$ (and we may then assume that $\,]X[_{V} =V$).
We will denote by $\mathbf{Ad}^\wedge$ be the full subcategory of $\mathbf{Ad}^\dagger$ consisting of convergent spaces.
This is an overconvergent site.
We can define, if  $T$ is any fibered category over $\mathbf{FS}$ the (over-) convergent site
\[
T^\wedge := \mathbf{Ad}^\wedge \times_{\mathbf{FS}} T
\]
in analogy with what we did above so that $T^\wedge \subset T^\dagger$.
Also, if $(X,V)$ is a convergent space, we may consider the full image $(X,V)^\wedge$ of the morphism $(X,V) \to X^\wedge$ and we have again $(X,V)^\wedge \subset (X,V)^\dagger$.
Given an overconvergent variety $(C,O)$ and a morphism of fibered categories $T \to C$, we may also set
\[
(T/O)^\wedge := (C,O) \times_{C^\dagger} T^\wedge
\]
so that $(T/O)^\wedge \subset (T/O)^\dagger$.
Finally, if we are given a convergent space $(X,V)$ and a morphism $(X,V) \to (C, O)$, we denote by $(X,V/O)^\wedge$ the full image of $(X, V) \to (X/O)^\wedge$ and again $(X,V/O)^\wedge \subset (X,V/O)^\dagger$.
We shall not discuss this situation in the present article but it is closely related to Ogus convergent site in \cite{Ogus90}.

\item The \emph{partially overconvergent} situation also fits into our pattern as we shall now explain.
If we are given a formal embedding $X \hookrightarrow Y$, we can then consider the morphism of overconvergent sites
\[
Y^\dagger \to X^\dagger, \quad (U, V) \mapsto (U \times_{Y} X, V)
\]
and denote by $(X \subset Y)^\dagger$ the full image of $Y^\wedge$ into $X^\dagger$.
An object of $(X \subset Y)^\dagger$ is an overconvergent space $(X' \hookrightarrow P' \leftarrow V')$ together with a morphism $X' \to X$ that has the following property: there exists a left cartesian diagram
\[
\xymatrix{X' \ar@{^{(}->}[r] \ar[d] & Y' \ar[d] \ar@{^{(}->}[r] & P' \\ X \ar@{^{(}->}[r] & Y}
\]
such that $\,]Y'[_{V'}$ is open in $V'$.
On can then define $(X \subset Y/O)^\dagger$ and $(X \subset Y, V/O)^\dagger$ as we did before.
Again, we shall not study this situation at all.
Even if it might sound convenient to rely on this construction for the purpose of localization, it is actually a lot more flexible to work locally on the adic side (see corollary 5.16 of \cite{LeStum17*}.
\item \emph{Analytic} overconvergent spaces form a full \emph{fibered} subcategory $\mathbf{An}^\dagger \subset \mathbf{Ad}^\dagger$: if $f : (Y,W) \to (X,V)$ is any morphism of overconvergent spaces with $(X,V)$ analytic, then $(Y,W)$ is also analytic.
In particular, any fibered category $T$ over $\mathbf{An}^\dagger$ is automatically an overconvergent site.
We will then call $T$ \emph{analytic}.
More generally, if $T$ is any overconvergent site, then we may consider the category
\[
T^{\mathrm{an}} := \mathbf{An}^\dagger \times_{\mathbf{Ad}^\dagger} T
\]
of analytic spaces over $T$ as an analytic overconvergent site.
This applies of course to all the previous examples.
This is only in the case of analytic overconvergent sites that we expect to do anything interesting.
\end{enumerate}

\subsection{Geometric materialization}

An overconvergent site is a very general object, but, unless there exists, at least locally, a \emph{geometric materialization} (definition below), we will feel mostly helpless.
The power of these geometric materializations relies on the fibration theorems of \cite{LeStum17*}.

Recall first from definitions 4.12 and 4.13 of \cite{LeStum17*} that a morphism of formal embeddings $(X \hookrightarrow P)  \to (C \hookrightarrow S)$ is said to be formally smooth (resp.\  partially proper) if there exists a neighborhood $Q$ of $X$ in $P$ (resp.\ a closed subspace $Y$ of $P$ containing $X$) such that the morphism $Q \to S$ (resp.\  $P^{/Y} \to S$) is  formally smooth (resp.\ partially proper).
Recall next from definition 5.17 of \cite{LeStum17*} that a formal morphism of overconvergent spaces
\begin{equation} \label{morphism}
\xymatrix{X \ar@{^{(}->}[r]\ar[d] & P \ar[d]^v & V \ar[l] \ar[d]\\ C \ar@{^{(}->}[r] & S &O.\ar[l]}
\end{equation}
is said to be partially proper (resp.\ formally smooth) if the left hand square is partially proper (resp.\ formally smooth) and the right hand square is \emph{cartesian} in the sense that $V$ is a neighborhood of $X$ in $(v^{\mathrm{ad}})^{-1}(O) := P^{\mathrm{ad}}\times_{S^{\mathrm{ad}}} O$.
We can apply the same process to any property of open (resp.\ closed) nature such as formally \'etale (resp.\ partially finite) for example.

\begin{dfn}
A \emph{geometric materialization} is a formal morphism of \emph{analytic} overconvergent spaces which is partially proper and formally smooth.
\end{dfn}

Recall that both conditions imply, by definition, that the morphism is also right cartesian.
We may simply say that a morphism $(X, V) \to (C,O)$ of analytic overconvergent spaces is a geometric materialization when it can be represented by a formal morphism as in \eqref{morphism} which is a geometric materialization.
In practice, we may even say that $V \to O$ is a geometric materialization of $X \to C$ or that $V$ is a geometric materialization of $X$ over $O$.

\begin{xmp}
\begin{enumerate}
\item If $X$ is a quasi-projective variety over $C$, then there exists an obvious geometric materialization
\[
\xymatrix{X \ar@{^{(}->}[r]\ar[d] & \mathbb P^n_{S} \ar[d] & \mathbb P^n_{O} \ar[l] \ar[d]\\ C \ar@{^{(}->}[r] & S &O.\ar[l]}
\]
\item Any formal scheme $X$ which is locally formally of finite type over $C$ has locally a geometric materialization over $O$.
\item Let $\mathcal V$ be a discrete valuation ring with residue field $k$ and fraction field $K$ of characteristic zero and $\mathcal W$ a finite unramified extension of $\mathcal V$ with residue field $l$ and fraction field $L$.
Then 
\[
\xymatrix{\mathrm{Spec}(l) \ar@{^{(}->}[r]\ar[d] & \mathrm{Spf}(\mathcal W) \ar[d] & \mathrm{Spa}(L) \ar[l] \ar[d]\\ \mathrm{Spec}(k) \ar@{^{(}->}[r] &  \mathrm{Spf}(\mathcal V)  & \mathrm{Spa}(K) \ar[l]}
\]
is a geometric materialization (the morphism is actually (partially) finite and (formally) étale).
\item
Let $\mathcal V$ be a discrete valuation ring with perfect residue field $k$ of characteristic $p > 0$ and fraction field $K$ of characteristic zero.
Any finite separable extension of $k((t))$ has the form $l((s))$ for a  finite separable extension 
$l$ of $k$.
If $\mathcal W$ is an unramified lifting of $l$ with fraction field $L$, then there exists a geometric materialization
\[
\xymatrix{\eta_{l} := \mathrm{Spec}(l((s))) \ar@{^{(}->}[r]\ar[d] & \mathbb A^{\mathrm{b}}_{\mathcal W} := \mathrm{Spf}(\mathcal W[[s]])  \ar[d] & \mathbb D^{\mathrm{b}}_{L} = \mathrm{Spa}(L \otimes_{\mathcal W} \mathcal W[[s]]) \ar[l] \ar[d]\\ \eta_{k} := \mathrm{Spec}(k((t)))  \ar@{^{(}->}[r] &  \mathbb A^{\mathrm{b}}_{\mathcal V} := \mathrm{Spf}(\mathcal V[[t]]) & \mathbb D^{\mathrm{b}}_{K} = \mathrm{Spa}(K \otimes_{\mathcal V} \mathcal V[[t]]) \ar[l]}
\]
(we only consider here the $p$-adic and not $t$-adic or $(p,t)$-adic topology).
\item (Monsky-Washnitzer setting) Let $R$ be a noetherian ring, $X$ a smooth affine \emph{scheme} of finite type over $R$ and  $S \to \mathrm{Spec}(R)$ a morphism of formal schemes.
From the choice of a presentation $X \hookrightarrow \mathbb A^n_R$, one easily deduces a geometric materialization of $X_{C}$ over $O$:
\[
\xymatrix{X_{C} \ar@{^{(}->}[r]\ar[d] & \mathbb P^n_{S} \ar[d] & X_{O} \ar[l] \ar[d]\\ C\ar@{^{(}->}[r] & S &O.\ar[l]}
\]
In the particular case $S = \mathrm{Spf}(R^{/I})$ and $C = \mathrm{Spf}(R/I)$ for some ideal $I \subset R$, it follows from \cite{Arabia01} that \emph{any} smooth affine scheme over $C$ has the above form $X_C$.
\end{enumerate}
\end{xmp}

The main purpose of  \cite{LeStum17*} was to prove the strong fibration theorem.
This has the following immediate consequence:

\begin{prop} \label{lsect}
If a geometric materialization $(X', V') \to (X, V)$ induces an isomorphism $X' \simeq X$, then this is a local epimorphism\footnote{What we call local epimorphism (resp.\ local isomorphism) is the same thing as a covering (resp.\ bicovering).} of overconvergent spaces.
\end{prop}

\begin{proof}
It follows from theorem 5.21 of \cite{LeStum17*} that, locally, the morphism $(X', V') \to (X, V)$ has a section.
\end{proof}

Recall that we defined in parts \eqref{mainob} and \eqref{compbs})  of section \ref{fundxmp}, the overconvergent sites of the form $(X/O)^\dagger$ and  $(X,V/O)^\dagger$.

\begin{cor} \label{cpro}
If a geometric materialization $(X', V') \to (X, V)$ over some overconvergent space $(C, O)$ induces an isomorphism $f \colon X' \simeq X$, then it induces a local isomorphism
\[
(X',V'/O)^\dagger \to (X,V/O)^\dagger.
\]
\end{cor}

\begin{proof}
Since $f$ is an isomorphism, we have $(X'/O)^\dagger \simeq (X/O)^\dagger$.
On the other hand, we know from proposition \ref{lsect} that the morphism $(X', V') \to (X, V)$ is a local epimorphism.
It follows that the map induced between their full images in $(X'/O)^\dagger$ and $(X/O)^\dagger$ is a local isomorphism. 
\end{proof}

In other words, in the situation of corollary \ref{cpro}, there exists an equivalence of topoi:
\[
\widetilde{(X', V'/O)^\dagger} \simeq \widetilde{(X,V/O)^\dagger}.
\]

More interesting, one can prove the following generalization of proposition 2.5.14 of \cite{LeStum11}):

\begin{thm} \label{strfib}
If $(X,V) \to (C, O)$ is a geometric materialization, then the induced morphism $(X, V) \to (X/O)^\dagger$ (resp.\ $(X,V/O)^\dagger \to (X/O)^\dagger$) is a local epimorphism (resp.\ a local isomorphism).
\end{thm}

\begin{proof}
We have to show that any morphism $(Y, W) \to (X/O)^\dagger$ factors locally through $(X,V)$.
In other words, we are given a commutative diagram
\[
\xymatrix{Y \ar@{^{(}->}[r] \ar[d] & Q \ar@/^1pc/[dd] & \ar[l] W \ar@{-->}[d] \ar@/^1pc/[dd]\\ X \ar@{^{(}->}[r] \ar[d] & P \ar[d] & \ar[l] V \ar[d] \\ C \ar@{^{(}->}[r] & S & \ar[l] O}
\]
with a geometric materialization at the bottom and we need to (locally) fill the dotted arrow up.
It is sufficient to consider the commutative diagram
\[
\xymatrix{Y \ar@{^{(}->}[r] \ar@{=}[d] & Q & \ar[l] W \\ Y \ar@{^{(}->}[r] \ar[d] & P \times_S Q\ar[d]^{p_1}\ar[u]^{p_2} & \ar[l] W' \ar[d]\ar[u] \\ X \ar@{^{(}->}[r] \ar[d] & P \ar[d] & \ar[l] V \ar[d] \\ C \ar@{^{(}->}[r] & S & \ar[l] O}
\]
and apply the strong fibration theorem 5.21 of \cite{LeStum17*} to the upper rectangle.
The other assertion is then obtained exactly as in corollary \ref{cpro}.
\end{proof}

In the situation of theorem \ref{strfib}, there exists therefore an equivalence of topoi:
\[
\widetilde{(X,V/O)^\dagger} \simeq \widetilde{(X/O)^\dagger}.
\]
It means that, when $X$ has a geometric materialization $V$, we can work directly on $V$ in order to obtain general results about $X$.
It also means that the choice of $V$ does not matter as long as it is a geometric materialization.
This last argument was Berthelot's philosophy for the development of rigid cohomology.

As a consequence of the \'etale version of the strong fibration theorem, we also have the following invariance with respect to the basis :

\begin{prop} \label{etis}
Assume that the following formal morphism of \emph{analytic} overconvergent spaces
\[
\xymatrix{ C' \ar@{^{(}->}[r] \ar[d] & S'  \ar[d]&  O'  \ar[l] \ar[d] \\
C \ar@{^{(}->}[r] & S  & O  \ar[l]}
\]
is partially proper and formally \emph{\'etale}.
If $T$ is a fibered category over $C'$, then 
\[
(T/O')^\dagger \simeq (T/O)^\dagger.
\]
\end{prop}

Recall again that our assumptions require the morphism to be right cartesian.

\begin{proof}
We can assume that $T = C'$ and then pull the whole thing back.
We give ourselves a formal morphism of overconvergent spaces
\[
(U \hookrightarrow P \leftarrow V) \to (C \hookrightarrow S \leftarrow O)
\]
such that $U \to C$ factors through $C'$ and we have to show that our morphism factors uniquely through $(C' \hookrightarrow S' \leftarrow O')$.
Thanks to theorem 5.20 of  \cite{LeStum17*}, we can replace $(U \hookrightarrow P \leftarrow V)$ with $(U \hookrightarrow P_{S'} \leftarrow V_{O'})$ and the question then becomes trivial.
\end{proof}

\begin{xmp}
\begin{enumerate}
\item
Consider a right cartesian formal morphism
\[
\xymatrix{C' \ar@{^{(}->}[r] \ar[d] & S'  \ar[d]&  O'  \ar[l] \ar[d] \\
C \ar@{^{(}->}[r] & S  & O  \ar[l]}
\]
with $O$ analytic.
We denote by $\overline C$ (resp.\ $\overline C'$) the closure of $C$ in $S$ (resp.\ $C'$ in $S'$).
Assume that $\overline C' \to \overline C$ is finite \'etale.
If $T$ is a fibered  category over $C'$, then $(T/O)^\dagger \simeq (T/O')^\dagger$.
\item This is a particular case of the previous example: we can consider a finite extension of ultrametric fields $K \to K'$.
Then, if $T$ is an overconvergent site over the residue field $k'$ of $K'$, we have $(T/K)^\dagger \simeq (T/K')^\dagger$.
\end{enumerate}
\end{xmp}

\section{Sheaves}

We now study the category of sheaves of sets or abelian sheaves on an overconvergent site.
In the end, we will investigate more closely the behavior of sheaves under a formal embedding.
This is the first step for understanding the notion of a constructible crystal.

\subsection{Sites associated to an overconvergent space}

As already mentioned, we may always consider an overconvergent space as an overconvergent site.
However, this is a big site and it is necessary for computations to be able to work on a small site.

Thus, we can consider the big and small sites of an overconvergent space $(X,V)$:
\[
(X,V) \quad \left(= \mathbf{An}^\dagger_{/(X,V)}\right) \quad \mathrm{and} \quad \,]X[_{V} \quad \left(= \mathbf{Open}(\,]X[_V\right)
\]
(in which we consider the ordered set of open subsets as a site).
It should be clear from the context if we see $(X,V)$ as an overconvergent space (resp.\ $\,]X[_{V}$ as a topological space) or as a site.
There exists a functor
\begin{align} \label{smalbig}
\psi_{V} : \,]X[_{V} \to (X,V), \quad \,]X[_{V'} \mapsto (X,V')
\end{align}
(this is well defined) from which we can derive the usual morphisms relating big and small sites:

\begin{lem} \label{adjret}
Let $(X,V)$ be an overconvergent space.
Then, there exists a morphism of sites (resp.\ topoi)
\[
\varphi_{V} : (X,V) \to \,]X[_{V} \quad \left(\mathrm{resp.} \quad \psi_{V} :\widetilde {\,]X[_{V}} \to  \widetilde{(X,V)}\right)
\]
such that $\widetilde \varphi_V$ is an adjoint retraction for $\psi_V$.
\end{lem}

\begin{proof}
Follows from the fact that the functor \eqref{smalbig} is continuous, cocontinuous, fully faithful and left exact (use then \cite[\href{https://stacks.math.columbia.edu/tag/00XU}{Tag 00XU}]{stacks-project} for example).
\end{proof}

We may write $\varphi_{\,]X[_V}$ and  $\psi_{\,]X[_V}$ if we want to specify the formal scheme $X$ in the notations.
We may also write $\varphi_{X}$ and  $\psi_{X}$ when $V$ is understood from the context.
As a consequence of the proposition, we see that there exists a sequence of three adjoint functors on sheaves:
\[
\varphi_{V}^{-1} = \psi_{V!}, \quad \varphi_{V*} = \psi_{V}^{-1}, \quad \varphi^!_{V} = \psi_{V*}.
\]
In particular, $\varphi_{V*}$ commutes with all limits, colimits, tensor products and internal Hom.
Moreover, $\varphi_{V}^{-1}$ is exact and fully faithful so that $\varphi_{V*} \circ \varphi_{V}^{-1} = \mathrm{Id}$.

There exists an alternative approach.
We can consider the functor
\[
\Phi : \mathbf{Ad}^\dagger \to \mathbf{Ger}, \quad (X, V) \mapsto \,]X[_V^\dagger
\]
from the absolute overconvergent site to the category of germs of adic spaces (prepseudo-adic spaces up to strict neighborhoods).
Here $\,]X[_V^\dagger$ denotes the germ of the prepseudo-adic space $(V, \,]X[_V)$.
The functor $\Phi$ is continuous, cocontinuous and left exact (but not fully faithful).
It induces a couple of adjoint morphisms of sites/topoi $\Pi$ and $\Phi$ (in that order).
Moreover, there exists a commutative diagram
\[
\xymatrix{\mathbf{Ad}^\dagger \ar[rr]^{\Phi} &&  \mathbf{Ger} \\
(X,V) \ar[u] && \,]X[_V \ar[u] \ar[ll]_{\psi_V} } 
\]
implying various compatibilitiies (the right vertical functor is also continuous, cocontinuous and left exact - and moreover fully faithful).

The following result is elementary but fundamental:

 \begin{lem}  \label{comphi}
 If $(f,u) : (X',V') \to (X,V)$ is a morphism of overconvergent spaces, then the diagram
\[
\xymatrix{
\widetilde{(X',V')}\ar[rr]^{\varphi_{V'}}  \ar[d]^{(f,u)}  &&  \widetilde{\,]X'[}_{V'}  \ar[d]^{]f[_{u}}   \\ \ar[rr]^{\varphi_{V}}   \widetilde{(X,V)} && \widetilde{\,]X[}_{V} 
 }
 \]
is commutative.
 \end{lem}
 
 \begin{proof}
If $W \subset V$ is an open subset, we have
\begin{align*}
(\varphi_{V} \circ(f,u))^{-1}(\,]X[_{W}) &=  (f,u)^{-1}(\varphi_{V}^{-1}(\,]X[_{W}))
\\&=  (f,u)^{-1}(X,W)
\\&=  (X',V') \times_{(X,V)} (X,W)
\\&=:  (X',W')
 \end{align*}
 and
 \begin{align*}
(|f[_{u} \circ \varphi_{V'})^{-1}(\,]X[_{W}) &= \varphi_{V'}^{-1} (|f[_{u}^{-1}(\,]X[_{W}))
\\&= \varphi_{V'}^{-1} (\,]X'[_{W'}) \\ &= (X',W'). \qedhere
 \end{align*}
\end{proof}

If $(X,V)$ is an overconvergent space, then we will usually denote by $i_X : \,]X[_V \hookrightarrow V$ the inclusion of the tube and we may sometimes write $\mathcal F_{|\,]X[_V} := i_X^{-1}\mathcal F$.
In order to make the link with Berthelot's original theory, we can mention the following:

\begin{lem} \label{FKlim}
If $V$ is \emph{analytic} and $\mathcal F$ is a sheaf on $V$, then
\[
\varinjlim j'_*j'^{-1}\mathcal F \simeq i_{X*}i_X^{-1} \mathcal F
\]
when $j' : V' \hookrightarrow V$ runs through the inclusions of open neighborhoods of $\,]X[_V$ in $V$.
\end{lem}

\begin{proof}
Since $V$ is analytic, $\,]X[_V$ is stable under generalization (consequence of \cite{LeStum17*}, proposition 4.27) and therefore $\,]X[_V = \bigcap V'$.
The assertion is then readily checked on the stalks.
\end{proof}

\subsection{Realization}

We will see that a sheaf on an overconvergent site is fully understood when we know its realizations on various overconvergent spaces.
We let $T$ be an overconvergent site.

\begin{dfn}
If $E$ is a sheaf on $T$ and $(X, V)$ is an overconvergent space over $T$, then the \emph{realization} of $E$ on $(X,V)$ (we will simply say on $V$) is $E_{V} := \varphi_{V*}E_{|(X,V)}$ on $\,]X[_V$.
\end{dfn}

It is worth mentioning that the realization functor $E \mapsto E_{V}$ commutes with all limits, colimits, tensor products an internal Hom.
Note that, even if it doesn't show, $E_V$ does depend on the structural map $(X,V) \to T$ and not only on $(X,V)$.
Again, we may write $E_{\,]X[_V}$ if we want to specify the role of $X$ or else $E_X$ when this is $V$ which is understood from the context.
Also, by definition, if $T' \to T$ is a morphism of overconvergent sites, and $(X',V')$ is an overconvergent space over $T'$, then $(E_{|T'})_{V'} = E_{V'}$.

\begin{lem} \label{pulbk}
If $(f, u) : (X', V') \to (X,V)$ is a morphism of overconvergent spaces and $\mathcal F$ is a sheaf on $\,]X[_{V}$, then
\begin{align} \label{phif}
(\varphi_{V}^{-1}\mathcal F)_{V'} = ]f[_{u}^{-1}\mathcal F.
\end{align}
\end{lem}

\begin{proof}
This follows from lemma \ref{comphi}:
\[
(\varphi_{V}^{-1}\mathcal F)_{V'} = \varphi_{V'*}(f,u)^{-1}\varphi_{V}^{-1}\mathcal F = \varphi_{V'*}\varphi_{V'}^{-1}]f[_{u}^{-1}\mathcal F= ]f[_{u}^{-1}\mathcal F. \qedhere
\]
\end{proof}

Assume that we are given a morphism $(f,u) : (X',V') \to (X,V)$ of overconvergent spaces over $T$.
Then, if $E$ is a sheaf on $T$, there exists a \emph{transition map}
\[
\xymatrix{
]f[_{u}^{-1}E_{V} \ar@{=}[d] \ar[rr]^{\eta^{-1}}&& E_{V'}\ar@{=}[d].\\ ]f[_{u}^{-1}\varphi_{V*}E_{|V}\ar[r] & \varphi_{V'*}(f,u)^{-1}E_{|V} \ar@{=}[r]&\varphi_{V'*}E_{|V'}.}
\]
We wrote $\eta^{-1}$ because the morphism $(f,u)$ also comes with a natural transformation $\eta$ and the transition map also depends on this transformation.
It should be noticed that, by adjunction, the transition maps may as well be given on the form $E_{V} \to ]f[_{u*}E_{V'}$.

There exists another interpretation of the transition map:

\begin{lem} \label{alttr}
If $(f,u) : (X',V') \to (X,V)$ is a morphism of overconvergent spaces over $T$, then the transition map for a sheaf $E$ on $T$ is the realization on $V'$ of the adjunction map $\varphi_{V}^{-1}\varphi_{V*}E \to E$.
\end{lem}

\begin{proof}
Follows from lemma \ref{pulbk} applied to $\mathcal F = E_{V}$.
\end{proof}

\begin{prop} \label{realsh}
Giving a sheaf $E$ on $T$ is equivalent to giving the collection of its realizations $E_{V}$ for $s : (X,V) \to T$ together with transition maps $\eta^{-1} : ]f[_{u}^{-1}E_{V} \to E_{V'}$ for all morphisms
\begin{align*}
\eta : \xymatrix{
 (X',V') \xtwocell[0,2]{}\omit{<2>} \ar[rd]_{(f,u)} \ar@/^.2cm/[rr]^{s'}
&& T \\
& (X,V) \ar[ru]_s &}
\end{align*}
in such a way that the diagram
\[
\xymatrix{ ]f'[_{u'}^{-1}(]f[_{u}^{-1}E_{V}) \ar[rr]^-{]f'[_{u'}^{-1}(\eta^{-1})} \ar@{=}[d] && ]f'[_{u'}^{-1} E_{V'} \ar[d]^{\eta'^{-1}}
\\ ]f\circ f'[_{u \circ u'}^{-1}E_{V} \ar[rr]^-{(f'^{-1}(\eta)\circ \eta')^{-1}} && E_{V''}
}
\]
commutes when
\begin{align*}
\eta' : \xymatrix{
 (X'',V'') \xtwocell[0,2]{}\omit{<2>} \ar[rd]_{(f',u')} \ar@/^.2cm/[rr]^{s''}
&& T \\
& (X',V') \ar[ru]_{s'} &}
\end{align*}
is another morphism (cocycle condition) and $]\mathrm{Id}[_{u}^{-1} E_{V} = E_{V'}$ when $u : V' \hookrightarrow V$ is an open immersion (normalization condition).
\end{prop}

\begin{proof}
Due to the general description of sheaves on a fibered category over a site (see section \ref{topfib}), we may assume that $T = (X,V)$ is representable.
The family $E_{V'}$ for $(X',V') \to (X,V)$ being given, we simply set
\[
E(X',V') := \Gamma(\,]X'[_{V'}, E_{V'}).
\]
If, for $(f,u) : (X'', V'') \to (X',V')$, we interpret the transition map as a morphism $E_{V'} \to ]f[_{u*}E_{V''}$, then we obtain the restriction maps
\[
E(X',V') = \Gamma(\,]X'[_{V'}, E_{V'}) \to  \Gamma(\,]X'[_{V'}, ]f[_{u*}E_{V''}) = \Gamma(\,]X''[_{V''}, E_{V''}) = E(X'',V'').
\]
The cocycle condition implies that this does define a presheaf and the normalization condition shows that it is actually a sheaf.
It remains to do all the tedious verifications that we do obtain an equivalence of categories.
\end{proof}

It is important to recall that, in the proposition, $E_{V}$ actually depends on the structural map $s : (X,V) \to T$ and not only on $(X,V)$.

There exists a more formal interpretation of the previous statement.
Let us consider the fibered category\footnote{This is actually a category bifibered in dual of topoi.} $\mathrm{Sets}^\dagger$ over $\mathbf {Ad}^\dagger$ whose fiber over $(X,V)$ is opposite to the category of sheaves of sets on $]X[_V$.
If $(f,u) : (Y,W) \to (X,V)$ is a morphism and $\mathcal F$ (resp. $\mathcal G$) is a sheaf on $]X[_V$ (resp. $]Y[_W$), then a morphism $\mathcal G$ to $\mathcal F$ over $(f,u)$ is a morphism of sheaves $\mathcal F \to ]f[_{u*}\mathcal G$, or equivalently a morphism $]f[_u^{-1}\mathcal F \to \mathcal G$.
Now, if $T$ is an overconvergent site, then proposition \ref{realsh} may be rephrased by saying that there exists an equivalence of topoi:
\[
\widetilde T \simeq \mathrm{Hom}_{\mathbf {Ad}^\dagger}(T, \mathrm{Sets}^\dagger).
\]
Alternatively, we may consider the fibered category $\mathrm{Sets}_T^\dagger  := T \times_{\mathbf{Ad}^\dagger} \mathrm{Sets}^\dagger$ over $T$ and obtain an equivalence
\[
\widetilde T \simeq \mathrm{Hom}_T(T, \mathrm{Sets}_T^\dagger)
\]
with the topos of sections of $\mathrm{Sets}^\dagger_T$.

The following is worth mentioning:

\begin{cor} \label{comlim}
If $\underline E \to \underline E'$ is a morphism of diagrams of sheaves on $T$, then
\[
\varprojlim \underline E \simeq \varprojlim \underline E' \Leftrightarrow \forall (X,V) \to T, \varprojlim \underline {E_{V}} \simeq \varprojlim \underline {E'_{V}}
\]
and
\[
\varinjlim \underline E \simeq \varinjlim \underline E' \Leftrightarrow \forall (X,V) \to T, \varinjlim \underline {E_{V}} \simeq \varinjlim \underline {E'_{V}}
\]
\end{cor}

\begin{proof}
Follows from proposition \ref{realsh} which implies that the family of realizations is conservative and commutes with all limits and colimits.
\end{proof}

Note also that there exists an analogous assertion for tensor products and internal Hom.

As a consequence proposition \ref{realsh}, we also see that the topos $\widetilde T$ has enough points: it is actually sufficient to consider the (genuine) points $x \in \,]X[_{V}$ for all overconvergent spaces $(X,V)$ over $T$.
There exists an obvious analog of corollary \ref{comlim} with points instead of realizations but it only holds for \emph{finite} limits and general colimits (because inverse image does not preserve infinite limits in general).

For further use, let us also show the following :

\begin{lem} \label{cartcdb}
Let
\[
\xymatrix{(X',V') \ar[r]^{(f,u)} \ar[d]^{s'} & (X,V) \ar[d]^{s} \\ T' \ar[r]^v & T }
\]
be a $2$-cartesian diagram of overconvergent sites with a representable upper map.
If $E$ is a sheaf on $T'$, then
\[
(v_{*}E)_{V} \simeq ]f[_{u*}E_{V'}. 
\]
For a complex $E$ of abelian sheaves, we have $(\mathrm R v_{*}E)_{V} \simeq \mathrm R]f[_{u*}E_{V'}$.
\end{lem}

\begin{proof}
Since the diagram is $2$-cartesian, it follows from lemma \ref{basfib} in the appendix that $s^{-1} \circ v_{*} = (f,u)_{*} \circ s'^{-1}$ and we know from proposition \ref{comphi} that $\varphi_{V*} \circ (f,u)_{*} = ]f[_{u*} \circ \varphi_{V'*}$.
It follows that
\[
(v_{*}E)_{V} = \varphi_{V*}s^{-1}v_{*}E = \varphi_{V*}(f,u)_{*}s'^{-1}E = ]f[_{u*}\varphi_{V'*}s'^{-1}E = ]f[_{u*}E_{V'}.
\]
The last assertion then follows from the fact that $\varphi_{V*}$ and $\varphi_{V'*}$ as well as $s^{-1}$ and $s'^{-1}$ are all exact and preserve injectives.
\end{proof}

\subsection{Cohomology}

We will define here absolute and relative cohomology and explain how one can move back and forth between them.

If $T$ is an overconvergent site over some overconvergent space $(C,O)$ with structural morphism $f : T \to (C,O)$, then we will denote by
\[
p_{T/O} := \varphi_{O} \circ f :  T \to ]C[_{O}
\]
the composite morphism of sites.

\begin{dfn} \label{absco}
Let $T$ be an overconvergent site over some overconvergent space $(C,O)$.
If $E$ is a complex of abelian sheaves on $T$, then its \emph{absolute cohomology} is $\mathrm Rp_{T/O*}E$.
\end{dfn}

We may also write
\[
\mathcal H^k(T/O, E) = \mathrm R^kp_{T/O*}E.
\]
In the case $T = (X/O)^\dagger$ or $T = (X,V/O)^\dagger$, we will simply write $\mathcal H^k((X/O)^\dagger, E)$ or $\mathcal H^k((X,V/O)^\dagger, E)$ respectively (and not repeat $O$).

\begin{xmp}
\begin{enumerate}
\item
This absolute cohomology is meant to be a generalization of Berthelot's rigid cohomology.
More precisely, assume $\mathcal V$ is a complete discrete valuation ring with residue field $k$ and $S$ is a formal scheme topologically of finite type over $\mathcal V$.
If $E$ is an overconvergent isocrystal on some algebraic variety $X/S$, then one can define $\mathcal H_{\mathrm{rig}}^k(X/S, E)$.
One may then interpret $E$ as a sheaf on $(X/S)^\dagger$ and show that
\[
\mathcal H_{\mathrm{rig}}^k(X/S, E) = \mathcal H^k((X/S)^\dagger, E)
\]
(see theorem \ref{crismic} and corollary \ref{cohdR} below).
\item
Using lemma \ref{comphi}, we see that if $(f,u) : (X,V) \to (C,O)$ is a morphism of overconvergent spaces and $E$ is a complex of abelian sheaves on $(X,V)$, then
\[
\mathrm Rp_{X,V/O*}E = \mathrm R]f[_{u*}E_{V}
\]
only depends on the realization of $E$ on $V$ (in this a very specific situation).
\end{enumerate}
\end{xmp}

It should be mentioned that, if $g : T' \to T$ is a morphism of overconvergent sites and $E'$ is a complex of abelian sheaves on $T'$, then
\[
\mathrm Rp_{T'/O*}E' = \mathrm Rp_{T/O*}\mathrm Rg_*E'.
\]

We can switch between absolute and relative cohomology (derived direct image) as follows:

\begin{prop} \label{bicom}
\begin{enumerate}
\item If $T$ is an overconvergent site over an overconvergent space $(C,O)$ with structural morphism $f : T \to (C,O)$ and $E$ is a complex of abelian sheaves on $T$, then
\[
\mathrm Rp_{T/O*}E \simeq (\mathrm Rf_{*}E)_{O}.
\]
\item If $f : T' \to T$ is a morphism of overconvergent sites, $(X,V)$ is an overconvergent space over $T$, and $E$ is a complex of abelian sheaves on $T'$ then
\[
(\mathrm Rf_{*}E)_{V} \simeq \mathrm Rp_{T'_{V}/V*}E_{|T'_{V}}
\]
with $T'_{V} := T' \times_{T} (X,V) \to (X,V)$.
\end{enumerate}
\end{prop}

\begin{proof}
Since $\varphi_{V*}$ is exact, the first assertion is trivial and the second one follows from lemma \ref{basfib} in the appendix.
More precisely, there exists a $2$-cartesian diagram
\[
\xymatrix{T'_{V} \ar[r]^{f'}\ar[d]^{j'}  & (X,V) \ar[d]^j \\ T' \ar[r]^f & T}
\]
of overconvergent sites so that
\[
\xymatrix{
(\mathrm Rf_{*}E)_{V} \ar@{=}[d] \ar@{=}[rr]&& \mathrm Rp'_{V*}E_{|T'_V}\ar@{=}[d].\\ \varphi_{V*}j^{-1}\mathrm Rf_{*}E\ar@{=}[r] & \varphi_{V*} Rf'_{*}j'^{-1}E \ar@{=}[r]&\mathrm  R(\varphi_{V*} \circ f'_{*})j'^{-1}E }\qedhere
\]
\end{proof}

\begin{xmp}
\begin{enumerate}
\item
If $f : X \to C$ is a morphism of formal schemes, $(C', O')$ is an overconvergent space over $C$ (which simply means that we are given a morphism of formal schemes $C' \to C$) and $E$ is a complex of abelian sheaves on $X^\dagger$, then
\[
(\mathrm Rf_{*}E)_{O'} \simeq \mathrm Rp_{X_{C'}/O'*}E_{|X_{C'}/O'}.
\]
\item More interesting: if $(C,O)$ is an overconvergent space, $f : Y \to X$ is a morphism of formal schemes over $C$, $(U,V)$ is an overconvergent space over $(X/O)^\dagger$ and $E$ is a complex of abelian sheaves on $(Y/O)^\dagger$ then
\[
(\mathrm Rf_{*}E)_{V} \simeq \mathrm Rp_{Y \times_{X} U/V*}E_{|Y \times_{X} U/V}.
\]
This is a powerful tool to reduce the study of relative cohomology to absolute cohomology. We shall come back to this question later.
\item Another interesting case is the following: if $(X,V) \to (C,O)$ is a morphism of overconvergent spaces, we may consider the morphism
\[
j_V : (X,V) \to (X,V/O)^\dagger.
\]
If $E$ is a complex of abelian sheaves on $(X,V)$ and $(X',V')$ is an overconvergent space over $(X,V/O)^\dagger$, then we have
\[
(\mathrm Rj_{V*}E)_{V'} \simeq \mathrm R]p_{1}[_*E_{V' \times_O V}.
\]
The same formula holds with $(X/O)^\dagger$ instead of $(X,V/O)^\dagger$.
\end{enumerate}
\end{xmp}

\subsection{Embeddings}

Formal embeddings will play an important role in the future because they are used to reduce the study of constructible crystals to finitely presented crystals (classically called overconvergent isocrystals).

So, we want now to investigate the behaviour of abelian sheaves on an overconvergent site along a formal embedding.
We start with the following general consideration: if $f : Y \to X$ is morphism of formal schemes and $T$ is an overconvergent site over $X^\dagger$, we let $T_{Y} = T \times_{X^\dagger} Y^\dagger\ (= T \times_{X} Y)$ and still denote by $f : T_{Y} \to T$ the corresponding map.
If $E$ is a sheaf on $T$, we will also write $E_{|Y} := f^{-1}E$ which is a sheaf on $T_{Y}$.
If $g : T' \to T$ is a morphism of overconvergent sites, $g_Y : T'_Y\to T_Y$ denotes its pullback and $E$ is a sheaf (resp.\ a complex of abelian sheaves) on $T'$, then it follows from lemma \ref{basfib} that
\[
(g_*E)_{|Y} \simeq g_{Y*}E_{|Y} \quad (\mathrm{resp.}\ (\mathrm Rg_*E)_{|Y} \simeq \mathrm Rg_{Y*}E_{|Y}).
\]

Now, we let $\gamma : Y \hookrightarrow X$ be a formal embedding and $T$ an overconvergent site over $X^\dagger$.

\begin{lem}
$\widetilde T_{Y}$ is an open subtopos of $\widetilde T$.
Moreover, if $Y' \hookrightarrow X$ is another formal embedding with $Y \cap Y'= \emptyset$, then $\widetilde T_{Y} \cap \widetilde T_{Y'} = \widetilde \emptyset$.
\end{lem}

\begin{proof}
The first assertion formally follows from the fact that a formal embedding is a monomorphism of formal schemes and the second one is also automatic (pullback of disjoint open subtopoi).
\end{proof}

As a first consequence, we see that the morphism of topoi $\gamma : \widetilde T_{Y} \hookrightarrow \widetilde T$ is an embedding which means that $\gamma^{-1} \circ \gamma_* = \mathrm{Id}$ or equivalently that $\gamma_*$ (or $\gamma_{!}$) is fully faithful.
Also, if $\gamma' : Y' \hookrightarrow X$ is another formal embedding with $Y \cap Y'= \emptyset$, then we will have $\gamma^{-1} \circ \gamma'_{*} = 0$ (use lemma \ref{basfib} of the appendix for example).
Note however that $\gamma'_{*} \circ \gamma^{-1} \neq 0$ in general.

The open subtopos $\widetilde T_{Y}$ has a closed complement (as a subtopos) and there exists therefore a functor $\underline \Gamma_{\complement Y}$ of sections with support outside $Y$ that may also be defined as
\[
\Gamma_{\complement Y}E := \ker(E \to \gamma_{*}E_{|Y})
\]
on abelian sheaves.
This functor has an exact adjoint and therefore preserves injectives and commutes with all limits.
As a consequence, there exists a distinguished triangle
\[
\mathrm R\Gamma_{\complement Y}E  \to E \to \mathrm R\gamma_{*}E_{|Y} \to \cdots
\]
when $E$ is a complex of abelian sheaves on $T$.
This will be especially interesting when $Y = U$ is an open subset or when $Y=Z$ is a closed subset in which case we will be able to interpret the complement as a genuine closed or open subset.

In what follows, if $E$ a sheaf on $T$, we will rather write $E_{X'}$ than $E_{V'}$ for the realization on some $(X', V')$.
We will give now an explicit description of direct image under a formal embedding (note that we simply write $]\gamma[$ for the morphism induced on the tubes):

\begin{prop} \label{gamet}
If $(X',V')$ is an overconvergent space over $T$, $E$ is a sheaf on $T_{Y}$ and $\gamma' : Y' \hookrightarrow X'$ denotes the inverse image of $\gamma$, then we have 
\[
(\gamma_{*}E)_{X'} = ]\gamma'[_{*}E_{Y'}.
\]
If $E$ is a complex of abelian sheaves on $T_{|Y}$, then $(\mathrm R\gamma_{*}E)_{X'} = \mathrm R]\gamma'[_{*}E_{Y'}$.
\end{prop}

\begin{proof}
Since the diagram
\[
\xymatrix{(Y',V') \ar[r]^{(\gamma',\mathrm{Id})} \ar[d]^{j'} & (X',V') \ar[d]^{j} \\ T_{Y} \ar[r]^\gamma & T }
\]
is $2$-cartesian, this follows from lemma \ref{cartcdb} (or directly from lemma \ref{basfib} of the appendix).
\end{proof}

Be careful however that if $E$ is a sheaf on $T$, then $(\gamma^{-1}E)_{V'} \neq ]\gamma[^{-1}E_{V'}$ in general
 (again from lemma \ref{cartcdb}).
 
 Let us finish with the following remark: if we are given an overconvergent space $(X \overset \delta \hookrightarrow P \leftarrow V)$, then we have $]P[_V = V$ and we can therefore identify the inclusion map $i_X : \,]X[_V \hookrightarrow V$ with $]\delta[$.
There exists a partial converse: if $(X \overset \delta \hookrightarrow P \leftarrow V)$ is an analytic convergent ($X$ closed in $P$) space, then $W := \,]X[_V$ is an open subset of $V$, and if $\gamma : Y \hookrightarrow X$ is a formal embedding, then we can identify $]\gamma[$ with the inclusion map $i_Y : \,]Y[_W \hookrightarrow W$.

\subsection{Overconvergent direct image}

Usual direct image under a formal embedding will not preserve crystals in general and will have to be replaced by a more subtle functor.

As before, we denote by $\gamma : Y \hookrightarrow X$ a formal embedding.
We let $T$ be an \emph{analytic} overconvergent site over $X^\dagger$.
Recall that analytic means that all overconvergent spaces $(X,V)$ defined over $T$ are analytic.

We first recall that if $h$ is a locally closed embedding of topological spaces, then there exists a couple of adjoint functors $h_{!}, h^!$ on abelian sheaves such that $h_{!} = h_{*}$ (resp.\ $h^!=h^{-1}$) when $h$ is a closed (resp.\ an open) embedding.
The functor $h_{!}$ (resp.\ $ h^!$) commutes with inverse (resp.\ direct) images on cartesian diagrams.
We will have to be careful in using these exceptional direct and inverse images because the tube turns closed into open and conversely (on analytic spaces).

We may now introduce two new functors (still writing $E_{X'}$ instead of $E_{V'}$ for the realization):

\begin{prop}The assignments:
\begin{enumerate}
\item
for an abelian sheaf $E$ on $T_{Y}$,
\[
(X', V') \mapsto (\gamma_{\dagger}E)_{X'} := ]\gamma'[_{!}E_{Y'},
\]
\item
for an abelian sheaf $E$ on $T$,
\[
(X', V') \mapsto (j^{\dagger}_{Y}E)_{X'} := ]\gamma'[_{!}]\gamma'[^{-1}E_{X'},
\]
\end{enumerate}
where $\gamma' : Y' \hookrightarrow X'$ denotes the inverse image of $\gamma$, both define a sheaf on $T$.
\end{prop}

\begin{proof}
Since $T$ is analytic, it follows from corollary 4.26 of \cite{LeStum17*} that the map $]\gamma'[: \,]Y'[_{V'} \hookrightarrow \,]X'[_{V'}$ is a locally closed embedding (of topological spaces) and  $]\gamma'[_{!}$ therefore makes sense.
Assume now that we are given a morphism
\[
(f,u) : (X'', V'') \to (X',V')
\]
over $T$, and denote by $g : Y'' \to Y'$ the map induced between the inverse images of $Y$ and by $\gamma'' : Y'' \hookrightarrow X''$ the inverse image of $\gamma$.
Then there exists a cartesian diagram of topological spaces
\[
\xymatrix{\,]Y''[\ar@{^{(}->}[rr]^{]\gamma''[}\ar[d]^{]g[} &&\,]X''[ \ar[d]^{]f[} \\\,]Y'[ \ar@{^{(}->}[rr]^{]\gamma'[} &&\,]X'[}
\]
and it follows that $]f[^{-1} \circ ]\gamma'[_{!} = ]\gamma''[_{!} \circ ]g[^{-1}$.
If $E$ is a sheaf on $T_{Y}$, then there exists therefore a canonical map
\[
]f[^{-1} ]\gamma'[_{!}E_{Y'} = ]\gamma''[_{!} ]g[^{-1}E_{Y'}  \to  ]\gamma''[_{!}E_{Y''}
\]
and the first assertion follows.
The second one is proved in the same way: if $E$ is a sheaf on $T$, we have
\[
\xymatrix{
]f[^{-1} ]\gamma'[_{!}]\gamma'[^{-1}E_{X'} \ar@{=}[d] \ar@{-->}[rr]&&  ]\gamma''[_{!}]\gamma''[^{-1}E_{X''}.\\ ]\gamma''[_{!} ]g[^{-1}]\gamma'[^{-1}E_{X'} \ar@{=}[rr] && ]\gamma''[_{!} ]\gamma''[^{-1} ]f[^{-1} E_{X'} \ar[u]}\qedhere
\]
\end{proof}

It would be interesting to investigate if there exists analogous functors for more general maps $f : Y \to X$.
Note however that there exists no such assertion with respect to the exceptional inverse images $]\gamma[^!$'s.

\begin{dfn} \label{newf}
The functor $\gamma_{\dagger}$ (resp.\ $j^{\dagger}_{Y}$) is the \emph{overconvergent direct image functor} (resp.\ the \emph{overconvergent sections functor}).
\end{dfn}

It is clear that, if $T' \to T$ is any morphism of overconvergent sites, then
\[
\gamma_\dagger E_{|T'_Y} \simeq (\gamma_\dagger E)_{|T'} \quad \mathrm{and} \quad j_\dagger E_{|T'} \simeq (j_\dagger E)_{|T'}.
\]

The functor $\gamma_{\dagger}$ should not be confused with $\gamma_{!}$ and the functor $j^\dagger_{Y}$ should not be confused with $\gamma_{!} \circ \gamma^{-1}$.
Be also careful that these functors do \emph{not} preserve injectives (doing so may lead to erroneous conclusions).
Note also that if $\alpha : U \hookrightarrow X$ is a formal \emph{open} embedding, then we have $\alpha_{*} = \alpha_{\dagger}$.
It follows that all the properties or overconvergent direct image will automatically apply to direct image for formal open embeddings.

\begin{dfn}
Assume $T$ is defined over some overconvergent space $(C,O)$.
Then, the \emph{absolute cohomology with support in $Y$} of a complex $E$ of abelian sheaves on $T$ is
$
 \mathrm Rp_{T/O*}j^{\dagger}_{Y}E.
$
\end{dfn}

We may then also write
\[
\mathcal H^k_Y(T/O, E) := \mathrm R^kp_{T/O*}j^{\dagger}_{Y}E.
\]

\begin{prop}
Both functors $\gamma_\dagger$ and $j_Y^\dagger$ are exact and preserve all colimits.
Moreover, $\gamma_\dagger$ is fully faithful and $\gamma_\dagger = j^\dagger_Y \circ \gamma_*$.
\end{prop}

\begin{proof}
Follows the analogous assertions about $]\gamma[_{!}$, $]\gamma[_{!} \circ ]\gamma[^{-1}$ and $]\gamma[_*$ (and from proposition \ref{gamet}).
\end{proof}

In particular, we see that the functor $\alpha_{*}$ is exact and fully faithful (and preserves colimits) if $\alpha : U \hookrightarrow X$ is a formal open embedding.

Let us also mention that our definition of $j_U^\dagger$ is compatible with Berthelot's (\cite{Berthelot96c*}, 2.1.1):

\begin{prop} \label{limBer}
Assume $\alpha : U \hookrightarrow X$ is an open embedding.
if $E$ is an abelian sheaf on $T$ and $(X', V')$ is an overconvergent space over $T$, then we have
\[
(j_U^\dagger E)_{X'} = \varinjlim j_*j^{-1} E_{X'}
\]
when $j$ runs through the inclusions of open neighborhoods of $\,]U'[_{V'}$ in $\,]X'[_{V'}$ (if $U'$ denotes the inverse image of $U$ in $X'$).
\end{prop}

\begin{proof}
Checked on stalks as in lemma \ref{FKlim}.
\end{proof}

\begin{xmp}
Assume that we are given a formal open embedding $\alpha : U \hookrightarrow X$ with closed complement $\beta : Z \hookrightarrow X$.
Then,
\begin{enumerate}
\item There exist two triple of adjoint functors $\alpha_{!}, \alpha^{-1}, \alpha_{*} : \widetilde T_{U} \to \widetilde T$ as well as $\beta_{!}, \beta^{-1}, \beta_{*}  : \widetilde T_{Z} \to \widetilde T$ with $\alpha_{*}$ exact and all direct images fully faithful (be careful that the exceptional direct image has to be modified a bit for abelian sheaves).
Moreover $\alpha^{-1} \circ \beta_{*} = 0$ and $\beta^{-1} \circ \alpha_{*} = 0$.
\item
There exists a couple of functors $\underline \Gamma_{U}$ and $\underline \Gamma_{Z}$ on abelian sheaves on $T$ which preserve injectives and colimits.
They give rise to exact triangles
\[
\mathrm R\underline \Gamma_{U}E  \to E \to \mathrm R\beta_{*}E_{|Z} \to \cdots \quad \mathrm{and} \quad \mathrm R\underline \Gamma_{Z}E  \to E \to \alpha_{*}E_{|U} \to \cdots
\]
\item There exists a  fully faithful functor $\beta_\dagger$ from abelian sheaves on $T_{Z}$ to abelian sheaves on $T$ (we have $\alpha_{\dagger} = \alpha_{*}$) and two functors $j^\dagger_{U}, j^\dagger_{Z}$ on abelian sheaves on $T$.
They are all exact an preserve all colimits (but not injectives).
\end{enumerate}
\end{xmp}

We can improve on the previous examples as follows:

\begin{prop} \label{dagab}
Assume that we are given a formal open embedding $\alpha : U \hookrightarrow X$ with closed complement $\beta : Z \hookrightarrow X$.
Then,
\begin{enumerate}
\item We have $\beta^{-1} \circ \beta_{\dagger} = \mathrm{Id}$ and $\alpha^{-1} \circ \beta_{\dagger} = 0$ .
\item If $E'$ is an abelian sheaf on $T_{Z}$, then there exists an exact sequence
\begin{align} \label{fondseq}
0 \to \beta_{\dagger}E' \to \beta_{*}E' \to j_{U}^\dagger\beta_{*}E' \to 0.
\end{align}
\item \label{dagap3}Any exact sequence
\[
0 \to \beta_{\dagger}E' \to E \to \alpha_{*}E'' \to 0
\]
 of abelian sheaves on $T$ is the pullback of the previous one through a unique map $\alpha_{*}E'' \to j_{U}^\dagger\beta_{*}E'$.
\end{enumerate}
\end{prop}

\begin{proof}
Let $(X',V')$ be an overconvergent space over $T$ and $\alpha', \beta'$ the inverse images of $\alpha, \beta$.
If $(X',V')$ lives over $T_{Z}$ (resp.\ $T_{U}$) then we have $\beta'= \mathrm{Id}$ (resp.\ $\beta'= 0$).
The first assertion follows.
If $E'$ is an abelian sheaf on $T_{Z}$ and $Z'$ denotes the inverse image of $Z$ in $X'$, then there exists an exact sequence
\[
0 \to ]\beta'[_{!}E'_{Z'} \to ]\beta[_{*}E'_{Z'} \to ]\alpha[_{*}]\alpha[^{-1}]\beta[_{*}E'_{Z'} \to 0
\]
(always true for a complement pair of open and closed subsets) and we get the second assertion.
Finally, the last assertion will follow from the second isomorphism in lemma \ref{exthom} below.
\end{proof}

As a consequence of the second assertion, if $E'$ is a complex of abelian sheaves on $T_{Z}$, then there exists a distinguished triangle
\[
\beta_{\dagger}E' \to \mathrm R\beta_{*}E' \to j_{U}^\dagger\mathrm R\beta_{*}E' \to \cdots.
\]

Let us also mention this other consequence of the proposition:

\begin{cor} \label{cartone}
If $E$ is an abelian sheaf on $T$, then the diagram
\[
\xymatrix{ E \ar[r] \ar[d] & j_U^\dagger E \ar[d] \\ \beta_{*}\beta^{-1} E \ar[r] &j_U^\dagger \beta_{*}\beta^{-1} E
}
\]
is cartesian.
\end{cor}

\begin{proof}
Follows from assertion \eqref{dagap3} of proposition \ref{dagab} which implies the existence of a morphism of exact sequences
\[
\xymatrix{ 0 \ar[r] & \beta_{\dagger}\beta^{-1}E \ar[r] \ar@{=}[d] & E \ar[r] \ar[d] &j_U^\dagger E \ar[d] \ar[r] & 0\\ 0 \ar[r] & \beta_{\dagger}\beta^{-1}E \ar[r] & \beta_{*}\beta^{-1} E \ar[r] & j_U^\dagger \beta_{*}\beta^{-1} E \ar[r] & 0.
} \qedhere
\]
\end{proof}

Note that there exists an analogous (homotopy) cartesian diagram in the derived category.

We used above the following result:

\begin{lem} \label{exthom}
Assume that we are given a formal open embedding $\alpha : U \hookrightarrow X$ with closed complement $\beta : Z \hookrightarrow X$.
If $E'$ (resp.\ $E''$) is an abelian sheaf on $T_Z$ (resp.\ $T_U$), then
\begin{enumerate}
\item 
$
\mathrm{Hom}(\alpha_{*}E'', \beta_{\dagger}E') = 0$,
\item
$
\mathrm{Hom}(\alpha_{*}E'', j_{U}^\dagger\beta_{*} E') \simeq \mathrm{Ext}(\alpha_{*}E'', \beta_{\dagger}E'),
$
\item 
$
\mathrm{RHom}(\beta_{\dagger}E',\alpha_{*}E'') = 0.
$
\end{enumerate}
\end{lem}

\begin{proof}
We know that $\beta^{-1}\alpha_{*}E'' = 0$ and it follows that
\[
\mathrm{Hom}(\alpha_{*}E'', \beta_{*}E') = \mathrm{Hom}(\beta^{-1}\alpha_{*}E'', E') = 0.
\]
We obtain the first assertion because $\beta_{\dagger}E' \subset \beta_{*}E'$.
Moreover, this being true for any sheaf, we actually have
\[
\mathrm R\mathrm{Hom}(\alpha_{*}E'', \mathrm R\beta_{*}E') = 0.
\]
Using the exact sequence \eqref{fondseq}, we obtain
\[
\mathrm{RHom}(\alpha_{*}E'', j_{U}^\dagger\mathrm R\beta_{*} E') \simeq \mathrm{RHom}(\alpha_{*}E'', \beta_{\dagger}E'[1]).
\]
The second assertion formally follows.
Finally, since $\alpha^{-1}\beta_{\dagger}E'' \subset \alpha^{-1}\beta_{*}E'' = 0$, we have
\[
\mathrm{Hom}(\beta_{\dagger}E'', \alpha_{*}E') = \mathrm{Hom}(\alpha^{-1}\beta_{\dagger}E'', E') = 0
\]
and we obtain the last assertion because the result holds for all abelian sheaves.
\end{proof}

The same proof provides a local version:
\begin{align*}
&\mathcal H\mathrm{om}(\alpha_{*}E'', \beta_{\dagger}E') = 0,
\\&
\mathcal H\mathrm{om}(\alpha_{*}E'', j_{U}^\dagger \beta_{*} E') \simeq \mathcal E\mathrm{xt}(\alpha_{*}E'', \beta_{\dagger}E'),
\\&
\mathrm R\mathcal H\mathrm{om}(\beta_{\dagger}E'', \alpha_{*}E') = 0.
\end{align*}

For the next statement, we need some preparation (the result will be used to do noetherian induction when we come to constructible crystals).

\begin{dfn} \label{fildef}
A finite filtration
\[
X = X_0 \supset X_1 \supset \cdots \supset X_d \supset X_{d+1} = \emptyset
\]
is said to be \emph{constructible} if, for $i=0, \ldots, d$, each $X_{i+1}$ is closed and nowhere dense in $X_i$.
\end{dfn}

Equivalently, it means that the open complement $U_i$ of $X_{i+1}$ in $X_i$ is everywhere dense in $X_i$: this is called a \emph{good stratification}\footnote{We will avoid this terminology because will have to consider later the notion of a stratification on a module.} in \cite[\href{https://stacks.math.columbia.edu/tag/09Y0}{Tag 09Y0}]{stacks-project}. 

We denote for each $0 \leq i < j \leq d$, by $\beta_{i,j} : X_j \hookrightarrow X_i$ the inclusion map.
If $I = (i_1, \ldots, i_k)$ is an ordered subset of $(0, \ldots, d)$, then we consider for an abelian sheaf $E$ on $T$,  the following barbarian expression:
\[
j_{I}^{\dagger} E := \beta_{0i_1*}j_{U_{i_1}}^\dagger \beta_{i_1i_2*}j_{U_{i_2}}^\dagger \ldots \beta_{i_{k-1}i_k*}j_{U_{i_k}}^\dagger \beta_{0i_k}^{-1} E.
\]
For example, we have
\[
j^\dagger_{(0)}E= j^\dagger_{U_0}E, \quad  j^\dagger_{(d)} =\beta_{0d*}\beta_{0d}^{-1}E \quad \mathrm{and} \quad j_{(0,d)}^{\dagger} E := j_{U_{0}}^\dagger \beta_{0d*}\beta_{0d}^{-1} E.
\]
If $I \subset J$, then there exists a canonical map $j_{I}^{\dagger} E \to j_{J}^{\dagger}  E$ and we denote by $j_{\bullet}^{\dagger}  E$ the corresponding diagram.
Note that there exists an augmentation map $E \to j_{\bullet}^{\dagger}  E$.

This diagram has a nicer description on a realization $(X',V')$: if we denote for each $i = 0, \ldots, d$, by $\gamma'_i : U'_i \hookrightarrow X'$ the inverse image of the inclusion map, then we may set for any sheaf on $\,]X'[_{V'}$,
\[
]\gamma'_{I}[_*]\gamma'_{I}[^{-1}\mathcal F := ]\gamma'_{i_1}[_*]\gamma'_{i_1}[^{-1} \ldots ]\gamma'_{i_k}[_*]\gamma'_{i_k}[^{-1} \mathcal F.
\]
It is not difficult to see that
\[
\left(j_{\bullet}^{\dagger}  E\right)_{V'} \simeq ]\gamma'_{\bullet}[_*]\gamma'_{\bullet}[^{-1} E_{V'}.
\]

We have as in proposition 2.1.1 of \cite{AbeLazda22}:

\begin{prop} \label{limstr}
If $E$ is an abelian sheaf on $T$ and $X$ is endowed with a constructible filtration, then
\[
E \simeq \varprojlim j_{\bullet}^{\dagger}  E.
\]
\end{prop}

\begin{proof}
We consider the induced filtration on $X_1$.
For $I = (i_1, \ldots, i_k) \subset (1, \ldots, d)$,  if $E'$ is an abelian sheaf on $T_{X_1}$, we will write
\[
j'^{\dagger}_{I} E' := \beta_{1i_2*}j_{U_{i_2}}^\dagger \beta_{i_2i_3*}j_{U_{i_3}}^\dagger \ldots \beta_{i_{k-1}i_k*}j_{U_{i_k}}^\dagger \beta_{1i_k}^{-1} E'.
\]
Note that, given $I = (i_1, i_2, \ldots, i_k) \subset (0, \ldots, d)$, we have the following alternatives
\[
j_{I}^{\dagger} E = \left\{ \begin{array}{lll}
\beta_{01*}j'^{\dagger}_{I} \beta_{01}^{-1}E& \mathrm{if}\ I \subset (1, \ldots, d)
\\
\\  j^\dagger_{U_0} \beta_{01*}j'^{\dagger}_{I'} \beta_{01}^{-1}E &\mathrm{if}\   I = (0) \cup I'.
\end{array}\right.
\]
Now, we may assume by induction that if $E'$ is an abelian sheaf on $T_{X_1}$, then
\[
E' \simeq \varprojlim j'^{\dagger}_{\bullet}  E'.
\]
As a consequence, we get
\[
\beta_{01*}\beta_{01}^{-1} E \simeq \beta_{01*} \varprojlim j'^{\dagger}_{\bullet} \beta_{01}^{-1}E \simeq \varprojlim \beta_{01*}j'^{\dagger}_{\bullet} \beta_{01}^{-1}E
\]
and
\[
j^\dagger_{U_0} \beta_{01*}\beta_{01}^{-1} E \simeq  j^\dagger_{U_0}  \varprojlim \beta_{01*}j'^{\dagger}_{\bullet} \beta_{01}^{-1}E \simeq \varprojlim  j^\dagger_{U_0}   \beta_{01*}j'^{\dagger}_{\bullet} \beta_{01}^{-1}E.
\]
On the other hand, we know from corollary \ref{cartone} that the diagram
\[
\xymatrix{ E \ar[r] \ar[d] &  j^\dagger_{U_0} E \ar[d] \\ \beta_{01*}\beta_{01}^{-1} E \ar[r] & j^\dagger_{U_0} \beta_{01*}\beta_{01}^{-1} E
}
\]
is cartesian.
Putting all this together provides the formula.
\end{proof}

There exists again an analog for a complex $E$ of abelian sheaves on $T$: we define
\[
\mathrm Rj_{I}^{\dagger} E := \mathrm R\beta_{0i_1*}j_{U_{i_1}}^\dagger \mathrm R\beta_{i_1i_2*}j_{U_{i_2}}^\dagger \ldots \mathrm R\beta_{i_{k-1}i_k*}j_{U_{i_k}}^\dagger \mathrm R\beta_{0i_k}^{-1} E
\]
(this may not be the derived functor of $j_{I}^{\dagger} E$ because $j^\dagger_{U_i}$ does not preserve injectives).
Then we have
\[
E \simeq \mathrm R\varprojlim \mathrm Rj_{\bullet}^{\dagger}  E
\]
(see the proof of proposition 2.1.1 in \cite{AbeLazda22}).

\section{Crystals}

We want now to define the notion of a crystal on an overconvergent site which is the natural coefficient for cohomology.
Again, we will finish with a fine study of the behavior of a crystal along a formal embedding.

\subsection{The structural sheaf}

An overconvergent site is naturally a ringed site as we shall now explain.

Recall that if $(X,V)$ is an overconvergent space, then we denote by $i_{X} : \,]X[_{V} \hookrightarrow V$ the inclusion map.
We endow the tube $\,]X[_V$ with the sheaf of rings
$
\mathcal O_{V}^\dagger := i_{X}^{-1}\mathcal O_{V}.
$
We will sometimes write $\mathcal O_{\,]X[_{V}}^\dagger$ in order to specify the role of $X$ but we may as well write $\mathcal O^\dagger_X$ when this is $V$ that plays a minor role.
When we consider an $\mathcal O_V^\dagger$-module, we may say $\mathcal O_V^\dagger$-module \emph{on $X$} in order to specify the formal scheme.
Note that the ring $\mathcal O_V^\dagger$ is identical to the structural ring $\mathcal O_{\,]X[_{V}^\dagger}$ of the germ $\,]X[_V^\dagger$ of the prepseudo-adic space $(V, \,]X[_V)$.

We let $(f,u) : (X',V') \to (X,V)$ be a morphism of overconvergent spaces.
We promote the map $]f[_{u} : \,]X'[_{V'} \to \,]X[_{V}$ to a morphism of ringed spaces
\begin{align} \label{basmor}
]f[_{u} : (\,]X'[_{V'},\mathcal O_{V'}^\dagger)   \to (\,]X[_{V}, \mathcal O_{V}^\dagger)
\end{align}
using the canonical map
\begin{align} \label{cantr}
]f[_{u}^{-1}\mathcal O_{V}^\dagger =]f[_{u}^{-1} i_{X}^{-1}\mathcal O_{V}= i_{X'}^{-1} u^{-1}\mathcal O_{V}  \to i_{X'}^{-1} \mathcal O_{V'} = \mathcal O_{V'}^\dagger.
\end{align}
In order to avoid any confusion, we will denote by $]f[_{u}^\dagger$ the inverse image functor (instead of $]f[_{u}^*$).
We may also write $]f[^\dagger$ (resp.\ $u^\dagger$) in the case $u = \mathrm{Id}_V$ (resp.\ $f = \mathrm{Id}_X$). 
Alternatively, we could again consider the germ $\,]X[_V^\dagger$ that comes with its natural structural sheaf $\mathcal O_{\,]X[^\dagger_V}$ and $]f[_u$ is then the natural morphism of ringed germs.

There exists another description of the inverse image functor that may be sometimes more convenient to use:

\begin{lem}
If  $\mathcal F$ is an $\mathcal O_{V}^\dagger$-module, then
\[
]f[_{u}^\dagger \mathcal F =i_{X'}^{-1}u^*i_{X*}\mathcal F.
\]
\end{lem}

\begin{proof} We have
\begin{align*}
]f[_{u}^\dagger \mathcal F &=  \mathcal O_{V'}^\dagger \otimes_{]f[_{u}^{-1}\mathcal O_{V}^\dagger } ]f[_{u}^{-1} \mathcal F
\\&= i_{X'}^{-1}\mathcal O_{V'} \otimes_{]f[_{u}^{-1}i_{X}^{-1}\mathcal O_{V} }  ]f[_{u}^{-1} i_{X}^{-1}i_{X*}\mathcal F
\\&=  i_{X'}^{-1}\mathcal O_{V'} \otimes_{i_{X'}^{-1} u^{-1}\mathcal O_{V}}i_{X'}^{-1}u^{-1}i_{X*}\mathcal F
\\&=  i_{X'}^{-1} (\mathcal O_{V'} \otimes_{u^{-1}\mathcal O_{V}}u^{-1}i_{X*}\mathcal F)
\\& = i_{X'}^{-1}u^*i_{X*}\mathcal F. \qedhere
\end{align*}
\end{proof}

The same computations also show that, more generally, if $\mathcal F'$ is an $\mathcal O_V$-module such that $i_X^{-1}\mathcal F' = \mathcal F$, then $]f[_{u}^\dagger \mathcal F =i_{X'}^{-1}u^*\mathcal F'$.
For example, $]f[_{u}^\dagger \mathcal F = i_{X'}^{-1}u^*i_{X!}\mathcal F$ when $V$ is analytic (this last condition is made to insure that exceptional direct image does exist).

As a particular case, note that if $(X,V)$ is an overconvergent space, $\gamma : Y \hookrightarrow X$ is a formal embedding and $\mathcal F$ is an $ \mathcal O_{V}^\dagger$-module on $X$, then, for the morphism of overconvergent spaces $(Y, V) \hookrightarrow (X,V)$, we have
\[
]\gamma[^\dagger \mathcal F = i_{Y}^{-1}i_{X*}\mathcal F =  ]\gamma[^{-1} i_{X}^{-1}i_{X*} \mathcal F =  ]\gamma[^{-1} \mathcal F.
\]

For further use, let us mention the following (recall that a morphism of ringed sites is said to be flat when inverse image is exact):

\begin{lem} \label{flatan}
If $(X,V)$ analytic and $u$ is flat in the neighborhood of $X$, then the morphism of ringed spaces $]f[_{u}$ also is flat.
\end{lem}

\begin{proof}
Since $(X,V)$ is analytic, $\,]X[_{V}$ is locally closed in $V$ and we may therefore assume that it is closed.
Then, $i_{X'}^{-1}$, $u^*$ and $i_{X*}$ are all exact functors.
\end{proof}

We can now make the following fundamental definition:

\begin{dfn}
The \emph{structural sheaf}  $\mathcal O^\dagger$ on the absolute overconvergent site $\mathbf{Ad}^\dagger$ is defined by
\[
\Gamma((X,V), \mathcal O^\dagger) := \Gamma(\,]X[_{V}, \mathcal O_{V}^\dagger)
\]
with obvious transition maps.
\end{dfn}

It follows from proposition \ref{realsh} that $\mathcal O^\dagger$ is a sheaf of rings with realization $\mathcal O_{V}^\dagger$ on $(X,V)$.
Any overconvergent site $T$ inherits a structural sheaf $\mathcal O^\dagger_{T}  := \mathcal O^\dagger_{|T}$ and any morphism of overconvergent sites $f : T' \to T$ is automatically a flat morphism of ringed sites since $f^{-1}\mathcal O_{T}^\dagger = \mathcal O_{T'}^\dagger$.
When $T = (X/O)^\dagger$ or $T = (X,V/O)^\dagger$, we will simply write $\mathcal O^\dagger_{X/O}$ and $\mathcal O^\dagger_{X,V/O}$ respectively (and not repeat the $\dagger$).

\subsection{Modules}

We need to specialize our former results on abelian sheaves on overconvergent sites to the case of modules on the corresponding ringed site.

If $(X, V)$ is an overconvergent space, we can promote $\varphi_{V}$ (resp.\ $\psi_V$) to a morphism of ringed sites (resp.\ ringed topoi) and obtain the following analog to lemma \ref{adjret}:

\begin{lem}
If $(X, V)$ is an overconvergent space, then there exists a morphism of ringed sites
\[
\varphi_{V} : \left((X,V), \mathcal O^\dagger_{(X,V)}\right)  \to \left(\,]X[_{V}, \mathcal O_{V}^\dagger\right)
\]
and a morphism of ringed topoi
\[
\psi_V : \left(\widetilde{\,]X[_{V}}, \mathcal O_{V}^\dagger\right) \to \left(\widetilde{(X,V)}, \mathcal O^\dagger_{(X,V)}\right)
\]
such that $\widetilde \varphi_V$ is an adjoint retraction to $\psi_{V}$.
\end{lem}

\begin{proof}
By definition, we have $
\psi_{V}^{-1}\mathcal O^\dagger_{(X,V)} =  \mathcal O_{V}^\dagger.
$
The assertion formally follows.
\end{proof}

As a consequence of the proposition, we see that there exists again a sequence of adjoints
\[
\varphi_{V}^{*}, \quad \varphi_{V*} = \psi_{V}^{-1}, \quad \psi_{V*}.
\]
Moreover, $\varphi_{V}^{*}$ is fully faithful, or equivalently, $\varphi_{V*} \circ \varphi_{V}^{*} = \mathrm{Id}$ (but the inverse image $\varphi_{V}^{*}$ is not exact anymore).

The commutative diagram of lemma \ref{comphi} is actually a commutative diagrams of \emph{ringed sites} so that we always have
\[
\varphi_{V'}^* \circ ]f[_u^\dagger = (f,u)^{-1} \circ \varphi_V^*.
\]

There  also exists an analog to lemma \ref{pulbk}:

\begin{lem} \label{crisfi2}
If $(f,u) : (X',V') \to (X,V)$ is a morphism of overconvergent spaces and $\mathcal F$ is an $\mathcal O_{V}^\dagger$-module, then 
\[
(\varphi_{V}^{*}\mathcal F)_{V'} = ]f[_{u}^{\dagger}\mathcal F.
\]
\end{lem}

\begin{proof}
Follows as lemma \ref{pulbk} from lemma \ref{comphi} (in which the morphisms are seen as morphisms of ringed sites):
\[
(\varphi_{V}^{*}\mathcal F)_{V'} = \varphi_{V'*}(f,u)^{-1}\varphi_{V}^{*}\mathcal F = \varphi_{V'*}\varphi_{V'}^*]f[_{u}^{\dagger}\mathcal F= ]f[_{u}^{\dagger}\mathcal F. \qedhere
\]
\end{proof}

If $T$ is an overconvergent site, $E$ is a an $\mathcal O^\dagger_{T}$-module (we will simply say \emph{a module on $T$}) and $(f,u) : (X',V') \to (X,V)$ is a morphism of overconvergent spaces over $T$, then the transition map $]f[_{u}^{-1} E_{V} \to E_{V'}$ extends uniquely to an $\mathcal O_{V}^\dagger$-linear transition map $]f[_{u}^{\dagger} E_{V} \to E_{V'}$.
Then there exists also an analog to lemma \ref{alttr}:

\begin{lem} \label{alttr2}
If $T$ is an overconvergent site and $(f,u) : (X',V') \to (X,V)$ is a morphism of overconvergent spaces over $T$, then the linear transition map for a module $E$ on $T$ is the realization on $V'$ of the adjunction map $\varphi_{V}^{*}\varphi_{V*}E \to E$.
\end{lem}

\begin{proof}
Follows from lemma \ref{crisfi2}.
\end{proof}

Finally, there also exists an analog to proposition \ref{realsh}:

\begin{prop} \label{realsh2}
If $T$ is an overconvergent site, then there exists an equivalence between the category of modules $E$ on $T$ and the category made of collections of $\mathcal O_{V}^\dagger$-modules $E_{V}$ endowed with $\mathcal O_{V'}^\dagger$-linear transition maps $\eta^\dagger : ]f[_{u}^{\dagger}E_{V} \to E_{V'}$ when $(f,u) : (X', V') \to (X, V)$ is a morphism over $T$, satisfying the usual compatibility conditions.
\end{prop}

\begin{proof}
Follows from proposition \ref{realsh} since the \emph{linear} transition maps $]f[_{u}^{\dagger}E_{V} \to E_{V'}$ correspond bijectively to the (automatically semi-linear) transition maps $]f[_{u}^{-1}E_{V} \to E_{V'}$.
\end{proof}

Again, there exists a more formal interpretation.
Let us denote by $\mathrm{Mod}^\dagger$ the fibered category\footnote{This is actually a stack.} over $\mathbf {Ad}^\dagger$ whose fiber over $(X,V)$ is opposite to the category of $\mathcal O_V^\dagger$-modules.
If $(f,u) : (Y,W) \to (X,V)$ is a morphism and $\mathcal F$ (resp. $\mathcal G$) is an $\mathcal O_V^\dagger$-module (resp. an $\mathcal O_W^\dagger$-module), then a morphism $\mathcal G$ to $\mathcal F$ over $(f,u)$ is a morphism of sheaves $]f[_u^{\dagger}\mathcal F \to \mathcal G$.
If $T$ is an overconvergent site, then proposition \ref{realsh2} may be rephrased by saying that there exists an equivalence
\[
\mathrm{Mod}(T) \simeq \mathrm{Hom}_{\mathbf{Ad}^\dagger}(T, \mathrm{Mod}^\dagger)
\]
or even, if we pull back $\mathrm{Mod}^\dagger$ to a fibered category $\mathrm{Mod}_T^\dagger$ over $T$, an equivalence
\[
\mathrm{Mod}(T) \simeq \mathrm{Hom}_T(T, \mathrm{Mod}_T^\dagger)
\]
between the category of $\mathcal O^\dagger_T$-modules and the category of sections of $\mathrm{Mod}^\dagger_T$.

Recall that a module $E$ on $T$ is said to be \emph{flat} if the functor
\[
E' \mapsto E \otimes_{\mathcal O_{T}^\dagger} E'
\]
is exact on modules.
Then we have the following elementary properties:

\begin{prop} \label{flatmod}
\begin{enumerate}
\item If $T$ is an overconvergent site, then a module $E$ on $T$ is flat if and only $E_{V}$ is a flat $\mathcal O_{V}^\dagger$-modules for all realizations of $E$.
\item If $f : T' \to T$ is a morphism of overconvergent sites, then $f^{-1}$ preserves flatness.
\item \label{flatmod3}If $\gamma : Y \hookrightarrow X$ is a formal embedding and $T$ is an analytic overconvergent site over $X^\dagger$, then $\gamma_{\dagger}$ preserves flatness.
\end{enumerate}
\end{prop}

\begin{proof}
Since the points $x \in \,]X[_{V}$ where $(X,V)$ is an overconvergent space over $T$ form a conservative family of points, we know that a module $E$ on $T$ is flat if and only if all stalks $E_{V,x}$ are flat.
The first assertion follows.
The second one is completely formal.
The last assertion follows from the first one because, if $(X', V')$ is an overconvergent space over $T$ and $\gamma' : Y' \hookrightarrow X'$ denote the inverse image of $\gamma$, then we have $(\gamma_{\dagger} E)_{X'} = ]\gamma'[_{!}E_{Y'}$ and  we know that $]\gamma'[_{!}$ preserves flatness (as any exceptional direct image does).
\end{proof}

It follows from the last assertion of the proposition that $\alpha_{*}$ preserves flatness when $\alpha : U \hookrightarrow X$ is a formal open embedding (in the analytic case).

For the sake of completeness, let us also state the following:

\begin{prop} \label{inthom}
Let $T$ be an overconvergent site and $E, E'$ two modules on $T$.
Then for any overconvergent space $(X,V)$ over $T$, we have
 \[
(E \otimes_{\mathcal O_{T}^\dagger} E')_{V} \simeq E_{V} \otimes_{\mathcal O_{V}^\dagger} E'_{V}.
\]
and
\[
\mathcal H\mathrm{om}_{\mathcal O_{T}^\dagger}(E, E')_{V} \simeq \mathcal H\mathrm{om}_{\mathcal O_{V}^\dagger}(E_{V}, E'_{V}).
\]
\end{prop}

\begin{proof}
Completely formal since realization is essentially a pullback (see \cite{SGA4}, IV,12 for example).
\end{proof}

\subsection{Cohomological base change}

In order to study preservation of crystals under direct image, it will be essential to investigate cohomological base change.
We let $T$ be an overconvergent site over an overconvergent space $(C,O)$ and $E$ be a complex of modules on $T$.
Recall that we denote by $p_{T/O} : \widetilde T \to \widetilde{]C[}_{O}$ the canonical map.

\begin{dfn}
Let $(f,u) : (C',O') \to (C,O)$ be a morphism of overconvergent spaces.
We will say that $E$ \emph{commutes with (cohomological) base change to $(C',O')$} if
\[
\forall k \in \mathbb N, \quad ]f[_{u}^\dagger \mathrm R^kp_{T/O*}E \simeq \mathrm R^kp_{T_{O'}/O'*}E.
\]
We say that $E$ \emph{commutes with (cohomological) base change} if this holds for all such $(C',O')$.
\end{dfn}

Alternatively, we could consider a \emph{derived base change} condition:
\[
\mathrm L]f[_{u}^\dagger \mathrm Rp_{T/O*}E \simeq \mathrm Rp_{T'/O'*}E
\]
(which is equivalent to ordinary base change when cohomology is flat - or when $u$ is flat if $(C,O)$ is analytic).

Clearly, these properties are of local nature on $(C,O)$.

Cohomological base change is usually hard to prove.
Let us however mention the following:

\begin{lem} \label{transba}
\begin{enumerate}
\item
Assume $E$ satisfies base change to $(C',O')$ and let $(g,v) : (C'' , O'') \to (C', O')$ be another morphism of overconvergent spaces, then $E$ satisfies base change to $(C'', O'')$ if and only if $E_{|T_{O'}}$ satisfies base change to $(C'', O'')$.
\item If $E$ satisfies base change, and $(C',O') \to (C,O)$ is any morphism of overconvergent spaces, then $E_{|T_{O'}}$ also satisfies base change.
\end{enumerate}
\end{lem}

\begin{proof} The second assertion follows from the first one which itself results from the commutativity of the diagram
\[
\xymatrix{
]f\circ g[^\dagger \mathrm R^kp_{T/O*}E \ar[rr]^\simeq \ar[dr] && ]g[^\dagger\mathrm R^kp_{T_{O'}/O'*} E_{|T_{O'}'} \ar[dl]
\\
&\mathrm R^kp_{T_{O''}/O''*} E_{|T_{O''}}.
}  \qedhere
\]
\end{proof}

\subsection{Crystals}

At last, we arrive at the definition of a crystal.


We let $T$ be an overconvergent site.

\begin{dfn}
A module $E$ on $T$ is a \emph{crystal} if all transition maps are isomorphisms
\[
\eta^{-1} : ]f[_{u}^{\dagger}E_{V} \simeq E_{V'}.
\]
\end{dfn}

One may also say \emph{isocrystal} when $T$ is an analytic overconvergent site (which is the only case where one can really say something).
The cohomology of an isocrystal (with respect to some structural map) is also called \emph{rigid cohomology}.

We will denote by $\mathrm{Cris}(T) \subset \mathrm{Mod}(T)$ the full subcategory of crystals on $T$.
There exists an alternative description of this category of crystal as a category of \emph{cartesian} sections:
\[
\mathrm{Cris}(T) \simeq \mathrm{Hom}_{\mathbb Fib(T)}(T, \mathrm{Mod}^\dagger_T)
\]
When $T$ is fibered in groupoids, we also have
\[
\mathrm{Cris}(T) \simeq \mathrm{Hom}_{\mathbb Fib(\mathbf {Ad}^\dagger)}(T, \mathrm{Mod}^\dagger)
\]
(but not otherwise).

One may also define a \emph{derived crystal} as a complex $E$ of modules on $T$ such that all \emph{derived} transitions maps are isomorphisms $\mathrm L]f[_{u}^{\dagger}E_{V} \simeq E_{V'}$.
For flat modules, both notions coincide.

\begin{xmp}
\begin{enumerate}
\item 
A module $E$ on $T$ is quasi-coherent (resp.\ finitely presented, resp. locally finite free) if and only if it is a \emph{crystal} with quasi-coherent (resp.\ finitely presented, resp. locally finite free) realizations.
We use the terminology of \cite[\href{https://stacks.math.columbia.edu/tag/03DL}{Tag 03DL}]{stacks-project}.
Note however that $\mathcal O_T^\dagger$ is \emph{not} a coherent ring in general.
\item
Let $\mathcal V$ be discrete valuation ring with residue field $k$ and fraction field $K$ of characteristic zero.
Let $S$ be a formal scheme which is locally topologically of finite type over $\mathcal V$.
Let $X$ be a variety over $S_k$.
As a consequence of theorem \ref{crismic} below, the category of overconvergent isocrystals on $X/S$ (definition 2.3.6 of \cite{Berthelot96c*}) is equivalent to the category of finitely presented modules on $X/S^{\mathrm{an}}$ (or equivalently finitely presented crystals).
\end{enumerate}
\end{xmp}

Clearly, a module $E$ on $T$ is a crystal if and only if $E_{|(X,V)}$ is a crystal for all overconvergent spaces $(X,V)$ over $T$.
The next result will therefore be quite useful:

\begin{lem} \label{eqrep}
If $(X,V)$ is an overconvergent space, then $\varphi_{V*}$ and $\varphi_{V}^*$ induce an equivalence
\[
\mathrm{Cris}(X,V) \simeq \mathrm{Mod}(\mathcal O_V^\dagger).
\]
In particular, an $\mathcal O^\dagger_{(X,V)}$-module $E$ is a crystal if and only if $\varphi_{V}^*\varphi_{V*}E \simeq E$.
\end{lem}

\begin{proof}
We already know that $\varphi^*_V$ is fully faithful.
Now, a module $E$ on $T$ belongs to the essential image of $\varphi^*_V$ if and only if the adjunction map $\varphi_{V}^*\varphi_{V*}E \to E$ is bijective.
But we know from lemma \ref{alttr2} that the realization of this adjunction map on some $(X', V')$ over $(X,V)$ is exactly the transition map  $]f[_{u}^{\dagger}E_{V} \to E_{V'}$.
And this is an isomorphism if and only if $E$ is a crystal.
\end{proof}

As a consequence, we see that $\mathrm{Cris}(X,V)$ is an abelian category with tensor product and internal Hom that has enough injectives.
Be careful however that  $\mathrm{Cris}(X,V)$ is \emph{not} a weak Serre subcategory of $\mathrm{Mod}(X,V)$: the inclusion functor is only right exact.
In general, we have the following:

\begin{prop} \label{catcrys}
$\mathrm{Cris}(T)$ is an additive subcategory of $\mathrm{Mod}(T)$ which is stable under colimit, extension and tensor product.
\end{prop}

\begin{proof}
We may assume that $T = (X,V)$ is representable in which case the assertion follows from lemma \ref{eqrep} since  $\varphi_{V}^*$ preserves colimits, extensions and tensor products.
\end{proof}

We want to stress out the fact that $\mathrm{Cris}(T)$ is \emph{not} an abelian subcategory (and in particular not a weak Serre subcategory) of the category of all modules on $T$ (and that there is no internal Hom either in general).
In particular, the extension assertion means that any extension of crystals in the category of modules is necessarily a crystal but there might exists other extensions in the category of crystals (which do not come from an extension in the category of all modules).
Nevertheless, we have the following:

\begin{lem} \label{flatex}
A sequence $0 \to E' \to E \to E'' \to 0$ in $\mathrm{Cris}(T)$ is always exact when it is exact in $\mathrm{Mod}(T)$ and the converse holds when $E''$ is flat.
\end{lem}

\begin{proof}
Again, we may assume that $T = (X,V)$ is representable and the statement then follows from the formal properties of $\varphi_{V}^*$.
\end{proof}

Usually, when we call a sequence \emph{exact}, we always mean in the strong sense, that is in $\mathrm{Mod}(T)$ but we will try to make it clear when some ambiguity might persist.

\begin{xmp}
We consider the overconvergent space $(\mathbb A_{\mathbb F_p} \hookrightarrow \mathbb A \leftarrow \mathbb D_{\mathbb Q_p})$ and denote by $t$ the coordinate.
Realization provides an equivalence between crystals on $(\mathbb A_{\mathbb F_p} , \mathbb D_{\mathbb Q_p})$ and $\mathcal O_{\mathbb D_{\mathbb Q_p}}$-modules.
Multiplication by $t$ on $\mathcal O_{\mathbb D_{\mathbb Q_p}}$ corresponds to a morphism of crystals $\mathcal O^\dagger_{(\mathbb A_{\mathbb F_p} , \mathbb D_{\mathbb Q_p})} \overset t \to \mathcal O^\dagger_{(\mathbb A_{\mathbb F_p} , \mathbb D_{\mathbb Q_p})}$ whose kernel (in the category of crystals) is $0$ (because multiplication by $t$ is injective in $\mathcal O_{\mathbb D_{\mathbb Q_p}}$).
We consider now the overconvergent space $(\mathrm{Spec}(\mathbb F_p) \hookrightarrow \mathrm{Spec}(\mathbb Z) \leftarrow \mathrm{Spa}(\mathbb Q_p))$ as living over $(\mathbb A_{\mathbb F_p} , \mathbb D_{\mathbb Q_p})$ via the zero section.
On this space, our morphism of crystals induces the zero map $\mathbb Q_p \overset 0 \to \mathbb Q_p$ (by definition because the zero section is defined by $t=0$) whose kernel is $\mathbb Q_p$ which is not zero.
This shows that a sequence might be exact in $\mathrm{Cris}(T)$ but not in $\mathrm{Mod}(\mathcal O_{T}^\dagger)$.
\end{xmp}

\begin{lem} \label{loccrys}
Let $f:T' \to T$ be a morphism of overconvergent sites and $E \in \mathrm{Mod}(T)$.
If $E \in \mathrm{Cris}(T)$, then $E_{|T'} \in \mathrm{Cris}(T')$ and the converse holds when $f$ is local epimorphism.
\end{lem}

\begin{proof}
Only the converse implication needs a proof and it is sufficient to consider the case where $T = (X,V)$.
The question being local on $V$, we may assume that $T' \to (X,V)$ has a section and are done.
\end{proof}

As an obvious consequence, we see that the same property holds for a local isomorphism even if the definition of a crystal is site-theoretic (and not topos-theoretic).
Note also that, since $\mathrm{Mod^\dagger}$ is a stack on $\mathbf{Ad}^\dagger$, the functor 
\[
f^{-1} : \mathrm{Cris}(T') \to \mathrm{Cris}(T)
\] 
is fully faithful (resp.\ an equivalence) when $T' \to T$ is a local epimorphism (resp.\ a local isomorphism).

\begin{lem} \label{homcrys}
If $E, E' \in \mathrm{Cris}(T)$, then
\[
E \otimes_{\mathcal O_{T}^\dagger} E' \in \mathrm{Cris}(T) \quad (\mathrm{resp.} \
\mathcal H\mathrm{om}_{\mathcal O_{T}^\dagger}(E, E') \in \mathrm{Cris}(T)
\]
when $E$ is \emph{finitely presented}).
\end{lem}

\begin{proof}
Follows from the fact that $]f[_u^\dagger$ will commute with tensor product (resp.\ internal hom as long as the first argument is finitely presented).
\end{proof}

Finally, our main theorem of \cite{LeStum17*} has the following implication:

\begin{thm}\label{geomcrys}
If $V$ is a geometric materialization of a formal scheme $X$ over an analytic space $O$, then there exists an equivalence of categories
\[
\mathrm{Cris}(X/O)^\dagger \simeq \mathrm{Cris}(X,V/O)^\dagger.
\]
\end{thm}

\begin{proof}
Follows from lemma \ref{loccrys} and theorem \ref{strfib}.
\end{proof}

This will prove to be the fundamental argument that will allow us to give a down to earth description of crystals.

\subsection{Cohomologically crystalline complexes}

It is usually very hard to study preservation of crystals under direct image and this is a very important issue in the theory (see for example \cite{Lazda16}).
We can however reduce the question to base change in cohomology (which is unfortunately not easy to deal with either).

Since we are interested in computing cohomology, we introduce the following:

\begin{dfn}
Let $f:T' \to T$ be a morphism of overconvergent sites.
A complex $E$ of $\mathcal O_{T'}$-modules is called \emph{(cohomologically) $f$-crystalline} if $\mathrm R^kf_{*}E$ is a crystal for all $k \in \mathbb N$.
\end{dfn}

We may also say that $E$ is \emph{derived $f$-crystalline} if $\mathrm Rf_{*}E$ is a derived crystal.
When all $\mathrm R^kf_{*}E$ are flat, both definitions are equivalent.

It is very hard in general to show that a complex $E$ is $f$-crystalline.

\begin{xmp}
\begin{enumerate}
\item
It follows from proposition \ref{alphacr} below that if $\alpha : U \hookrightarrow X$ is a formal open embedding and $T$ is an analytic overconvergent site over $X^\dagger$, then any crystal on $T_{U}$ is $\alpha$-crystalline.
\item This is not true anymore for closed embeddings:
if $\beta : \{0\} \hookrightarrow \mathbb A$ is the inclusion of the origin, then $\mathcal O^\dagger_{\{0\}}$ is \emph{not} $\beta$-crystalline.
\item If $T$ is an overconvergent site, then there exists a morphism of overconvergent sites $f : \mathbb A^{n}_T \to T$ and one can show that $\mathcal O^\dagger_{\mathbb A^{n}_T}$ is $f$-crystalline when $T$ is analytic and defined over $\mathbb Q$.
More precisely, we will show later, as a consequence of corollary \ref{cohdR} below, that in this case, $\mathrm Rf_*\mathcal O^\dagger_{\mathbb A^{n}_T} = \mathcal O^\dagger_{T}$.
\item Let $(C \hookrightarrow S \leftarrow O)$ be an analytic convergent space over $\mathbb Q$.
Let $X$ be a smooth scheme over $C$ and $f : Y \to X$ a finite étale morphism.
We will show in theorem \ref{finet} that any finitely presented crystal on $(Y/O)^\dagger$ is $f$-crystalline.
\item One can conjecture that the previous results extends to a proper smooth morphism.
This seems out of reach at this point.
\end{enumerate}
\end{xmp}

Being $f$-crystalline is local on the base:

\begin{lem} \label{locry}
Assume that we are given a $2$-cartesian diagram of overconvergent sites
\[
\xymatrix{T''' \ar[r]^{f'} \ar[d]^{g'} &T'' \ar[d]^{g} \\ T' \ar[r]^{f}& T},
\]
and a complex $E$ of modules on $T'$.
If $E$ is $f$-crystalline, then $g'^{-1}E$ is $f'$-crystalline and the converse holds when $g$ is a local epimorphism.\end{lem}

\begin{proof}
We have
\[
\mathrm R^kf'_{*} g'^{-1} E = g^{-1} \mathrm R^kf_{*}E
\]
because $g^{-1}$ is exact and preserves injective and crystals.
Moreover, if $g$ is a local epimorphism, then we showed in lemma \ref{loccrys} that $g^{-1}$ detects crystals.
\end{proof}

Assume that we are given a sequence of morphisms $T'' \overset f \to T' \overset g \to T$.
Then, it directly follows from the definition that if $E$ is derived $f$-crystalline, then $\mathrm Rf_{*} E$ is derived $g$-crystalline if and only if $E$ est derived $(g \circ f)$-crystalline.
Also in a distinguished triangle $E' \to E \to E'' \to \cdots$, if two of them are derived $f$-crystalline, so is the third.
There does not exist a non-derived analog of these statements and it is necessary to rely on spectral sequences in practice.

Being $f$-crystalline is closely related to satisfying cohomological base change:

\begin{prop} \label{crsiba}
Let $f:T' \to T$ be a morphism of overconvergent sites.
A complex $E$ on $T'$ is $f$-crystalline if and only if for all overconvergent space $(X, V)$ over $T$, the complex $E_{|T'_{V}}$ satisfies (cohomological) base change over $V$.
\end{prop}

\begin{proof}
We consider the property for $E$ of being $f$-crystalline.
It means that
\[
]f[_{u}^\dagger (\mathrm R^kf_{*}E)_{V} \simeq (\mathrm R^kf_{*}E)_{V'}
\]
whenever $(f,u) : (X',V') \to (X,V)$ is a morphism of overconvergent spaces.
Using proposition \ref{bicom}, this condition also reads
\[
]f[_{u}^\dagger \mathrm R^kp_{T'_{V}/V*}E_{|T'_{V}} \simeq \mathrm R^kp_{T'_{V'}/V'*}E_{|T'_{V'}}
\]
This exactly means that $E_{|T'_{V}}$ satisfies base change to $(X',V')$.
\end{proof}

Note that there exists again an obvious derived analog to this statement.

\begin{xmp}
\begin{enumerate}
\item
If $f:T \to (X,V)$ is a morphism of overconvergent sites, then a complex $E$ on $T$ is $f$-crystalline if and only if the complex $E$ satisfies (cohomological) base change over $V$.
This follows from lemma \ref{transba}.
\item
If $f:T \to (X,V/O)^\dagger$ is a morphism of overconvergent sites, then a complex $E$ on $T$ is $f$-crystalline if and only if the complex $E_{|T_V/V}$ satisfies (cohomological) base change over $V$.
This follows from lemma \ref{locry}.
Moreover, theorem \ref{strfib} implies that the same applies to $f:T \to (X/O)^\dagger$ when $V$ is a geometric materialization.
\item Let $(X,V) \to (C,O)$ be a morphism of overconvergent spaces and $f : Y \to X$ be morphism of formal schemes.
We consider the induced morphism of overconvergent sites $f : (Y/O)^\dagger \to (X,V/O)^\dagger$ (using the same letter $f$).
Then a crystal $E$ on $(Y/O)^\dagger$ is $f$-crystalline if and only $E_{|(Y/V)^\dagger}$ satisfies base change over $V$: we have
\[
(Y/O)^\dagger \times_{(X,V/O)^\dagger} (X,V) = (Y/V)^\dagger.
\]
\item Let $f : Y \to X$ be morphism of formal schemes over some overconvergent analytic space $(C,O)$.
We consider the induced morphism of overconvergent sites $f : (Y/O)^\dagger \to (X/O)^\dagger$ (using the same letter $f$ again).
Let $V$ be a geometric materialization for $X$ over $O$.
Then a crystal $E$ on $(Y/O)^\dagger$ is $f$-crystalline if and only $E_{|(Y/V)^\dagger}$ satisfies base change over $V$.
This follows again from theorem \ref{strfib}.
\item In the last two examples, if we are given a morphism $(Y,W) \to (X,V)$, we cannot replace $(Y/O)^\dagger$ by $(Y,W/O)^\dagger$ because the diagram
\[
\xymatrix{(Y,W/V)^\dagger \ar@{^{(}->}[r] \ar[d] & (Y/V)^\dagger \ar[d]
\\ (Y,W/O)^\dagger \ar@{^{(}->}[r] & (Y/O)^\dagger }
\]
is not cartesian in general.
\end{enumerate}
\end{xmp}

Here is a non trivial (even if it may sound so) example:

\begin{prop}
Let $(f,u) : (X',V') \to (X,V)$ be a morphism of overconvergent spaces.
Assume $f$ is bijective and $u : V' \to V$ is a finite universal homeomorphism.
Then any crystal $E'$ on $(X',V')$ is $(f,u)$-crystalline.
\end{prop}

\begin{proof}
It follows from lemma \ref{eqrep} and from lemma \ref{univhome} below that
\[
\mathrm R(f,u)_*E' \simeq \varphi_V^*]f[_{u*}\varphi_{V'*}E'
\]
is a crystal.
\end{proof}

As a consequence (and equivalently), we see that $E'$ satisfies base change with respect to $(f,u)$.

We used in the above proof the following result which is completely analogous to proposition 4.2.4 in \cite{LeStum11}:

\begin{lem} \label{univhome}
Let $(f,u) : (X',V') \to (X,V)$ be a morphism of overconvergent spaces.
Assume $f$ is bijective and $u : V' \to V$ is a finite universal homeomorphism.
If $\mathcal F'$ is an $\mathcal O_{V'}$-module, then
\[
\mathrm R(f,u)_*\varphi_{V'}^*\mathcal F' \simeq \varphi_V^*]f[_{u*}\mathcal F'.
\]
\end{lem}

\begin{proof}
If $(g,v) : (Y,W) \to (X,V)$ is a morphism of overconvergent spaces, then there exists a cartesian diagram
\[
\xymatrix{(Y',W') \ar[r]^{(f',u')} \ar[d]^{(g',v')} & (Y,W) \ar[d]^{(g,v)} \\ (X',V') \ar[r]^{(f,u)} & (X,V) }
\]
Since $]f'[_{u'*}$ is a homeomorphism, it follows from lemma \ref{cartcdb} that
\[
(\mathrm R(f,u)_*\varphi_{V'}^*\mathcal F')_W = \mathrm R]f'[_{u'*}(\varphi_{V'}^*\mathcal F')_{W'} = ]f'[_{u'*}(\varphi_{V'}^*\mathcal F')_{W'} = ]f'[_{u'*}]g'[_{v'}^\dagger \mathcal F'.
\]
On the other hand, we have
\[
(\varphi_V^*]f[_{u*}\mathcal F')_W = ]g[_v^\dagger]f[_{u*}\mathcal F'
\]
and we are therefore reduced to prove the base change
\[
]g[_v^\dagger]f[_{u*}\mathcal F' \simeq  ]f'[_{u'*}]g'[^\dagger \mathcal F'.
\]
If we set $\mathcal F :=  i_{X'*}\mathcal F'$, then we have $\mathcal F' = i_{X'}^{-1}\mathcal F$.
Then, we compute on one hand
\[
]g[_v^\dagger]f[_{u*}\mathcal F' = i_{Y}^{-1}v^*i_{X*}]f[_{u*} i_{X'}^{-1}\mathcal F =  i_{Y}^{-1}v^*u_*i_{X'*} i_{X'}^{-1}\mathcal F = i_{Y}^{-1}v^*u_*\mathcal F.
\]
Moreover, since both $u'$ and $]f'[_{u'}$ are homeomorphisms, we have $]f'[_{u'*} \circ i_{Y'}^{-1} = i_Y^{-1} \circ u'_*$ and we get on the other hand
\[
]f'[_{u'*}]g'[^\dagger \mathcal F' = ]f'[_{u'*}i_{Y'}^{-1}v'^{*}i_{X'*} \mathcal F' = i_Y^{-1} u'_*v'^{*}\mathcal F.
\]
It is therefore sufficient to finally prove the base change formula
\[
v^*u_*\mathcal F \simeq  u'_*v'^{*}\mathcal F.
\]
This may be checked on stalks and since $u'$ is bijective, it is sufficient to prove that for $y' \in W'$, we have
\[
(v^*u_*\mathcal F)_{u'(y')} \simeq  (u'_*v'^{-1}\mathcal F)_{u'(y')}.
\]
Using the fact that both $u$ and $u'$ are homeomorphisms, we have on the one hand
\begin{align*}
(v^*u_*\mathcal F)_{u'(y')} &\simeq \mathcal O_{W,u'(y')} \otimes_{\mathcal O_{V,v(u'(y'))}} (u_*\mathcal F)_{v(u'(y'))}
\\ &\simeq \mathcal O_{W,u'(y')} \otimes_{\mathcal O_{V,u(v'(y'))}} (u_*\mathcal F)_{u(v'(y'))}
\\ &\simeq \mathcal O_{W,u'(y')} \otimes_{\mathcal O_{V,u(v'(y'))}} \mathcal F_{v'(y')}
\end{align*}
and on the other hand, since $u$ is finite,
\begin{align*}
(u'_*v'^{*}\mathcal F)_{u'(y')} &\simeq (v'^{*}\mathcal F)_{y'}
\\&\simeq \mathcal O_{W',y'} \otimes_{\mathcal O_{V',v'(y')}} \mathcal F_{v'(y')}
\\&\simeq (\mathcal O_{W,u'(y')} \otimes_{\mathcal O_{V,u(v'(y'))}} \mathcal O_{V',v'(y')}) \otimes_{\mathcal O_{V',v'(y')}} \mathcal F_{v'(y')}
\\ &\simeq \mathcal O_{W,u'(y')} \otimes_{\mathcal O_{V,u(v'(y'))}} \mathcal F_{v'(y')}.\qedhere
\end{align*}
\end{proof}

\subsection{Embeddings again}

We will now investigate the behavior of a crystal along a formal embedding.
We let $\gamma : Y \hookrightarrow X$ be a formal embedding and $T$ an \emph{analytic} overconvergent site over $X^\dagger$ but we first recall the following general result:

\begin{lem} \label{dirchan}
If $h : W \hookrightarrow V$ is a locally closed embedding of topological spaces, $\mathcal A$ a sheaf of rings on $V$, $\mathcal F$ an $\mathcal A$-module and $\mathcal G$ an $h^{-1}\mathcal A$-module, then
\[
h_{!}(h^{-1}\mathcal F \otimes_{h^{-1}\mathcal A} \mathcal G) \simeq \mathcal F \otimes_{\mathcal A} h_{!}\mathcal G
\]
and
\[
\mathcal H\mathrm{om}_{\mathcal A} (h_!\mathcal F, \mathcal G) \simeq h_{*}\mathcal H\mathrm{om}_{h^{-1}\mathcal A} (\mathcal F, h^{-1} \mathcal G).
\]
\end{lem}

\begin{proof}
This is readily checked on stalks.
\end{proof}

Before going any further, let us mention the following consequence of the first isomorphism of lemma \ref{dirchan}: if $\mathcal B$ is a sheaf of rings on $W$, then $\mathrm{Mod}(\mathcal B) \simeq \mathrm{Mod}(h_!\mathcal B)$ and in particular $\mathrm{Mod}(h^{-1}\mathcal A) \simeq \mathrm{Mod}(h_!h^{-1}\mathcal A)$ (and we can write $h_*$ instead of $h_!$ when $h$ is a closed embedding).

We can apply these considerations in the following situations:
\begin{enumerate}
\item If $(X,V)$ is any analytic space, then
\[
\mathrm{Mod}(\mathcal O_{V}^\dagger) \simeq \mathrm{Mod}(i_{X!}i_X^{-1}\mathcal O_V)
\]
(and we can use $i_{X*}$ when $\,]X[_V$ is closed in $V$).
\item With $\gamma : Y \hookrightarrow X$ as above,
\[
\mathrm{Mod}(\mathcal O_{Y}^\dagger) \simeq \mathrm{Mod}(]\gamma[_{!}]\gamma[^{-1}\mathcal O_{X}^\dagger).
\]
(and we can use $]\alpha[_*$ in the case of a formal open embedding $\alpha$).
\end{enumerate}

Recall that we introduced in definition \ref{newf} two new functors: overconvergent direct image $\gamma_\dagger$ and overconvergent sections $j_Y^\dagger$.

\begin{prop}
We have
$
\mathrm{Mod}(\mathcal O_{T_Y}) \simeq \mathrm{Mod}(j_Y^\dagger \mathcal O_{T})
$.
\end{prop}

\begin{proof}
Giving a $j_Y^\dagger \mathcal O_{T}$-module is equivalent to giving a compatible family of $]\gamma'[_{!}]\gamma'[^{-1}\mathcal O_{V'}^\dagger$-modules when $(X',V')$ is an overconvergent space over $T$ and $\gamma'$ denotes the inverse image of $\gamma$.
The assertion therefore follows from the previous considerations.
\end{proof}

\begin{lem} \label{comgam}
Let $(f,u) : (X'', V'') \to (X',V')$ be a morphism over $T$, $g : Y'' \to Y'$ the map induced between the inverse images of $Y$, and $\gamma', \gamma''$ the inverse images of $\gamma$.
Then,
\[
]f[^\dagger \circ  ]\gamma'[_{!} =  ]\gamma''[_{!} \circ ]g[^\dagger.
\]
\end{lem}

\begin{proof}
There exists a commutative diagram of topological spaces
\[
\xymatrix{\,]Y''[\ar@{^{(}->}[rr]^{]\gamma''[}\ar[d]^{]g[} \ar@/^2pc/[rrrr]^{i_{Y''}}&&\,]X''[ \ar[d]^{]f[} \ar@{^{(}->}[rr]^{i_{X''}}&& V'' \ar[d]^u
\\\,]Y'[ \ar@{^{(}->}[rr]^{]\gamma'[}  \ar@/_2pc/[rrrr]^{i_{Y'}}&&\,]X'[ \ar@{^{(}->}[rr]^{i_{X'}}&& V' }
\]
with a cartesian left hand square so that $]f[^{-1} \circ ]\gamma'[_{!} = ]\gamma''[_{!} \circ ]g[^{-1}$.
We compute on the one side
\begin{align*}
]f[^\dagger ]\gamma'[_{!} \mathcal F &= \mathcal O_{V''}^\dagger  \otimes_{]f[^{-1}\mathcal O_{V'}^\dagger } ]f[^{-1}]\gamma'[_{!} \mathcal F
\\ &= i_{X''}^{-1} \mathcal O_{V''} \otimes_{i_{X''}^{-1} u^{-1}\mathcal O_{V'} } ]\gamma''[_{!}  ]g[^{-1}\mathcal F,
\end{align*}
and on the other side
\begin{align*}
 ]\gamma''[_{!} ]g[^\dagger \mathcal F &=]\gamma''[_{!} (\mathcal O_{V''}^\dagger  \otimes_{]g[^{-1}\mathcal O_{V'}^\dagger  } ]g[^{-1} \mathcal F)
\\&=]\gamma''[_{!} (i_{Y''}^{-1} \mathcal O_{V''} \otimes_{i_{Y''}^{-1}u^{-1} \mathcal O_{V'} } ]g[^{-1} \mathcal F)
\\&=]\gamma''[_{!} (]\gamma''[^{-1}i_{X''}^{-1} \mathcal O_{V''} \otimes_{]\gamma''[^{-1}i_{X''}^{-1}u^{-1} \mathcal O_{V'} } ]g[^{-1} \mathcal F).
\end{align*}
Our assertion therefore follows from the first isomorphism of lemma \ref{dirchan}.
\end{proof}

\begin{prop} \label{crisdag}
Both functors $\gamma_\dagger$ and $j_Y^\dagger$ preserve crystals.
Moreover, if $E$ is a crystal on $T$, then $j_{Y}^\dagger E = \gamma_\dagger\gamma^{-1}E$.
\end{prop}

\begin{proof}
The last assertion directly follows from the definitions and the assertion on $j_Y^\dagger$ is then a consequence of the assertion on $\gamma_{\dagger}$ which itself results from lemma \ref{comgam}.
\end{proof}

As a consequence, we see that if $T$ is defined over some (analytic) overconvergent space $(C,O)$, then the absolute cohomology with support in $Y$ of a \emph{crystal} $E$ is 
\[
\mathrm Rp_{T/O*}j^{\dagger}_{Y}E = \mathrm Rp_{T/O*}\gamma_\dagger E_{|Y}.
\]
In other words, we have for all $k \in \mathbb N$,
\[
\mathcal H^k_Y(T/O, E) = \mathcal H^k(T/O, \gamma_\dagger E_{|Y}).
\]

For further use, we also mention the following consequence of the proposition:

\begin{cor} \label{comphi2}
If $(X,V)$ is an analytic overconvergent space, then
\[
\varphi^*_{X,V} \circ ]\gamma[_! = \gamma_\dagger \circ \varphi_{Y,V}^* \quad \mathrm{and} \quad \varphi^*_{X,V} \circ ]\gamma[_! \circ ]\gamma[^{-1} = j_Y^\dagger \circ \varphi^*_{X,V}.
\]
\end{cor}

\begin{proof}
If $\mathcal G$ is an $\mathcal O^\dagger_{V}$-module on $Y$, then $\varphi^*_{Y,V}\mathcal G$ is a crystal on $(Y,V)$ and therefore $\gamma_\dagger \varphi^*_{Y,V}\mathcal G$ is a crystal on $(X,V)$.
It follows that
\begin{align*}
\gamma_\dagger \varphi^*_{Y,V}\mathcal G &= \varphi^*_{X,V} \varphi_{X,V*}\gamma_\dagger \varphi^*_{Y,V} \mathcal G.
\\&=\varphi^*_{X,V}  (\gamma_\dagger\varphi^*_{Y,V} \mathcal G)_{X,V} 
\\&= \varphi^*_{X,V}]\gamma[_!(\varphi^*_{Y,V} \mathcal G)_{Y,V} 
\\&= \varphi^*_{X,V}]\gamma[_! \mathcal G.
\end{align*}
If $\mathcal F$ is an  $\mathcal O^\dagger_{V}$-module on $X$, then
\begin{align*}
j_Y^\dagger \varphi^*_{X,V}\mathcal F & = \gamma_\dagger \gamma^{-1} \varphi^*_{X,V} \mathcal F
\\ &=  \gamma_\dagger \varphi^*_{Y,V} ]\gamma[^{-1}\mathcal F
\\ &= \varphi^*_{X,V}]\gamma[_!  ]\gamma[^{-1}\mathcal F. \qedhere
\end{align*}

\end{proof}

\begin{lem} \label{lochom}
\begin{enumerate}
\item
If $E$ is a crystal on $T$ and $E'$ is a module on $T_{|Y}$, then
\[
\gamma_{\dagger}(\gamma^{-1}E \otimes_{\mathcal O^\dagger_{T_{Y}}} E') \simeq E \otimes_{\mathcal O^\dagger_{T}} \gamma_{\dagger}E'
\]
and
\[\mathcal H\mathrm{om}_{\mathcal O^\dagger_{T}}(\gamma_{\dagger}E', E) \simeq \gamma_{*}\mathcal H\mathrm{om}_{\mathcal O^\dagger_{T_{Y}}}(E', \gamma^{-1} E).
\]
\item
If $E$ is a crystal on $T$ and $E'$ is a module on $T$, then
\[
j_Y^\dagger(E \otimes_{\mathcal O^\dagger_{T}} E') \simeq E \otimes_{\mathcal O^\dagger_{T}} j_Y^\dagger E'
\]
and
\[\mathcal H\mathrm{om}_{\mathcal O^\dagger_{T}}(j_Y^\dagger E', E) \simeq j_Y^\dagger \mathcal H\mathrm{om}_{\mathcal O^\dagger_{T}}(E', E).
\]
\end{enumerate}
\end{lem}

\begin{proof}
The second assertion formally follows from the first one that we shall now prove.
We let $(X',V')$ be an overconvergent space over $T$ and we denote as usual by $\gamma' : Y' \hookrightarrow X'$ the inverse image of $\gamma$.
Note that since $E$ is a crystal, we have $(\gamma^{-1}E)_{Y',V'} =  ]\gamma'[^{-1}E_{X',V'}$.
Then, using proposition \ref{inthom} and lemma \ref{dirchan}, we can compute
\begin{align*}
(\gamma_{\dagger}(\gamma^{-1}E \otimes E'))_{X',V'} &= ]\gamma'[_{!}(\gamma^{-1}E \otimes E')_{Y',V'}
\\&= ]\gamma'[_{!}((\gamma^{-1}E)_{Y',V'} \otimes E'_{Y',V'})
\\&= ]\gamma'[_{!}(]\gamma'[^{-1}E_{X',V'} \otimes E'_{Y',V'})
\\&= E_{X',V'} \otimes  ]\gamma'[_{!}E'_{Y',V'}
\\&= E_{X',V'} \otimes  (\gamma_{\dagger}E')_{X',V'}
\\&= (E \otimes  \gamma_{\dagger}E')_{X',V'}.
\end{align*}
In the same way, we have
\begin{align*}
\mathcal H\mathrm{om}(\gamma_{\dagger}E', E)_{X',V'} &= \mathcal H\mathrm{om}((\gamma_{\dagger}E')_{X',V'}, E_{X',V'})
\\&= \mathcal H\mathrm{om}(]\gamma'[_{!}E'_{Y',V'}, E_{X',V'})
\\&= ]\gamma'[_{*}\mathcal H\mathrm{om} (E'_{Y',V'}, ]\gamma'[^{-1}E_{X',V'})
\\&= ]\gamma'[_{*}\mathcal H\mathrm{om}(E'_{Y',V'}, (\gamma^{-1}E)_{Y',V'})
\\&= ]\gamma'[_{*}\mathcal H\mathrm{om}(E', (\gamma^{-1}E)_{Y',V'})
\\&= (\gamma_{*}\mathcal H\mathrm{om}(E', \gamma^{-1} E))_{X',V'}. \qedhere
\end{align*}
\end{proof}

\begin{prop} \label{equivgam}
There exists a pair of adjoint functors
\[
\gamma_{\dagger} : \mathrm{Cris}(T_{Y}) \leftrightarrows \mathrm{Cris}(T) : \gamma^{-1}
\]
(resp.
\[
j_Y^{\dagger} : \mathrm{Cris}(T) \leftrightarrows \mathrm{Cris}(T) : \gamma_*\gamma^{-1}).
\]
\end{prop}

\begin{proof}
The second assertion follows from the first one.
Moreover, we know from lemma \ref{lochom} that if $E' \in \mathrm{Cris}(T_{Y})$ and $E \in \mathrm{Cris}(T)$, then
\[
\mathrm{Hom}(\gamma_{\dagger}E', E) \simeq \mathrm {Hom}(E', \gamma^{-1} E). \qedhere
\] 
\end{proof}

As a consequence, we see that $\gamma_{\dagger}$ induces an equivalence between $\mathrm{Cris}(T_{Y})$ and the full subcategory of $\mathrm{Cris}(T)$ made of crystals $E$ satisfying $j^\dagger_{Y}E\simeq E$.

\subsection{Complementary embeddings}

We shall now apply our previous considerations to the case of a complementary pair made of an open and a closed embeddings.
Thus, we denote by $\alpha : U \hookrightarrow X$ a formal open embedding and by $\beta : Z \hookrightarrow X$ the embedding of a closed complement. 
We let $T$ be an analytic overconvergent site over $X^\dagger$.

\begin{prop} \label{alphacr}
If $E$ is a crystal on $T$, then,
\begin{enumerate}
\item \label{alphacr1} both $\alpha_{*}$ and $j_U^\dagger$ preserve crystals and $j_{U}^\dagger E = \alpha_*\alpha^{-1}E$,
\item There exists a pair of adjoint functors
\[
\alpha_{*} : \mathrm{Cris}(T_{U}) \leftrightarrows \mathrm{Cris}(T) : \alpha^{-1},
\]
\item if $(X',V')$ is an overconvergent space over $T$ and $U'$ denotes the inverse image of $U$, then
\[
(\mathrm R\underline \Gamma_{U} E)_{V'} \simeq \mathrm R\underline \Gamma_{\,]U'[} E_{V'}.
\]
\end{enumerate}
\end{prop}

\begin{proof}
Only the last assertion deserves a proof because the first ones are just particular cases of propositions \ref{crisdag} and \ref{equivgam}.
If we denote by $\beta' : Z' \hookrightarrow X'$ the inverse image of $\beta$, then there exists two distinguished triangles
\[
\mathrm R\underline \Gamma_{U}E  \to E \to \mathrm R\beta_{*}\beta^{-1}E \to \cdots
\]
and
\[
\mathrm R\underline \Gamma_{\,]U'[}E_{X',V'}  \to E_{X',V'} \to \mathrm R]\beta'[_{*}]\beta'[^{-1}E_{X',V'} \to \cdots
\]
It is therefore sufficient to notice that
\begin{align*}
(\mathrm R\beta_{*}\beta^{-1}E)_{X',V'} = \mathrm R]\beta'[_{*}(\beta^{-1}E)_{Z',V'} = \mathrm R]\beta'[_{*}]\beta'[^{-1}E_{X',V'}
\end{align*}
because $E$ is a crystal.
\end{proof}

\begin{prop} \label{clop}
\begin{enumerate}
\item The functor $\underline \Gamma_{Z}$ induces an exact functor
\[
\underline \Gamma^\dagger_{Z} : \mathrm{Cris}(T) \to \mathrm{Cris}(T),
\]
\item if $E \in \mathrm{Cris}(T)$, then
\[
\underline \Gamma^\dagger_{Z}E = \mathrm R\underline \Gamma_{Z}E = \beta_{\dagger}E_{|Z} = j_{Z}^\dagger E,
\]
\item \label{clop3} If $E \in \mathrm{Cris}(T)$, then there exists an exact sequence
\[
0 \to \underline \Gamma^\dagger_{Z} E \to E \to j^\dagger_{U}E \to 0.
\]
\end{enumerate}
\end{prop}

\begin{proof}
If $(X',V')$ is an overconvergent space over $T$ and $\alpha', \beta'$ denote the inverse images of $\alpha, \beta$, then
the distinguished triangle
\[
\mathrm R\underline \Gamma_{Z}E \to E \to \alpha_{*}\alpha^{-1}E \to \cdots
\]
realizes itself on $(X',V')$ as
\[
(\mathrm R\underline \Gamma_{Z}E)_{V'} \to E_{V'} \to ]\alpha'[_{*}]\alpha'[^{-1}E_{V'} \to \cdots
\]
because $E$ is a crystal.
On the other hand, there exists a standard short exact sequence
\[
0 \to ]\beta'[_{!}]\beta'[^{-1}E_{V'} \to E_{V'} \to ]\alpha'[_{*}]\alpha'[^{-1}E_{V'} \to 0.
\]
We obtain an identification
\[
(\mathrm R\underline \Gamma_{Z}E)_{V'} \simeq ]\beta'[_{!}]\beta'[^{-1}E_{V'} = (\beta_{\dagger}E_{|Z})_{V'},
\]
again because $E$ is a crystal.
Everything then follows from our previous results.
\end{proof}

As a consequence of proposition \ref{clop}, we see that if $E$ is a crystal on $T$, then
\[
\underline \Gamma^\dagger_{Z}E=E \Leftrightarrow j^\dagger_{U}E = 0 \quad \mathrm{and} \quad \underline \Gamma^\dagger_{Z}E=0 \Leftrightarrow E =j^\dagger_{U}E.
\]

Note also that if $E$ and $E'$ are two crystals on $T$, then
\[
\mathcal H\mathrm{om}_{\mathcal O^\dagger_{T}}(j_U^\dagger E', \underline \Gamma_Z^\dagger E) = 0
\]
and lemma \ref{lochom} therefore implies that
\[
\mathcal H\mathrm{om}_{\mathcal O^\dagger_{T}}(j_U^\dagger E', j_U^\dagger E) \simeq
\mathcal H\mathrm{om}_{\mathcal O^\dagger_{T}}(j_U^\dagger E', E) \simeq j_U^\dagger \mathcal H\mathrm{om}_{\mathcal O^\dagger_{T}}(E', E).
\]
It then follows from lemma \ref{limBer} that if $(X', V')$ is an overconvergent space over $T$, then
\begin{align*}
\mathcal H\mathrm{om}_{\mathcal O^\dagger_{T}}(j_U^\dagger E', j_U^\dagger E)_{V'}
& \simeq (j_U^\dagger \mathcal H\mathrm{om}_{\mathcal O^\dagger_{T}}(E', E))_{V'}
\\ & \simeq \varinjlim j_*j^{-1} \mathcal H\mathrm{om}_{\mathcal O^\dagger_{T}}(E', E)_{V'}
\\ & \simeq \varinjlim j_*j^{-1} \mathcal H\mathrm{om}_{\mathcal O^\dagger_{V'}}(E'_{V'}, E_{V'})
\end{align*}
when $j$ runs through the inclusions of open neighborhoods of $\,]Y[_{V'}$ in $\,]X[_{V'}$.

It is not true at all that $\underline \Gamma_{Y}E = j_{Y}^\dagger E$ in general but in the case of an open embedding, we have the following comparison theorem:

\begin{cor}
If $E \in \mathrm{Cris}(T)$, then there exists a distinguished triangle
\[
\mathrm R\underline \Gamma_{U}E \to j_{U}^\dagger E \to  j_{U}^\dagger \mathrm R\beta_{*}E_{|Z} \to \cdots.
\]
\end{cor}

\begin{proof}
In the following diagram, all other lines and columns are distinguished, and it follows that the bottom line must be distinguished too:
\[
\xymatrix{
& \underline \Gamma^\dagger_{Z}E \ar[r]^-{\simeq} \ar[d] &  \underline \Gamma^\dagger_{Z} \mathrm R\beta_{*}E_{|Z} \ar[d]  &
\\ \mathrm R\underline \Gamma_{U}E \ar[r] \ar@{=}[d] &E \ar[r] \ar[d]& \mathrm R\beta_{*}E_{|Z} \ar[r] \ar[d] &  \cdots
\\ \mathrm R\underline \Gamma_{U}E \ar[r] & j_{U}^\dagger E \ar[r] \ar[d]&  j_{U}^\dagger \mathrm R\beta_{*}E_{|Z} \ar[r] \ar[d] &  \cdots
\\ & \vdots & \vdots &
}
\]
(the isomorphism upstairs is left as an exercise).
\end{proof}

Assume that $T$ is defined over some overconvergent space $(C,O)$.
Then the absolute cohomology with support in $U$ (resp. $Z$) of a \emph{crystal} $E$ is
\[
\mathrm Rp_{T/O*}j^{\dagger}_{U}E = \mathrm Rp_{T_U/O*} E_{|U} \quad \mathrm{(resp.}\ \mathrm Rp_{T/O*}j^{\dagger}_{Z}E = \mathrm Rp_{T/O*}\beta_{\dagger}E_{|Z}).
\]
In other words, we have for all $k \in \mathbb N$,
\[
\mathcal H^k_U(T/O, E) = \mathcal H^k(T_U/O, E_{|U}) \quad \mathrm{(resp.}\
\mathcal H^k_Z(T/O, E) = \mathcal H^k(T/O, \beta_\dagger E_{|Z})).
\]

\section{Calculus}

So far, the notion of a crystal is a very abstract notion.
We explain here how it is related to differential calculus.

\subsection{Overconvergent stratifications}

We let $(X \hookrightarrow P \leftarrow V) \to (C \hookrightarrow S \leftarrow O)$ be a formal morphism of overconvergent spaces.
We will assume that the morphism $P \to S$ is locally noetherian and that $X$ and has self-products\footnote{Otherwise, it would be necessary to invoke hypercoverings.}.
This last condition means that, for all $k \in \mathbb N$, the fibered products
\[
V(k) := \underbrace {V \times_{O} \cdots \times_{O} V}_{k+1}
\]
are representable (in our category of adic spaces locally of noetherian type).
Note that this condition is always satisfied when $V$ is locally of finite type over $O$.
We may then consider the overconvergent space $(X \hookrightarrow P(k) \leftarrow V(k))$.
We will denote by $\delta : V \to V(1)$ the diagonal map and by $p_{1}, p_{2} : V(1) \to V$ as well as $p_{12}, p_{13}, p_{23} : V(2) \to V(1)$ the various projections.
In this situation, we will still use the same letters for the morphisms induced on the tubes in order to make the notation lighter.

\begin{dfn} \label{overst}
An \emph{overconvergent stratification} on an $\mathcal O_{V}^\dagger$-module $\mathcal F$ is a \emph{(Taylor) isomorphism}
$\epsilon : p_{2}^\dagger\mathcal F \simeq p_{1}^\dagger\mathcal F$ of $\mathcal O_{V(1)}^\dagger$-modules such that
\[
p_{13}^\dagger(\epsilon) =p_{12}^\dagger(\epsilon) \circ p_{23}^\dagger(\epsilon)
\]
on $\,]X[_{V(2)}$ (cocycle condition).
\end{dfn}

Note that the \emph{normalization condition} $\delta^\dagger(\epsilon) = \mathrm{Id}_{\mathcal F}$ on $\,]X[_{V}$ is then automatic.

With obvious morphisms, $\mathcal O_{V}^\dagger$-modules endowed with an overconvergent stratification form a category $\mathrm{Strat}(X,V/O)^\dagger$ which is functorial in $(X,V) \to (C,O)$. 

\begin{lem} \label{straO}
If $V$ is flat over $O$ in the neighborhood of $X$, then the category $\mathrm{Strat}(X,V/O)^\dagger$ is abelian and the forgetful functor
\[
\mathrm{Strat}(X,V/O)^\dagger \to \mathrm{Mod}(\mathcal O_{V}^\dagger)
\]
is exact, faithful and detects (and preserves) flatness when $O$ is analytic.
\end{lem}

\begin{proof}
Detection (and preservation) of flatness follows from lemma \ref{flatan} and the other properties are easily checked.
\end{proof}

Without the flatness hypothesis, it is still true that the category is additive with cokernels and that the forgetful functor is right exact and faithful.
Note also that, if $\mathcal F, G \in \mathrm{Strat}(X,V/O)^\dagger$, then there exists a canonical overconvergent stratification on
\[
\mathcal F \otimes_{\mathcal O_{V}^\dagger} \mathcal G \quad (\mathrm{resp.} \
\mathcal H\mathrm{om}_{\mathcal O_{V}^\dagger}(\mathcal F, \mathcal G)
\]
when $\mathcal F$ is \emph{finitely presented}).

There exists a cohomology theory for overconvergent stratifications that we shall describe now.
Let us denote for all $k \in \mathbb N$, by $p(k) : V(k) \to O$ the structural map and $p_{1}(k) : V(k) \to V$ the first projection.
Then, $\mathcal H^{k\dagger}(\mathcal F)$ is the cohomology of the ``bicomplex''
\[
\mathrm R p(0)_{*}p_{1}(0)^\dagger \mathcal F \overset{d^0} \to \mathrm R p(1)_{*}p_{1}(1)^\dagger \mathcal F \overset{d^1}\to \cdots \overset{d^{k-1}}\to \mathrm R p(k)_{*}p_{1}(k)^\dagger \mathcal F \overset{d^k}\to \cdots
\]
on $]C[_{O}$ with horizontal differentials
\[
d^k = (\epsilon \otimes \mathrm{Id}) \circ p_{\widehat 1}(k)^{\dagger} - p_{\widehat 2}(k)^{\dagger} + \cdots + (-1)^k p_{\widehat {k+2}}(k)^{\dagger}
\]
if $p_{\widehat i}(k) : V(k+1) \to V(k)$ denotes the projection that forgets the $i$th component for $i=1, \ldots, k+2$.

In particular, we have
\[
\mathcal H^{0\dagger}(\mathcal F) = \ker (\epsilon \circ p_{2}^{\dagger} - p_{1}^{\dagger} : p(0)_{*}\mathcal F \to p(1)_{*}p_{1}^\dagger \mathcal F).
\]
Also, if $\mathcal F, G \in \mathrm{Strat}(X,V/O)^\dagger$ with $\mathcal F$ finitely presented, then
\[
\mathrm{Hom}_{\mathrm{Strat}(X,V/O)^\dagger}(\mathcal F, \mathcal G) \simeq \Gamma\left(]C[_O, \mathcal H^{0\dagger}\left(\mathcal H\mathrm om_{\mathcal O_{V}^\dagger}(\mathcal F, \mathcal G)\right)\right).
\]

Recall now that we denote by $(X,V/O)^\dagger$ the category of overconvergent spaces $(Y,W)$ endowed with a morphism $f : Y \to X$ that can be extended to \emph{some} morphism $(f,u) : (Y,W) \to (X,V)$ (but the morphism $u : W \to V$ itself is not part of the data).
Then, it follows from lemma \eqref{eqrep} that if $E$ is a crystal on $(X,V)$, it is equivalent to give an overconvergent stratification on $E_V$ or a descent datum (see section \ref{apdesc} of the appendix) on $E$ with respect to the projection $j_V : (X,V) \to (X,V/O)^\dagger$ (or equivalently $(X,V) \to (X/O)^\dagger$).
In other words, there exists an equivalence of categories:
\[
\mathrm{Strat}(X,V/O)^\dagger \simeq \mathrm{Cris}\left((X,V) \to (X,V/O^\dagger)\right)
\]
(where the right hand side denotes the category of objects endowed with a descent datum).
Moreover, the cohomology of the stratified module is related to cohomological descent (see section \ref{formaco} of the appendix\footnote{We only do the case of abelian sheaves in the appendix but our results are still valid here.}).
More precisely, one can identify $\mathcal H^{k\dagger}(E_V)$ with the value on $\,]X[_O$ of the cohomology of the complex $\mathrm Rj_{V\epsilon *}j_{V\epsilon}^{-1}E$.
We shall not use here these considerations and rely on more down to earth arguments.

\begin{prop} \label{equiV}
There exists an equivalence of categories
\[
\xymatrix@R=0cm{ \mathrm{Cris}(X,V/O)^\dagger \ar[r]^\simeq & \mathrm{Strat}(X,V/O)^\dagger \\ E \ar[r]& (E_{V}, \epsilon : p_{2}^\dagger E_{V} \simeq E_{V^2} \simeq p_{1}^\dagger E_{V}).
}
\]
\end{prop}

\begin{proof}
We describe a quasi-inverse.
Let $(\mathcal F, \epsilon) \in \mathrm{Strat}(X,V/O)^\dagger$.
If $(Y,W)$ is an overconvergent space over $(X,V/O)^\dagger$ (so that $f : Y \to X$ comes with the data), then we pick up some factorization $(f,u) : (Y,W) \to (X,V)$ and we set $E_{W} := ]f[_{u}^\dagger \mathcal F$.
If $(g,v) : (Y',W') \to (Y,W)$ is any morphism over $(X,V/O)^\dagger$ and $(f',u') : (Y',W') \to (X,V)$ is the chosen factorization for $(Y',W')$, then we pull the stratification $\epsilon$ back along the morphism
\[
(f', (v \circ u, u')) : (Y',W') \to (X,V(1))
\]
in order to obtain an isomorphism $]g[_{v}^\dagger E_{W} \simeq ]g[_{v}^\dagger]f[_{u}^\dagger \mathcal F \simeq ]f'[_{u'}^\dagger \mathcal F = E_{W'}$.
We can then use the cocycle condition to verify that this defines a crystal.
\end{proof}

Alternatively, the assertion formally follows from the fact that the morphism $j_V : (X,V) \to (X,V/O)^\dagger$ is a local epimorphism and in particular a morphism of effective descent with respect to  (the stack of) crystals.

\begin{cor} \label{crisab}
If $V$ is flat over $O$ in the neighborhood of $X$, then the category $\mathrm{Cris}(X,V/O)^\dagger$ is abelian and the forgetful functor
\[
\mathrm{Cris}(X,V/O)^\dagger \to \mathrm{Mod}(\mathcal O_{V}^\dagger)
\]
is exact, faithful and detects (and preserves) flatness when $O$ is analytic.
\end{cor}

\begin{proof}
Follows from lemma \ref{straO}.
\end{proof}

There also exists an isomorphism in cohomology:

\begin{lem} \label{coiso}
If $E$ is a crystal on $(X,V/O)^\dagger$, then 
\[
\mathcal H^k((X,V/O)^\dagger, E) \simeq \mathcal H^{k\dagger}(E_{V}).
\]
\end{lem}

\begin{proof}
Let us denote by
\[
p : (X,V/O)^\dagger \to (C,O) \quad \mathrm{and} \quad p(k) : (X,V(k)) \to (C,O)
\]
the structural maps.
Since $j_V : (X,V) \to (X,V/O)^\dagger$ is a local epimorphism, there exists an isomorphism
\[
\mathrm Rp_{*}E \simeq [\mathrm Rp(0)_{*}E_{|(X,V(0))} \overset {d^0}\to \mathrm Rp(1)_{*}E_{|(X,V(1))} \overset {d^1}\to \cdots].
\]
On the right, we mean the bicomplex with horizontal differentials given by the usual formulas $d^k = \sum_{i=0}^{k+1} d_{k}^{i}$ where $d_{k}^{i}$ denotes the $i$th coface map in degree $k$.
If we apply $\varphi_{O}$ to this complex, we obtain an isomorphism\footnote{The complex on the right may be called the \emph{\v Cech-Alexander complex} of $E$ on $V$.}
\[
\mathrm Rp_{(X,V/O)^\dagger*}E \simeq [\mathrm Rp_{V(0)*}E_{V(0)} \overset {d^0}\to \mathrm Rp_{V(1)*}E_{V(1)} \overset {d^1}\to \cdots]
\]
on $]C[_{O}$.
If we denote by $p_{1(k)} : (X,V(k)) \to (X,V)$ the first projection, then we can pull at each level back along the transition map $p_{1(k)}^\dagger E_{V} \simeq E_{V(k)}$ and obtain an isomorphism
\[
\mathrm Rp_{(X,V/O)^\dagger*}E \simeq [\mathrm Rp_{V(0)*}p_{1(0)}^\dagger E_{V} \overset {d^0}\to \mathrm Rp_{V(1)*}p_{1(1)}^\dagger E_{V} \overset {d^1}\to \cdots].
\]
It only remains to check that the horizontal differentials coincide which should be clear.
\end{proof}

This last result can also be formally obtained from cohomological descent.

Recall that, if $(C,O)$ is an overconvergent space and $X \to C$ is a morphism of formal schemes, then we denote by $(X/O)^\dagger$ the category of overconvergent spaces $(Y,W)$ over $(C,O)$ endowed with a factorization $Y \to X \to C$.
Also a geometric materialization of $X$ over $O$ is a partially proper formally smooth (and right cartesian) formal morphism of analytic spaces $(X \hookrightarrow P \leftarrow V) \to (C  \hookrightarrow S \leftarrow O)$.

\begin{thm}
Let $(C,O)$ be an overconvergent space,  $X$ a formal scheme over $C$ and $V$ a geometric materialization of $X$ over $O$.
Then,
\begin{enumerate}
\item
the category $\mathrm{Cris}(X/O)^\dagger$ is abelian,
\item
the realization functor $\mathrm{Cris}(X/O)^\dagger \to \mathrm{Mod}(\mathcal O_{V}^\dagger)$ is exact, faithful and detects (and preserves) flatness,
\item it induces an equivalence
\[
\mathrm{Cris}(X/O)^\dagger \simeq \mathrm{Strat}(X,V/O)^\dagger,
\]
\item 
if $E$ is a crystal on $(X/O)^\dagger$, then for all $k \in \mathbb N$,
\[
\mathcal H^k((X/O)^\dagger, E) \simeq \mathcal H^{k\dagger}(E_{V}).
\] 
\end{enumerate}
\end{thm}

\begin{proof}
It is shown in theorem \ref{strfib} (as a consequence of the main result of \cite{LeStum17*}) that 
that there exists a local isomorphism $(X,V/O)^\dagger \to (X/O)^\dagger$.
Then our assertions follow from proposition \ref{equiV}, corollary \ref{crisab} and  lemma \ref{coiso}.
\end{proof}

Of course, the first assertion, being local in nature, is still valid if we only assume that $X$ is locally formally of finite type over $C$ and $O$ is analytic.

\subsection{Overconvergent connections}

We will now develop an infinitesimal version of the notion of an overconvergent stratification and study its relation with the notion of an integrable connection.
This is standard material and we will not give too many details.

We let $(X,V) \to (C,O)$ be a morphism of overconvergent spaces with $V$ locally of finite type over $O$.

We denote by $V^{(n)}$ the $n$th infinitesimal neighborhood of $V$ in $V(1)$: if $\mathcal I$ denotes the ideal of the diagonal embedding $V \hookrightarrow V(1)$, then $V^{(n)}$ is the subspace defined by $\mathcal I^{n+1}$ in $V(1)$.
By definition, the diagonal map $\delta^{(n)} : V \to V^{(n)}$ is a nilpotent immersion that allows us to identify the underlying spaces of $V^{(n)}$ and $V$.
We will denote by $p_{1}^{(n)}, p_{2}^{(n)} : V^{(n)} \to V$ the projections which are all finite universal homeomorphisms.
Unless otherwise specified, when we consider $V^{(n)}$ as an adic space over $V$, we always mean that we are using $p_{1}^{(n)}$ (and make it specific when we use $p_{2}^{(n)}$).

\begin{dfn} \label{stratdef}
A \emph{stratification} on an $\mathcal O_{V}^\dagger$-module $\mathcal F$ is a compatible family of $ \mathcal O_{V^{(n)}}^\dagger$-linear \emph{(Taylor) isomorphisms} $\epsilon_{n} : p_{2}^{(n)\dagger}\mathcal F \simeq p_{1}^{(n)\dagger}\mathcal F$ that satisfy the usual cocycle condition.
\end{dfn}

Note that, since the tubes all have the same underlying space, we may as well write
\[
\epsilon_{n} : \mathcal O^\dagger_{V^{(n)}} \otimes'_{\mathcal O^\dagger_{V}} \mathcal F \simeq  \mathcal F \otimes_{\mathcal O^\dagger_{V}} \mathcal O^\dagger_{V^{(n)}}
\]
where the $\otimes'$ indicates that we are using $p_{2}^{(n)}$ as structural map on the left.
With obvious morphisms, $\mathcal O_{V}^\dagger$-modules endowed with a stratification form a category $\mathrm{Strat}(X,V/O)$ which is functorial in $(X,V) \to (C,O)$.
This is an additive category with cokernels and the forgetful functor $\mathrm{Strat}(X,V/O) \to \mathrm{Mod}(\mathcal O_{V}^\dagger)$ is right exact and faithful.
Actually, if $V$ is smooth in the neighborhood of $X$, then this is even an abelian category and the forgetful functor is exact.
Moreover, if $\mathcal F, G \in \mathrm{Strat}(X,V/O)$, then there exists a canonical stratification on
\[
\mathcal F \otimes_{\mathcal O_{V}^\dagger} \mathcal G \quad (\mathrm{resp.} \
\mathcal H\mathrm{om}_{\mathcal O_{V}^\dagger}(\mathcal F, \mathcal G)
\]
if $\mathcal F$ is \emph{finitely presented}).

There also exists a cohomology theory for stratifications:
let us denote by $p^{(n)} : V^{(n)} \to O$ (and $p = p^{(0)}$ for short) the structural map and by $p_{1}^{(n)} : V^{(n)} \to V$ the first projection.
Then, $\mathcal H^{k}(\mathcal F)$ is the total cohomology of a ``bicomplex''\footnote{It would be necessary to introduce more notation such as $V^{(n)}(k)$ in order to write down the higher terms.}
\[
\mathrm R p_{*} \mathcal F \overset{d^0} \to \mathrm R\varprojlim_{n} \mathrm R p^{(n)}_{*}p_{1}^{(n)\dagger} \mathcal F \overset{d^1}\to \cdots 
\]
with horizontal differentials given by $d_n^0 : \epsilon_{n} \circ {p_2^{(n)\dagger}} - {p_1^{(n)\dagger}}$, etc.
As a particular case, we have
\[
\mathcal H^{0}(\mathcal F) = \varprojlim_{n} \ker \left(\epsilon_n \circ {p_2^{(n)\dagger}} - {p_1^{(n)\dagger}} : p_{*}\mathcal F \to p^{(n)}_{*}p_{1}^{(n)\dagger} \mathcal F\right).
\]
Also, if $\mathcal F, G \in \mathrm{Strat}(X,V/O)$ with $\mathcal F$ finitely presented, then
\[
\mathrm{Hom}_{\mathrm{Strat}(X,V/O)}(\mathcal F, \mathcal G) \simeq \Gamma\left(]C[_O, \mathcal H^{0}\left(\mathcal H\mathrm om_{\mathcal O_{V}^\dagger}(\mathcal F, \mathcal G)\right)\right).
\]

Recall that we introduced in definition \ref{overst} the notion of an overconvergent stratification.
There exists an obvious forgetful functor
\begin{equation} \label{forgstr}
\mathrm{Strat}(X,V/O)^\dagger \to \mathrm{Strat}(X,V/O)
\end{equation}
which is faithful and induces a morphism in cohomology $\mathcal H^{k\dagger}(\mathcal F) \to \mathcal H^k(\mathcal F)$ (actually defined at the complex level).
Note however that there is no reason for this forgetful functor to be fully faithful (nor for the morphism in cohomology to be bijective) in general.

Now, we let $\Omega^{1\dagger}_{V}$ be the the kernel of ${\delta^{(1)}}^{\dagger}$ so that we have an exact sequence
\[
0 \to \Omega^{1\dagger}_{V} \to \mathcal O_{V^{(1)}}^\dagger \to \mathcal O_{V}^\dagger \to 0.
\]
The map
\[
{p_2^{(1)\dagger}} - {p_1^{(1)\dagger}} :  \mathcal O_V \to \mathcal O_{V^{(1)}}
\]
induces a derivation $d : \mathcal O_{V}^\dagger \to \Omega^{1\dagger}_{V}$.
Alternatively, we have $\Omega^{1\dagger}_{V} = i_X^{-1}\Omega^{1}_{V}$ and the derivation of $ \mathcal O_{V}^\dagger$ is induced by the derivation $\mathrm d$ of $\mathcal O_V$.

\begin{dfn}
A \emph{connection} on an $\mathcal O_{V}^\dagger$-module $\mathcal F$ is an $\mathcal O_{O}^\dagger$-linear map
\[
\nabla : \mathcal F \mapsto \mathcal F \otimes_{\mathcal O_{V}^\dagger} \Omega^{1\dagger}_{V}
\]
that satisfies the Leibniz rule.
\emph{Integrability} also is defined in the usual way.
A \emph{horizontal} map is an $\mathcal O_{V}^\dagger$-linear map which is compatible with the connections.
\end{dfn}

Together with horizontal maps, $\mathcal O_{V}^\dagger$-modules endowed with an integrable connection form a category $\mathrm{MIC}(X,V/O)$ which is functorial in $(X,V/O)$.
We can always build the de Rham complex
\[
\mathcal F \otimes_{\mathcal O_{V}^\dagger}\Omega^{\bullet\dagger}_{V}
\]
of an $\mathcal O_{V}^\dagger$-module $\mathcal F$ endowed with an integrable connection, and define its de Rham cohomology $\mathcal H^k_{\mathrm dR}(\mathcal F)$ as the cohomology of the complex
\[
\mathrm Rp_{\mathrm{dR}} \mathcal F := \mathrm Rp_*\left(\mathcal F \otimes_{\mathcal O_{V}^\dagger}\Omega^{\bullet\dagger}_{V}\right).
\]
If we are given a stratification $\epsilon$ on an $\mathcal O_{V}^\dagger$-module $\mathcal F$, then the map
\[
\epsilon \circ {p_2^{(1)\dagger}} - {p_1^{(1)\dagger}} : \mathcal F \to p_1^{(1)\dagger}\mathcal F
\]
induces a connection $\nabla$ on $\mathcal F$.
This provides us with a faithful functor
\[
\mathrm{Strat}(X,V/O) \to \mathrm{MIC}(X,V/O).
\]

\begin{dfn}
An integrable connection $\nabla$ on  an $\mathcal O_{V}^\dagger$-module $\mathcal F$ is said to be \emph{overconvergent} if it comes from an overconvergent stratification (via the functor \eqref{forgstr}).
\end{dfn}

We will denote by $\mathrm{MIC}(X,V/O)^\dagger$ the category of $\mathcal O_{V}^\dagger$-modules endowed with an overconvergent integrable connection.

If $V$ is smooth over $O$, then $V^{(n)}$ is (finite) flat over $V$.
More precisely, if we are given étale coordinates $\underline x := x_1, \ldots, x_d$ on $V$, and we let for all $i=1, \ldots, d$, $\xi_i := p_2^{(n)*}(x_i) - p_1^{(n)*}(x_i)$, then $\mathcal O_{V^{(n)}} \simeq \mathcal O_{V}[\underline \xi]/(\underline \xi)^{n+1}$ as an $\mathcal O_{V}$-algebra.
Here, as usual, we are using $p_1^{(n)}$ as structural map and we can define the higher divided derivatives through the formula
\[
p_2^{(n)*}: \mathcal O_{V} \to \mathcal O_{V^{(n)}}, \quad f \mapsto \sum_{\underline k \in \mathbb N^d} \underline\partial^{[\underline k]}(f) \xi^{\underline k}.
\]
Moreover, $\Omega^1_V$ is a free module on $\mathrm dx_1, \ldots, dx_d$ and we can write
\[
\mathrm d: \mathcal O_{V} \to \Omega^1_V, \quad f \mapsto \sum_{i=1}^d \partial_i(f) \mathrm dx_i.
\]
When $O$ is defined over $\mathbb Q$, then $\partial^{[\underline k]} = \frac 1{\underline k!} \underline \partial^{\underline k}$.
Using this explicit description, it is not hard to show the following:

\begin{lem} \label{MICQ}
Assume $V$ is smooth over $O$ in the neighborhood of $X$ and $O$ is defined over $\mathbb Q$.
Then there exists an equivalence
\[
\mathrm{Strat}(X,V/O) \simeq \mathrm{MIC}(X,V/O)
\]
and, if $\mathcal F$ is an $\mathcal O_{V}^\dagger$-module endowed with a stratification, for all $k \in \mathbb N$, an isomorphism 
\[
\mathcal H^k(\mathcal F) \simeq \mathcal H^k_{\mathrm{dR}}(\mathcal F)
\]

\end{lem}

\begin{proof}
Standard. Locally, the equivalence is given by
\[
\nabla(s) = \sum_{i=1}^d \nabla_i(s) \otimes \mathrm dx_i \Leftrightarrow \epsilon_n(s) = \sum_{|\underline k| \leq n} \frac 1{\underline k!} \underline \nabla^{\underline k}(s) \otimes \underline \xi^{\underline k}. \qedhere
\] 
\end{proof}

We can gather our results in a general perspective.
If we are given a morphism of overconvergent spaces $(X,V) \to (C,O)$ with $V$ flat locally of finite type over $O$, then there exists a commutative diagram
\begin{equation} \label{gloper}
\xymatrix{
\mathrm{Cris}(X/O)^\dagger \ar[r] \ar[d]^{(\star)} &  \mathrm{MIC}(X,V/O)^\dagger \ar@{^{(}->}[r] &\mathrm{MIC}(X,V/O)
\\ \mathrm{Cris}(X,V/O)^\dagger \ar[r]^-\simeq &  \mathrm{Strat}(X,V/O)^\dagger \ar[r] \ar@{->>}[u] &\mathrm{Strat}(X,V/O) \ar[u]^{(\star)}.}
\end{equation}
Moreover, when $V$ is a geometric materialization for $X$ over $O$ defined over $\mathbb Q$, then the $(\star)$'s also are equivalences.

\subsection{Differential operators}

We will now develop the theory of \emph{overconvergent} differential operators.
Again, what follows is quite standard and we will omit to write down the details.

Let $(f,u) : (X,V) \to (C,O)$ be a morphism of overconvergent spaces with $V$ locally of finite type over $O$.

Recall that we denote by $V^{(n)}$ the $n$th infinitesimal neighborhood of $V$, we use the diagonal map $\delta^{(n)} : V \to V^{(n)}$ to identify the underlying spaces of $V^{(n)}$ and $V$ and we denote by $p_{1}^{(n)}, p_{2}^{(n)} : V^{(n)} \to V$ the projections.
In order to keep track of the ring structure, it will be convenient here to maintain $p_{1*}^{(n)}, p_{2*}^{(n)}$ in the notations (even if they act trivially).

\begin{dfn}
An $\mathcal O_O^\dagger$-linear map $D : \mathcal F \to \mathcal G$ between  two $\mathcal O_{V}^\dagger$-modules is a \emph{differential operator} of order at most $n$ if it factors as
\[
D : \mathcal F \to p_{2*}^{(n)}p_{2}^{(n)\dagger}\mathcal F \simeq p_{1*}^{(n)}p_{2}^{(n)\dagger}\mathcal F \overset{\widetilde D}\to \mathcal G
\]
where the first map is the adjunction map and $\widetilde D$ is $\mathcal O_{V}^\dagger$-linear. 
\end{dfn}

Be careful that the middle isomorphism is \emph{not} $\mathcal O_{V}^\dagger$-linear so that $D$ is not $\mathcal O_{V}^\dagger$-linear either.
Alternatively, we could write
\[
D : \mathcal F \to \mathcal O_{V^{(n)}}^\dagger \otimes'_{\mathcal O_V^\dagger} \mathcal F \to \mathcal G
\]
because our spaces share the same underlying topological space.
Then, the first map is ``right linear'' and the second one is ``left linear''.
Note also that $\widetilde D$ is uniquely determined by $D$.
This notion is functorial downstairs in the sense that when $(g,v) : (X',V') \to (X,V)$ is a pullback along some morphism $(C',O') \to (C,O)$, then $D$ pulls back to a differential operator $]g[_V^\dagger D : ]g[_V^\dagger\mathcal F \to ]g[_V^\dagger \mathcal G$.
Also, when $O$ is analytic, if $\gamma : Y \hookrightarrow X$ is a formal embedding, then both $]\gamma[^{-1}$ and $]\gamma[_!$ send a differential operator to a differential operator.

We will denote by
\[
\mathcal D\mathrm{iff}_n(\mathcal F,\mathcal G) \simeq \mathcal H\mathrm{om}_{\mathcal O_{V}^\dagger}(p_{1*}^{(n)}p_{2}^{(n)\dagger}\mathcal F, \mathcal G)
\]
the $\mathcal O_{V}^\dagger$-module of differential operators of order at most $n$ and set
\[
\mathcal D\mathrm{iff}(\mathcal F,\mathcal G) =\bigcup_{n \in \mathbb N} \mathcal D\mathrm{iff}_n(\mathcal F,\mathcal G).
\]
We will simply write $\mathcal D\mathrm{iff}_n(\mathcal F)$ and $\mathcal D\mathrm{iff}(\mathcal F)$ and when $\mathcal F = \mathcal G$.
A standard example of differential operators of order $1$ is given by the differentials in the de Rham complex $\mathcal F \otimes_{\mathcal O_{V}^\dagger}\Omega^{\bullet\dagger}_{V}$ of an $\mathcal O_{V}^\dagger$-module  $\mathcal F$ endowed with an integrable connection.
The composition of two differential operators $D$ and $E$ of orders at most $n$ and $m$ respectively, is a differential operator of order at most $n+m$ and we have
\[
\widetilde{D \circ E} = \widetilde D \circ {p_2^{(n)\dagger}}(\widetilde E).
\]
As a byproduct, we meet the $\mathcal O_{V}^\dagger$-algebra of differential operators  which can be defined as
\[
\mathcal D_{V^\dagger} = \mathcal D\mathrm{iff}\left(\mathcal O_{V}^\dagger \right)
\]
(we keep the dagger downstairs in order to avoid confusion with Berthelot's $\mathcal D^\dagger$ ring which is much more sophisticated).
Any stratification on an $\mathcal O_{V}^\dagger$-module $\mathcal F$ will define a structure of $\mathcal D_{V^\dagger}$-module on $\mathcal F$.
More precisely, the composite map
\[
\mathcal O_{V^{(n)}}^\dagger \otimes'_{\mathcal O_{V}^\dagger} \mathcal F \overset{\epsilon_n}\simeq  \mathcal F \otimes_{\mathcal O_{V}^\dagger} \mathcal O_{V^{(n)}}^\dagger \to \mathcal H\mathrm{om}_{\mathcal O_{V}^\dagger}(\mathcal H\mathrm{om}_{\mathcal O_{V}^\dagger}(\mathcal O_{V^{(n)}}^\dagger,  \mathcal O_{V}^\dagger), \mathcal F)
\]
provides by adjunction a map
\[
\xymatrix{
 \mathcal D\mathrm{iff}_n\left(\mathcal O_{V}^\dagger \right) \ar@{=}[d] \ar[r] &  \mathcal D\mathrm{iff}_n\left(\mathcal F\right)  \ar@{=}[d]
\\
\mathcal H\mathrm{om}( \mathcal O_{V^{(n)}}^\dagger,  \mathcal O_{V}^\dagger) \ar[r] & \mathcal H\mathrm{om}( \mathcal O_{V^{(n)}}^\dagger \otimes'_{\mathcal O_{V}^\dagger} \mathcal F, \mathcal F)}
\]
which defines a structure of $\mathcal D_{V^\dagger}$-module on $\mathcal F$.
When $V$ is smooth in the neighborhood of $X$, this is an equivalence of categories.
This formally follows from the fact that $V^{(n)}$ is then finite \emph{flat} over $V$.
Actually, if we are given étale coordinates $\underline x$ on $V$, then the divided higher derivatives $\underline \partial^{[\underline k]}$ for $\underline k \in \mathbb N^d$ are differential operators of order $|\underline k|$ on $\,]X[_V$ and they form a basis for $\mathcal D_{V^\dagger}$.
Note that we can enhance our diagram \eqref{gloper} by adding the morphism
\[
  \mathrm{Strat}(X,V/O)^\dagger \overset {(\star)} \longrightarrow \mathrm{Mod}(\mathcal D_{V^\dagger}).
\]

\subsection{Linearization of differential operators}

We will explain now how one can build a canonical complex on an overconvergent site from a complex of differential operators.
This technic will allow us to interpret later the cohomology of a crystal as a de Rham cohomology.

Let $(f,u) : (X,V) \to (C,O)$ be a morphism of overconvergent spaces with $V$ locally of finite type over $O$.
We will denote as usual by $j_V : (X,V) \mapsto (X,V/O)^\dagger$ the canonical map.

We first need to extend a bit our notion of a differential operator:

\begin{dfn}
An $\mathcal O_{(C,O)}^\dagger$-linear map $D : E \to E'$ between  two $\mathcal O_{(X,V)}^\dagger$-modules is a \emph{differential operator} of order at most $n$ if it factors as
\[
D : E \to p_{2*}^{(n)}(p_{2}^{(n)})^{-1}E \simeq p_{1*}^{(n)}(p_{2}^{(n)})^{-1}E \overset{\widetilde D}\to E'
\]
where the first map is the adjunction map and $\widetilde D$ is $\mathcal O_{(X,V)}^\dagger$-linear. 
\end{dfn}

Then, we have the following:

\begin{lem} \label{eqdif}
The functors $\varphi_{V*}$ and $\varphi_V^*$ induce an equivalence between the category of crystals and differential operators on $(X,V)$ and the category of $\mathcal O_V^\dagger$-modules and differential operators.
\end{lem}

\begin{proof}
This is \emph{not} completely trivial.
The point is to show that if $\mathcal F$ is an $\mathcal O_V^\dagger$-module and $i=1,2$, then
\[
p_{i*}^{(n)}(p_{2}^{(n)})^{-1}\varphi_V^*\mathcal F = \varphi_V^*p_{i*}^{(n)}p_{2}^{(n)\dagger}\mathcal F.
\]
Since $p_i$ is a finite universal homeomorphism, this follows from lemma \ref{univhome}.
\end{proof}

The next lemma is the first step for the fundamental process of linearization of differential operators.

\begin{lem} \label{lin}
If $E$ and $E'$ are two $\mathcal O_{(X,V)}^\dagger$-modules and $D : E \to E'$ is a differential operator, then $j_{V*}D$ is $\mathcal O^\dagger_{X,V/O}$-linear.
\end{lem}

\begin{proof}
It is sufficient to show that if $E$ is an $\mathcal O_{(X,V)}^\dagger$-module, then the isomorphism
\[
j_{V*}p_{2*}^{(n)}(p_{2}^{(n)})^{-1} E \simeq  j_{V*}p_{1*}^{(n)}(p_{2}^{(n)})^{-1}E
\]
is $\mathcal O^\dagger_{X,V/O}$-linear.
But, for $i=1,2$, we have $j_{V} \circ p_{i}^{(n)} = j_{V^{(n)}}$ and we actually have an equality.
\end{proof}

\begin{lem} \label{enough}
The category of $\mathcal O_{(X,V)}^\dagger$-modules and differential operators has enough injectives.
\end{lem}

\begin{proof}
If we are given an injective $\mathcal O_{(X,V)}^\dagger$-linear morphism $E \hookrightarrow I$ where $I$ is an injective module, then we get an injective $\mathcal O_{(X,V)}^\dagger$-linear morphism
\[
p_{1*}^{(n)}(p_{2}^{(n)})^{-1}E \hookrightarrow p_{1*}^{(n)}(p_{2}^{(n)})^{-1}I
\]
by exactness of $p_{1*}^{(n)}$ and $(p_{2}^{(n)})^{-1}$.
The assertion follows.
\end{proof}

Thanks to lemmas \ref{lin} and \ref{enough}, we can make the following definition:

\begin{dfn}
If $\mathcal F$ is a complex of $\mathcal O_{V}^\dagger$-modules and differential operators, then its \emph{(derived) linearization} is the complex of $\mathcal O^\dagger_{X,V/O}$-modules
\[
\mathrm{R}L\mathcal F := \mathrm Rj_{V*}\varphi_V^*\mathcal F.
\]
\end{dfn}

We could also introduce the notion of (underived) linearization $L\mathcal F$ but then $\mathrm RL$ would not be the  derived functor of $L$.
We can make a list of basic properties:

\begin{lem} \label{RLstuff}
Let $\mathcal F$ be a complex of $\mathcal O_{V}^\dagger$-modules and differential operators.
Then,
\begin{enumerate}
\item  $\mathrm Rp_{X,V/O*}\mathrm RL\mathcal F \simeq  \mathrm R]f[_{u*}\mathcal F$.
\item  If $(X',V')$ is an overconvergent space over $(X,V/O)^\dagger$, then
\[
(\mathrm{R}L\mathcal F)_{V'} = \mathrm R]p_{1}[_*]p_2[^\dagger \mathcal F
\]
where $p_1 : V' \times_O V \to V'$ and $p_2 : V' \times_O V \to V$ denote the projections.
\item If $(g,v) : (X',V') \to (X,V)$ is a pullback along some morphism $(C',O') \to (C,O)$, then $(\mathrm RL\mathcal F)_{|(X',V')} \simeq \mathrm RL]g[_v^\dagger\mathcal F$.
\item If $\gamma : Y \hookrightarrow X$ is a formal embedding, then $(\mathrm RL\mathcal F)_{|Y} \simeq \mathrm RL\mathcal F_{|\,]Y[_V}$.
\item Assume $O$ is analytic and let $\alpha : U \hookrightarrow X$ be a formal open embedding.
If $\mathcal G$ is a complex of $\mathcal O_V^\dagger$-modules and differential operators on $U$, then $\alpha_*\mathrm RL\mathcal G \simeq \mathrm RL]\alpha[_*\mathcal G$.
\end{enumerate}
\end{lem}

\begin{proof}
From the second example after definition \ref{absco}, we get
\begin{align*}
\mathrm Rp_{X,V/O*}\mathrm RL\mathcal F &=\mathrm Rp_{X,V/O*}\mathrm Rj_{V*}\varphi_V^*\mathcal F \\&=\mathrm Rp_{X,V/O*}\varphi_V^*\mathcal F \\&= \mathrm R]f[_{u*}(\varphi_V^*\mathcal F)_V \\&=  \mathrm R]f[_{u*}\mathcal F
\end{align*}
and the first assertion is proved.
From proposition \ref{bicom} (and the examples thereafter), we derive the second assertion:
\begin{align*}
(\mathrm{R}L\mathcal F)_{V'} &= (\mathrm Rj_{V*}\varphi_V^*\mathcal F)_{V'} \\&= \mathrm R]p_{1}[_*(\varphi_V^*\mathcal F)_{V' \times_O V} \\&= \mathrm R]p_{1}[_*]p_2[^\dagger \mathcal F.
\end{align*}
We now prove the third assertion: we have thanks to lemmas \ref{basfib} and \ref{comphi},
\begin{align*}
(\mathrm RL\mathcal F)_{|(X',V')} &= (g,v)^{-1}\mathrm Rj_{V*}\varphi_V^*\mathcal F \\&= \mathrm Rj_{V'*}(g,v)^{-1}\varphi_V^*\mathcal F \\&=\mathrm Rj_{V'*}\varphi_{V'}^* ]g[_v^\dagger \mathcal F \\&= \mathrm RL]g[_v^\dagger\mathcal F.
\end{align*}
The arguments for the fourth assertion are the same.
We use indices $X$ and $Y$ (instead of $V$ which is fixed) for the various functors.
We have, thanks again to lemmas \ref{basfib} and \ref{comphi},
\begin{align*}
(\mathrm RL\mathcal F)_{|Y} &= \gamma^{-1}\mathrm Rj_{X*}\varphi_X^*\mathcal F \\&=\mathrm Rj_{Y*} \gamma^{-1}\varphi_X^*\mathcal F \\&=\mathrm Rj_{Y*}\varphi_Y^* \gamma^{-1}\mathcal F  \\&= \mathrm RL\mathcal F_{|\,]Y[_V}.
\end{align*}
We know prove the last assertion (with $X$ and $U$ as indices) using the fact that $\alpha_*$ is exact and that $\varphi_X^* \circ ]\alpha[_* = \alpha_* \circ \varphi_U^*$ from lemma \ref{comphi}:
\begin{align*}
\alpha_*\mathrm RL\mathcal G &= \alpha_*\mathrm Rj_{U*}\varphi_U^*\mathcal G \\&= \mathrm Rj_{X*}\alpha_*\varphi_U^*\mathcal G \\&=  \mathrm Rj_{X*}\varphi_X^*]\alpha[_*\mathcal G
\\&= \mathrm RL]\alpha[_*\mathcal G. \qedhere
\end{align*}
\end{proof}

Unfortunately, the last assertion will not hold for $\gamma_*$ in general (and it is not likely to hold for $\gamma_\dagger$ either\footnote{Or else I missed something.}).

\subsection{de Rham crystals}

Roughly speaking, a de Rham crystal is a crystal whose cohomology can be computed \emph{à la} de Rham.
We will make this precise.

Let $(f,u) : (X,V) \to (C,O)$ be a morphism of overconvergent spaces with $V$ locally of finite type over $O$.

If $\mathcal F$ is an $\mathcal O_V^\dagger$-module endowed with an integrable connection, then we will write
\[
\mathrm R]f[_{u,\mathrm{dR}}\mathcal F := \mathrm R]f[_{u*}\left(\mathcal F \otimes_{\mathcal O_{V}^\dagger}\Omega^{\bullet\dagger}_V\right) \quad \mathrm{and} \quad \mathrm{R}L_{\mathrm{dR}}\mathcal F := \mathrm{R}L\left(\mathcal F \otimes_{\mathcal O_{V}^\dagger}\Omega^{\bullet\dagger}_{V}\right).
\]

The following is an immediate consequence of lemma \ref{RLstuff} that we state for future reference:

\begin{prop} \label{RLstuff2}
Let $\mathcal F$ be an $\mathcal O_V^\dagger$-module endowed with an integrable connection.
\begin{enumerate}
\item We have
\[
\mathrm Rp_{X,V/O*}\mathrm RL_{\mathrm{dR}} \mathcal F =  \mathrm R]f[_{u,\mathrm{dR}}\mathcal F.
\]
\item If $(X',V')$ is an overconvergent space over $(X,V/O)^\dagger$, then
\[
(\mathrm{R}L_{\mathrm{dR}}\mathcal F)_{V'} = \mathrm R]p_{1}[_{\mathrm{dR}}]p_2[^\dagger \mathcal F.
\]
\item 
If $(g,v) : (X',V') \to (X,V)$ is a pullback along some morphism $(C',O') \to (C,O)$, then
\[
(\mathrm RL_{\mathrm{dR}}\mathcal F)_{|(X',V')} \simeq \mathrm RL_{\mathrm{dR}}]g[_v^\dagger\mathcal F.
\]
\item If $Y \hookrightarrow X$ is a formal embedding, then
\[
(\mathrm RL_{\mathrm{dR}}\mathcal F)_{|Y} \simeq \mathrm RL_{\mathrm{dR}}\mathcal F_{|Y}.
\]
\item Assume $O$ is analytic and let $\alpha : U \hookrightarrow X$ be a formal open embedding.
If $\mathcal G$ an $\mathcal O_{V}^\dagger$-module with an integrable connection on $U$, then
\[
\alpha_*\mathrm RL_{\mathrm{dR}}\mathcal G \simeq \mathrm RL_{\mathrm{dR}}]\alpha[_*\mathcal G. \qed
\]
\end{enumerate}
\end{prop}
\begin{lem}
If $E$ is a crystal on $(X,V/O)^\dagger$, then there exists an adjunction map
\[
E \to \mathrm{RL}_{\mathrm{dR}}E_V.
\]
\end{lem}

\begin{proof}\footnote{The last argument in the proof of \cite{LeStum11}, proposition 4.3.10.2) is not correct as written.}
Since, for $i=1,2$, we have $j_{V} \circ p_{i}^{(1)} = j_{V^{(1)}}$, functoriality of adjunction maps provides us with a commutative diagram (we simply write $p_i$ instead of $p_i^{(1)}$ in oder to make the notations lighter)
\[
\xymatrix{E \ar[r]_-{j_V^{-1}}  \ar@/^2pc/[rr]^-{j_{V^{(1)}}^{-1}} &j_{V*} j_V^{-1}E \ar[r]_-{p_i^{-1}} & j_{V^{(1)*}}j_{V^{(1)}}^{-1}E},
\]
from which we derive
\[
\xymatrix{E \ar[rr]_-{j_V^{-1}} \ar@/^2pc/[rrrr]^-{0} &&j_{V*} j_V^{-1}E \ar[rr]_-{p_2^{-1}-p_1^{-1}} && j_{V^{(1)*}}j_{V^{(1)}}^{-1}E}.
\]
Since $E$ is a crystal on $(X,V/O)^\dagger$, then $j_V^{-1}E$ is a crystal on $(X,V)$ and therefore we have
\[
j_V^{-1}E = \varphi_V^*\varphi_{V*}j_V^{-1}E = \varphi_V^*E_V.
\]
It follows that $j_{V*} j_V^{-1}E = j_{V*}\varphi_V^*E_V$.
We also have
\begin{align*}
j_{V^{(1)*}}j_{V^{(1)}}^{-1}E &= j_{V^{*}} p_{1*}(p_1)^{-1}j_{V}^{-1}E
\\&= j_{V^{*}} p_{1*}(p_1)^{-1}\varphi_V^*E_V
\\&= j_{V^{*}} p_{1*}\varphi_{V^{(1)}}^*]p_1[^{\dagger}E_V
\\& = j_{V*}\varphi_V^*]p_1[_*]p_1[^\dagger E_V.
\end{align*}
Now, $p_1^{-1}$ (resp.\ $p_2^{-1}$) acts as $j_{V*}\varphi_V^*(p_1^\dagger)$ (resp.\ $j_{V*}\varphi_V^*(\epsilon_1 \circ p_2^\dagger)$)
and there exists therefore a commutative diagram
\[
\xymatrix{j_{V*} j_V^{-1}E \ar[rrr]^-{p_2^{-1}-p_1^{-1}} \ar@{=}[d] &&& j_{V^{(1)*}}j_{V^{(1)}}^{-1}E \ar@{=}[d]\\  j_{V*}\varphi_V^*E_V \ar[rrr]^-{j_{V*} j_V^{-1}(\epsilon_1 \circ p_2^{\dagger}-p_1^{\dagger})}   &&& j_{V*}\varphi_V^*]p_1[_*]p_1[^\dagger E_V.}
\]
It follows that the following diagram is also commutative
\[
\xymatrix{E \ar[rr] \ar@/^2pc/[rrrr]^-{0} &&j_{V*}\varphi_V^*E_V \ar[rr]_-{j_{V*} j_V^{-1}(\nabla)}   && j_{V*}\varphi_V^*(E_V \otimes_{\mathcal O_{V}^\dagger}\Omega^{1\dagger}_{V})}
\]
and we obtain the expected adjunction
\[
E \to j_{V*} \varphi_V^*(E_V \otimes_{\mathcal O_{V}^\dagger}\Omega^{\bullet\dagger}_{V})  \to \mathrm Rj_{V*} \varphi_V^*(E_V \otimes_{\mathcal O_{V}^\dagger}\Omega^{\bullet\dagger}_{V}) = \mathrm{RL}_{\mathrm{dR}}E_V. \qedhere
\]
\end{proof}

We can now write down the condition that will make it possible to compute the cohomology of a crystal as a de Rham cohomology:

\begin{dfn}
A crystal $E$ on $(X,V/O)^\dagger$ is \emph{de Rham} if the adjunction map is an isomorphism:
\[
E \simeq \mathrm{RL}_{\mathrm{dR}}E_V.
\]
A crystal $E$ on $(X/O)^\dagger$ is \emph{de Rham on $V$} if $E_{|(X,V/O)^\dagger}$ is de Rham.
\end{dfn}

It is worth mentioning as a consequence of proposition \ref{RLstuff2} that \emph{if} $E$ is de Rham, then
\[
\mathrm Rp_{X,V/O*}E =  \mathrm R]f[_{u,\mathrm{dR}}E_V
\]
and, if $(X',V')$ is an overconvergent space over $(X,V/O)^\dagger$, then
\[
E_{V'} \simeq  \mathrm R]p_{1}[_{\mathrm{dR}}]p_2[^\dagger E_V.
\]

This notion will allow us to compute the cohomology of a crystal as follows:

\begin{thm} \label{cohdr}
Let $(C,O)$ be an overconvergent analytic space, $X$ a formal scheme over $C$ and $E$ a crystal on $(X/O)^\dagger$.
If $E$ is de Rham on a geometric materialization $V$ of $X$ over $O$, then
\[
\mathrm Rp_{X/O*}E =  \mathrm R]f[_{u,\mathrm{dR}}E_V.
\]
\end{thm}

\begin{proof}
Follows from theorem \ref{strfib} as usual (and the first assertion of lemma \ref{RLstuff}).
\end{proof}

Of course, for this result to be useful, we will have to produce de Rham crystals but it will be necessary for that to introduce some finiteness conditions.
For the moment, we want to list some formal properties.
First of all, this notion behaves quite well with respect to formal embeddings:
\begin{prop} \label{dRbas}
\begin{enumerate}
\item
If $Y \hookrightarrow X$ is a formal embedding and $E$ is a de Rham crystal on $(X,V/O)^\dagger$, then $E_{|Y}$ is a de Rham crystal on $(Y,V/O)^\dagger$.
\item
Assume $O$ is analytic et let $\alpha : U \hookrightarrow X$ be a formal open embedding.
If $E$ de Rham crystal on $(U,V/O)^\dagger$, then $\alpha_*E$ is a de Rham crystal on $(X,V/O)^\dagger$.
\item \label{dRbas3} Assume $V$ is smooth over $O$ analytic and let $\beta : Z \hookrightarrow X$ be a formal closed embedding.
If $E$ is a de Rham crystal on $(X,V/O)^\dagger$, then $\underline \Gamma^\dagger_Z E$ is a also a de Rham crystal.
\end{enumerate}
\end{prop}

\begin{proof}
The first two assertions are immediate consequences of the last assertions of lemma \ref{RLstuff}.
For the third one, we denote by $\alpha : U \hookrightarrow X$ the open complement of $Z$.
Then, we know from proposition \ref{clop} that there exists an exact sequence
\[
0 \to  \Gamma^\dagger_{Z} E \to E \to j^\dagger_{U}E \to 0
\]
and  $j^\dagger_{U}E = \alpha_*E_{|U}$.
Our assertion therefore follows from the first part and lemma \ref{exdr} below.
\end{proof}

As a consequence of the last assertion of this proposition, we see that if $E$ is a crystal on $(Z, V/O)^\dagger$ that extends to some de Rham crystal on $(X,V/O)^\dagger$, then $\beta_\dagger E$ is de Rham.
Unfortunately, we do not know if the same holds without this extension property.
This may be a deep problem.

We used above the following result:

\begin{lem} \label{exdr}
Assume $V$ is smooth over $O$.
Let $0 \to E' \to E \to E'' \to 0$ be a sequence of crystals on $(X,V/O)^\dagger$ which is exact in $\mathrm{Mod}(X,V/O)^\dagger$.
If two of them are de Rham, then so is the third.
\end{lem}

\begin{proof}
Since this is an exact sequence of $\mathcal O^\dagger_{X,V/O}$-modules, the sequence
\[
0 \to j_V^{-1}E' \to  j_V^{-1}E \to  j_V^{-1}E'' \to 0
\]
is also exact.
Since we are dealing with crystals, this is the same as the sequence
\[
0 \to \varphi_V^*E'_V \to \varphi_V^*E_V \to \varphi_V^*E''_V \to 0.
\]
Now, since $V$ is smooth over $O$, each $\Omega^{k\dagger}_{V}$ is flat and the same therefore holds for $\varphi_V^{*}\Omega^{k\dagger}_{V}$. It follows that the sequence of de Rham complexes
\[
0 \to \varphi_V^*(E'_V \otimes_{\mathcal O_{V}^\dagger}\Omega^{\bullet\dagger}_{V}) \to \varphi_V^*(E_V \otimes_{\mathcal O_{V}^\dagger}\Omega^{\bullet\dagger}_{V}) \to \varphi_V^*(E''_V \otimes_{\mathcal O_{V}^\dagger}\Omega^{\bullet\dagger}_{V}) \to 0
\]
is also exact.
Applying $\mathrm Rj_{V*}$ provides us with a morphism of distinguished triangles
\[
\xymatrix{E' \ar[r] \ar[d] & E \ar[r] \ar[d] & E'' \ar[r] \ar[d]  & \cdots \\  \mathrm{RL}_{\mathrm{dR}}E'_V \ar[r] &  \mathrm{RL}_{\mathrm{dR}}E_V \ar[r] &   \mathrm{RL}_{\mathrm{dR}}E''_V \ar[r] & \cdots 
}
\]
If two of the arrows are bijective, then so is the third.
\end{proof}

Recall from lemma \ref{flatex} that when $E''$ is flat, exactness in $\mathrm{Mod}(X,V/O)^\dagger$ is equivalent to exactness in $\mathrm{Cris}(X,V/O)^\dagger$ but we have to be careful in general.

Our next result will show that, when $E$ is de Rham, cohomological base change reduces to base change in de Rham cohomology:

\begin{lem}
If $(g,v) : (X',V') \to (X,V)$ is a pullback along some morphism $(C',O') \to (C,O)$ and $E$ is a de Rham crystal on $(X,V/O)^\dagger$, then $E_{|(X',V'/O)^\dagger}$ is also a de Rham crystal.
Moreover, $E$ commutes with (cohomological) base change to $(C',O')$ if and only if
\[
\forall k \in \mathbb N, \quad ]g[_{v}^\dagger \mathrm R^k]f[_{u,\mathrm{dR}}E_V \simeq \mathrm R]f'[_{u',\mathrm{dR}} ]g'[_{v'}^\dagger E_V
\]
(with obvious notations).
\end{lem}

\begin{proof}
This is a consequence of the third assertion of lemma \ref{RLstuff}.
\end{proof}

As usual, the same holds on $(X/O)^\dagger$ when $V$ is a geometric materialization of $X$ over $O$.

\begin{prop} \label{dRcrys}
Let $(Y, W) \to (X,V) \to (C,O)$ be a sequence of morphisms of overconvergent spaces.
We consider the induced morphism of overconvergent sites $f : (Y/O)^\dagger \to (X,V/O)^\dagger$.
We assume that $W$ is a geometric materialization for $Y$ over $V$.
Then a crystal $E$ on $(Y/O)^\dagger$ which is de Rham on $W$ is $f$-crystalline if and only if for all $(g,v) : (X',V') \to (X,V)$, we have
\[
\forall k \in \mathbb N, \quad ]g[_{v}^\dagger \mathrm R^k]f[_{u,\mathrm{dR}}E_V \simeq \mathrm R]f'[_{u',\mathrm{dR}} ]g'[_{v'}^\dagger E_V
\]
(with obvious notations).
\end{prop}

\begin{proof}
We know from the examples after proposition \ref{crsiba} that $E$ is $f$-crystalline if and only if $E_{|(Y/V)^\dagger}$ satisfies base change over $V$.
\end{proof}

Again, the same holds on $(X/O)^\dagger$ when $V$ is a geometric materialization of $X$ over $O$.
It is however important to impose that $W$ is a geometric materialization and we cannot remove this assumption and just replace $(Y/O)^\dagger$ with some $(Y,W/O)^\dagger$.

\section{Constructible crystals}

We will study the notion of a finitely presented crystal which corresponds to the classical notion of overconvergent isocrystal and introduce the more general notion of a constructible crystal.

\subsection{Cohomology on tubes}

We want to compute the cohomology of abelian sheaves on tubes but we start with a very general definition (we replace the usual Hausdorff condition with quasi-separatedness):

\begin{dfn}
A quasi-separated topological space $V$ is
\begin{enumerate}
\item \emph{paracompact} if any open covering has a locally finite refinement,
\item \emph{countable at infinity} if it has a countable covering by quasi-compact subspaces,
\item \emph{taut} if the closure of any quasi-compact subspace if quasi-compact.
\end{enumerate}
\end{dfn}

When $V$ is Hausdorff (which is not our concern here), one recovers the usual notions (and the last one is automatic).
A quasi-separated \emph{analytic} space is paracompact if and only if has an open covering by coherent open subspaces meeting only finitely many others if and only if it is a disjoint union of analytic subspaces that are countable at infinity and taut.
A quasi-separated \emph{analytic} space $V$ is countable at infinity (resp.\ and taut) if and only if $V = \bigcup_{n \in \mathbb N} V_{n}$ with $V_{n}$ quasi-compact open (resp.\ and $\overline V_{n} \subset V_{n+1}$).
All these equivalences actually hold when $V$ is a locally closed subset of an analytic space (and in particular this applies to tubes).

\begin{xmp}
\begin{enumerate}
\item A coherent analytic space is paracompact.
\item As a particular case, an affinoid analytic space or a proper space over a coherent analytic space is paracompact.
\item If $V$ is a quasi-projective space over a coherent analytic space, then $V$ is paracompact.
\item A tube in a coherent analytic space is paracompact (this is proved in lemma \ref{tateret} below).
\end{enumerate}
\end{xmp}

Many results below generalize quite well to locally spectral spaces (or more precisely to valuative spaces) but we will stick to the geometric situation.

Unless otherwise specified, we let $(X, V)$ be an overconvergent space with $V$ Tate (there exists a global topologically nilpotent unit) and paracompact.
We may then also say that the overconvergent space $(X,V)$ is \emph{Tate paracompact}.

\begin{lem} \label{tateret}
The tube $\,]X[_{V}$ also is paracompact.
\end{lem}

\begin{proof}
By definition, we are given a locally closed embedding $X \hookrightarrow P$ and a morphism $V \to P^{\mathrm{ad}}$.
We may assume that $X$ is closed in $P$ because a closed subset of a paracompact analytic space is automatically paracompact\footnote{Be careful that a closed subspace of a quasi-separated space is not quasi-separated in general. This is nevertheless true for locally spectral spaces.} (and the tube is anti-continuous on analytic spaces).
We may also assume that $V = \bigcup_{n \in \mathbb N} V_{n}$ with $V_{n}$ quasi-compact open in $V$ and $\overline V_{n} \subset V_{n+1}$.
We know from proposition 4.28 of \cite{LeStum17*} that $\,]X[_{V} = \bigcup_{n \in \mathbb N}[X]_{V,n}$ with $[X]_{V,n}$ open and retrocompact in $V$  and $\overline {[X]}_{V,n}\subset[X]_{V,n+1}$.
It follows that $\,]X[_{V} = \bigcup_{n \in \mathbb N}([X]_{V,n} \cap V_{n})$ with $[X]_{V,n} \cap V_{n}$ quasi-compact open in $V$ and $\overline {[X]_{V,n} \cap V_{n}} \subset[X]_{V,n+1} \cap V_{n+1}$.
\end{proof}

Note that this applies in particular when $V$ is quasi-compact in which case the tube itself need not be quasi-compact.

The proof shows that the analogous result holds with ``countable at infinity (and taut)'' instead of ``paracompact''.
Note also that, conversely, if we only assume $V$ to be quasi-separated and taut but that $\,]X[_V$ is known to be paracompact (or countable at infinity), then there exists a paracompact (or countable at infinity) neighborhood of $\,]X[_V$ in $V$.

\begin{prop} \label{limsh}
If $\mathcal F$ is a sheaf on $V$, then
\[
\varinjlim_{V'} \Gamma(V', \mathcal F) \simeq \Gamma(\,]X[_{V}, i_{X}^{-1}\mathcal F )
\]
where $V'$ runs through all open neighborhoods of $\,]X[_{V}$ in $V$.
\end{prop}

\begin{proof}
Injectivity is automatic and only surjectivity needs a proof.
Assume first that $\,]X[_{V}$ is quasi-compact.
As mentioned after proposition 4.27 of \cite{LeStum17*}, $\,]X[_{V}$ is stable under generalization.
It follows that $\,]X[_{V} = \bigcap_{V'} V'$ when $V'$ runs through the (quasi-compact) open neighborhoods of $\,]X[_{V}$ in $V$.
Our assertion then follows from lemma \ref{FKlim} because filtered direct limits commute with global sections on coherent topological spaces.
We now turn to the general case.
Thanks to lemma \ref{tateret}, we can replace $V$ with $]\overline X[_{V}$ where $\overline X$ denotes a closure of $X$ and assume that $\,]X[_{V}$ is closed in $V$.
We can also assume that $V = \bigcup_{n \in \mathbb N} V_{n}$ with $V_{n}$ quasi-compact open and $\overline V_{n} \subset V_{n+1}$.
Now, for each $n \in \mathbb N$, $\,]X[_{V_{n}}$ is closed in $V_{n}$ which is quasi-compact and therefore $\,]X[_{V_{n}}$ itself is quasi-compact.
Therefore, if $s \in \Gamma(\,]X[_{V}, i_{X}^{-1}\mathcal F )$, then there exists an open subset $U'_{n} \subset V_{n}$ containing $\,]X[_{V_{n}}$ and $s'_{n} \in \Gamma(U'_{n}, \mathcal F)$ such that $s'_{n|\,]X[_{V_{n}}} = s_{|\,]X[_{V_{n}}}$.
In particular, we have $s'_{n+1|\,]X[_{V_{n}}}= s'_{n|\,]X[_{V_{n}}}$ and there exists therefore an open subset $U_{n} \subset U'_{n} \cap U'_{n+1}$ containing $\,]X[_{V_{n}}$ such that $s'_{n+1|U_{n}}= s'_{n|U_{n}}$ (using again the fact that $\,]X[_{V_{n}}$ is quasi-compact).
Now, we consider the open subset $W_{n} := U_{n} \cap V_{n} \setminus \overline V_{n-2}$ (with the convention that $V_{n} = \emptyset$ if $n < 0$) and we define $s_{n} := s'_{n|W_{n}} \in \Gamma(W_{n}, \mathcal F)$.
Since $W_{n} \subset U_{n}$ (and $W_{n+1} \subset U_{n+1})$, we have:
\[
s_{n+1|W_{n+1} \cap W_{n}} = s'_{n+1|W_{n+1} \cap W_{n}} = s'_{n|W_{n+1} \cap W_{n}} = s_{n|W_{n+1} \cap W_{n}}.
\]
Moreover, $W_{n} \cap W_{m} \subset V_{n} \setminus \overline V_{n-2} \cap V_{m} \setminus \overline V_{m-2} = \emptyset$ if $|n-m| >1$.
Therefore, if we set $W = \bigcup_{n \in \mathbb N} W_{n}$, then there exists a unique $\widetilde s \in \Gamma(W, \mathcal F)$ such that $\widetilde s_{|W_{n}} = s_{n}$ for all $n \in \mathbb N$.
It follows that
\[
\widetilde s_{|\,]X[_{W_{n}}} = \widetilde s_{|W_{n}|\,]X[_{W_{n}}} = s_{n|\,]X[_{W_{n}}} = s_{|\,]X[_{W_{n}}}.
\]
It only remains to check that $\,]X[_{V} \subset W$, and we will then have $\widetilde s_{|\,]X[_{V}} = s_{|\,]X[_{V}}$.
But if $x \in V$, then there exists $n \in \mathbb N$ such that $x \in V_{n}$ and $x \notin V_{n-1}$ (or else $V$ is quasi-compact).
If we assume that $x \in \,]X[_{V}$, then we will have $x \in \,]X[_{V_{n}} \subset U_{n}$ but also $x \notin \overline V_{n-2} \subset V_{n-1}$ so that $x \in W_{n}$ and we are done. 
\end{proof}

\begin{xmp}
In proposition \ref{limsh}, it is not sufficient to assume that $\bigcap_{V'} V' = \,]X[_V$: it is necessary to let $V'$ run through all (or a cofinal system of) open neighborhoods of $\,]X[_V$ in $V$.
In the case $X = \emptyset$ and $V = \mathbb D^-_{\mathbb Q_p}(0,1) $, if we let $V_n = \mathbb D^-_{\mathbb Q_p}(0, 1) \setminus \mathbb D_{\mathbb Q_p}(0, p^{1/n})$, then
\[
\varinjlim_{V_n} \Gamma(V_n, \mathcal O_V) = \mathcal R \neq \{0\} = \Gamma(\,]X[_{V}, i_{X}^{-1}\mathcal O_V)
\]
where $\mathcal R$ is the Robba ring.
\end{xmp}

\begin{cor} \label{limhom}
If $\mathcal F$ and $\mathcal G$ are two $\mathcal O_{V}$-modules, then
\[
\varinjlim_{V'} \mathrm{Hom}_{\mathcal O_{V'}}(\mathcal F_{|V'}, \mathcal G_{|V'}) = \mathrm{Hom}_{\mathcal O_{V}^\dagger}(i_{X}^{-1}\mathcal F, i_{X}^{-1}\mathcal G)
\]
where $V'$ runs through the open neighborhoods of $\,]X[_{V}$ in $V$.
\end{cor}

\begin{proof}
It follows from proposition \ref{limsh} that
\begin{align*}
\mathrm{Hom}_{\mathcal O_{V}^\dagger}(i_{X}^{-1}\mathcal F, i_{X}^{-1}\mathcal G) &= \Gamma(\,]X[_{V}, \mathcal H\mathrm{om}_{\mathcal O_{V}^\dagger}(i_{X}^{-1}\mathcal F, i_{X}^{-1}\mathcal G) )
\\&= \Gamma(\,]X[_{V}, i_{X}^{-1}\mathcal H\mathrm{om}_{\mathcal O_{V}}(\mathcal F, \mathcal G) )
\\&= \varinjlim \Gamma(V', \mathcal H\mathrm{om}_{\mathcal O_{V}}(\mathcal F_{|V}, G_{|V})_{|V'})
\\&= \varinjlim\Gamma(V', \mathcal H\mathrm{om}_{\mathcal O_{V'}}(\mathcal F_{|V'}, G_{|V'}) )
\\&= \varinjlim \mathrm{Hom}_{\mathcal O_{V'}}(\mathcal F_{|V'}, G_{|V'}). \qedhere
\end{align*}
\end{proof}

Note that the analogous result also holds for sheaves of sets or abelian sheaves for example (with the same proof).

For further use, we also prove the following:

\begin{lem} \label{locshea}
Let $\mathcal F$ be an $\mathcal O_{V}^\dagger$-module, $V = \bigcup_{i\in I} V_{i}$ an open covering and, for each $i\in I$, $\mathcal G_{i}$ an $\mathcal O_{V_{i}}$-module such that  $\mathcal G_{i|\,]X[_{V_{i}}} \simeq \mathcal F_{|\,]X[_{V_{i}}}$.
Then, there exists a neighborhood $V'$ of $\,]X[_{V}$ in $V$ and an $\mathcal O_{V'}$-module $\mathcal G$ such that $\mathcal G_{|\,]X[_{V}} \simeq \mathcal F$ and for all $i \in I$, $\mathcal G_{|V' \cap V_{i}} \simeq \mathcal G_{i|V' \cap V_{i}}$.
\end{lem}

\begin{proof}
We will make an extensive use of corollary \ref{limhom}.
We first assume that $\,]X[_{V}$ is quasi-compact.
We may then also assume that $I$ is finite and all $V_{i}$ are quasi-compact.
The identity on $\mathcal F_{|\,]X[_{V_{i} \cap V_{j}}}$ provides a collection of compatible isomorphism $\mathcal G_{i|\,]X[_{V_{i} \cap V_{j}}} \simeq \mathcal G_{j|\,]X[_{V_{i} \cap V_{j}}}$ for $i,j \in I$.
They extend to a family of isomorphisms $\mathcal G_{i|V_{ij}} \simeq \mathcal G_{j|V_{ij}}$ defined on a neighborhood $V_{ij}$ of $\,]X[_{V_{i} \cap V_{j}}$.
Moreover, there exists for all $i,j,k \in I$ an neighborhood $V_{ijk}$ of $\,]X[_{V_{i} \cap V_{j} \cap V_{k}}$ on which the isomorphisms are compatible.
Actually, since $I$ is finite, we can assume thanks to corollary 2.2.12 of \cite{FujiwaraKato18}, that $V_{ij} = V_{i} \cap V_{j}$ and $V_{ijk} = V_{i} \cap V_{j} \cap V_{k}$.
And we are done.
We know turn to the general case and we may assume that $\,]X[_{V}$ is closed in $V$ and $V = \bigcup_{n \in \mathbb N} V_{n}$ with $V_{n}$ quasi-compact open and $\overline V_{n} \subset V_{n+1}$.
Then, we know from the first part that there exists for each $n \in \mathbb N$, an open subset $U_{n} \subset V_{n}$ containing $\,]X[_{V_{n}}$ and a $\mathcal O_{U_{n}}$-module $\mathcal G_{n}$ such that $\mathcal G_{n|\,]X[_{V_{n}}} \simeq \mathcal F_{|\,]X[_{V_{n}}}$ and, for all $i \in I$, $\mathcal G_{n|U_{n} \cap V_{i}} \simeq \mathcal G_{i|U_{n} \cap V_{i}}$.
We can now shrink $U_{n}$ and assume that $\mathcal G_{n+1|U_{n}} = \mathcal G_{n}$.
Our assertion follows.
\end{proof}

We now turn to cohomological concern:

\begin{lem} \label{corhom}
If $\mathcal F$ is a flasque abelian sheaf on $V$, then $i_{X}^{-1}\mathcal F$ is acyclic:
\[
\forall k> 0, \quad \mathrm H^k(\,]X[_{V}, i_{X}^{-1}\mathcal F ) = 0.
\]
\end{lem}

\begin{proof}
Recall from section 2 of \cite{Kempf80} that a sheaf $\mathcal G$ on $\,]X[_V$ is said to be \emph{quasi-flasque} if whenever $W' \subset W$ is an inclusion of \emph{quasi-compact} open subsets of $\,]X[_V$, then the restriction map $\Gamma(W, \mathcal G) \to \Gamma(W', \mathcal G)$ is surjective.
Note that any quasi-compact open subset of $\,]X[_V$ is the tube of $X$ in some quasi-compact open subset of $V$.
It then follows from (the easy case of) proposition \ref{limsh} that if $\mathcal F$ is flasque on $V$, then $i_X^{-1}\mathcal F$ is quasi-flasque on $\,]X[_V$.
In particular, it is sufficient to prove that quasi-flasque sheaves on $\,]X[_V$ are acyclic.
We will show that the assumptions from \cite[\href{https://stacks.math.columbia.edu/tag/05T8}{Tag 05T8}]{stacks-project} are satisfied.
First of all any sheaf $\mathcal G$ on $\,]X[_V$ embeds into a quasi-flasque sheaf: there exists an injection $i_{X*}\mathcal G \hookrightarrow \mathcal I$ into an injective and therefore $\mathcal G = i_X^{-1}i_{X*}\mathcal G \hookrightarrow i_X^{-1} \mathcal I$ (which is quasi-flasque from what we saw before).
Assume now that $0 \to \mathcal G' \to \mathcal G \to \mathcal G'' \to 0$ is an exact sequence with $\mathcal G, \mathcal G'$ quasi-flasque.
Then, it formally follows from proposition 4 in \cite{Kempf80} that $\mathcal G''$ also is quasi-flasque.
Now, since $\,]X[_V$ is paracompact, we may assume that $\,]X[_V = \bigcup_{n \in \mathbb N} W_n$ with $W_n$ quasi-compact open and $W_n \subset W_{n+1}$.
It follows from proposition 4 in \cite{Kempf80} again that, for each $n \in \mathbb N$, the sequence
\[
0 \to \mathrm \Gamma(W_n, \mathcal G') \to \Gamma(W_n,\mathcal G) \to \Gamma(W_n,\mathcal G'') \to 0
\]
is exact.
Since $\mathcal G'$ is quasi-flasque, the restriction maps $\mathrm \Gamma(W_{n+1}, \mathcal G') \to  \mathrm \Gamma(W_n, \mathcal G')$ are surjective and we can invoke Mittag-Leffler to conclude that the sequence
\[
0 \to \mathrm \Gamma(\,]X[_V, \mathcal G') \to \Gamma(\,]X[_V,\mathcal G) \to \Gamma(\,]X[_V,\mathcal G'') \to 0
\]
is also exact.
\end{proof}

We can now improve on proposition \ref{limsh}:

\begin{prop} \label{limcoh}
If $\mathcal F$ is a complex of abelian sheaves on $V$, then
\[
\forall k \in \mathbb N, \quad \varinjlim_{V'} \mathrm H^k(V', \mathcal F) \simeq \mathrm H^k(\,]X[_{V}, i_{X}^{-1}\mathcal F )
\]
where $V'$ runs through the open neighborhoods of $\,]X[_{V}$ in $V$.
\end{prop}

\begin{proof}
We choose an injective resolution $\mathcal F \simeq \mathcal I^\bullet$.
Since filtered direct limits are exact, we have
\[
\varinjlim_{V'} \mathrm H^k(V', \mathcal F) \simeq \mathrm H^k\left(\varinjlim_{V'} \Gamma(V', \mathcal I^\bullet) \right).
\]
On the other hand, it follows from lemma \ref{corhom} that all $i_X^{-1}\mathcal I^k$ are acyclic, and therefore
\[
\mathrm H^k(\,]X[_{V}, i_{X}^{-1}\mathcal F )\simeq  \mathrm H^k\left(\Gamma(]X[_{V}, i_X^{-1} \mathcal I^\bullet) \right).
\]
Our assertion then follows from proposition \ref{limsh}.
\end{proof}

Let us finish with a relative variant of proposition \ref{limsh}:

\begin{prop} \label{dirim}
Let
\[
\xymatrix{Y \ar@{^{(}->}[r] \ar[d]^f & Q \ar[d] & \ar[l] W \ar[d]^u \\ X \ar@{^{(}->}[r] & P & \ar[l] V}
\]
be a left cartesian morphism of analytic overconvergent spaces with $u$ quasi-compact quasi-separated (and any $V$).
If $\mathcal F$ is an abelian sheaf on $W$, then
\[
\mathrm R]f[_{u*}i_Y^{-1}\mathcal F \simeq i_X^{-1}\mathrm Ru_*\mathcal F.
\]
\end{prop}

\begin{proof}
We could rely on our previous results or refer to corollary 3.17 of \cite{FujiwaraKato18} but we shall rather derive it from corollary 1.2.2 of \cite{AbeLazda20}.
First of all, the diagram
\[
\xymatrix{\,]Y[_W \ar@{^{(}->}[r] \ar[d]^{]f[_u} & W \ar[d]^u \\ \,]X[_V \ar@{^{(}->}[r] & V}
\]
is cartesian.
Now, let $v \in \,]X[_V$ and $Y_v := u^{-1}(\mathrm G(v))$ where $\mathrm G(v)$ denotes the set of all generalizations of $v$ in $V$.
Note that $\,]X[_V$ is stable under generalization because $V$ is analytic.
Also, the fact that $u$ is quasi-compact quasi-separated implies the same for $]f[_u$.
Therefore, if we denote by $i_v : Y_v \hookrightarrow W$ the inclusion map, then
\[
\forall k \in \mathbb N, \quad \left(\mathrm R^k]f[_{u*}i_Y^{-1}\mathcal F\right)_v = \mathrm H^k(Y_v, i_v^{-1}\mathcal F) = \left(\mathrm R^ku_*\mathcal F\right)_v. \qedhere
\]
\end{proof}

Note that the condition for $u$ of being \emph{quasi-compact} is necessary in this statement but, unfortunately, it is not satisfied in practice as the case of the relative affine space $\mathbb A_V$ or open unit  disk $\mathbb D^-_V$ shows (see the example after proposition \ref{dirim} below).

\subsection{Cohomology and formal embeddings}

We will apply the results of the previous section to formal embeddings.

We fix here an overconvergent analytic space $(X,V)$ and a formal embedding $\gamma : Y \hookrightarrow X$.
Our results below are somehow more general than those in the previous section in the sense that if we are given an overconvergent space $(X \overset \gamma \hookrightarrow P \leftarrow V)$, then we have $i_X = ]\gamma[$.

\begin{prop} \label{corgam}
If $V$ is Tate and paracompact and $\mathcal F$ is a sheaf on $\,]X[_V$, then
\[
\varinjlim_{W} \Gamma(W, \mathcal F) \simeq \Gamma(\,]Y[_{V}, ]\gamma[^{-1}\mathcal F )
\]
where $W$ runs through the open neighborhoods of $\,]Y[_{V}$ in $\,]X[_V$.
\end{prop}

\begin{proof}
We may assume that $W = \,]X[_{W'}$ when $W'$ runs through the open neighborhoods of $\,]Y[_V$ in $V$ (possibly repeating $W$).
It then follows from proposition \ref{limsh} that
\begin{align*}
\varinjlim_{W} \Gamma(W, \mathcal F) & \simeq \varinjlim_{W'} \Gamma(W', i_{X*}\mathcal F)
\\ & \simeq \Gamma(\,]Y[_{V}, i_Y^{-1}i_{X*}\mathcal F )
\\ & \simeq \Gamma(\,]Y[_{V}, ]\gamma[^{-1}\mathcal F ). \qedhere
\end{align*}
\end{proof}

\begin{lem} \label{flaskgm}
If $V$ is Tate and paracompact and $\mathcal F$ is a flasque abelian sheaf on $\,]X[_V$, then $]\gamma[^{-1}\mathcal F$ is acyclic.
\end{lem}

\begin{proof}
Follows from lemma \ref{corhom} because
\[
\mathrm H^k(\,]Y[_{V}, ]\gamma[^{-1}\mathcal F) = \mathrm H^k(\,]Y[_{V}, i_{Y}^{-1} i_{X*} \mathcal F). \qedhere
\]
\end{proof}

There also exists an analog to proposition \ref{limcoh} that is proved exactly in the same way.

Recall that extension by zero does not preserve injectives in general (and the same problem occurs with overconvergent direct image).
Nevertheless, we have the following:

\begin{lem} \label{flask}
If $V$ is Tate and paracompact and $\mathcal F$ is a flasque abelian sheaf on $\,]Y[_V$, then $]\gamma[_!\mathcal F$ is acyclic.
\end{lem}

\begin{proof}
We can assume that $Y$ is open or closed in $X$.
In the case of an open embedding $\alpha : U \hookrightarrow X$, then $]\alpha[_!\mathcal F = ]\alpha[_*\mathcal F$ is flasque.
In the case of a closed embedding $\beta : Z \hookrightarrow X$, if we denote by $\alpha$ the inclusion of its open complement, then there exists an exact sequence
\[
0 \to ]\beta[_!\mathcal F \to ]\beta[_*\mathcal F  \to ]\alpha[_*]\alpha[^{-1}]\beta[_*\mathcal F \to 0.
\]
Since $]\beta[_*\mathcal F$ is flasque, it is sufficient to recall that 
\[
\Gamma(\,]X[_V, ]\beta[_*\mathcal F)  \to \Gamma(\,]U[_V, ]\alpha[^{-1}]\beta[_*\mathcal F).
\]
is surjective thanks to proposition \ref{corgam} and that
\[
H^k(\,]U[_V, ]\alpha[^{-1}]\beta[_*\mathcal F) = 0
\]
when $k > 0$ thanks to lemma \ref{flaskgm}.
\end{proof}

Here again, there exists a relative version:

\begin{prop} \label{dirim}
Let $(f,u) : (X', V') \to (X,V)$ be a morphism of analytic overconvergent spaces with $]f[_u$ \emph{quasi-compact and quasi-separated}.
Denote by $\gamma' : Y' \hookrightarrow X'$ and $f' : Y' \to Y$ the respective inverse images of $\gamma$ and $f$.
Then,
\begin{enumerate}
\item
if $\mathcal F$ is a complex of abelian sheaves on $\,]X'[_{V'}$,
\[
]\gamma[^{-1}\mathrm R]f[_{u*}\mathcal F \simeq \mathrm R]f'[_{u*}]\gamma'[^{-1}\mathcal F,
\]
\item
if $\mathcal F$ is a complex of abelian sheaves on $\,]Y'[_{V'}$,
\[
]\gamma[_!\mathrm R]f'[_{u*}\mathcal F \simeq \mathrm R]f[_{u*}]\gamma'[_!\mathcal F.
\]
\end{enumerate}
\end{prop}

\begin{proof}
Let us start with the first question which is local on $V$.
We can assume that $V$ is Tate affinoid.
We have a commutative diagram with cartesian left hand square:
\[
\xymatrix{\,]Y'[_{V'} \ar@{^{(}->}[r]^{]\gamma[} \ar[d]^{]f'[_u} & \,]X'[_{V'} \ar@{^{(}->}[r]^{i_{X'}} \ar[d]^{]f[_u} &V' \ar[d]^u
\\ \,]Y[_V \ar@{^{(}->}[r]^{]\gamma[}  &  \,]X[_V \ar@{^{(}->}[r]^{i_X}  & V .  }
\]
It follows from lemma \ref{corhom} that if $\mathcal F$ is a flasque abelian sheaf, then $]\gamma'[^{-1}\mathcal F$ is acyclic for $]f'[_{u*}$.
It is therefore sufficient to show that
\[
]\gamma[^{-1} ]f'[_{u*}\mathcal F \simeq ]f'[_{u*}]\gamma'[^{-1}\mathcal F.
\]
Since the question is local on $V$, it is even sufficient to show that
\[
\Gamma(\,]Y[_V, ]\gamma[^{-1} ]f'[_{u*}\mathcal F) \simeq \Gamma(\,]Y'[_{V'}, ]\gamma'[^{-1}\mathcal F).
\]
We know from proposition \ref{corgam} that
\[
\varinjlim_{W} \Gamma(W, ]f'[_{u*} \mathcal F) \simeq \Gamma(\,]Y[_{V}, ]\gamma[^{-1} ]f'[_{u*} \mathcal F )
\]
where $W$ runs through the open neighborhoods of $\,]Y[_{V}$ in $\,]X[_V$.
On the other hand, since $]f[_u$ is quasi-compact, the inverse images $W'$ of $W$ in $\,]X'[_{V'}$ form a cofinal system of neighborhoods of $\,]Y'[_V$ and therefore
\[
\varinjlim_{W} \Gamma(W, ]f'[_{u*} \mathcal F) \simeq \varinjlim_{W'} \Gamma(W', \mathcal F) \simeq \Gamma(\,]Y'[_{V'}, ]\gamma'[^{-1}\mathcal F).
\]

For the second question, it is sufficient to consider the case of an open embedding $\alpha : U \hookrightarrow X$ or a closed embedding $\beta : Z \hookrightarrow X$.
The first case is trivial and we may now assume that $U$ is the open complement of $Z$.
There exists a morphism of distinguished triangles:
\[
\xymatrix{
]\beta[_! \mathrm R]f'[_{u*}\mathcal F \ar[r] \ar[d] &  \mathrm R]\beta[_* \mathrm R]f'[_{u*}\mathcal F  \ar[r]  \ar[d]  & ]\alpha[_*]\alpha[^{-1} \mathrm R]\beta[_*\mathrm R]f'[_{u*}\mathcal F  \ar[r]  \ar[d]  & \cdots
\\
\mathrm R]f[_{u*}]\beta'[_! \mathcal F \ar[r] & \mathrm R]f[_{u*}\mathrm R]\beta'[_* \mathcal F  \ar[r] & \mathrm R]f[_{u*}]\alpha'[_*]\alpha'[^{-1} \mathrm R]\beta'[_* \mathcal F  \ar[r] & \cdots.
}
\]
The middle map is clearly bijective but the right hand map also thanks to the first part.
\end{proof}

\begin{xmp}
Be careful that the assertion is not valid anymore when $]f[_u$ is \emph{not quasi-compact}: take $f : X' =  \mathbb A \setminus \{0\} \hookrightarrow X = \mathbb A$, $Y = \{0\}$  and $V = V' = \mathbb D_{\mathbb Q_p}$.
\end{xmp}

There exists an obvious analog for de Rham cohomology that we'd rather state now for further use:

\begin{cor} \label{dRChang}
In the situation of the proposition,
\begin{enumerate}
\item
if $\mathcal F$ is an $\mathcal O_{V'}^\dagger$-module with an integrable connection on $X'$, then
\[
]\gamma[^{-1}\mathrm R]f[_{u\mathrm{dR}}\mathcal F \simeq \mathrm R]f'[_{u\mathrm{dR}}]\gamma'[^{-1}\mathcal F,
\]
\item
if $\mathcal F$ is an $\mathcal O_{V'}^\dagger$-module with an integrable connection on $Y'$, then
\[
]\gamma[_!\mathrm R]f'[_{u\mathrm{dR}}\mathcal F \simeq \mathrm R]f[_{u\mathrm{dR}}]\gamma'[_!\mathcal F.
\]
\end{enumerate}
\end{cor}

\begin{proof}
If we write indices as $X$ and $Y$ instead of $V$, we have
\begin{align*}
]\gamma[^{-1}\mathrm R]f'[_{u\mathrm{dR}}\mathcal F & \simeq ]\gamma[^{-1}\mathrm R]f[_{u*}\left(\mathcal F \otimes_{\mathcal O_{X}^\dagger}\Omega^{\bullet\dagger}_X\right)
\\ & \simeq \mathrm R]f'[_{u*}]\gamma'[^{-1}\left(\mathcal F \otimes_{\mathcal O_{X}^\dagger}\Omega^{\bullet\dagger}_X\right)
\\ & = \mathrm R]f'[_{u*}\left(]\gamma'[^{-1} \mathcal F \otimes_{\mathcal O_{Y}^\dagger}\Omega^{\bullet\dagger}_Y\right)
\\ & \simeq \mathrm R]f'[_{u\mathrm{dR}}]\gamma'[^{-1}\mathcal F
\end{align*}
and also, thanks to lemma \ref{dirchan},
\begin{align*}
]\gamma[_!\mathrm R]f'[_{u\mathrm{dR}}\mathcal F & \simeq ]\gamma[_!\mathrm R]f[_{u*}\left(\mathcal F \otimes_{\mathcal O_{Y}^\dagger}\Omega^{\bullet\dagger}_Y\right) 
\\& \simeq \mathrm R]f'[_{u*}]\gamma'[_!\left(\mathcal F \otimes_{\mathcal O_{Y}^\dagger}\Omega^{\bullet\dagger}_Y\right)
\\ & \simeq \mathrm R]f'[_{u*}\left(]\gamma'[_! \mathcal F \otimes_{\mathcal O_{X}^\dagger}\Omega^{\bullet\dagger}_X\right)
\\ & \simeq \mathrm R]f'[_{u\mathrm{dR}}]\gamma'[_!\mathcal F. \qedhere
\end{align*}
\end{proof}

\subsection{Coherent modules}

We will show that, in the Tate paracompact case, a coherent module on a tube always extends to a coherent module on a neighborhood of the tube.

We will always denote with an index ``$\mathrm{coh}$'' the full subcategory made of objects whose underlying module is coherent.

\begin{thm} \label{glumd}
If $(X,V)$ is an analytic overconvergent space, then $\mathcal O_{V}^\dagger$ is a coherent sheaf.
If $V$ is Tate and paracompact, then there exists an equivalence of categories
\[
\varinjlim_{V'} \mathrm{Mod}_{\mathrm{coh}}(\mathcal O_{V'}) \simeq \mathrm{Mod}_{\mathrm{coh}}(\mathcal O_{V}^\dagger) 
\]
when $V'$ runs through the open neighborhoods of $X$ in $V$.
\end{thm}

\begin{proof}
We first prove that $\mathcal O_{V}^\dagger$ is coherent.
Thus, we want to show that, after replacing $V$ with any open subset, the kernel of any morphism $\mathcal O_{V}^{\dagger n} \to \mathcal O_{V}^\dagger$ is of finite type.
This is a local question and we may therefore assume $(X,V)$ is Tate paracompact (even affinoid).
Using corollary \ref{limhom}, we may shrink $V$ a little bit and extend our morphism to a morphism $\mathcal O_{V}^n \to \mathcal O_{V}$.
Since $\mathcal O_{V}$ is coherent, the kernel of this morphism is of finite type and our assertion follows from the fact that pulling back along $i_{X}$ is exact.

We now turn to the second assertion.
It follows from corollary \ref{limhom} that the functor
\[
\varinjlim_{V'} \mathrm{Mod}_{\mathrm{coh}}(\mathcal O_{V'}) \to \mathrm{Mod}_{\mathrm{coh}}(\mathcal O_{V}^\dagger)
\]
is fully faithful and it only remains to show that it is essentially surjective.
If $\mathcal F$ is a coherent $\mathcal O_{V}^\dagger$-module, then there exists a covering $V = \bigcup_{i \in I} V_{i}$ and for each $i \in I$ a finite presentation
\[
\mathcal O_{V_i}^{\dagger n_{i}}  \to \mathcal O_{V_i}^{\dagger m_{i}} \to \mathcal F_{V_{i}} \to 0.
\]
Using corollary \ref{limhom}, we can shrink each $V_{i}$ a little bit and extend the first map to a morphism $\mathcal O_{V_i}^{n_{i}}  \to \mathcal O_{V_i}^{m_{i}}$.
If we denote by $\mathcal G_{i}$ its cokernel, then there exists an isomorphism $\mathcal G_{i|\,]X[_{V_{i}}} \simeq \mathcal F_{|\,]X[_{V_{i}}}$ and we can finish with lemma \ref{locshea}.
\end{proof}

\begin{prop} \label{hknul}
Let $(X \hookrightarrow P \leftarrow V)$ be an analytic overconvergent space with $P$ affine and $V$ Tate affinoid.
Assume either that
\begin{enumerate}
\item
$X$ closed in $P$ (convergent space), or
\item
$X$ is the complement of a hypersurface in $P$.
\end{enumerate}
If $\mathcal F$ be a coherent $\mathcal O_V^\dagger$-module on $X$, then $\mathrm H^k(\,]X[_V, \mathcal F) = 0$ for $k > 0$.
\end{prop}

\begin{proof}
In the first case, we can write $\,]X[_V = \bigcup_{n \in \mathbb N} [X]_{V,n}$ as in proposition 4.28 of \cite{LeStum17*}.
Each $[X]_{V,n}$ is an affinoid open subset of $V$, we have $[X]_{V,n} \subset [X]_{V,n+1}$ and, if $V := \mathrm{Spa}(B, B^+)$, each $\Gamma([X]_{V,n}, \mathcal O_V)$ is a completion of the ring $B$ for some topology.
Hence, the assertion follows from the adic analog of theorem 2.4 in \cite{Kiehl67b}.

In the second case, we denote the hypersurface as $h = 0$ and let $\pi$ be a topologically nilpotent unit.
Then, we set for $n \in \mathbb N$,
\[
V_n := \left\{v \in V \colon v(\pi) \geq v(h^n) \right\}.
\]
This is a rational open subset of $V$ and $\,]X[_V = \bigcap_{n\in \mathbb N} V_n$ thanks to corollary 4.30 of \cite{LeStum17*}.
Since $]X[_V$ is quasi-compact, they forme a fundamental system of open neighborhoods.
It follows from theorem \ref{glumd} that $\mathcal F$ extends to a coherent $\mathcal O_{V_m}$-module $\widetilde {\mathcal F}$ for some $m \in \mathbb N$.
Therefore,
\[
\mathrm H^k(\,]X[_V, \mathcal F) = \varinjlim_{n \geq m} \mathrm H^k(V_n, \widetilde {\mathcal F}) = 0
\]
when $k > 0$ by proposition \ref{limcoh}.
\end{proof}

We may actually do a little better in the second case: it is sufficient to assume that $X$ is a disjoint union of open subsets $X'$ whose locus at infinity is a hypersurface (think of $P = \mathrm{Spec}(\mathbb Z[x,y]/xy)$ and $X$ the complement of $x=y=0$ for example).

\begin{cor}
Let $(X \hookrightarrow P \leftarrow V)$ be an analytic \emph{convergent} space.
Let $\beta : Y \hookrightarrow X$ be a closed formal embedding.
If $\mathcal F$ is a coherent $\mathcal O_V^\dagger$-module on $Y$, then $\mathrm R^k]\beta[_*\mathcal F = 0$ for $k > 0$.
\end{cor}

\begin{proof}
Since $\mathrm R^k]\beta[_*\mathcal F$ is the sheaf associated to $\,]X[_{V'} \mapsto \mathrm H^k(\,]Y[_{V'}, \mathcal F)$, it is sufficient to show that, locally on $V$, we have $\mathrm H^k(\,]Y[_V, \mathcal F) = 0$.
This follows from proposition \ref{hknul}.
\end{proof}

Recall from definition \ref{stratdef} that if $V \to O$ is a morphism of adic spaces which is locally of finite type, we denote by $V^{(n)}$ the $n$th infinitesimal neighborhood of $V$ in $V(1)$ and by $p_{1}^{(n)}, p_{2}^{(n)} : V^{(n)} \to V$ the projections.
A \emph{stratification} on an $O_{V}$-module $\mathcal F$ is a compatible family of $ \mathcal O_{V^{(n)}}$-linear \emph{(Taylor) isomorphisms} $\epsilon_{n} : p_{2}^{(n)*}\mathcal F \simeq p_{1}^{(n)*}\mathcal F$ that satisfy the usual cocycle condition.
They form a category $\mathrm{Strat}(V/O)$.
We can also consider the category $\mathrm{MIC}(V/O)$ of modules on $V$ endowed with an integrable connection with respect to $O$.
There exists a forgetful functor $\mathrm{Strat}(V/O) \to \mathrm{MIC}(V/O)$ which is an equivalence when $O$ is defined over $\mathbb Q$ and $V$ is smooth over $O$.

\begin{cor}
Let $(X,V) \to (C,O)$ be a morphism of overconvergent spaces.
Assume $V$ is Tate,  paracompact and locally of finite type over $O$.
Then, there exists an equivalence of categories
\[
\mathrm{Strat}_{\mathrm{coh}}(X,V/O) \simeq \varinjlim_{V'} \mathrm{Strat}_{\mathrm{coh}}(V'/O) 
\]
when $V'$ runs through the open neighborhoods of $X$ in $V$.
\end{cor}

\begin{proof}
Follows from theorem \ref{glumd} applied to both $V$ and $V^{(n)}$.
\end{proof}

\begin{cor}
Let $(X,V) \to (C,O)$ be a morphism of overconvergent spaces.
Assume $V$ is Tate, paracompact, smooth over $O$ in the neighborhood of $X$ and $O$ is defined over $\mathbb Q$.
Then, there exists an equivalence of categories
\[
\mathrm{MIC}_{\mathrm{coh}}(X,V/O) \simeq \varinjlim_{V'} \mathrm{MIC}_{\mathrm{coh}}(V'/O) 
\]
when $V'$ runs through the open neighborhoods of $X$ in $V$.
\end{cor}

\begin{proof}
Follows from lemma \ref{MICQ}.
\end{proof}

\subsection{Constructible crystals}

We introduce here the general notion of a constructible crystal that will play an increasing role.

Recall that, if $X$ is a formal scheme, then we denote by $X^\dagger$ the overconvergent site whose objects are overconvergent spaces $(U,V)$ endowed with a morphism $U \to X$.

\begin{dfn}
Let $T$ be an overconvergent site and $E$ a module on $T$.
\begin{enumerate}
\item
Let $X$ be a formal scheme and $T \to X^\dagger$ a morphism of overconvergent sites.
Then $E$ is said to be \emph{$X$-constructible} if there exists a locally finite covering of $X$ by locally closed formal subschemes $Y$ such that each $E_{|Y}$ is finitely presented.
\item
$E$ is said to be \emph{constructible} if there exists a formal scheme $X$ and a morphism $T \to X^\dagger$ such that $E$ is $X$-constructible.
\end{enumerate}
\end{dfn}

The formal scheme $X$ will usually be understood from the context and we shall simply say \emph{constructible} in both cases.
This should not matter in practice thanks to the following:

\begin{lem}
Assume
\begin{enumerate}
\item $T = (X,V)$ where $(X,V)$ be an overconvergent space, or
\item $T = (X/O)^\dagger$ where $(C,O)$ be an overconvergent space and $X \to C$ a morphism of formal schemes, or
\item $T = (X,V/O)^\dagger$ where $(X,V) \to (C,O)$ is morphism of overconvergent spaces.
\end{enumerate}
Then, a module $E$ on $T$ is constructible if and only if it is $X$-constructible.
\end{lem}

\begin{proof}
In all three cases, there exists a canonical morphism $T \to X^\dagger$.
Using assertion \eqref{changx} of lemma \ref{cons1} below, it will be sufficient to show that, given any morphism $T \to X'^\dagger$, there exists a morphism $X \to X'$ such that $T \to X'^\dagger$ factors through $X^\dagger$.
The first case is a particular case of the second one from which one easily derives the last one because $(X,V/O)^\dagger$ is by definition the image of $(X,V)$ into $(X/O)^\dagger$.
We are left with the case $T = (X/O)^\dagger$.
We may then consider the composite map $(X, \emptyset) \to  (X/O)^\dagger \to X'^\dagger$.
It provides a morphism of formal schemes $X \to X'$ which clearly induces a factorization $(X/O)^\dagger \to X^\dagger \to X'^\dagger$ as expected.
\end{proof}

Recall that we introduced the notion of a \emph{constructible filtration} in definition \ref{fildef} (this is also sometimes called a good stratification).

\begin{lem} \label{cons2}
Let $X$ be a formal scheme and $T$ on overconvergent site over $X^\dagger$.
For a module $E$ on $T$, the following are equivalent:
\begin{enumerate}
\item $E$ is $X$-constructible
\item There exists a locally finite covering of $X$ by locally closed formal subschemes $Y$ such that each $E_{|Y}$ is $Y$-constructible.
\item \label{compcon} There exists a dense open subset $U$ with closed complement $Z$ such that $E_{|U}$ is finitely presented and $E_{|Z}$ is $Z$-constructible.
\item (if $\dim(X) = d$) there exists a constructible filtration $\{X_i\}_{i=0}^d$ such that $E_{|X_i\setminus X_{i+1}}$ is finitely presented.
\end{enumerate}
\end{lem}

\begin{proof}
Let us assume first that $E$ is $X$-constructible.
If $\xi$ is a maximal point of $X$, then there exists a locally closed subset $Y_\xi$ containing $\xi$ such that $E_{|{Y_\xi}}$ is finitely presented.
Since $\xi$ is maximal, there exists an open neighborhood $U_\xi$ of $\xi$ in $Y_\xi$ and we can set $U := \bigcup_\xi U_\xi$.
This shows that $(1) \Rightarrow (3)$. Moreover $(2) \Rightarrow (1)$ thanks to transitivity of locally finite locally closed coverings, $(3) \Rightarrow (4)$ by noetherian induction and we clearly have $(3) \Rightarrow (1)$ and $(4) \Rightarrow (1)$.
\end{proof}

In assertion $(3)$ (resp.\ $(4)$), we can assume that $U$ (resp.\ each $X_i\setminus X_{i+1}$) is locally irreducible.
When $X$ is affine, we may also assume that $U$ (resp.\ each $X_i \setminus X_{i+1}$) is a disjoint union of open subsets whose locus at infinity is a hypersurface.

Note also that, from assertion \ref{changx} of lemma \ref{cons1} below, it is equivalent in lemma \ref{cons2} to say $Y$-constructible (resp.\ $Z$-constructible) or $X$-constructible.

\begin{lem} \label{cons1}
Let $X$ be a formal scheme and $T$ on overconvergent site over $X^\dagger$.
\begin{enumerate}
\item If $T' \to T$ is a morphism of overconvergent sites and $E$ is an $X$-constructible module on $T$, then $E_{|T'}$ is also $X$-constructible.
The converse holds if $T' \to T$ is a local epimorphism.
\item \label{changx} If $f : X' \to X$ is a morphism of formal schemes and $T$ is defined over $X'^\dagger$, then an $X$-constructible module on $T$ is also $X'$-constructible.
The converse holds if $f$ is formally quasi-finite.
\end{enumerate}
\end{lem}

\begin{proof}
The first assertion follows from stability under pullback of locally finite locally closed coverings and the fact that being finitely presented is a local notion.
The same argument shows that the condition is sufficient in the second assertion.
For the converse, we let $E$ be an $X'$-constructible module on $T$.
Assertion \eqref{compcon} of lemma \ref{cons2} implies that there exists a dense open subset $U'$ with closed complement $Z'$ such that $E_{|U'}$ is finitely presented and $E_{|Z'}$ is constructible.
If we set $Z := \overline {f(Z')}$ and let $U$ be the open complement of $Z$, then $U$ is dense in $X$ because $f$ is formally quasi-finite (\cite[\href{https://stacks.math.columbia.edu/tag/03J2}{Tag 03J2}]{stacks-project}), $E_{|U}$ is still finitely presented because $f^{-1}(U) \subset U'$ and $E_{|Z}$ is constructible.
\end{proof}

The second assertion of lemma \ref{cons1} applies in particular to a formal embedding $\gamma : Y \hookrightarrow X$,  and we may therefore say that $E$ is $X$-constructible instead of $Y$-constructible (and we shall simply say constructible in practice).

Be careful that, in the definition of a constructible module, each $E_{|Y}$ is automatically a crystal but this need not be the case for $E$ itself (example below).

\begin{prop}
Let $X$ be a formal scheme and $T$ an overconvergent site over $X^\dagger$.
Then, $X$-constructible modules (resp.\ crystals) on $T$ form an additive subcategory of $\mathrm{Mod}(\mathcal O_T)$ which is stable under cokernel, extension and tensor product.
\end{prop}

\begin{proof}
Since the restriction maps $E \mapsto E_{|Y}$ are exact and commute with tensor product, it is enough to consider the case of a finitely presented crystal.
This is then completely formal according to proposition \ref{catcrys}.
\end{proof}

$X$-constructible \emph{modules} on $T$ are also stable under internal Hom but this need not be the case for $X$-constructible crystals in general (example below).
There also exists an analog to the proposition for constructible (and not merely $X$-constructible) objects. 
This follows from the fact that the category of morphisms $T \to X^\dagger$ is filtered (use fibered product).

\subsection{The analytic case}

In the analytic situation, there exists a very nice way to describe constructible crystals as we shall shortly see.

\begin{lem} \label{dirconv}
Let $X$ be a formal scheme and $T$ an analytic overconvergent site over $X^\dagger$.
If $\gamma : Y \hookrightarrow X$ is a formal embedding and $E$ is a module on $T_Y$, then $E$ is constructible if and only if $\gamma_\dagger E$ is constructible.
\end{lem}

\begin{proof}
The condition is sufficient because then $E = (\gamma_\dagger E)_{|T_Y}$.
For the converse implication, we may assume that $\gamma$ is either an open or a closed immersion with complement $Y'$ and recall that then $(\gamma_\dagger E)_{|Y} = E$ and $(\gamma_\dagger E)_{|Y'} = 0$.
\end{proof}

\begin{xmp}
Let $X$ be a formal scheme and $T$ an analytic overconvergent site over $X^\dagger$.
\begin{enumerate}
\item Any finitely presented module (or equivalently crystal) is constructible (this is what people usually call an ``overconvergent isocrystal'').
Note that the converse also holds if $X$ is a (locally noetherian) formal scheme of dimension zero.
\item More generally, if $\gamma : Y \hookrightarrow X$ is a formal embedding and $E$ is a finitely presented module on $T_{Y}$, then $\gamma_{\dagger} E$ is a constructible crystal on $T$ (but usually not finitely presented).
\item As a particular case, if $\alpha : U \hookrightarrow X$ is a formal open embedding and $E$ is a finitely presented module on $T_{U}$, then $\alpha_{*} E$ is a constructible crystal on $T$ (but usually not finitely presented).
\item If $\beta : Z \hookrightarrow X$ is a formal closed embedding and $E$ is finitely presented on $T_{Z}$, then $\beta_{*} E$ is a constructible module but this is \emph{not} a crystal in general.
\item The usual dual to a constructible crystal is not a (constructible) crystal in general because $(\beta_{\dagger}E)^\vee = \beta_{*}E^\vee$ in the notations of the previous examples.
There should however exist a duality theory for complexes with constructible crystalline cohomology.
\end{enumerate}
\end{xmp}

The main tool for studying constructible crystals on analytic overconvergent sites is the following:

\begin{prop} \label{extcons}
Let $X$ be a formal scheme and $T$ an analytic overconvergent site over $X^\dagger$.
A crystal $E$ on $T$ is $X$-constructible if and only if there exists an open subset $U$ of $X$ with closed complement $Z$, a constructible crystal $E'$ on $T_{Z}$ (resp.\ $E''$ on $T_{U}$) and a short exact sequence
\[
0 \to \beta_{\dagger}E' \to E \to \alpha_{*}E'' \to 0
\]
where $\alpha : U \hookrightarrow X$ and $\beta : Z \hookrightarrow X$ denote the corresponding formal embeddings.
Moreover, we may assume that $U$ is dense in $X$ and $E''$ is finitely presented.
\end{prop}

\begin{proof}
Given such data, we will have $E_{|Z} = E'$ and $E_{|U} = E''$ and $E$ is therefore constructible.
Conversely, if $E$ is constructible, we may invoke lemma \ref{cons2}.
Thus, there exists a dense open subset $U$ with closed complement $Z$ such that $E'' := E_{|U}$ is finitely presented and $E' := E_{|Z}$ is constructible.
It only remains to write down the exact sequence of assertion \eqref{clop3} in proposition \ref{clop}.
\end{proof}

As a consequence of proposition \ref{extcons}, we see that, when $X$ is finite dimensional, the category of $X$-constructible crystals on $T$  is the smallest category of crystals (or even modules) which is stable under extension and contains $\gamma_\dagger E$ whenever $\gamma : Y \hookrightarrow X$ is a formal embedding and $E$ is a finitely presented crystal (or equivalently module) on $T_Y$.

As a first application, we prove the following result that will be useful later when we consider descent issues:

\begin{lem} \label{bijcons}
Let $X$ be a formal scheme and $f : T \to S$ a morphism of analytic overconvergent sites over $X^\dagger$.
Assume that for all formal embedding $Y \hookrightarrow X$, $f_Y^{-1}$ is fully faithful on finitely presented crystals.
Then the same holds for $X$-constructible crystals.
\end{lem}

\begin{proof}
It is sufficient to consider the case $Y=X$ and we can proceed by induction on the dimension of $X$ (since this is a local question).
We give ourselves two constructible crystals $E_1$ and $E_2$ on $S$ and we want to show that
\[
\mathrm{Hom} (E_1,E_2)  \simeq \mathrm{Hom} (E_{1T}, E_{2T}).
\]
We know that there exists a dense open subset $U$ of $X$ such that the restrictions $E_1''$ and $E_2''$ to $U$ of $E_1$ and $E_2$ are finitely presented.
We denote by $E_1'$ and $E_2'$ the restrictions of $E_1$ and $E_2$ to the complement $Z$ of $U$.
We will denote as usual by $\alpha : U \hookrightarrow X$ (resp.\  $\beta : Z \hookrightarrow X$) the inclusion map. 
Our assumptions imply that the result holds for $E_1''$ and $E_2''$ (resp.\ for $E'_1$ and $E'_2$ by induction).
Therefore, we have by adjunction,
for any constructible crystal $E$  on $S$ (assuming $E_{|U}$ is finitely presented in the first case),
\begin{align*}
&\mathrm{Hom} (E, \alpha_*E_2'')  \simeq \mathrm{Hom} (E_{T}, \alpha_*E_{2T}'') \quad \mathrm{and}
\\
&\mathrm{Hom} (E,\beta_\dagger E_2')  \simeq \mathrm{Hom} (E_{T},\beta_\dagger E_{2T}').
\end{align*}
Recall also from lemma \ref{exthom} that we actually have
\begin{align*}
&\mathrm{Hom} (\beta_\dagger E_1', \alpha_*E_2'')  = \mathrm{Hom} (\beta_\dagger E_{1T}', \alpha_*E_{2T}'') = 0  \quad \mathrm{and}
\\ &\mathrm{Hom} (\alpha_*E_1'',\beta_\dagger E_2')  = \mathrm{Hom} (\alpha_*E_{1T}'', \beta_\dagger E_{2T}') = 0.
\end{align*}
As a consequence, the upper map as well as the vertical maps are bijective in the diagram
\[
\xymatrix{
\mathrm{Hom} (\beta_\dagger E_1',\beta_\dagger E_2') \ar[d]^-\simeq \ar[r]^-\simeq & \mathrm{Hom} (\beta_\dagger E_{1T}',\beta_\dagger E_{2T}') \ar[d]^-\simeq \\ 
\mathrm{Hom} (\beta_\dagger E_1',E_2) \ar@{^{(}->}[r]  & \mathrm{Hom} (\beta_\dagger E_{1T}', E_{2T}).
}
\]
It follows that we also have
\begin{align} \label{pf4}
\mathrm{Hom} (\beta_\dagger E_1', E_2)  \simeq \mathrm{Hom} (\beta_\dagger E_{1T}', E_{2T}).
\end{align}
Assume now that we are given an extension
\[
0 \to \beta_\dagger E_2' \to E \to E_1 \to 0
\]
that splits over $T$.
Then there exists a section $E_T \to \beta_\dagger E_{2T}'$ which comes necessarily from a section $E \to \beta_\dagger E_{2}'$.
It follows that we have an injection
\[
\mathrm{Ext} (E_1,\beta_\dagger E_2') \hookrightarrow \mathrm{Ext} (E_{1T},\beta_\dagger E_{2T}').
\]

We observe now the following commutative diagram with exact columns:
\[
\xymatrix{
0  \ar[d] & 0  \ar[d] \\
\mathrm{Hom} (\alpha_*E_1'', E_2) \ar[d] \ar@{^{(}->}[r] & \mathrm{Hom} (\alpha_*E_{1T}'', E_{2T}) \ar[d] \\ 
\mathrm{Hom} (\alpha_*E_1'', \alpha_*E_2'')  \ar[d] \ar[r]^-\simeq & \mathrm{Hom} (\alpha_*E_{1T}'', \alpha_*E_{2T}'')  \ar[d] \\ 
\mathrm{Ext} (\alpha_*E_1'',\beta_\dagger E_2') \ar@{^{(}->}[r] & \mathrm{Ext} (\alpha_*E_{1T}'',\beta_\dagger E_{2T}').
}
\]
It follows from the five lemma that the upper map is necessarily bijective: we have
\begin{align} \label{pf2}
\mathrm{Hom} (\alpha_*E_1'', E_2)  \simeq \mathrm{Hom} (\alpha_*E_{1T}'', E_{2T}'').
\end{align}
Now, for the same reason as above (the section argument), we have an injection
\[
\mathrm{Ext} (\alpha_*E_1'', E_2) \hookrightarrow \mathrm{Ext} (\alpha_*E_{1T}'', E_{2T}).
\]
We may then consider the commutative diagram with exact columns:
\[
\xymatrix{
0  \ar[d] & 0  \ar[d] \\
\mathrm{Hom} (\alpha_*E_1'', E_2) \ar[d] \ar[r]^\simeq & \mathrm{Hom} (\alpha_*E_{1T}'', E_{2T}) \ar[d] \\
\mathrm{Hom} (E_1, E_2) \ar[d] \ar@{^{(}->}[r] & \mathrm{Hom} (E_{1T}, E_{2T}) \ar[d] \\ 
\mathrm{Hom} (\beta_{\dagger}E_1',E_2)  \ar[d] \ar[r]^-\simeq & \mathrm{Hom} (\beta_{\dagger}E_{1T}', E_{2T})  \ar[d] \\ 
\mathrm{Ext} (\alpha_*E_1'', E_2) \ar@{^{(}->}[r] & \mathrm{Ext} (\alpha_*E_{1T}'', E_{2T}).
}
\]
It is then sufficient to apply the five lemma again to get
\[
\mathrm{Hom} (E_1, E_2) \simeq \mathrm{Hom} (E_{1T}, E_{2T}). \qedhere
\]
\end{proof}

Using the same section argument as in the proof, we see that we always have injective maps
\[
\mathrm{Ext}(E_1,E_2) \hookrightarrow \mathrm{Ext}(E_{1T},E_{2T}).
\]

\subsection{Constructible modules on a tube}

It is also necessary to develop a theory of constructible modules on the small site of an overconvergent space and we explain now how to do that.

\begin{dfn}
Let $(X,V)$ be an overconvergent space.
An $\mathcal O_{V}^\dagger$-module $\mathcal F$ on $X$ is \emph{constructible} if there exists a locally finite covering of $X$ by locally closed formal subschemes $Y$ such that each $\mathcal F_{|\,]Y[_V}$ is a coherent $\mathcal O_{V}^\dagger$-module on $Y$.
\end{dfn}

\begin{xmp}
\begin{enumerate}
\item
The category of constructible \emph{crystals} on $(X,V)$ is equivalent to the category of constructible $\mathcal O_{V}^\dagger$-modules on $X$.
In particular, most results about constructible crystals have an analog for constructible $\mathcal O_{V}^\dagger$-modules as we shall see below.
\item
If $(X', V')$ is any overconvergent space over $T$ and $E$ is a constructible module on $T$, then $E_{V'}$ is a constructible $\mathcal O^\dagger_V$-module on $X$.
But the converse is not true.
One may call such a module $E$ \emph{weakly constructible}.
\end{enumerate}
\end{xmp}

Let us first derive some basic properties.

\begin{lem} \label{consbis}
\begin{enumerate}
\item If $(f,u) : (X', V') \to (X, V)$ is a morphism of overconvergent spaces and $\mathcal F$ is a constructible $\mathcal O_{V}^\dagger$-module on $X$, then $]f[^\dagger\mathcal F$ is also constructible.
\item If $V$ is analytic, $\gamma : Y \hookrightarrow X$ is a formal embedding and $\mathcal G$ is an $\mathcal O_{V}^\dagger$-module on $Y$, then $\mathcal G$ is constructible if and only if $]\gamma[_! \mathcal G$ is constructible.
\end{enumerate}
\end{lem}

\begin{proof}
Follows from lemmas \ref{cons1} and \ref{dirconv} applied to crystals on $(X,V)$.
\end{proof}

\begin{lem} \label{consbis}
Let $(X,V)$ be an overconvergent space.
Then, for an $\mathcal O_{V}^\dagger$-module $\mathcal F$ on $X$, the following are equivalent:
\begin{enumerate}
\item $\mathcal F$ is constructible,
\item there exists a locally finite covering of $X$ by locally closed formal subschemes $Y$ such that each $\mathcal F_{|\,]Y[_V}$ is constructible,
\item there exists a dense open subset $U$ with closed complement $Z$ such that $\mathcal F_{|\,]U[_V}$ is finitely presented and $\mathcal F_{|\,]Z[}$ is $Z$-constructible.
\item \label{consbis3}
(if $V$ is analytic) there exists a formal open embedding $\alpha : U \hookrightarrow X$, a complementary closed embedding $\beta : Z \hookrightarrow X$,  a constructible $\mathcal O_{V}^\dagger$-module $\mathcal F''$ on $U$, a constructible $\mathcal O_{V}^\dagger$-module $\mathcal F'$ on $Z$ and a short exact sequence
\begin{align} \label{extdag}
0 \to ]\beta[_{!} \mathcal F' \to \mathcal F \to ]\alpha[_{*}\mathcal F'' \to 0.
\end{align}
Actually, we can assume that $U$ is dense in $X$ and $\mathcal F''$ is coherent.
\end{enumerate}
\end{lem}

\begin{proof}
Follows from lemma \ref{cons2} and proposition \ref{extcons} applied to crystals on $(X,V)$
\end{proof}

In order to deal with constructible modules, it will be necessary to have a better grasp on that kind of extension that appears in lemma \ref{consbis}.

\begin{lem} \label{extlim}
Let $(X,V)$ be a Tate paracompact space, $\alpha : U \hookrightarrow X$ a formal open embedding,  $\beta : Z \hookrightarrow X$ a complementary closed embedding,  $\mathcal F$ a coherent $\mathcal O_{V}^\dagger$-module on $U$ and $\mathcal G$ a constructible $\mathcal O_{V}^\dagger$-module on $Z$.
Then $\mathcal F$ extends to some $\mathcal F'$ on any sufficiently small neighborhood $V'$ of $U$ in $V$ and 
\[
\mathrm{Ext}_{\mathcal O_{\,]X[_V}^\dagger} (]\alpha[_*\mathcal F,]\beta[_! \mathcal G) \simeq \varinjlim \mathrm {Hom}_{\mathcal O_{\,]Z[_{V'}^\dagger}}(\mathcal F'_{|\,]Z[_{V'}},\mathcal G_{|\,]Z[_{V'}}).
\]
\end{lem}

\begin{proof}
First of all, it is completely formal that
\[
\mathrm{Hom}_{\mathcal O_{\,]X[_V}^\dagger}  (\mathcal F,]\alpha[^{-1}]\beta[_* \mathcal G) \simeq \mathrm{Ext}_{\mathcal O_{\,]U[_{V}}^\dagger}  (]\alpha[_*\mathcal F,]\beta[_! \mathcal G).
\]
We can then use corollary \ref{limhom} which gives
\[
\mathrm{Hom}_{\mathcal O_{\,]X[_V}^\dagger}  (\mathcal F,]\alpha[^{-1}]\beta[_* \mathcal G) \simeq \varinjlim \mathrm {Hom}_{\mathcal O_{\,]X[_{V'}^\dagger}}(\mathcal F', ]\beta[^{-1}\mathcal G_{|V'})\]
and finish by adjunction.
\end{proof}

\subsection{Stratifications on constructible modules}

In order to recover constructible crystals from constructible modules on tubes, it is necessary to endow them with an (overconvergent) stratification or, better, an integrable connection.
We will denote with an index ``$\mathrm{cons}$'' the full subcategories made of objects whose underlying module is constructible.

\begin{prop} \label{consab}
\begin{enumerate}
\item
If $(X,V)$ is an overconvergent space, then
$\mathrm{Mod}_{\mathrm{cons}}(\mathcal O_{V}^\dagger)$ is a weak Serre subcategory of $\mathrm{Mod}(\mathcal O_{V}^\dagger)$.
\item
Let $(X,V) \to (C,O)$ be a morphism of overconvergent spaces.
Assume $V$ has self products and is flat over $O$ in the neighborhood of $X$.
Then, $\mathrm{Strat}_{\mathrm{cons}}(X,V/O)^\dagger$ is a weak Serre subcategory  of $\mathrm{Strat}(X,V/O)^\dagger$.
\end{enumerate}
In both cases, this is an abelian category and the inclusion functor is exact.
\end{prop}

\begin{proof}
Since $\mathrm{Mod}(\mathcal O_{V}^\dagger)$ is abelian and, thanks to lemma \ref{straO}, the same holds for $\mathrm{Strat}(X,V/O)^\dagger$  under our hypothesis, then, in both cases, the last assertion follows from the first one.
The same lemma tells us that the forgetful functor $\mathrm{Strat}(X,V/O)^\dagger \to \mathrm{Mod}(\mathcal O_{V}^\dagger)$ is exact and faithful.
Moreover, restriction is exact.
Therefore, both assertions follow from the analogous statement for coherent $\mathcal O_{V}^\dagger$-modules (\cite[\href{https://stacks.math.columbia.edu/tag/01BY}{Tag 01BY}]{stacks-project}).
\end{proof}

The next result may sound quite specific but we should not forget where we come from:

\begin{prop}
Assume $O = \mathrm{Spa}(K)$ were $K$ is a non-archimedean field.
Let $(X,V) \to (C,O)$ be a morphism of analytic overconvergent spaces with $V$ locally of finite type over $K$.
Then, a constructible $\mathcal O^\dagger_{V}$-module $\mathcal F$ on $X$ endowed with a stratification is flat.
\end{prop}

\begin{proof}
This is a local question.
By induction on the dimension of $X$, we can assume that $\mathcal F$ is coherent: this follows from assertion \eqref{consbis3} of lemma \ref{consbis} since both $]\beta[_{!}$ and $]\alpha[_{*}$ preserve flatness.
Our assertion now follows from theorem \ref{glumd} and (the classical) lemma \ref{locfree} below.
\end{proof}

\begin{lem} \label{locfree}
Let $K$ be a non-archimedean field and $V$ an adic space locally of finite type over $K$.
Then, a coherent $\mathcal O_V$-module $\mathcal F$ endowed with a stratification with respect to $K$ is automatically locally free.
\end{lem}

\begin{proof} This is classic. Since $V$ is locally of finite type, the \emph{rigid} points (meaning $\kappa(v)/K$ is finite) $v$ of $V$ form a conservative family for coherent $\mathcal O_V$-modules.
It is therefore sufficient to show that $\mathcal F_{V,v}$ is a free $\mathcal O_{V,v}$-module for all such $v$.
In this situation, $\mathcal O_{V,v}$ is noetherian and it is then sufficient to show that $\mathcal F_{v}/\mathrm m_{V,v}^n\mathcal F_v$ is free over $A_n := \mathcal O_{V,v}/\mathrm m_{V,v}^n$ for all $n \in \mathbb N$.
We can therefore assume that $V = T_n := \mathrm{Spa}(A_n)$.
Since $T_0 \hookrightarrow T_n$ is nilpotent, we have $\mathrm{Strat}(T_n/K) \simeq \mathrm{Strat}(T_0/K)$.
We may therefore assume that $n=0$ and we are done since $A_0 = \kappa(v)$ is a field.
\end{proof}

The next statement will prove itself fundamental in the interpretation of a crystal as a module with a connection.
The first part of the proof follows the same pattern as the proof of lemma \ref{bijcons} but I was unable to make both results a consequence of a common lemma.

\begin{lem} \label{mainstrat}
If
\[
\xymatrix{X \ar@{^{(}->}[r]\ar[d] & P \ar[d]_v & V \ar[l] \ar[d]\\ C \ar@{^{(}->}[r] & S &O\ar[l]}
\]
is a formal morphism of analytic overconvergent spaces which is formally smooth and locally formally of finite type, then the forgetful functor
\[
\mathrm{Strat}_{\mathrm{cons}}(X,V/O)^\dagger \to \mathrm{Strat}_{\mathrm{cons}}(X,V/O)
\]
is fully faithful.
\end{lem}

\begin{proof}
We already know that the functor is faithful.
We give ourselves two constructible modules $\mathcal F$ and $\mathcal G$ on $\,]X[_V$ endowed with an overconvergent stratification.
We have to show that any morphism which is compatible with the usual stratifications is automatically compatible with the overconvergent ones.

Let us assume for the moment that the result is known for coherent modules (on any such $X$).
Since $\mathcal F$ and $\mathcal G$ are constructible, there exists a dense open subset $U$ of $X$ such that the respective restrictions $\mathcal F''$ and $\mathcal G''$ to $U$ of $\mathcal F$ and $\mathcal G$ are coherent.
The property therefore holds for $\mathcal F''$ and $\mathcal G''$ on $U$.
Let us denote as usual by $\alpha : U \hookrightarrow X$ the inclusion map.
Since $\,]\alpha[_{*}$ is fully faithful (even with respect to stratifications or overconvergent stratifications), the property is satisfied by $\,]\alpha[_{*}\mathcal F''$ and $\,]\alpha[_{*}\mathcal G''$.
In other words, if we shorten our categories names as $\mathrm{Strat}^\dagger$ and $\mathrm{Strat}$, we have a bijection
\[
\mathrm{Hom}_{\mathrm{Strat}^\dagger}(\,]\alpha[_{*}\mathcal F'', \,]\alpha[_{*}\mathcal G'')  \simeq \mathrm{Hom}_{\mathrm{Strat}}(\,]\alpha[_{*}\mathcal F'', \,]\alpha[_{*}\mathcal G'').
\]
We denote now by $\beta : Z \hookrightarrow X$ the inclusion of a closed complement of $U$ and let $\mathcal F'$ and $\mathcal G'$ be the restrictions of $\mathcal F$ and $\mathcal G$ to $Z$.
Assume that we are given an extension
\[
0 \to  \,]\beta[_{!}\mathcal G' \to \mathcal H \to \,]\alpha[_{*}\mathcal F'' \to 0
\]
in $\mathrm{Strat}^\dagger$ that splits as a sequence of $\mathcal O_{V}^\dagger$-modules.
Then, there exists an isomorphism
\[
\mathcal H \simeq \,]\alpha[_{*}\mathcal F'' \oplus  \,]\beta[_{!}\mathcal G'
\]
of $\mathcal O_{V}^\dagger$-modules and the overconvergent stratification of $\mathcal H$ is given by an isomorphism
\[
\epsilon : p_{2}^\dagger(\,]\alpha[_{*}\mathcal F'' \oplus  \,]\beta[_{!}\mathcal G') \simeq p_{1}^\dagger(\,]\alpha[_{*}\mathcal F'' \oplus  \,]\beta[_{!}\mathcal G')
\]
of $\mathcal O_{V(1)}^\dagger$-modules.
Thanks to lemma \ref{comgam}, this isomorphism may be rewritten
\[
\epsilon : \,]\alpha[_{*}p_{2}^\dagger \mathcal F'' \oplus  \,]\beta[_{!}p_{2}^\dagger\mathcal G' \simeq \,]\alpha[_{*}p_{1}^\dagger\mathcal F'' \oplus  \,]\beta[_{!}p_{1}^\dagger\mathcal G'.
\]
Since
\[
\mathrm{Hom} \left(\,]\alpha[_{*}p_{2}^\dagger \mathcal F'',  \,]\beta[_{!}p_{1}^\dagger\mathcal G'\right) = \mathrm{Hom} \left(  \,]\beta[_{!}p_{2}^\dagger\mathcal G', \,]\alpha[_{*}p_{1}^\dagger \mathcal F''\right) = 0,
\]
our sequence will necessary split in $\mathrm{Strat}^\dagger$.
It follows that the map
\[
\mathrm{Ext}_{\mathrm{Strat}^\dagger}(\,]\alpha[_{*}\mathcal F'', \,]\beta[_{!}\mathcal G') \to \mathrm{Ext}_{\mathrm{Strat}}(\,]\alpha[_{*}\mathcal F'', \,]\beta[_{!}\mathcal G')
\]
is injective.
We may now contemplate the following commutative diagram:
\[
\xymatrix{
0  \ar[d] & 0  \ar[d] \\
\mathrm{Hom}_{\mathrm{Strat}^\dagger}(\,]\alpha[_{*}\mathcal F'', \mathcal G) \ar[d] \ar@{^{(}->}[r] & \mathrm{Hom}_{\mathrm{Strat}}(\,]\alpha[_{*}\mathcal F'', \mathcal G) \ar[d] \\ 
\mathrm{Hom}_{\mathrm{Strat}^\dagger}(\,]\alpha[_{*}\mathcal F'', \,]\alpha[_{*}\mathcal G'')  \ar[d] \ar[r]^-\simeq & \mathrm{Hom}_{\mathrm{Strat}}(\,]\alpha[_{*}\mathcal F'', \,]\alpha[_{*}\mathcal G'')  \ar[d] \\ 
\mathrm{Ext}_{\mathrm{Strat}^\dagger}(\,]\alpha[_{*}\mathcal F'', \,]\beta[_{!}\mathcal G') \ar@{^{(}->}[r] & \mathrm{Ext}_{\mathrm{Strat}}(\,]\alpha[_{*}\mathcal F'', \,]\beta[_{!}\mathcal G').
}
\]
The columns are exact because $\mathrm{Hom}(\,]\alpha[_{*}\mathcal F'', \,]\beta[_{!}\mathcal G') = 0$.
It follows from the five lemma that the upper map is necessarily bijective: we have
\begin{align} \label{pf2}
\mathrm{Hom}_{\mathrm{Strat}^\dagger}(\,]\alpha[_{*}\mathcal F'', \mathcal G)  \simeq \mathrm{Hom}_{\mathrm{Strat}}(\,]\alpha[_{*}\mathcal F'', \mathcal G'').
\end{align}
Assume now that we are given an extension
\[
0 \to \mathcal G \to \mathcal H \to \,]\alpha[_{*}\mathcal F'' \to 0
\]
in $\mathrm{Strat}^\dagger$ that splits in $\mathrm{Strat}$.
Then there exists a section $\,]\alpha[_{*}\mathcal F'' \to \mathcal H$ for this sequence in $\mathrm{Strat}$.
From what we already showed, we know that the property that we aim at proving holds for $\,]\alpha[_{*}\mathcal F''$ and $\mathcal H$.
Therefore, the section is actually defined in $\mathrm{Strat}^\dagger$.
It follows that the map
\begin{align} \label{pf5}
\mathrm{Ext}_{\mathrm{Strat}^\dagger}(\,]\alpha[_{*}\mathcal F'', \mathcal G)  \hookrightarrow \mathrm{Ext}_{\mathrm{Strat}}(\,]\alpha[_{*}\mathcal F'', \mathcal G)
\end{align}
is injective.
We proceed now by induction on the dimension of $X$.
Then, the proposition is valid for $\mathcal F'$ and $\mathcal G'$ on $Z$, and since $\,]\beta[_{!}$ is fully faithful (even with respect to stratifications or overconvergent stratifications), it also holds for $\,]\beta[_{!}\mathcal F'$ and $\,]\beta[_{!}\mathcal G'$.
Hence, we have
\begin{align} \label{pf3}
\mathrm{Hom}_{\mathrm{Strat}^\dagger}(\,]\beta[_{!}\mathcal F', \,]\beta[_{!}\mathcal G')  \simeq \mathrm{Hom}_{\mathrm{Strat}}(\,]\beta[_{!}\mathcal F', \,]\beta[_{!}\mathcal G').
\end{align}
We consider now the commutative square
$$
\xymatrix{
\mathrm{Hom}_{\mathrm{Strat}^\dagger}(\,]\beta[_{!}\mathcal F', \,]\beta[_{!}\mathcal G') \ar[d]^-\simeq \ar[r]^-\simeq & \mathrm{Hom}_{\mathrm{Strat}}(\,]\beta[_{!}\mathcal F', \,]\beta[_{!}\mathcal G') \ar[d]^-\simeq \\ 
\mathrm{Hom}_{\mathrm{Strat}^\dagger}(\,]\beta[_{!}\mathcal F',\mathcal G) \ar@{^{(}->}[r]  & \mathrm{Hom}_{\mathrm{Strat}}(\,]\beta[_{!}\mathcal F', \mathcal G).
}
$$
The vertical maps are bijective because $\mathrm{Hom}(\,]\beta[_{!}\mathcal F', \,]\alpha[_{*}\mathcal G'') = 0$.
If follows that we have an isomorphism
\begin{align} \label{pf4}
\mathrm{Hom}_{\mathrm{Strat}^\dagger}(\,]\beta[_{!}\mathcal F', \mathcal G)  \simeq \mathrm{Hom}_{\mathrm{Strat}}(\,]\beta[_{!}\mathcal F', \mathcal G).
\end{align}
Finally, we observe the commutative diagram with exact columns:
\[
\xymatrix{
0  \ar[d] & 0  \ar[d] \\
\mathrm{Hom}_{\mathrm{Strat}^\dagger}(\,]\alpha[_{*}\mathcal F'', \mathcal G) \ar[d] \ar[r]^\simeq & \mathrm{Hom}_{\mathrm{Strat}}(\,]\alpha[_{*}\mathcal F'', \mathcal G) \ar[d] \\
\mathrm{Hom}_{\mathrm{Strat}^\dagger}(\mathcal F, \mathcal G) \ar[d] \ar@{^{(}->}[r] & \mathrm{Hom}_{\mathrm{Strat}}(\mathcal F, \mathcal G) \ar[d] \\ 
\mathrm{Hom}_{\mathrm{Strat}^\dagger}(\,]\beta[_{\dagger}\mathcal F',\mathcal G)  \ar[d] \ar[r]^-\simeq & \mathrm{Hom}_{\mathrm{Strat}}(\,]\beta[_{\dagger}\mathcal F', \mathcal G)  \ar[d] \\ 
\mathrm{Ext}_{\mathrm{Strat}^\dagger}(\,]\alpha[_{*}\mathcal F'', \mathcal G) \ar@{^{(}->}[r] & \mathrm{Ext}_{\mathrm{Strat}}(\,]\alpha[_{*}\mathcal F'', \mathcal G).
}
\]
It is then sufficient to apply the five lemma again to see that the property is satisfied by $\mathcal F$ and $\mathcal G$.

It remains however to treat the case of coherent modules.
There exists now an internal Hom and it is therefore sufficient to prove that if $\mathcal F$ is a coherent $\mathcal O_{V}^\dagger$-module with an overconvergent stratification, then
\[
\mathcal H^{\dagger 0}(\mathcal F) = \mathcal H^0(\mathcal F).
\]
Recall that there exists,  for all $n \in \mathbb N$, a commutative diagram
\[
\xymatrix{V(1) \ar[rd]^{p_{1}} \ar@/^1pc/[rrd]^{p(1)} \\ & V \ar[r]^p & O \\ V^{(n)} \ar@{^{(}->}[uu] \ar[ru]^{p_{1}^{(n)}} \ar@/_1pc/[rru]_{p^{(n)}(1)} }
\]
where $V(1) := V \times_O V$ and $V^{(n)}$ denotes the $n$th infinitesimal neighborhood of $V$ in $V(1)$.
If we still denote by the same letters the maps induced on the tubes, then there exists a morphism of left exact sequences
\[
\xymatrix{
0 \ar[r] & \mathcal H^{0\dagger}(\mathcal F) \ar[r] \ar[d] & p_{*}\mathcal F \ar[r] \ar@{=}[d] & p(1)_{*}p_{1}^\dagger \mathcal F \ar[d]
\\
0 \ar[r] & \mathcal H^{0}(\mathcal F) \ar[r] & p_{*}\mathcal F \ar[r] &  \varprojlim_{n} p^{(n)}_{*}p_{1}^{(n)\dagger} \mathcal F.
}
\]
It is therefore sufficient to show that the right hand map is injective and we will actually show that  the canonical map
\[
p_{1*}p_{1}^\dagger \mathcal F \to   \varprojlim_{n} p_{1}^{(n)\dagger} \mathcal F
\]
is injective (since we may then push down onto $O$).
Recall that we may always assume that $X$ is open in $P$ after replacing $P$ with its completion along $\overline X$.
We can then consider the open immersions $j :X_V \hookrightarrow \,]X[_V$ and $j(1): X_{V(1)} \hookrightarrow \,]X[_{V(1)}$ of the fibers (which are also the naive tubes here).
Since the fibers are dense in the tubes, the map
\[
p_{1*}p_{1}^\dagger \mathcal F \to j_*p_{1*}j(1)^{-1}p_{1}^\dagger \mathcal F
\]
(where $p_{1*}$ denotes the projection for the fibers on the right) is injective.
It is therefore sufficient to prove our assertion with the tubes replaced by the fibers.
Now, since the fibers are open, we may actually assume that $V = X_V$ and therefore also that $P = X$ (and $V = \,]X[_P$).
Since the question is local and $v$ is formally smooth and locally formally of finite type, it is differentially smooth.
We may therefore assume that there exists a formally étale map $P \to \mathbb A^d_S$, or in other words, formally étale coordinates $\underline x$ (we use multi-index notation).
Then we also obtain étale coordinates $\underline \xi := 1 \otimes \underline x - \underline s \otimes 1$ for the (first) projection $p_1 = P(1) \to P$.
The corresponding formally étale map $P(1) \to \mathbb A^d_P$ provides an isomorphism $\,]X[_{V(1)} \simeq \mathbb D^{-,d}_V$ (first assertion of lemma 4.11 of \cite{LeStum17*}).
Also, if we set (only in this proof) $Z^{(n)} := \mathrm{Spv}(\mathbb Z[\underline \xi]/(\underline \xi)^n)$, then there exists an isomorphism $V^{(n)} \simeq V \times Z^{(n)} =: Z^{(n)}_V$.
We consider now the following diagram
\[
\xymatrix{\mathbb D^{-,d}_V \ar[rd]^{p} \\ & V \\  Z^{(n)}_V \ar@{^{(}->}[uu] \ar[ru]^{p^{(n)}} }
\]
and we have to show that if $\mathcal F$ is a coherent sheaf on $V$, then the map
\[
p_{*}p^* \mathcal F \to   \varprojlim_{n} p^{(n)*} \mathcal F
\]
is injective.
This is done as in lemma 3.4.9 of \cite{LeStum11}.
We can assume that $V = \mathrm{Spa}(A,A^+)$ where $A$ is a Tate algebra with topologically nilpotent unit $\pi$ and it is sufficient to prove that
\[
\Gamma(\mathbb D^{-,d}_V, p^* \mathcal F) \to \Gamma(V, \varprojlim_{n} p^{(n)*} \mathcal F)
\]
is injective.
Actually, we can replace $\mathbb D^{-,d}_V$ with a closed polydisc $\mathbb D_V^{d}(0, \pi^{\frac 1n})$ (because they form on open covering) and then take the limit.
If we write $M := \Gamma(V, \mathcal F)$, we are therefore reduced to showing that the map
\[
A\{\underline \xi/\pi^{\frac 1n}\} \otimes_A M \to A[[\underline \xi]] \otimes_A M
\]
is injective.
By induction on the number of generators of $M$, we can assume that $M$ is a quotient of $A$ so that $M$ becomes itself a Tate algebra and we may finally assume that $M=A$ in which case this is clear.
\end{proof}

In the course of the proof, we showed that
\[
\mathcal H^{\dagger 0}(\mathcal F) \simeq \mathcal H^0(\mathcal F)
\]
when $\mathcal F$ is coherent.
Actually, the same is true more generally when $\mathcal F$ is constructible because then
\[
\xymatrix{
\Gamma(]C[_O, \mathcal H^{\dagger 0}(\mathcal F)) \ar[d]^{\simeq} &  \Gamma(]C[_O, \mathcal H^{0}(\mathcal F)) \ar[d]^{\simeq}  \\  \mathrm{Hom}_{\mathrm{Strat}(X,V/O)^\dagger}(\mathcal O^\dagger_{V}, \mathcal F) \ar[r]^{\simeq}  &\mathrm{Hom}_{\mathrm{Strat}(X,V/O)}(\mathcal O^\dagger_{V}, \mathcal F)
}
\]
(and the same holds for any open subset of $O$).

Note also that it formally follows from the theorem that if $\mathcal F$ and $\mathcal G$ are constructible $\mathcal O_V^\dagger$-modules endowed with an overconvergent stratification, then we have an injection
\[
\mathrm{Ext}_{\mathrm{Strat}^\dagger}(\mathcal F, \mathcal G) \hookrightarrow \mathrm{Ext}_{\mathrm{Strat}}(\mathcal F, \mathcal G).
\]
Unfortunately, it was necessary to do it by hand several times in the first part of the proof.

\section{Cohomology}

We will show here that a constructible crystal may be seen as a module with an integrable connection and that its cohomology is computed as a de Rham cohomology.

\subsection{The comparison theorem}

We show that a constructible crystal may be seen as a constructible module on a tube endowed with an integrable connection.
If $T$ is an overconvergent site, we denote with an index ``$\mathrm{cons}$'' the full subcategory of $\mathrm{Mod}(T)$ (resp.\ $\mathrm{Cris}(T)$) consisting of constructible modules (resp.\ constructible crystals) on $T$.

\begin{lem} \label{adlem}
Let $(X,V) \to (C,O)$ be a morphism of analytic overconvergent spaces.
Assume that $V$ has self products and is flat over $O$ in the neighborhood of $X$.
Then $\mathrm{Cris}_\mathrm{cons}(X,V/O)^\dagger$ is an abelian category and the realization functor is an equivalence
\[
\mathrm{Cris}_\mathrm{cons}(X,V/O)^\dagger \simeq \mathrm{Strat}_\mathrm{cons}(X,V/O)^\dagger.
\]
\end{lem}

\begin{proof}
It will follow from proposition \ref{consab} and lemma \ref{straO} that our category is abelian.
Now, we already know from proposition \ref{equiV} that
\[
\mathrm{Cris}(X,V/O)^\dagger \simeq \mathrm{Strat}(X,V/O)^\dagger.
\]
Moreover, the morphism of overconvergent sites $(X,V) \to (X,V/O)^\dagger$ being a local epimorphism, it detects constructibility.
Since
\[
\mathrm{Cris}_\mathrm{cons}(X,V) \simeq \mathrm{Mod}_\mathrm{cons}(\mathcal O_V^\dagger),
\]
the equivalence follows from the faithfulness of the forgetful functors.
\end{proof}

There exists a variant of the lemma:
if $V$ is a good realization for $X$ over $O$, then $\mathrm{Cris}_\mathrm{cons}(X/O)^\dagger$ is an abelian category and the realization functor is an equivalence
\[
\mathrm{Cris}_\mathrm{cons}(X/O)^\dagger \simeq \mathrm{Strat}_\mathrm{cons}(X,V/O)^\dagger.
\]
We can however do a lot better (in the classical case, this is the main result of \cite{LeStum16}):

\begin{thm} \label{crismic}
If $V$ is a geometric materialization of a formal scheme $X$ over an analytic space $O$ defined over $\mathbb Q$, then $\mathrm{Cris}_\mathrm{cons}(X/O)^{\dagger}$ is an abelian category and
\[
\mathrm{Cris}_\mathrm{cons}(X/O)^{\dagger} \simeq \mathrm{MIC}_\mathrm{cons}(X,V/O)^{\dagger}.
\]
\end{thm}

\begin{proof}
We know from theorem \ref{strfib} that $(X,V/O)^\dagger \to (X/O)^\dagger$ is a local isomorphism.
It therefore follows from lemmas \ref{loccrys} and \ref{cons1} that
\[
\mathrm{Cris}_\mathrm{cons}(X/O)^\dagger \simeq \mathrm{Cris}_\mathrm{cons}(X,V/O)^\dagger
\]
and that this is an abelian category.
Now we showed in lemma \ref{adlem} that
\[
\mathrm{Cris}_\mathrm{cons}(X,V/O)^\dagger  \simeq \mathrm{Strat}_\mathrm{cons}(X,V/O)^\dagger.
\]
Finally, lemmas \ref{MICQ} and \ref{mainstrat} imply that
\[
\mathrm{Strat}_\mathrm{cons}(X,V/O)^\dagger \simeq \mathrm{MIC}_\mathrm{cons}(X,V/O)^\dagger. \qedhere
\]
\end{proof}

In the situation of the theorem, we have the following comparison diagram
\[
\xymatrix{
\mathrm{Cris}_\mathrm{cons}(X/O)^\dagger \ar[r]^-\simeq \ar[d]^\simeq &  \mathrm{MIC}_\mathrm{cons}(X,V/O)^\dagger \ar@{^{(}->}[r] &\mathrm{MIC}_\mathrm{cons}(X,V/O)
\\ \mathrm{Cris}_\mathrm{cons}(X,V/O)^\dagger \ar[r]^-\simeq &  \mathrm{Strat}_\mathrm{cons}(X,V/O)^\dagger \ar@{^{(}->}[r] \ar[u]^\simeq &\mathrm{Strat}_\mathrm{cons}(X,V/O) \ar[u]^\simeq.}
\]

\begin{xmp}
\begin{enumerate}
\item (Overconvergent isocrystals)
Let $\mathcal V$ be discrete valuation ring with residue field $k$ and fraction field $K$ of characteristic zero.
Let $S$ be a formal scheme which is locally topologically of finite type over $\mathcal V$.
Let $X$ be a variety over $S_k$.
Then the category of overconvergent isocrystals on $X/S$ (definition 2.3.6 of \cite{Berthelot96c*}) is equivalent to the category of finitely presented modules on $X/S^{\mathrm{an}}$.
\item (Finite separable field extension)
Let $\mathcal V$ be a discrete valuation ring with residue field $k$ and fraction field $K$ of characteristic zero.
Let $l/k$ be a finite separable extension of $k$.
Then there exists an unramified extension $\mathcal W$ of $\mathcal V$ with residue field $l$.
If we denote by $L$ the fraction field of $\mathcal W$, then the category of constructible (or equivalently finitely presented) crystals on $l/\mathcal V$ is equivalent to the category of finite dimensional $L$-vector spaces.
\item (Lazda-Pàl version)
Let $\mathcal V$ be a discrete valuation ring with perfect residue field $k$ of characteristic $p > 0$ and fraction field $K$ of characteristic zero.
We endow $\mathcal V[[t]]$ with the $p$-adic topology and we work over the overconvergent space
\[
(\eta_{k} := \mathrm{Spec}(k((t))) \hookrightarrow \mathbb A^{\mathrm{b}}_{\mathcal V} := \mathrm{Spf}(\mathcal V[[t]])\leftarrow \mathbb D^{\mathrm{b}}_{K} = \mathrm{Spa}(K \otimes_{\mathcal V} \mathcal V[[t]])).
\]
Any finite separable extension of $\eta_k$ has the form $\eta_l := \mathrm{Spec}(l((s)))$ for a finite separable extension 
$l$ of $k$.
If $\mathcal W$ is an unramified lifting of $l$ with fraction field $L$ and $\mathcal E_L^{\dagger}$ denotes the bounded Robba ring of $L$, then the category of constructible crystals on $\eta_l/\mathbb D^{\mathrm{b}}_{K}$ is equivalent to the category of finite dimensional $\mathcal E_L^{\dagger}$-vector spaces.
\end{enumerate}
\end{xmp}

\subsection{The Roos complex}

We will explain here a method introduced by Abe and Lazda for studying constructibility in \cite{AbeLazda22}.

We let $(X \hookrightarrow P \leftarrow V)$ be a \emph{Tate} (there exists a topologically nilpotent unit defined on $V$) paracompact \emph{convergent} ($X$ closed in $P$) space.

We start with the following technical condition:

\begin{dfn}
A sheaf $\mathcal F$ on $\,]X[_V$ is \emph{affinoid-acyclic} if, whenever $W$ is an open affinoid subspace of  $\,]X[_V$ that factors as
\[
\xymatrix{Q^{\mathrm{ad}} \ar@{^{(}->}[d] & W \ar[l] \ar@{^{(}->}[d]\\ P^{\mathrm{ad}}  & V \ar[l]}
\]
with $Q$ an affine open subset of $P$, then $\mathrm H^k(W, \mathcal F) = 0$ for $k > 0$.
\end{dfn}

In order to study the stability of affinoid-acyclicity under open embeddings, we introduce the following:

\begin{dfn}
A formal open embedding $U \hookrightarrow X$ \emph{very affine} if $U$ is the disjoint union of open subsets  whose locus at infinity is locally principal (the embedding is then affine).
\end{dfn}

\begin{lem} \label{afcov}
Let $U \hookrightarrow X$ be a very affine formal open embedding.
Assume $P$ is affine.
Then, there exists a decreasing family $\{V_m\}_{m \in \mathbb N} $ of open subsets of $V$ such that, if $W \subset \,]X[_V$ is an affinoid open subset, then the family $\{W \cap V_m\}_{m \in \mathbb N} $ is a cofinal system of affinoid neighborhoods of $\,]U[_{W}$ in $W$.
\end{lem}

\begin{proof}
We can replace $U$ with one of its connected components and then $P$ with its completion along the closure $\overline U$ of $U$ in $P$.
In this situation, the complement of $U$ in $P$ is a hypersurface $h = 0$.
We set for $m \in \mathbb N$:
\[
V_m := \left\{v \in V \colon v(\pi) \geq v(h^m) \right\}.
\]
Then $W \cap V_m$ is a rational open subset of $W$ and they form a cofinal system of affinoid open neighborhoods of $\,]U[_W$ thanks to corollary 4.30 of \cite{LeStum17*}.
\end{proof}

In order to study the stability of affinoid-acyclicity under closed embeddings, we introduce the following:

\begin{dfn}
Let $\beta : Z \hookrightarrow X$ be a closed formal embedding and $\mathcal F$ an abelian sheaf on $\,]Z[_V$.
If, for $n \in \mathbb N$, we denote by $j_n : [Z]_{V,n} \hookrightarrow \,]X[_V$ the inclusion map, then the \emph{Roos complex} of $\mathcal F$ is
\[
\mathcal R_\beta^\bullet \mathcal F \simeq \left[\begin{array} c \prod_n j_{n*}\mathcal F_{|[Z]_{V,n}} \\ \downarrow d \\  \prod_n  j_{n*}\mathcal F_{|[Z]_{V,n}} \end{array} \right]
\]
with $d(\{s_n\}_{n \in \mathbb N}) = \{ s_n - s_{n+1|V_n} \}_{n \in \mathbb N}$.
\end{dfn}

We can now state and prove stability for affinoid-acyclicity:

\begin{prop} \label{acyc}
\begin{enumerate}
\item 
If $\alpha : U \hookrightarrow X$ is a very affine formal open embedding and $\mathcal F$ is an affinoid-acyclic abelian sheaf on $\,]X[_V$, then $]\alpha[_*]\alpha[^{-1}\mathcal F$ is also affinoid-acyclic.
\item 
If $\beta : Z \hookrightarrow X$ is a closed formal embedding and $\mathcal F$ an affinoid-acyclic abelian sheaf on $\,]Z[_V$, then the terms of $\mathcal R_\beta\mathcal F$ are affinoid-acyclic and
\[
\mathrm R ]\beta[_*\mathcal F \simeq \mathcal R_\beta^\bullet\mathcal F.
\]
\end{enumerate}
\end{prop}

\begin{proof}
For the first assertion, we have to show that if $P$ is affine and $V = \,]X[_V$ is affinoid, then $\mathrm H^k(\,]U[_V, ]\alpha[^{-1}\mathcal F) = 0$ for $k > 0$.
We know from lemma \ref{afcov} that there exists a cofinal system of affinoid neighborhoods $V_m$ of $\,]U[_V$ and therefore
\[
\mathrm H^k(\,]U[_V, ]\alpha[^{-1}\mathcal F) = \varinjlim_{n\in \mathbb N} \mathrm H^k(V_n,  \mathcal F) = 0
\]
when $k > 0$ by proposition \ref{limcoh}.

The second assertion is exactly proposition 2.2.6 of \cite{AbeLazda22}.
\end{proof}

\begin{lem} \label{Gamloc}
In the situation of proposition \ref{acyc}, when $P$ is affine and $V$ is affinoid, we have
\begin{enumerate}
\item 
$
\Gamma(\,]X[_V, ]\alpha[_*]\alpha[^{-1}\mathcal F) =   \varinjlim_{\varphi} \Gamma(V_\varphi, \mathcal F)
$
when $V_\varphi$ runs through the open neighborhoods of $\,]U[_V$ in $V$,
\item
$
\Gamma(\,]X[_V, \mathcal R^0_{I} \mathcal F) = \Gamma(\,]X[_V, \mathcal R^1_{I} \mathcal F)=   \prod_{n} \Gamma([Z]_{V,n}, \mathcal F).
$
\end{enumerate}
\end{lem}

\begin{proof}
The first assertion follows from proposition \ref{corgam} and the second one is clear since global sections commutes with infinite products.
\end{proof}

Recall from definition \ref{fildef} that a filtration
\[
X = X_0 \supset X_1 \supset \cdots \supset X_d \supset X_{d+1} = \emptyset
\]
is said to be \emph{constructible} if, for $i=0, \ldots, d$, each $X_{i+1}$ is closed and nowhere dense in $X_i$.
If we denote by $\gamma_i : U_i \hookrightarrow X$ the inclusion map, and $I = (i_1, \ldots, i_k)$ is an ordered subset of $(0, \ldots, d)$, then we set
\[
]\gamma_{I}[_*]\gamma_{I}[^{-1}\mathcal F := ]\gamma_{i_1}[_*]\gamma_{i_1}[^{-1} \ldots ]\gamma_{i_k}[_*]\gamma_{i_k}[^{-1} \mathcal F.
\]
They form a diagram in an obvious way and we have:

\begin{lem} \label{limfa}
If $X$ is endowed with a constructible filtration and $\mathcal F$ is an abelian sheaf on $\,]X[_V$, then
\[
\mathcal F \simeq \varprojlim ]\gamma_{\bullet}[_*]\gamma_{\bullet}[^{-1}\mathcal F.
\]
\end{lem}

\begin{proof}
This is proved exactly as (and may also be derived from) proposition \ref{limstr}.
\end{proof}

We will actually use the derived version: if $\mathcal F$ is a complex of abelian sheaves, it follows from lemma \ref{flaskgm} that
\[
\mathrm R ]\gamma_{I}[_*]\gamma_{I}[^{-1}\mathcal F := \mathrm R ]\gamma_{i_1}[_*]\gamma_{i_1}[^{-1} \ldots \mathrm R ]\gamma_{i_k}[_*]\gamma_{i_k}[^{-1} \mathcal F,
\]
\emph{is} the derived functor.
As a consequence, we will have
\[
\mathcal F \simeq \mathrm R\varprojlim \mathrm R ]\gamma_{\bullet}[_*]\gamma_{\bullet}[^{-1}\mathcal F
\]
(this is the original statement of Abe and Lazda in proposition 2.1.1 of \cite{AbeLazda22}).

\begin{dfn}
Assume $X$ is endowed with a constructible filtration $\{X_i\}_{i=0}^d$ and $\mathcal F$ is an abelian sheaf on $\,]X[_V$.
If we denote by $\beta_{ij} : X_j \hookrightarrow X_i$ and $\alpha_i : U_i := X_i \setminus X_{i+1} \hookrightarrow X_i$ the inclusion maps, then the \emph{Roos complex} of $\mathcal F$ on $I := (i_1, \ldots, i_k)$ is
\[
\mathcal R_{I}^\bullet\mathcal F := \mathcal R_{\beta_{0i_1}}^\bullet]\alpha_{i_1}[_{*}]\alpha_{i_1}[^{-1} \mathcal R_{\beta_{i_1i_2}}^\bullet]\alpha_{i_2}[_{*}]\alpha_{i_2}[^{-1} \ldots \mathcal R_{\beta_{i_{k-1}i_k}}^\bullet]\alpha_{i_k}[_{*}]\alpha_{i_k}[^{-1}]\beta_{0i_k}[^{-1}\mathcal F.
\]
\end{dfn}

We say that the filtration of $X$ is \emph{good} is each $U_i \hookrightarrow X_i$ is very affine.

\begin{lem} \label{RR}
We assume that $X$ is endowed with a good constructible filtration $\{X_i\}_{i=0}^d$ and we write $U_i = X_i \setminus X_{i+1}$.
We let $\mathcal F$ be a sheaf of $\mathcal O_V^\dagger$-modules on $X$.
If we are given $I = (i_1, \ldots, i_k)$ such that $\mathcal F$ is a coherent $\mathcal O_V^\dagger$-module on $U_{i_k}$, then the terms of $\mathcal R_{I}(\mathcal F)$ are affinoid-acyclic and
\[
\mathrm R ]\gamma_{I}[_*]\gamma_{I}[^{-1}\mathcal F \simeq \mathcal R_{I}^\bullet(\mathcal F).
\]
\end{lem}

\begin{proof}
We keep the previous notations so that we have
\[
\xymatrix{U_0 \ar@{^{(}->}[d]_{\alpha_0} & U_1 \ar@{^{(}->}[d]_{\alpha_1} & U_2 \ar@{^{(}->}[d]_{\alpha_2} & \cdots & U_{d-1} \ar@{^{(}->}[d]_{\alpha_{d-1}}& U_{d} \ar@{=}[d]
\\ X_0 & X_1  \ar@{_{(}->}[l]^{\beta_{01}} & X_2 \ar@{_{(}->}[l]^{\beta_{12}} & \cdots \ar@{_{(}->}[l] & X_{d-1} \ar@{_{(}->}[l] & X_{d} \ar@{_{(}->}[l]^{\beta_{(d-1)d}}
}
\]
and $\gamma_i = \beta_{0i} \circ \alpha_i$.
Now, if $i < j$, then
\[
]\beta_{0i}[^{-1} \mathrm R  ]\beta_{0j}[_* = ]\beta_{0i}[^{-1} \mathrm R  ]\beta_{0i}[_*  \mathrm R ]\beta_{ij}[_* = \mathrm R ]\beta_{ij}[_*,
\]
and therefore
\begin{align*}
\mathrm R ]\gamma_i[_*]\gamma_i[^{-1} ]\mathrm R \gamma_j[_*]\gamma_j[^{-1} &= \mathrm R ]\beta_{0i}[_* ]\alpha_i[_*]\alpha_i[^{-1}]\beta_{0i}[^{-1} \mathrm R  ]\beta_{0j}[_* ]\alpha_j[_*]\alpha_j[^{-1}]\beta_{0j}[^{-1}
\\ &= \mathrm R ]\beta_{0i}[_* ]\alpha_i[_*]\alpha_i[^{-1} \mathrm R  ]\beta_{ij}[_* ]\alpha_j[_*]\alpha_j[^{-1}]\beta_{0j}[^{-1}.
\end{align*}
It follows that
\[
\mathrm R ]\gamma_{I}[_*]\gamma_{I}[^{-1}\mathcal F:= \mathrm R ]\beta_{0i_1}[_* ]\alpha_{i_1}[_{*}]\alpha_{i_1}[^{-1}  \ldots \mathrm R ]\beta_{i_{k-1}i_k}[_*]\alpha_{i_k}[_{*}]\alpha_{i_k}[^{-1} ]\beta_{0i_k}[^{-1} \mathcal F
\]
and our assertion is therefore consequence of lemma \ref{acyc} and proposition \ref{hknul} for the right hand side.
\end{proof}

\begin{lem} \label{GamRoos}
In the situation of lemma \ref{RR}, if $P$ is affine and $V$ is affinoid, then for all $i = 0, \dots, 2k$, we have
\[
\Gamma(\,]X[_V, \mathcal R^i_{I} \mathcal F) =   \prod_{n_1} \varinjlim_{m_1} \cdots \prod_{n_k} \varinjlim_{m_k}\Gamma(V_{\underline n, \underline m}, \mathcal F)
\]
with each $V_{\underline n, \underline m}$ affinoid.
\end{lem}

\begin{proof}
Thanks to lemma \ref{afcov}, there exists for each $i = 0, \ldots, d$, a decreasing family $\{V_{i,m}\}_{m \in \mathbb N} $ of open subsets of $V$ such that, for all $n \in \mathbb N$, the family $\{[X_{i}]_{V,n} \cap V_{i,m}\}_{m \in \mathbb N}$ is a cofinal system of affinoid neighborhoods of $[X_{i}]_{V,n} \cap \,]U_i[_{V}$.
We set
\[
V_{\underline n, \underline m} =\bigcap_{j=0}^k\left([X_{i_j}]_{V,n_j} \cap V_{i_j, m_j}\right).
\]
If we apply recursively lemma \ref{Gamloc}, then we obtain the expected formula.
\end{proof}

\subsection{The Poincaré lemma}

As usual, Poincaré lemma is the key to the cohomological comparison theorem.
We assume in this section that all overconvergent spaces live over $\mathbb Q$ but we first recall the following general lemma that we will need later:

\begin{lem} \label{limco}
Let $V = \bigcup_{k \in\mathbb N} V_k$ be an increasing open covering of a topological space. 
Let $\mathcal F$ be a complex of abelian sheaves on $V$ such that
\[
\forall i > 1, \forall k \in \mathbb N, \forall j \in \mathbb N, \quad \mathrm H^i(V_k, \mathcal F^j) = 0.
\]
Then
\[
\mathrm R\Gamma(V, \mathcal F) \simeq \left[\begin{array} c \prod_k \Gamma(V_k, \mathcal F) \\ \downarrow d \\  \prod_k \Gamma(V_k, \mathcal F) \end{array} \right]
\]
with $d(\{s_k\}_{k \in \mathbb N}) = \{ s_k - s_{k+1|V_k} \}_{k \in \mathbb N}$.
\end{lem}

\begin{proof}
We always have $\Gamma(V,\mathcal F) \simeq \mathrm \varprojlim_k \Gamma(V_k,\mathcal F)$ and our hypothesis implies that $\Gamma(V_k,\mathcal F) \simeq \mathrm R \Gamma(V_k,\mathcal F)$.
It follows that
\[
\mathrm R\Gamma(V,\mathcal F) \simeq \mathrm R\varprojlim_k \Gamma(V_k,\mathcal F)
\]
and we may therefore use the explicit description of derived filtered limits (\cite[\href{https://stacks.math.columbia.edu/tag/08TC}{Tag 08TC}]{stacks-project}).
\end{proof}

\begin{lem}[Poincar\'e lemma]\label{Poinc}
If $(f,u) : (X',V') \to (X,V)$ is a geometric materialization that induces an isomorphism $f \colon X' \simeq X$ and $\mathcal F$ is a constructible $\mathcal O_V^\dagger$-module on $X$, then
\[
\mathcal F \simeq \mathrm R ]f[_{u,\mathrm{dR}}]f[_u^\dagger \mathcal F.
\]
\end{lem}

We stress out the fact that $]f[_u^\dagger \mathcal F$ is endowed with the trivial connection.

\begin{proof}
First of all, we may assume that the intermediate formal scheme $P$ in $(X,V)$ is $X$ itself: if we denote by $h : X \hookrightarrow P$ the embedding, then we will have
\begin{align*}
\mathcal F & \simeq ]h[^{-1}]h[_! \mathcal F
\\ & \simeq ]h[^{-1} \mathrm R u_{\mathrm{dR}}u^\dagger ]h[_! \mathcal F
\\ & \simeq \mathrm R ]f[_{u,\mathrm{dR}} ]h[^{-1} ]h[_! u^\dagger  \mathcal F
\\ & \simeq \mathrm R ]f[_{u,\mathrm{dR}}]f[_u^\dagger \mathcal F.
\end{align*}
Also, using the strong fibration theorem (\cite{LeStum17*}, theorem 5.21), it is sufficient to consider the case where $X'=X$ is embedded into $P' = \mathbb A^{n-}_{P}$ using the zero section.
Moreover, thanks to lemma \ref{redone} below, we may assume that $n=1$.
The new picture is then
\[
\xymatrix{ X \ar@{^{(}->}[r] \ar@{=}[d] & \mathbb A^{-}_{X}  \ar[d]&  \mathbb D^{-}_{V}  \ar[l] \ar[d]^{u} \\
X \ar@{=}[r] & X  & V.  \ar[l]}
\]
There exists a constructible filtration $\{X_i\}_{i=0}^d$ such that, if we write $U_{i} = X_i \setminus X_{i+1}$, then $\mathcal F$ is coherent on $U_i$.
Since the question is local, we may actually assume that $X$ is affine and each $U_i \hookrightarrow X_i$ is very affine.
Recall, that, if we denote for each $i = 0, \ldots, d$, by $\gamma_i : U_i \hookrightarrow X$ the inclusion map, and $I$ is an ordered subset of $(0, \ldots, d)$, then we write
\[
\mathrm R ]\gamma_{I}[_*]\gamma_{I}[^{-1}\mathcal F := \mathrm R ]\gamma_{i_1}[_*]\gamma_{i_1}[^{-1} \ldots \mathrm R ]\gamma_{i_k}[_*]\gamma_{i_k}[^{-1} \mathcal F.
\]
Assume that we can prove that
\begin{align} \label{mainstep}
\mathrm R ]\gamma_{I}[_*]\gamma_{I}[^{-1}\mathcal F \simeq \mathrm R u_{\mathrm{dR}} \mathrm R ]\gamma_{I}[_*]\gamma_{I}[^{-1} u^\dagger \mathcal F.
\end{align}
Then we will have, thanks to lemma \ref{limfa}:
\begin{align*}
\mathcal F & \simeq \mathrm R\varprojlim \mathrm R ]\gamma_{\bullet}[_*]\gamma_{\bullet}[^{-1}\mathcal F
\\ & \simeq \mathrm R\varprojlim\mathrm R u_{\mathrm{dR}} \mathrm R ]\gamma_{\bullet}[_*]\gamma_{\bullet}[^{-1} u^\dagger \mathcal F
\\ & \simeq \mathrm R u_{\mathrm{dR}} \mathrm R\varprojlim \mathrm R ]\gamma_{\bullet}[_*]\gamma_{\bullet}[^{-1} u^\dagger \mathcal F
\\ & \simeq \mathrm R u_{\mathrm{dR}} u^\dagger \mathcal F.
\end{align*}
It remains to show the isomorphism \eqref{mainstep}.
Following Abe and Lazda, we showed in lemma \ref{RR} that there exists a natural isomorphism in the derived category
\[
\mathrm R ]\gamma_{I}[_*]\gamma_{I}[^{-1}\mathcal F \simeq \mathcal R_I^\bullet(\mathcal F)
\]
where $\mathcal R^\bullet_I(\mathcal F)$ is the Roos complex of $\mathcal F$ for $I$.
The terms of the Roos complex are affinoid-acyclic and it is therefore sufficient to prove that, if $V$ is Tate affinoid, then
\[
\Gamma(V, \mathcal R^\bullet_I(\mathcal F)) \simeq \mathrm R\Gamma\left(\mathbb D^-_V, \left[\mathcal R^\bullet_I(u^\dagger \mathcal F) \overset \partial \to \mathcal R^\bullet_I(u^\dagger \mathcal F)\right]\right)
\]
where $\partial$ denotes the standard derivation.
If $\pi$ is a topologically nilpotent unit, there exists an increasing affinoid covering
\[
\mathbb D_V^- = \bigcup_{k \in \mathbb N} \mathbb D_V(0, \pi^{1/k}).
\]
We may now invoke lemma \ref{limco} and we have to show that
\[
\Gamma(V, \mathcal R^\bullet_I(\mathcal F)) \simeq \left[\begin{array} {ccc} \prod_k \Gamma(\mathbb D_V(0, \pi^{1/k}), \mathcal R^\bullet_I(u^\dagger \mathcal F)) & \overset \partial \to &\prod_k \Gamma(\mathbb D_V(0, \pi^{1/k}), \mathcal R^\bullet_I(u^\dagger \mathcal F)) \\ \downarrow d && \downarrow d \\  \prod_k \Gamma(\mathbb D_V(0, \pi^{1/k}), \mathcal R^\bullet_I(u^\dagger \mathcal F)) & \overset \partial \to &\prod_k \Gamma(\mathbb D_V(0, \pi^{1/k}), \mathcal R^\bullet_I(u^\dagger \mathcal F)) \end{array} \right].
\]
In other words, if we fix $i \in \{0, \ldots, 2k\}$ and we set
\[
M := \Gamma(V, \mathcal R^i_I(\mathcal F)) \quad \mathrm{and} \quad M_k := \Gamma(\mathbb D_V(0, \pi^{1/k}), \mathcal R^i_I(u^\dagger \mathcal F)),
\]
we are reduced to showing that
\[
M
\simeq \left[\begin{array} {ccc}  \prod_k M_k & \overset \partial \to &  \prod_k M_k \\ \downarrow d & & \downarrow d \\   \prod_k M_k & \overset \partial \to &  \prod_k M_k\end{array} \right]
\]
with $d(\{s_k\}_{k \in \mathbb N}) = \{ s_k - s_{k+1|V_k} \}_{k \in \mathbb N}$, or equivalently that the sequence
\begin{align} \label{fundex}
0 \longrightarrow M \longrightarrow  \prod_k M_k \overset {(d, \partial)} \longrightarrow  \prod_k M_k \oplus  \prod_k M_k \overset {\partial - d} \longrightarrow   \prod_k M_k \longrightarrow 0
\end{align}
is exact.

Lemma \ref{GamRoos} provides an explicit description of $M$ and $M_k$.
We let
\[
A := \Gamma(V, \mathcal O_V) \quad \mathrm{so} \ \mathrm{that} \quad \Gamma(\mathbb D_V(0, \pi^{1/k}),  \mathcal O_{\mathbb D_V(0, \pi^{1/k})})) = A\{t/\pi^{1/k}\}
\]
(this is the completion of the polynomial ring $A[t]$ for the topology induced by $A[t^k/\pi]$).
Then, there exists a family of affinoid subsets $\{V_{\underline n, \underline m}\}_{\underline n, \underline m \in \mathbb N^{k+1}}$ of $V$ such that, if we set
\[
M^{\underline n, \underline m} := \Gamma(V_{\underline n, \underline m}, \mathcal F) \quad \mathrm{and} \quad M^{\underline n, \underline m}_k := M^{\underline n, \underline m} \otimes_A A\{t/\pi^{1/k}\},
\]
then
\[
M =   \prod_{n_1} \varinjlim_{m_1} \cdots \prod_{n_k} \varinjlim_{m_k}M^{\underline n, \underline m}
\quad \mathrm{and} \quad 
M_k =   \prod_{n_1} \varinjlim_{m_1} \cdots \prod_{n_k} \varinjlim_{m_k} M^{\underline n, \underline m}_k.
\]
Since $A$ is a $\mathbb Q$-algebra, there exists an integration map
\[
\int : M_k \to M_{k-1}, \quad t^r \mapsto \frac {t^{r+1}}{r+1}
\]
and we may also consider the canonical and the zero section maps:
\[
i : M \to M_k \quad \mathrm{and} \quad p: M_k \to M.
\]
We still use the same letters for the maps induced on the products
\[
\int : \prod_k  M_k \to \prod_k  M_{k}, \quad i : M^\mathrm N \to \prod_k  M_k \quad \mathrm{and} \quad p: \prod_k  M_k \to M^\mathrm N.
\]
One easily checks as in as in claim 6.5.8 of \cite{LeStum07} that
\[
d \circ \int - \int \circ d = 0, \quad \partial \circ \int - \int \circ \partial = i \circ p  \quad \mathrm{and} \quad \partial \circ \int = \mathrm{Id} + d.
\]
One then shows that the sequence \eqref{fundex} is exact as in claim 6.5.9 of \cite{LeStum07}.
\end{proof}

We used above the following:

\begin{lem} \label{redone}
If the hypothesis and conclusions of the Poincaré lemma are satisfied both for $(X'',V'') \to (X',V')$ and for $(X',V') \to (X,V)$, then they are also satisfied for their composition $(X'', V'') \to (X,V)$.
\end{lem}

\begin{proof}
Exactly as lemma 6.5.6 of \cite{LeStum07}.
\end{proof}

\subsection{Cohomology}

We have all the tools now to prove that a constructible crystal is de Rham.
In Berthelot's theory of rigid cohomology, this was only known for finitely presented crystals.
For constructible crystals, this was conjectured just after corollary 4.11 in \cite{LeStum16}).

We assume here that all adic spaces are defined over $\mathbb Q$.

\begin{thm} \label{consdR}
If $(X,V) \to (C,O)$ is a geometric materialization (defined over $\mathbb Q$), then any constructible crystal $E$ on $(X/O)^{\dagger}$ is de Rham on $V$.
\end{thm}

\begin{proof}
We have to show that the canonical map $E \mapsto \mathrm RL_{\mathrm{dR}}E_V$ is bijective.
Equivalently, we need to check that for all overconvergent space $(X',V')$ over $(X/O)^\dagger$, we have
\[
E_{V'} \simeq (\mathrm RL_{\mathrm{dR}}E_V)_{V'}.
\]
We denote by $P$ (resp.\ $P'$) the middle formal scheme and we consider the following morphism of overconvergent spaces
\[
\xymatrix{X' \ar@{^{(}->}[r] \ar@{=}[d] & P' \times_S P  \ar[d]^{p_1}& V' \times_O V  \ar[l] \ar[d]^{p_1} \\
X' \ar@{^{(}->}[r] & P'  & V'  \ar[l].}
\]
(and denote by $p_2$ the second projection).
It follows from proposition \ref{RLstuff2} that
\[
(\mathrm{R}L_{\mathrm{dR}} E_V)_{V'} = \mathrm R]p_{1}[_{\mathrm{dR}}]p_2[^\dagger E_V = \mathrm R]p_{1}[_{\mathrm{dR}} ]p_1[^\dagger E_{V'}
\]
because $E$ is a crystal.
Our assertion therefore follows from the Poincaré lemma \ref{Poinc} applied to the first projection $p_1$ (which is a geometric materialization).
\end{proof}

\begin{cor} \label{cohdR}
If $(f,u) : (X,V) \to (C,O)$ is a geometric materialization (defined over $\mathbb Q$) and $E$ is a constructible crystal on $(X/O)^{\dagger}$, then
\[
\mathrm Rp_{X/O*}E \simeq  \mathrm R]f[_{u,\mathrm{dR}}E_V.
\]
\end{cor}

\begin{proof}
This is a direct consequence of theorems \ref{cohdr} and \ref{consdR}.
\end{proof}

It follows that
\[
\forall k \in \mathbb N, \quad \mathcal H^k({X/O}, E) \simeq  \mathcal H^k_{\mathrm{dR}}(]X[_{V}, E_V)
\]
may be called the \emph{rigid cohomology} of $E$.
In particular, the notion of a finitely presented crystal and its cohomology generalizes the notion of an overconvergent isocrystal and its rigid cohomology in the sense of Berthelot.
This also works for cohomology with support in a subset $Y$ because we will have
\[
\mathrm Rp_{X/O*}j^\dagger_Y E \simeq  \mathrm R]f[_{u,\mathrm{dR}}]\gamma[_!]\gamma[^{-1}E_V.
\]
and therefore
\[
\forall k \in \mathbb N, \quad \mathcal H^k_Y({X/O}, E) \simeq  \mathcal H^k_{\mathrm{dR}}(]X[_{V}, ]\gamma[_!]\gamma[^{-1}E_V).
\]
Note that there is no analog to rigid cohomology with compact support yet. 
This is connected to the difficulty of defining a duality theory for constructible crystals (or equivalently a ``derived'' internal Hom).

In the finitely presented case, one can recover theorem \ref{crismic} form corollary \ref{cohdR} as follows:
\begin{align*}
\mathrm {Hom}(E, E') &= \Gamma\left((X/O)^\dagger, \mathcal H\mathrm{om}(E, E')\right)
\\ &= \Gamma\left(]C[_O, p_{X/O*} \mathcal H\mathrm{om}(E, E')\right)
\\ &= \Gamma\left(]C[_O,  \mathrm R^0]f[_{u,\mathrm{dR}}\mathcal H\mathrm{om}(E, E')_V\right)
\\ &= \Gamma\left(]C[_O,  \mathrm R^0]f[_{u,\mathrm{dR}}\mathcal H\mathrm{om}(E_V, E'_{V})\right)
\\ &= \Gamma\left(]C[_O,  \mathcal H\mathrm{om}_\nabla(E_V, E'_{V})\right)
\\ &=  \mathrm {Hom}_\nabla(E_V, E'_V).\qedhere
\end{align*}

The proof however is not valid in the constructible case due to the lack of internal hom.

Corollary \ref{cohdR} is equivalent to the following relative variant:

\begin{cor} \label{cohdR2}
If $(f,u) : (Y,W) \to (X,V)$ is a geometric materialization over $(C,O)$ and $E$ is a constructible crystal on $(Y/O)^{\dagger}$, then
\[
(\mathrm Rf_*E)_V \simeq  \mathrm R]f[_{u,\mathrm{dR}}E_W. \qed
\]
\end{cor}

\begin{xmp}
\begin{enumerate}
\item (Monsky-Washnitzer setting)
Let $R$ be a noetherian ring (with the discrete topology) and $S \to \mathrm{Spec}(R)$ any morphism of formal schemes.
Let $A$ be a smooth $R$-algebra and $X := \mathrm{Spec}(A)$.
Let $\Lambda$ be a Tate ring over $\mathbb Q$ with topologically nilpotent unit $\pi$ and $O := \mathrm{Spa}(\Lambda, \Lambda^+) \to S^{\mathrm{ad}}$ any morphism.
Then the category of finitely presented crystals on $(X_S/O)^\dagger$ is equivalent to the category of finite $A_\Lambda^\dagger $-modules endowed with an integrable connection with respect to $\Lambda$ and it induces an isomorphism on cohomology (see section \ref{MWsec} below for the details).
\item (The affine space) 
Recall that if $T$ is an overconvergent site, then there exists a morphism of overconvergent sites $f : \mathbb A^{n}_T \to T$.
Let us show that if $T$ is analytic and defined over $\mathbb Q$, then $\mathrm Rf_*\mathcal O^\dagger_{\mathbb A^{n}_T} = \mathcal O^\dagger_{T}$.
By induction, it is sufficient to consider the case $n=1$.
According to proposition \ref{bicom},2), it is sufficient to prove that if $(X \hookrightarrow P \leftarrow V)$ is an analytic overconvergent space, then
\[
\mathrm Rp_{(\mathbb A_{X}/V)*}\mathcal O^\dagger_{(\mathbb A_{X}/V)} \simeq \mathcal O_{V}^\dagger.
\]
This is a local question and we may therefore assume that $V = \mathrm{Spa}(A,A+)$ is Tate affinoid.
Also, using lemma \ref{basfib}, it is sufficient to consider the case $X=P$.
Thanks to corollary \ref{cohdR}, we can use the geometric materialization $(\mathbb A_{P} \hookrightarrow \mathbb P_{P}  \leftarrow \mathbb A_{V})
$ and we need to show that
\[
\mathrm R]f[_{u,\mathrm{dR}}\mathcal O_{\mathbb A_{V}}^\dagger \simeq \mathcal O_{V}^\dagger.
\]
Actually, this boils down to the short exact sequence
\[
0 \longrightarrow A \longrightarrow A[T]^\dagger \overset {\mathrm d/\mathrm d T} \longrightarrow A[T]^\dagger \longrightarrow 0.
\]
\end{enumerate}
\end{xmp}

For further use, let us prove the following:

\begin{lem} \label{finbas}
Let $(f,u) : (X,V) \to (C,O)$ be a geometric materialization over $\mathbb Q$ with $u$ \emph{finite}.
If $E$ is a finitely presented crystal on $(X/O)^\dagger$, then $E$ satisfies base change in cohomology.
Moreover,
\[
p_{X/O*}E = ]f[_{*}E_{V}  \quad \mathrm{and} \quad \mathrm R^kp_{X/O*}E = 0\ \mathrm{for}\ k > 0.
\]
\end{lem}

\begin{proof} 
This is a local question on $O$ and we may therefore assume that $E_{X,V} = i_X^{-1}\mathcal F$ where $\mathcal F$ is a coherent $\mathcal O_V$-module.
Also, since $u$ is smooth in the neighborhood of $X$ and (quasi-) finite, we have $\Omega^{1\dagger}_{\,]X[} =0$ and therefore
\[
\mathrm Rp_{X/O*}E =  \mathrm R]f[_{\mathrm {dR}} E_{V} = \mathrm R]f[_* E_{V} = \mathrm R]f[_* i_X^{-1}\mathcal F.
\]
Since $u$ is finite and $\mathcal F$ coherent, we have $\mathrm R^ku_*\mathcal F = 0$ for $q > 0$.
Therefore, it follows from lemma \ref{baspec} below that
\[
 \mathrm R]f[_* i_X^{-1}\mathcal F = i_C^{-1} \mathrm Ru_* \mathcal F =  i_C^{-1} u_* \mathcal F =  ]f[_* i_X^{-1}\mathcal F = ]f[_{*}E_{V}
\]
and the second assertion is proved.
Now, we give ourselves a morphism $(g,v) : (C',O') \to (C,O)$ and we consider the diagram
\[
\xymatrix{& V' \ar[dd]^{v'} \ar[rr]^{u'} &&O' \ar[dd]^{v} \\ \,]X'[_{V'} \ar[dd]^{]g'[} \ar[rr]^{]f'[} \ar@{^{(}->}[ru] && ]C'[_{O'} \ar[dd]^{]g[} \ar@{^{(}->}[ru]  \\  &V \ar[rr]^{u}&&O \\ \,]X[_{V} \ar[rr]^{]f[} \ar@{^{(}->}[ru]  &&]C[_{O} \ar@{^{(}->}[ru] }
\]
with $X' = X \times_{C} C'$ and $V' = V \times_{O} O'$.
We have to show that
\[
]g[^\dagger \mathrm R^i]f[_{\mathrm {dR}} E_{X,V} \simeq \mathrm R^i]f'[_{\mathrm {dR}} ]g['^\dagger E_{X,V}.
\]
According to the above computations, it reduces to
\[
]g[^\dagger]f[_{*} E_{X,V} \simeq ]f'[_{*} ]g['^\dagger E_{X,V}.
\]
On one hand, there exists the base change formula for finite maps and coherent sheaves:
\[
v^*u_*\mathcal F = u'_*v'^*\mathcal F.
\]
On the other hand, it follows from lemma \ref{baspec} again that
\[
i_C^{-1}u_*\mathcal F = ]f[_*i_X^{-1}\mathcal F \quad \mathrm{and} \quad  i_{C'}^{-1}u'_*v'^*\mathcal F = ]f'[_*i_{X'}^{-1}v'^*\mathcal F.
\]
Putting all this together gives the desired result.
\end{proof}

\begin{prop} \label{fincrys}
Let $(Y, W) \to (X,V) \to (C,O)$ be a sequence of morphisms of overconvergent spaces with $W \to V$ \emph{finite}.
We consider the induced morphism of overconvergent sites $f : (Y/O)^\dagger \to (X,V/O)^\dagger$ and assume that $W$ is a geometric materialization for $X$ over $V$.
Then a finitely presented crystal $E$ on $(Y/O)^\dagger$ is $f$-crystalline.
\end{prop}

\begin{proof}
According to proposition \ref{dRcrys}, this follows from lemma \ref{finbas}.
\end{proof}

We still have to check the base change theorem for locally spectral spaces that we used in the proof of lemma \ref{finbas}.
We call a morphism $g : V' \to V$ \emph{strongly generalizing} if
\[
\forall x' \in V', \forall y \rightsquigarrow g(x'), \exists ! y' \rightsquigarrow x', g(y') = y.
\]
For an inclusion, it simply means that $V'$ is stable under generalization in $V$.
This applies in particular to the inclusion of a tube into an analytic space.

\begin{lem} \label{baspec}
Let
\[
\xymatrix{W' \ar[r]^{f'} \ar[d]^{g'} &V ' \ar[d]^g  \\ W \ar[r]^f & V
}
\]
be a cartesian diagram of locally spectral spaces and $\mathcal F$ a sheaf on $W$.
If $f$ is quasi-compact quasi-separated and $g$ is strongly generalizing, then
\[
g^{-1}\mathrm Rf_*\mathcal F \simeq \mathrm Rf'_*g'^{-1}\mathcal F.
\]
\end{lem}

\begin{proof}
If $x' \in V'$, then we can consider the sets $\mathrm G(x')$ and $\mathrm G(g(x'))$ of generalizations of $x'$ and $g(x')$ respectively.
Then, we know from corollary 1.2.2 of \cite{AbeLazda20} that for all $k \in \mathbb N$, we have
\[
(g^{-1}\mathrm R^kf_*\mathcal F)_{x'} = \mathrm H^k(f^{-1}(\mathrm G(g(x')), \mathcal F)
\]
and
\[
(\mathrm R^kf'_*g'^{-1}\mathcal F)_{x'}= \mathrm H^k(f'^{-1}(\mathrm G(x')), g'^{-1} \mathcal F).
\]
Since $g$ is strongly generalizing and the diagram is cartesian, $g'$ induces a homeomorphism $f'^{-1}(\mathrm G(x')) \simeq f^{-1}(\mathrm G(g(x'))$ and we are done.
\end{proof}

\subsection{Monsky-Washnitzer cohomology} \label{MWsec}

We will explain how our results fit in the (generalized) formalism of Monsky-Washnitzer cohomology.

Let $\Lambda$ be a complete Tate ring defined over $\mathbb Q$ with topologically nilpotent unit $\pi$ and $O := \mathrm{Spa}(\Lambda, \Lambda^+)$.
We first consider the overconvergent space
\[
\mathbb A^N \hookrightarrow \mathbb P^N \leftarrow \mathbb A^N_{O}
\]
which is a geometric materialization of the affine space over $(\mathrm{Spec}(\mathbb Z) = \mathrm{Spec}(\mathbb Z)  \leftarrow O)$.
We have
\[
\,]\mathbb A^N[_{\mathbb A^N_{O}} = \bigcap_{n \in \mathbb N} \mathbb D^N_{O}(0, \pi^{-\frac 1n}) \subset \mathbb A^{N}_{O}
\]
where
\[
\mathbb D^N_{O}(0, \pi^{-\frac 1n}) := \left\{v \in  \mathbb A^N_{O} \colon v(\pi T_i^n) \geq 0 \right\}.
\]
Note that $\,]\mathbb A^N[_{\mathbb A^N_{O}} = \overline {\mathbb D}^N_{O}$ is the proper unit polydisc as a topological space but we want to consider the former as a germ in $\mathbb A^N_{O}$.
If we endow the polynomial ring $\Lambda[T_1, \ldots, T_N]$ with the topology coming from $\Lambda_0[\pi T_1^n, \ldots, \pi T_N^n]$ (with $\Lambda_0$ a ring of definition for $\Lambda$), then we have
\[
\mathbb D^N_{O}(0, \pi^{-\frac 1n})  = \mathrm{Spa}(\Lambda[T_1, \ldots, T_N], \Lambda^+ \cup \{\pi T_1^n, \ldots, \pi T_N^n\}).
\]
It is actually more convenient to introduce the completion
\[
\Lambda\{\pi^{1/n}T_1, \ldots, \pi^{1/n}T_N\} :=
 \left\{\sum_{\underline k \in \mathbb N^N} f_{\underline k} T^{\underline k} \colon \quad \pi^{-|\underline k|}f_{\underline k}^n \to 0 \right\}
\]
so that the topology is encoded into the notations.
We consider the weak completion of the polynomial ring (theorem 2.3 of \cite{MonskyWashnitzer68}):
\[
\Lambda_0[T_1, \ldots, T_N]^\dagger :=  \left\{\sum_{\underline k \in \mathbb N^N} f_{\underline k} T^{\underline k} \colon \quad \underline \lim\left( \frac {\mathrm{ord}(f_{\underline k})}{|\underline k|}\right) > 0 \right\}
\]
where $\mathrm{ord}(f) = \inf \{k \in \mathbb N \colon f \in \pi^k \Lambda_0\}$.
Then, we have
\[
\Lambda_0[T_1, \ldots, T_N]^\dagger[1/\pi] = \varinjlim_n \widehat \Lambda\{\pi^{1/n}T_1, \ldots, \pi^{1/n}T_N\}.
\]

Now, let $R$ be a noetherian ring and $X = \mathrm{Spec}(A)$ a smooth affine scheme over $R$.
If we write $A = R[T_1, \ldots, T_N]/I$, then we obtain a closed embedding $X \hookrightarrow \mathbb A^N_R$ and we may consider the corresponding locally closed embedding $X \hookrightarrow \mathbb P^N_{R}$.
We fix a morphism of formal schemes $S \to \mathrm{Spec}(R)$, a formal \emph{thickening} $C \hookrightarrow S$ and a morphism $O \to S^{\mathrm{ad}}$ of adic spaces.
The overconvergent space
\[
X_{C} \hookrightarrow \mathbb P^N_{S} \leftarrow X_{O}
\]
is a geometric materialization over $(C \hookrightarrow S \leftarrow O)$.
Note that we could as well consider the ``generic'' case $C=S = \mathrm{Spec}(R)$ which provides
$(
X \hookrightarrow \mathbb P^N_{R} \leftarrow X_{O}
)$
and we have $\,\,]X_{C}[_{X_{O}} = \,\,]X_{S}[_{X_{O}}  =  \,\,]X[_{X_{O}}$ so that there is no need to mention $C$ inside the tube.
Since
\[
\,\,]X[_{X_{O}} = \,]\mathbb A^{N}[_{ \mathbb A^N_{O}}  \cap X_{O} \subset \mathbb A^{N}_{O},
\]
we see that
\[
\,\,]X[_{X_{O}} = \bigcap_{n \in \mathbb N} V_n \quad \mathrm{with}\quad V_n := \mathbb D^N_{O}(0, \pi^{-\frac 1n}) \cap X_{O} \subset \mathbb A^{N}_{O}.
\]
Note that $V_n$ is affinoid and more precisely, $V_n = \mathrm{Spa}(A_{\Lambda,n}, A_{\Lambda,n}^+)$ where
\[
A_{\Lambda,n} := \widehat \Lambda\{\pi^{1/n}T_1, \ldots, \pi^{1/n}T_N\}/I\widehat \Lambda\{\pi^{1/n}T_1, \ldots, \pi^{1/n}T_N\}
\]
and $A_{\Lambda,n}^+$ is the union of $\Lambda^+$ and the images of $\pi T_1^n, \ldots, \pi T_N^n$ in $A_{\Lambda,n} $.
We can assume that the image of $R$ into $\Lambda$ is contained in $\Lambda_0$, then consider the weak completion
\[
A_{\Lambda_0}^\dagger := \Lambda_0[T_1, \ldots, T_N]^\dagger/I\Lambda_0[T_1, \ldots, T_N]^\dagger
\]
of $A_{\Lambda_0} := \Lambda_{0} \otimes_{R} A$ and set
\[
A_\Lambda^\dagger := A_{\Lambda_0}^\dagger[1/\pi] = \Lambda \otimes_{\Lambda_0} A_{\Lambda_0}^\dagger.
\]
Since (filtered) colimits are right exact, we have an isomorphism $A_\Lambda^\dagger = \varinjlim_n  A_{\Lambda,n}$
and an equivalence
\[
\mathrm{Mod}_{\mathrm{coh}}(A_\Lambda^\dagger) \simeq \varinjlim_n  \mathrm{Mod}_{\mathrm{coh}}(A_{\Lambda,n}).
\]
From this, we derive (with standard notations) an equivalence
\[
\mathrm{MIC}_{\mathrm{coh}}(A_\Lambda^\dagger) = \varinjlim_n  \mathrm{MIC}_{\mathrm{coh}}(A_{\Lambda,n})
\]
and,  if $M$ corresponds to $\{M_n\}_{n \geq n_0}$, an isomorphism (with standard notations again)
\[
M \otimes_{A_\Lambda^\dagger} \Omega^\bullet_{A_\Lambda^\dagger} \simeq \varinjlim_n \left(M \otimes_{A_{\Lambda,n}} \Omega^\bullet_{A_{\Lambda,n}}\right)\simeq A_\Lambda^\dagger \otimes_{A_{\Lambda,n_0}}  \left(M \otimes_{A_{\Lambda,n_0}} \Omega^\bullet_{A_{\Lambda,n_0}}\right).
\]

If $T$ is an overconvergent site, we will denote by $\mathrm{Crys}_{\mathrm{fp}}(T) = \mathrm{Mod}_{\mathrm{fp}}(T)$ the full subcategory of crystals, or equivalently modules, that are finitely presented.

\begin{prop}
There exists a fully faithful functor
\[
\xymatrix@R=0cm{
\mathrm{Crys}_{\mathrm{fp}}(X_C/O)^\dagger \ar[r] & \mathrm{MIC}_{\mathrm{coh}}(A_\Lambda^\dagger) 
\\ E \ar[r] &  \mathrm M(E) := \Gamma\left(\,]X[_{X_{O}}, E_{X_O}\right).
}
\]
Moreover, we have
\[
\mathrm R\Gamma(O, \mathrm Rp_{{X_C}/O*}E) =  M(E) \otimes_{A_\Lambda^\dagger} \Omega^\bullet_{A_\Lambda^\dagger} .
\]
\end{prop}

\begin{proof}
Since $\{V_n\}$ form a cofinal system of affinoid neighborhood of $X_{S}$ in $X_{O}$, this follows from the above discussion.
More precisely, the first (resp.\ second) assertion is a reinterpretation of theorem \ref{crismic} (resp.\ corollary \ref{cohdR}).
\end{proof}

When $M$ belongs to the essential image, we may call the connection \emph{overconvergent} and reformulate the statement as an equivalence.
It is also possible to give a more concrete description of this property.
We may also notice that, in general, a crystal $E$ is finitely presented if and only if $\mathrm M(E)$ is coherent.

We keep the above assumptions and notations and we also give ourselves a morphism of smooth affine schemes $f : Y \to X$ over $R$.
If we write $Y = \mathrm{Spec}(B)$, then we may choose a presentation $R[T_1, \ldots, T_M] \twoheadrightarrow B$ and deduce a locally closed embedding $Y \hookrightarrow \mathbb P^M_{R}$.
We may then consider the diagram
\[
\xymatrix{Y_C \ar@{=}[d] \ar@{^{(}->}[r] & \mathbb P^M_{S} & Y_O \ar[l] \ar@{=}[d]
\\ Y_C \ar@{^{(}->}[r] \ar[d]^{f_C} & \mathbb P^N_{S} \times \mathbb P^M_{S} \ar[d]^{p_{1}} \ar[u]^{p_{2}} & Y_O \ar[l] \ar[d]^{f_O}
\\X_C \ar@{^{(}->}[r] & \mathbb P^N_{S} & X_O \ar[l]}
\]
where the upper map is a strict neighborhood.

\begin{prop} \label{finMW}
\begin{enumerate}
\item If $E$ is a finitely presented crystal on $(X_C/O)^\dagger$, then $f_C^{-1}E$ is a finitely presented crystal on $(Y_C/O)^\dagger$ and
\[
M(f_C^{-1}E) \simeq B^\dagger_\Lambda \otimes_{A^\dagger_\Lambda}M(E).
\]
\item \label{finMW2} If $f$ is finite and $E$ is a finitely presented crystal on $(Y_C/O)^\dagger$, then $f_{C*}E$ is a finitely presented crystal on $(X_C/O)^\dagger$,
\[
M(f_{C*}E) \simeq M(E) \quad \mathrm{and} \quad \mathrm R^k f_{C*}E = 0 \ \mathrm{for}\ k > 0.
\]
\item \label{finMW3} If $f$ is a Galois covering and $E$ is a finitely presented crystal on $(X_C/O)^\dagger$, then $E$ is a direct factor in $f_{C*}f_C^{-1}E$.
\end{enumerate}
\end{prop}

\begin{proof}
The first assertion follows from the fact that if $E$ is a module on $(X_C/O)^\dagger$, then $(f_C^{-1}E)_{Y_O} = ]f[_C^\dagger E_{X_O}$.
The second one is a consequence of lemma \ref{finbas} (and proposition \ref{fincrys}).
The last one follows from the first two (using the trace map).
\end{proof}

We can globalize:

\begin{thm} \label{finet}
Let $(C \hookrightarrow S \leftarrow O)$ be an analytic convergent space over $\mathbb Q$.
Let $X$ be a smooth scheme over $C$ and $f : Y \to X$ a finite étale map.
If $E$ is a finitely presented crystal on $(Y/O)^\dagger$, then $f_{*}E$ is a finitely presented crystal on $(X/O)^\dagger$ and $R^kf_{*}E = 0$ for $k > 0$.
\end{thm}

\begin{proof}
The first question is local on $X$ and $O$.
In particular, we may assume that $O = \mathrm{Spa}(\Lambda, \Lambda^+)$ and that there exists a noetherian ring $R$ and an ideal $I \subset R$ such that $C = R/I$ and $S = R^{/I}$. We may also assume that $X$ is affine.
Thanks to \cite{Arabia01}, we can lift $f$ over $R$ and use assertion \ref{finMW2} of proposition \ref{finMW}.
\end{proof}

\section{Descent}

Effective descent in rigid cohomology was addressed only recently by Christopher Lazda in  \cite{Lazda19}.
Cohomological descent was studied earlier (\cite{ChiarellottoTsuzuki03}, \cite{Tsuzuki03}), mostly because the original definition of Berthelot needed some gluing.

There exists a more recent approach of cohomological descent by Zureick-Brown (\cite{ZureickBrown14}).
Unfortunately, a recurrent argument in his proof is not valid\footnote{The argument given on top of page of page 19 in \cite{ZureickBrown14} is not correct either.} and it seems actually necessary to use induction on the dimension as in Tsuzuki's.
More precisely, it is stated in \cite{ZureickBrown14} (remark 5.3 used twice in the proof of lemma 5.14) that if we have a morphism of analytic overconvergent spaces
\[
\xymatrix{Y \ar@{^{(}->}[r] \ar@{->>}[d]^f & Q  \ar[d]^v& V  \ar[l] \ar@{=}[d] \\
X \ar@{^{(}->}[r] & P  & V  \ar[l].}
\]
with $f$ surjective, then $\,]Y[_V = \,]X[_V$.
This is clearly false in general as the following example shows:
assume that $X = \mathbb A_{\mathbb Z_p}$, $P := \widehat {\mathbb A}_{\mathbb Z_p}$ and $V = \mathbb D_{\mathbb Q_p}(0, 1^+)$.
If $Q$ denotes the blowup of the origin in $P$ and $Y$ is the strict transform of $X$, then $\,]X[_V = \mathbb D_{\mathbb Q_p}(0, 1^+)$ is a disk but $\,]Y[_V = \mathbb A_{\mathbb Q_p}(|p|^-,1)$ is an annulus.

We will freely use our results from sections \ref{apdesc}, \ref{formaco} and \ref{apcdes} of the appendix.

\subsection{Descent on overconvergent spaces}

We consider here descent with respect to  crystals on overconvergent spaces.
In other words, our base site is the absolute overconvergent site.

We first need to translate descent conditions in terms of tubes.

Let $(f,u) : (Y,W) \to (X,V)$ be a morphism of overconvergent spaces which admits self products.
For $i \in \mathbf\Delta$, we let
\[
Y(i) := \underbrace{Y \times_X \cdots \times_X Y}_{i+1} \quad \mathrm{and} \quad W(i) :=  \underbrace{W \times_V \cdots \times_V W}_{i+1})
\]
and we write
\[
(Y,W)(i) := (Y(i), W(i)) \quad \mathrm{and} \quad \,]Y[_W^\dagger(i) := \,]Y(i)[_{W(i)}^\dagger.
\]
Note that $(Y,W)(i)$ (resp.\ $\,]Y[_W^\dagger(i)$) is the fibered product in the category of overconvergent spaces (resp.\ germs of adic spaces).
For the later, we can rely on proposition 4.20 in \cite{LeStum17*}.
It follows from lemma \ref{comphi} that there exists a morphism of (augmented) topoi
\[
\xymatrix{
\widetilde{(Y,W)(\bullet)} \ar[rr]^-{\varphi_W(\bullet)} \ar[d]^{(f,u)_\epsilon} && \widetilde{\,]Y[_W^\dagger(\bullet)} \ar[d]^{]f[_{u\epsilon}}
 \\ \widetilde{(X,V)} \ar[rr]^-{\varphi_V} && \widetilde{\,]X[_V^\dagger}
}
\]
(in which the $\epsilon$ indicates that the morphism lies over the final map $\mathbf\Delta \to \{0\}$).
We will consider $]f[_{u\epsilon}$ as a morphism of \emph{ringed} topoi in the obvious way.
If $\mathcal F$ is an $\mathcal O_V^\dagger$-module such that the adjunction map is an isomorphism
\[
\mathcal F \simeq \mathrm R ]f[_{u\epsilon*} ]f[_{u\epsilon}^{\dagger} \mathcal F,
\]
then we will say that $]f[_u$ satisfies \emph{cohomological descent} with respect to $\mathcal F$.
Be careful that we consider here cohomological descent with respect to \emph{module} pullback (meaning $]f[^\dagger$) and not naive pullback (meaning $]f[^{-1}$) as we do in the appendix.

Here is the first elementary result:

\begin{prop} \label{germarg}
If $(f, u) : (Y,W) \to (X,V)$ is a morphism of overconvergent spaces and $E$ is a crystal on $(X,V)$, then the following are equivalent:
\begin{enumerate}
\item $(f, u)$ satisfies universal cohomological descent with respect to $E$,
\item $(f, u)$ satisfies cohomological descent with respect to $E$,
\item given any cartesian diagram
\[
\xymatrix{(Y',W') \ar[d]^{(f',u')} \ar[r]^{(g',v')} & (Y,W) \ar[d]^{(f,u)} \\(X',V') \ar[r]^{(g,v)} & (X,V),}
\]
$]f'[_{u'}$ satisfies cohomological descent with respect to $E_{V'}$.
\end{enumerate}
\end{prop}

\begin{proof}
If $E'$ is a crystal on $(X', V')$, then we have (by definition of a crystal)
\[
\varphi_{V'}(\bullet)_*(f',u')_{\epsilon}^{-1}  E'=  ]f'[_{u'\epsilon}^{\dagger} \varphi_{V'*} E'.
\]
It follows that
\begin{align*}
(\mathrm R (f,u)_{\epsilon*} (f,u)_{\epsilon}^{-1} E)_{V'} &= \varphi_{V'*}(g,v)^{-1}\mathrm R (f,u)_{\epsilon*} (f,u)_{\epsilon}^{-1} E
\\ &= \varphi_{V'*}\mathrm R (f',u')_{\epsilon*} (g,v)(\bullet)^{-1} (f,u)_{\epsilon}^{-1} E
\\ &= \mathrm R ]f'[_{u'\epsilon*} \varphi_{V'}(\bullet)_* (f',u')_{\epsilon}^{-1}  (g,v)^{-1} E
\\ &= \mathrm R ]f'[_{u'\epsilon*}  ]f'[_{u'\epsilon}^{\dagger} \varphi_{V'*} (g,v)^{-1} E
\\ &= \mathrm R ]f'[_{u'\epsilon*}  ]f'[_{u'\epsilon}^{\dagger} E_{V'}.
\end{align*}
Thus, the last two conditions are equivalent and the very last therefore implies the first one because the condition is stable under pullback.
\end{proof}

We will also use systematically the fact that, thanks to lemma \ref{eqrep}, $(f, u)$ satisfies effective descent with respect to a crystal $E$ if and only if $]f[_{u}$ satisfies effective descent with respect to $E_{V}$.

Recall that a morphism is said to satisfy \emph{total descent} if it satisfies simultaneously universal cohomological descent and universal effective descent.

The very naive following consequence of proposition \ref{germarg} will prove itself quite useful:

\begin{cor} \label{cormag}
Let $(f,u) : (Y,W) \to (X,V)$ be a morphism of overconvergent spaces such that $\,]Y[_{W}^\dagger \simeq \,]X[_{V}^\dagger$.
Then, $(f,u)$ satisfies total descent with respect to  crystals.
\end{cor}

\begin{proof}
Our assertion follows from proposition \ref{germarg} (and lemma \ref{eqrep} in the case of effective descent).
\end{proof}

\begin{prop} \label{finffl}
A morphism of analytic overconvergent spaces
\[
\xymatrix{Y \ar@{^{(}->}[r] \ar[d]^f & Q \ar[d]^v & \ar[l] W \ar[d]^u \\ X \ar@{^{(}->}[r] & P & \ar[l] V}
\]
with both squares cartesians and $v$ finite faithfully flat satisfies total descent with respect to  constructible crystals.
\end{prop}

\begin{proof}
Let us first consider the case of finitely presented crystals.
By proposition \ref{germarg} (and lemma \ref{eqrep}), it is sufficient to prove that the morphism of germs $\,]Y[_{W}^\dagger \to \,]X[_{V}^\dagger$ satisfies cohomological and effective descent with respect to  coherent modules.
It follows respectively from proposition 3.10 and proposition 4.20 of \cite{LeStum17*} that $u$ is faithfully flat and that the diagram
\[
\xymatrix{\,]Y[^\dagger_W \ar@{^{(}->}[r] \ar[d] & W \ar[d]^u \\ \,]X[^\dagger_V \ar@{^{(}->}[r] & V}
\]
is cartesian.
Moreover, the question is local on $V$ that we may assume to be affinoid.
Thanks to theorem \ref{glumd} and proposition \ref{dirim}, our assertion will follow from cohomological and effective descent with respect to  coherent modules with respect to the finite faithfully flat morphism of Tate affinoid spaces $u : W \to V$.
Actually, thanks to theorems A and B, this turns out to be an immediate corollary of cohomological and effective descent with respect to faithfully flat ring maps (\cite[\href{https://stacks.math.columbia.edu/tag/023N}{Tag 023N}]{stacks-project} and \cite[\href{https://stacks.math.columbia.edu/tag/023M}{Tag 023M}]{stacks-project}).

We now consider the question of \emph{cohomological} descent with respect to  a general constructible crystal $E$ on $(X,V)$.
We know from proposition \ref{extcons} that there exists a dense open subset $U$ with closed complement $Z$, a constructible crystal $E''$ on $(Z,V)$, a finitely presented crystal $E'$ on $(U,V)$ and a short exact sequence
\[
0 \to \beta_{\dagger}E' \to E \to \alpha_{*}E'' \to 0
\]
where $\alpha : U \hookrightarrow X$ and $\beta : Z \hookrightarrow X$ denote the formal embeddings.
By induction on the dimension of $X$, we can assume that the pullback $Z' \to Z$ of $f$ satisfies total descent with respect to  $E'$.
Our assertion therefore follows from lemma \ref{desdag} below.

It remains to prove \emph{effective} descent with respect to  general constructible crystals.
It is important to recall from the second assertion in lemma \ref{cons1} that a crystal on $(Y,W)$ is constructible if and only if it is $X$-constructible.
There should therefore be no ambiguity here about the notion of constructibility.
Let us denote by $T$ the image of $(Y,W)$ in $(X,V)$, i.e.\ the sieve generated by $(f,u)$ (which is an overconvergent site).
We have to prove that the functor
\[
\mathrm{Cris}_{\mathrm{cons}}(X,V) \to \mathrm{Cris}_{\mathrm{cons}}(T)
\]
is an equivalence\footnote{Our assertion then means that $T \to (X,V)$ is a local isomorphism for the topology of total descent with respect to  constructible crystals.}.
Lemma \ref{bijcons} provides full faithfulness so that descent holds and it only remains to prove essential surjectivity to see that descent is effective.
If $F$ is a constructible crystal on $T$, then there exists a dense open subset $U \subset X$ with closed complement $Z$, a constructible crystal $F'$ on $T_Z$, a finitely presented crystal $F''$ on $T_U$ and a short exact sequence
\[
0 \to \beta_{\dagger}F' \to F' \to \alpha_{*}F'' \to 0
\]
where $\alpha : U \hookrightarrow X$ and $\beta : Z \hookrightarrow X$ denote the formal embeddings.
By induction on the dimension of $X$ (resp.\ from the finitely presented case) there exists a constructible crystal $E'$ on $(Z,V)$ (resp.\ a finitely presented crystal $E''$ on $(U,V)$) such that $F'=E'_T$ (resp.\ $F''=E''_T$).
It is therefore now sufficient to show that
\[
\mathrm{Ext} (\alpha_*E'',\beta_\dagger E')  \simeq \mathrm{Ext} (\alpha_*E_{T}'', \beta_\dagger E_{T}').
\]
The question is local on $V$ from the beginning.
We may therefore assume thanks to lemma \ref{extlim} that $E''$ extends to a finitely presented crystal on some sufficiently small neighborhood $V'$ of $U$ and that
\[
\mathrm{Ext} (\alpha_*E'',\beta_\dagger E') \simeq \varinjlim \mathrm{Hom} (\beta^{-1} E''_{|V'}, E'_{|V'}).
\]
Actually, lemma \ref{extlim} is stated on tubes but realization induces an equivalence between crystals on the overconvergent \emph{space} and modules on the tube.
For the same reason, we also have
\[
 \mathrm{Ext} (\alpha_*E_{T}'', \beta_\dagger E_{T}') \simeq \varinjlim \mathrm{Hom} (\beta^{-1} E''_{T|V'}, E'_{T|V'}).
\]
Actually, the property holds both on $(Y,W)$ and on the product $(Y \times_X Y, W \times_V W)$ and we have the interpretation of a module on $T$ as a module on $(Y,W)$ endowed with a descent datum.
Our assertion therefore follows from full-faithfulness applied on each $V'$.
\end{proof}

We used above the following lemma:

\begin{lem} \label{desdag}
In the situation of the proposition, assume that we are given a cartesian diagram
\[
\xymatrix{Y' \ar@{^{(}->}[r]^{\gamma'} \ar[d]^{f'} & Y \ar[d]^f \\ X' \ar@{^{(}->}[r]^\gamma & X}
\]
where $\gamma$ is a formal embedding.
If $(f', u)$ satisfies cohomological descent with respect to a crystal $E$, then $(f,u)$ satisfies cohomological descent with respect to $\gamma_\dagger E$.
\end{lem}

\begin{proof}
It is sufficient thanks to proposition \ref{germarg}, to show that
\[
]\gamma[_!\mathrm R ]f'[_{u\epsilon*}  ]f'[_{u\epsilon}^{\dagger} \mathcal F' \simeq \mathrm R ]f[_{u\epsilon*}  ]f[_{u\epsilon}^{\dagger} ]\gamma[_!\mathcal F'
\]
when $\mathcal F'$ is a module on $\,]X'[_V$.
From the fundamental spectral sequence, it is sufficient to prove that for all $i \geq 0$,
\[
]\gamma[_!\mathrm R ]f'(i)[_{u*}  ]f'(i)[_{u}^{\dagger} \mathcal F' \simeq \mathrm R ]f(i)[_{u*}  ]f(i)[_{u}^{\dagger} ]\gamma[_!\mathcal F'.
\]
Since our conditions are stable under pullback, it is sufficient to consider the case $i=0$ and prove that
\[
]\gamma[_!\mathrm R ]f'[_{u*}  ]f'[_{u}^{\dagger} \mathcal F' \simeq \mathrm R ]f[_{u*}  ]f[_{u}^{\dagger} ]\gamma[_!\mathcal F'.
\]
It follows from lemma \ref{comgam} that $]f[_{u}^{\dagger} \circ ]\gamma[_! = ]\gamma'[_! \circ ]f'[_{u}^{\dagger}$
and from proposition \ref{dirim} that
\[
\mathrm R ]f[_{u*} \circ ]\gamma'[_! = \mathrm R (]f[_{u*} \circ ]\gamma'[_!) =  \mathrm R (]\gamma[_! \circ ]f'[_{u*} ) = ]\gamma[_! \circ  \mathrm R ]f'[_{u*}.\qedhere
\]
\end{proof}

\subsection{Descent on overconvergent sites}

Recall that an overconvergent site is called \emph{classic} if the underlying category is fibered in sets - or equivalently in equivalence relations.
They form a topos that we can identify with the category of all presheaves on the absolute overconvergent site.
Here, we only consider descent with respect to classic overconvergent sites (that we simply call overconvergent sites to make notations lighter).
It is very likely that all results are valid on general overconvergent sites but it would then be necessary to work over the $2$-category of all overconvergent sites\footnote{And I do not feel comfortable with that.}.

We will make extensive use of the following:
\begin{enumerate}
\item Any local epimorphism satisfies total descent,
\item Any morphism which is dominated by a morphism of total descent is automatically of total descent itself,
\item The property is stable under composition,
\item The property is stable under pullback.
\end{enumerate}

Recall that, if $X$ is a formal scheme, then we denote by $X^\dagger$ (resp.\ $X^{\dagger,\mathrm{an}}$) the overconvergent site whose objects are overconvergent spaces (resp.\ overconvergent analytic spaces) $(Y,W)$ endowed with a morphism of formal schemes $Y \to X$.

\begin{lem} \label{Zcohd}
If $X = \bigcup_{i \in I} X_i$ is an open or a closed covering, then the family $\{X_{i}^{\dagger,\mathrm{an}} \hookrightarrow X^{\dagger,\mathrm{an}} \}_{i \in I}$ satisfies total descent with respect to  crystals.
\end{lem}

\begin{proof}
We have to show that if $T$ is an analytic overconvergent site over $X^\dagger$, then the family $\{T_{X_{i}} \to T_{X}\}$ satisfies total descent with respect to  crystals.
Since the question is local and the conditions are stable under pullback it is sufficient to consider the case $T = (X,V)$ with $V$ Tate affinoid.
Actually, the question is also local on the ``middle'' formal scheme in $(X,V)$ and we may therefore also assume that $I$ is finite.
Since $V$ is analytic, we obtain a finite closed or open covering $\,]X_i[_V^\dagger = \bigcup_{i \in I} \,]X_i[_V^\dagger$.
This covering clearly satisfies total descent with respect to  modules (using Mayer-Vietoris in the first case).
We finish with proposition \ref{germarg} (and lemma \ref{eqrep} for effective descent as usual).
\end{proof}

It is worth mentioning that if $Y \hookrightarrow X$ is a thickening, then $Y^\dagger \hookrightarrow X^\dagger$ satisfies total descent with respect to  crystals (the analytic case follows from the lemma).

The main ingredient for proper and flat descent is the next lemma which is due to Tsuzuki in the classical case (\cite{Tsuzuki03}, proposition 3.4.1).
For this purpose, we introduce Tsuzuki's condition:

$(A_d)$ : If $(X,V)$ is an analytic overconvergent space and $f : Y \to X$ is a partially proper surjective morphism of formal schemes of dimension $\leq d$, then $(Y/V)^\dagger \to (X,V)$ satisfies total descent with respect to   constructible crystals.

Note that, for the property to hold for $f$, it is sufficient that there exists a morphism $(Y,W) \to (X,V)$ that satisfies total descent with respect to   constructible crystals.
And conversely when this morphism is a geometric materialization (because $(Y,W) \to (Y/V)^\dagger$ is then a local epimorphism).

\begin{lem} \label{birat}
Assume $(A_{d-1})$ holds.
If $(X,V)$ is an analytic overconvergent space and $f : Y \to X$ is a partially proper birational morphism of formal schemes of dimension $\leq d$, then $(Y/V)^\dagger \to (X,V)$ satisfies total descent with respect to   constructible crystals.
\end{lem}

\begin{proof}
Since a closed covering satisfies total descent with respect to  crystals, we can assume that $X$ is an integral scheme.
After replacing $Y$ by one of its irreducible components that dominates $X$, we can assume that $Y$ also is an integral scheme.
Now, $f : Y \to X$ is a proper birational morphism of integral schemes.
Since the question is local for the Zariski topology of $X$, we can also assume that $X$ is affine.
We can then apply Chow's lemma (corollary I.5.7.14 of \cite{GrusonRaynaud71}) and reduce to the case where $f$ is projective.
It then follows from theorem 1.24 in \cite{Liu06} that $f$ is the blowing up of some closed subscheme $Z \subset X$.
We may assume that $Z$ has dimension $< d$ (otherwise $Y=X$ and we are done).

Let us denote by $(X \hookrightarrow P \leftarrow V)$ our overconvergent space.
After completion along the (reduced) closure $\overline X$ of $X$, we can assume that $P$ has dimension $\leq d$.
Now, we denote by $\overline Z$ the (reduced) closure of $Z$ in $P$ and we consider the blowing up $v : Q \to P$ of $\overline Z$ in $P$ so that $Y$ becomes the strict transform of $X$.
Since the question is local on $P$, we can assume that $Q \subset \mathbb P^n_P$.
Then our assertion is equivalent to showing that the geometric materialization $(Y,\mathbb P^n_V) \to (X,V)$ satisfies total descent with respect to constructible crystals.

Recall that we introduced in definition 4.6 of \cite{LeStum17*} the notion of a fiber and that there exists a closed immersion
\[
\overline Z_V := \overline Z^{\mathrm{ad}} \times_{P^{\mathrm{ad}}} V\hookrightarrow ]\overline Z[_V
\]
of the fiber of $\overline Z$ into its tube.
In particular, there exists an open covering
\[
V = ]\overline Z[_V \cup (V \setminus \overline Z_V)
\]
and we can therefore split the verification in two: either $V = ]\overline Z[_V$ or $\overline Z_V = \emptyset$.
Let us first assume that $V = ]\overline Z[_V$ so that $\,]Z[_{V} = \,]X[_{V}$.
Corollary \ref{cormag} then allows us to replace $X$ with $Z$, and consequently $Y$ with the exceptional divisor $E$ on $Y$.
We can conclude with our induction hypothesis.

In order to do the case $\overline Z_V = \emptyset$, we need some preparation.
We denote by $Y'$ the inverse image of $X$ in $Q$ and let $E'$ be the trace of the exceptional divisor on $Y'$.
The situation is as follows:
\[
\xymatrix{E \ar@{^{(}->}[r] \ar@{^{(}->}[d] &E' \ar@{^{(}->}[d]
\\ Y \ar@{^{(}->}[r] \ar[rd]^f &Y' \ar@{^{(}->}[r] \ar[d] & \mathbb P^n_P \ar[d] & \ar[l] \mathbb P^n_V \ar[d] \\ Z \ar@{^{(}->}[r] & X \ar@{^{(}->}[r] & P & \ar[l] V.}
\]
We consider the following morphisms of overconvergent spaces:
\[
\xymatrix{
(Y,\mathbb P^n_V) \coprod (E,\mathbb P^n_V) \ar[r] \ar[d] & (Y,\mathbb P^n_V) \coprod (Z,V) & (Y,\mathbb P^n_V) \coprod (E',\mathbb P^n_V) \ar[l] \ar[d] 
\\
(Y,\mathbb P^n_V) && (Y',\mathbb P^n_V).
}
\]
By induction, both horizontal maps satisfy total descent with respect to   constructible crystals (it is important here to consider geometric materializations).
The same holds trivially for the left hand map but also for the right hand one because $Y' = Y \cup E'$ is a closed covering and we can rely on lemma \ref{Zcohd}.
It is therefore equivalent to prove the property for $(Y,\mathbb P^n_V) \to (X,V)$ or for $(Y',\mathbb P^n_V) \to (X,V)$.

Since we assume now that $\overline Z_V = \emptyset$, the canonical map $V \to P^{\mathrm{ad}}$ factors through the open subset $P^{\mathrm{ad}} \setminus \overline Z^{\mathrm{ad}}$.
On the other hand, we showed in proposition 3.14 of \cite{LeStum17*} that the blowing up $v : Q \to P$ induces an isomorphism $Q^{\mathrm{ad}} \setminus v^{-1}(\overline Z)^{\mathrm{ad}}  \simeq P^{\mathrm{ad}} \setminus \overline Z^{\mathrm{ad}}$.
We may therefore consider the composite map
\[
V \to P^{\mathrm{ad}} \setminus \overline Z^{\mathrm{ad}} \simeq Q^{\mathrm{ad}} \setminus v^{-1}(\overline Z)^{\mathrm{ad}} \hookrightarrow Q^{\mathrm{ad}}
\]
and we obtain a morphism of overconvergent spaces
\[
\xymatrix{Y' \ar@{^{(}->}[r] \ar[d] & Q \ar[d]^v & \ar[l] V \ar@{=}[d] \\ X \ar@{^{(}->}[r] & P & \ar[l] V.}
\]
Since the left hand square is cartesian (and this is why we need to consider $Y'$ and not $Y$), we have $\,]Y'[_V = \,]X[_V$.
Proposition \ref{germarg} implies that $(Y',V) \to (X,V)$ satisfies total descent with respect to crystals.
Since it factors through $(Y',\mathbb P^n_V) \to (X,V)$, we are done.
\end{proof}

The next lemma also is due to Tsuzuki in the classical case (proposition 3.5.1 of \cite{Tsuzuki03}):

\begin{lem} \label{finsurj}
Assume $(A_{d-1})$ holds.
If $(X,V)$ is an analytic overconvergent space and $f : Y \to X$ is a finite surjective morphism of formal schemes of dimension $\leq d$, then $(Y/V)^\dagger \to (X,V)$ satisfies total descent with respect to   constructible crystals.
\end{lem}

\begin{proof}
With the same arguments as in the previous proof, we can assume that $X = \mathrm{Spec}(R)$ and $Y = \mathrm{Spec}(S)$ are integral affine \emph{schemes} and we proceed by induction on the degree of the field extension $K(X)\hookrightarrow K(Y)$.
When $K(X) \simeq K(Y)$, the assertion follows from the previous lemma \ref{birat} and we may assume from now on that the field extension is not trivial.
Then there exists $g \in S$ such that $g \notin K(X)$.
If we let $Y' := \mathrm{Spec}(R[g])$, then both maps in the factorization $Y \to Y' \to X$ are still finite surjective.
By induction on the degree, we may therefore assume that $S = R[g]$.
The minimal polynomial $\overline F$ of $g$ over $K(X)$ has integral coefficients $f_0, \dots f_d$ over $R$ and we can replace, thanks to lemma \ref{birat} again, $R$ and $S$ with $R[f_0, \dots f_d]$ and $S[f_0, \dots f_d]$ respectively.
We can therefore assume that $\overline F \in R[T]$ and $S \simeq R[T]/\overline F$.
In particular, $f$ is now finite faithfully flat.

We come to the proof of the assertion and we denote by $(X \hookrightarrow P \leftarrow V)$ our overconvergent space.
Since the question is local, we can assume that $V$ is affinoid, that $P = \mathrm{Spf}(A)$ is affine and that $X$ is supported by the complement $\mathcal U$ of a hypersurface $h=0$ in $P$.
We can lift $\overline F$ to a monic polynomial $F \in A[1/h]$ and after multiplication by a power of $h$, assume that $F \in A$.
The leading coefficient of $F$ will still be invertible in $A[1/h]$ and this provides a finite faithfully flat lifting of $f$.
We can then homogenize $F$ to $F_h \in  A[T_0, T_1]$ and consider the corresponding closed embedding $Q \hookrightarrow \mathbb P_P$.
We can now invoke the famous theorem I.5.2.2 of \cite{GrusonRaynaud71} and assume that the composite map $v : Q  \hookrightarrow \mathbb P_P \to P$ is flat.
More precisely, we can blow-up $\mathrm{Spec}(A)$ outside the open subset $\mathrm D(h)$ in such a way that the morphism $\mathrm{Proj}(A[T_0, T_1]/F_h) \to \mathrm{Spec}(A)$ becomes flat, and then take completions.
Note that $P$ may not be affine anymore but this doesn't matter.
Being proper flat and generically finite, the morphism $v$ is necessarily finite faithfully flat.
By construction, the right cartesian morphism
\[
\xymatrix{Y \ar@{^{(}->}[r] \ar[d]^f & Q \ar[d]^v & \ar[l] W \ar[d] \\ X \ar@{^{(}->}[r] & P & \ar[l] V}
\]
is also left cartesian.
It therefore follows from proposition \ref{finffl} that $(Y,W) \to (X,V)$ (and consequently $(Y/V)^\dagger \to (X,V)$) satisfies total descent with respect to   constructible crystals.
\end{proof}

We define the $h$-topology on the category of formal schemes as the topology generated by partially proper surjective morphisms and Zariski open coverings.

\begin{thm} \label{ffcoh}
If $\{X_i \to X\}_{i\in I}$ is a covering for the h-topology on formal schemes, then $\{X_i^{\dagger,\mathrm{an}} \to X^{\dagger,\mathrm{an}}\}_{i\in I}$ satisfies total descent with respect to   constructible crystals.
\end{thm}

\begin{proof}
This follows from lemma \ref{Zcohd} and the second assertion in corollary \ref{flpr} below.
\end{proof}

It remains to prove the following statement (which is actually a particular case) for the theorem to hold:

\begin{cor} \label{flpr}
If $f:Y \to X$ is a morphism of formal schemes which is
\begin{enumerate}
\item
either faithfully flat and locally formally of finite type
\item
or partially proper surjective,
\end{enumerate}
then $Y^{\dagger,\mathrm{an}} \to X^{\dagger,\mathrm{an}}$ satisfies total descent with respect to   constructible crystals.
\end{cor}

\begin{proof}
In both cases, we need to show that if $T$ is an overconvergent analytic space over $X^\dagger$, then $T_Y \to T$ satisfies total descent with respect to   constructible crystals.
Since the question is local on $T$, we may assume that $T =(X',V')$ for some $X' \to X$.
Since our hypothesis are stable under pullback along $X' \to X$, it is sufficient to consider the case $X' = X$.
We are therefore reduced to prove that, under our hypothesis, if $(X,V)$ is an overconvergent analytic space, then $(Y/V)^\dagger \to (X,V)$ satisfies total descent with respect to   constructible crystals.
We proceed by induction on the dimension $d$ of $X$.
We may assume that $(A_{d-1})$ holds.
In the first case, we can assume, thanks to \cite[\href{https://stacks.math.columbia.edu/tag/05WN}{Tag 05WN}]{stacks-project} and lemma \ref{Zcohd}, that $f$ is finite flat surjective and apply lemma \ref{finsurj}.
In the second case, we can assume that $f$ is a proper surjective morphism with $X$ an integral affine scheme.
It follows from theorem I.5.2.2 of \cite{GrusonRaynaud71} that there exists a blowing-up $X' \to X$ such that the strict transform $f' : Y' \to X'$ of $f$ is flat.
Since $f'$ is proper and dominant, it it necessarily faithfully flat.
The theorem therefore follows from lemma \ref{birat} and the faithfully flat case.
\end{proof}

\begin{xmp}
Let $(C,O)$ be an analytic overconvergent space and $X \to C$ is a morphism of formal schemes.
If $\{X_i \to X\}_{i\in I}$ is a covering for the $h$-topology, then the family $\{(X_i/O)^\dagger \to (X,O)\}_{i \in I}$ satisfies total descent with respect to   constructible crystals.
This applies to open or closed coverings, partially proper surjective maps as well as faithfully flat morphisms which are formally of finite type.
\end{xmp}

In the case of effective descent, we can also reformulate the main theorem in terms of stacks.
If we still call \emph{$h$-topology} the topology inherited by the category of overconvergent spaces through the forgetful functor $(X,V) \to X$, then we have:

\begin{cor}
The category of  constructible crystals on (classic) analytic overconvergent spaces is a stack for the $h$-topology. \qed
\end{cor}

\setcounter{section}{0}
\renewcommand{\theHsection}{\Alph{section}}
\appendix
\section{Appendix: Fibered categories}\label{fibcat}

\subsection{Definition}

Let us briefly review the theory (see \cite{Vistoli*} for example, or \cite{Streicher20} or better: the good old bible \cite{Giraud71}) and fix some notations.
Let $\mathbf B$ be a (base) category and $T$ a category defined over $\mathbf B$ through a morphism $j \colon T \to \mathbf B$ (we will write $j_{T}$ if we want to insist on the dependence on $T$).
A morphism $\psi \colon t \to s$ in $T$ with $j(\psi) = f$ is \emph{(strongly) cartesian} if, given any $g \colon W \to V$ and any $\rho :  u \to t$ with $j(\rho) = f \circ g$, then there exists a unique $\psi$ with $j(\psi) = g$ such that $\rho = \varphi \circ \psi$:
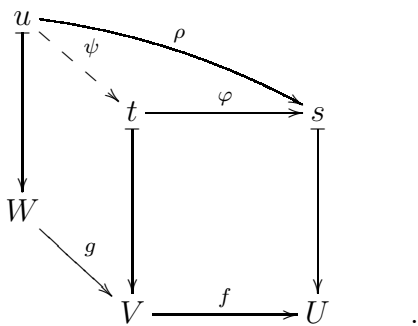
\[
\xymatrix{
u \ar@/^/[rrrd]^\rho \ar@{-->}[rd]^\psi \ar@{|->}[dd] & \\
& t \ar[rr]^\varphi \ar@{|->}[dd] && s \ar@{|->}[dd] \\
W \ar[rd]^g && \\
& V \ar[rr]^f && U &.
}
\]
The category $T$ is said to be \emph{fibered} over $\mathbf B$ if, given any $s \in T$ with $j(s) = U$ and any morphism $f \colon V \to U$ in $\mathbf B$, then,
there exists a \emph{(strongly) cartesian} morphism $\psi \colon t \to s$ with $j(\psi) = f$.
As a standard example, we can consider the codomain fibration $\mathbf B^{\mathbbm 2} \to \mathbf B$ when $\mathbf B$ has fibered products.

If $T$ and $T'$ are two fibered categories over $\mathbf B$, then a functor $u \colon T' \to T$ over $\mathbf B$ is said to be \emph{cartesian} if $u(\varphi)$ is cartesian whenever $\varphi$ is cartesian, and a transformation
\[
\eta \colon \xymatrix{T' \rtwocell^{u}_{v} & T}
\]
is said to be \emph{vertical}, if $j_{T}(\eta_{t}) = \mathrm{Id}_{j_{T'}(t)}$ whenever $t \in T'$.
Fibered categories over $\mathbf B$ endowed with cartesian functors and vertical transformations, make a $2$-category\footnote{Actually, there exists a fibered subcategory $\mathbb Fib \subset \mathbb Cat^{\mathbbm 2} \to \mathbb Cat$ of the codomain fibration whose fiber over $\mathbf B \in \mathbb Cat$ is equal to (the underlying category of) $\mathbb Fib(\mathbf B)$. } $\mathbb Fib(\mathbf B)$.
Any $2$-fibered product of fibered categories over $\mathbf B$ is automatically a fibered category over $\mathbf B$.
Also, pulling back along any functor $\mathbf C \to \mathbf B$ induces a $2$-functor
\[
\mathbb Fib(\mathbf B) \mapsto\mathbb Fib(\mathbf C), \quad T \mapsto T_{\mathbf C} = \mathbf C \times_{\mathbf B} T
\]
(the usual product is identical to the $2$-product in this situation).

\subsection{Fibered categories and presheaves}

Recall now that if $\mathcal T$ is a presheaf of sets on $\mathbf B$, then one can define the \emph{comma} category $\mathbf B_{/\mathcal T}$ as follows: an object is a pair $(U, s)$ with $U \in \mathbf B$ and $s \in \mathcal T(U)$ and a morphism $(V, t) \to (U, s)$ is a morphism $f \colon V \to U$ with $\mathcal T(f)(t) = s$.
The category $\mathbf B_{/\mathcal T}$ is fibered in sets over $\mathbf B$.
If we denote by $\widehat {\mathbf B}$ the category of all presheaves of sets on $\mathbf B$, then this construction provides a fully faithful functor
\[
\widehat {\mathbf B} \hookrightarrow \mathbb Fib(\mathbf B), \quad \mathcal T \mapsto \mathbf B_{/\mathcal T}.
\]
Its image (resp.\ essential image) is the full subcategory of categories fibered in sets (resp.\ in equivalence relations) on $\mathbf B$.
As we shall see below, there exists a more general construction for presheaves of \emph{categories} but this will  \emph{not} lead to a fully faithful functor in general.
At this point, we may also notice that, if $\mathbf C$ is a category fibered in \emph{sets} over $\mathbf B$ (for example, a fibered subcategory of $\mathbf B$) and $T \to \mathbf C$ is any functor, then $T$ is fibered over $\mathbf B$ if and only if $T$ is fibered over $\mathbf C$.

Before going any further, we also want to recall that there exists a fully faithful \emph{Yoneda} functor
\[
\mathbf B \hookrightarrow \widehat{\mathbf B}, \quad U \mapsto \widehat U := \mathrm{Hom}( -, U).
\]
In particular, we may apply the above considerations to the presheaf $\widehat U$ and consider the fibered category $\mathbf B_{/\widehat U}$ that we will simply denote by $\mathbf B_{/U}$.
Thus, an object of $\mathbf B_{/U}$ is a pair made of an object $V$ of $\mathbf B$ and a morphism $V \to U$ (we will say that $V$ is defined \emph{over} $U$).
A morphism $W \to V$ in $\mathbf B_{/U}$ is simply a morphism in $\mathbf B$ which is compatible with the structural maps\footnote{Alternatively, $\mathbf B_{/U}$ is the fiber over $U$ of the codomain fibration $\mathbf B^{\mathbbm 2} \to \mathbf B$.}.
Recall also that Yoneda's lemma then provides a sequence of natural \emph{bijections}
\[
\mathrm{Hom}_{\mathbb Fib(\mathbf B)}(\mathbf B_{/U}, \mathbf B_{/\mathcal T}) \simeq \mathrm{Hom}_{\widehat{\mathbf B}}(\widehat U, \mathcal T) \simeq \mathcal T(U)
\]
when $\mathcal T$ is a presheaf of sets on $\mathbf B$.
In particular, we may always interpret a section of a presheaf as a morphism of fibered categories.

If $T \in \mathbb Fib(\mathbf B)$ and $U \in \mathbf B$, then we let $T(U)$ denote the fiber of $T$ over $U$ which is the  subcategory of $T$ whose objects $s$ (resp.\ morphisms $\varphi$) are defined by the condition $j(s) = U$ (resp.\ $j(\varphi)=\mathrm{Id}_{U}$).
The \emph{$2$-Yoneda lemma} states that there exists a natural \emph{equivalence} of categories
\[
\mathrm{Hom}_{\mathbb Fib(\mathbf B)}(\mathbf B_{/U}, T) \simeq T(U)
\] 
(that depends on the choice of a \emph{cleavage}).
One can show that the assignment $T \colon U \mapsto T(U)$ extends to a \emph{pseudo-presheaf} of categories (but not a genuine presheaf).
On the other hand, the assignment
\[
\widehat T: U \mapsto \mathrm{Hom}_{\mathbb Fib(\mathbf B)}(\mathbf B_{/U}, T)
\]
does define a presheaf of categories on $\mathbf B$.
If we denote by $\mathbb Cat^{\mathbf B^{\mathrm{op}}}$ the $2$-category of presheaves of categories on $\mathbf B$, then we obtain a fully faithful $2$-functor
\[
\mathbb Fib(\mathbf B) \hookrightarrow \mathbb Cat^{\mathbf B^{\mathrm{op}}}, \quad T \mapsto \widehat T.
\]
Moreover this functor has a $2$-adjoint $\mathcal T \mapsto \mathbf B_{/\mathcal T}$ where the fibered category $\mathbf B_{/\mathcal T}$ is defined exactly as above, but now a morphism $(V, t) \to (U,s)$ is a pair $(f, \varphi)$ with $f \colon V \to U$ and $\varphi : t \to \mathcal T(f)(s)$.
Note that we always have $\mathbf B_{/\widehat T} \simeq T$ but the other adjunction map is \emph{not} an equivalence in general (even if, as we saw above, this is the case when $\mathcal T$ is a presheaf of sets).

In practice, we will make some wild identifications.
For example, we will identify an object $U \in \mathbf B$ with the corresponding presheaf of sets $\widehat U$ on $\mathbf B$, and a presheaf $\mathcal T$ (of sets) on $\mathbf B$ with the corresponding fibered category (in sets) $\mathbf B_{/\mathcal T}$ over $\mathbf B$.
By composition, we will identify the object $U \in \mathbf B$ with the fibered category in sets $\mathbf B_{/U}$ over $\mathbf B$.
Hopefully, this should not create any confusion and simply makes the notations lighter.

If $T$ is a fibered category over $\mathbf B$, then the equivalence $\mathbf B_{/\widehat T} \simeq T$ will allow us to consider an object of $T$ as (a couple made of an object $U \in \mathbf B$ and) a morphism $s : U \to T$ (of fibered categories over $\mathbf B$).
A morphism in $T$ then corresponds to a pair made of a morphism $f \colon V \to U$ in $\mathbf B$ and a transformation
\begin{align} \label{celldiag}
\eta : \xymatrix{
 V \xtwocell[0,2]{}\omit{<2>} \ar[rd]^f \ar@/^.2cm/[rr]^t
&& T \\
& U \ar[ru]^s &}
\end{align}
(the morphism is strongly cartesian exactly when $\eta$ is an equivalence and this is automatic when $T$ is fibered in groupoids).
This might sound like a complication but actually allows us to consider an object or a morphism in the fibered category $T$ as an object or a morphism in $\mathbf B$ living somehow \emph{over} $T$.

One can usually replace the base $\mathbf B$ with the category $\widehat{\mathbf B}$ (which has all limits and colimits and universal disjoint sums).
More precisely, there exists a ``fully faithful'' 2-functor\footnote{There exits actually a sequence of fully faithful functors $ \mathbf B \hookrightarrow \widehat{\mathbf B} \hookrightarrow \mathbb Fib(\mathbf B) \hookrightarrow \mathbb Fib(\widehat{\mathbf B})$.}
\[
\mathbb Fib(\mathbf B) \hookrightarrow \mathbb Fib(\widehat{\mathbf B}), \quad T \mapsto T^+
\]
with
\[
T^+(\mathcal F) = \mathrm{Hom}_{\mathbb Fib(\mathbf B)}(\mathbf B_{/\mathcal F}, T).
\]
By ``fully faithful'', we mean that there exists an equivalence
\[
\mathrm{Hom}_{\mathbb Fib(\mathbf B)}(T', T) \simeq \mathrm{Hom}_{\mathbb Fib(\widehat{\mathbf B})}(T'^+, T^+).
\]
In practice, we shall still write $T$ instead of $T^+$.

\subsection{Topology on fibered categories} \label{topfib}

When $T$ is a fibered category over a \emph{site} $\mathbf B$, $T$ will implicitly be endowed with the \emph{inherited} topology (the coarsest topology making $j \colon T \to \mathbf B$ cocontinuous).
When $T$ is fibered in groupoids, this is the same thing as the induced topology (finest topology making the map continuous).
There exists a crystalline description of sheaves on $T$: giving a sheaf $E$ on $T$ is equivalent to giving a sheaf $E_{U}$ on $U$ for each $s : U \to T$ (note that $E_{U}$ actually depends on $s$) and for each diagram \eqref{celldiag}, a morphism $\eta^{-1} : f^{-1}E_{U} \to E_{V}$ satisfying a cocycle condition.

A morphism of fibered categories $u \colon T' \to T$ over $\mathbf B$ is automatically continuous and cocontinuous, providing a sequence of three adjoint functors
\[
u_{!}, u^{-1}, u_{*} \colon \widetilde T' \to \widetilde {T}
\]
on the corresponding topoi (we always denote by $\widetilde {\mathcal C}$ the topos associated to a site $\mathcal C$), the last two of them defining a morphism of topoi.

Base change holds for fibered categories:

\begin{lem} \label{basfib}
If for $k=1,2$, $j_{k} : T_{k} \to T$ is a morphism of fibered categories over $\mathbf B$ and $p_{k} : T_{1} \times_{T} T_{2} \to T_{k}$ denotes the projection, then we have $j_{1}^{-1} \circ j_{2*} = p_{2*} \circ p_{1}^{-1}$.
\end{lem}

\begin{proof}
Follows from  \cite[\href{https://stacks.math.columbia.edu/tag/0FN1}{Tag 0FN1}]{stacks-project}. \end{proof}

It formally follows that we also have $j_{1}^{-1} \circ \mathrm R  j_{2*} = \mathrm R p_{2*} \circ p_{1}^{-1}$ on abelian sheaves.

If $T$ is a fibered category over $\mathbf B$ and we are given $s,t : U \to T$ for some $U \in \mathbf B$, then one defines the presheaf of morphisms $\mathcal H\mathrm{om}_U(s,t)$ on $U$ as follows: given $g : V \to U$, we set
\[
\mathcal H\mathrm{om}_U(s,t)(g) = \mathrm{Hom}(s \circ g ,t \circ g).
\]

\begin{dfn}
A morphism of fibered categories $f: T' \to T$ is a \emph{local epimorphism} (resp.\ a \emph{local isomorphism})\footnote{Also called a \emph{covering} (resp a \emph{bicovering}).} if
\begin{enumerate}
\item for all $s,t : U \to T'$, the morphism of presheaves
\[
\mathcal H\mathrm{om}_U(s,t) \to  \mathcal H\mathrm{om}_U(f \circ s,f \circ t)
\]
is a local epimorphism (resp.\ a local isomorphism).
\item for all $s : U \to T$, the full image of $T'_{U} \to U$ is a covering sieve of $U$.
\end{enumerate}
\end{dfn}

Usually, a sieve of $U$ is a family of maps $V \to U$ which is stable under composition with any map on the left.
This is equivalent to giving a subpresheaf of $\widehat U$ or a fibered subcategory of $\mathbf B_{/U}$, and we will mostly use the latter approach.

Being a local epimorphism (resp.\ local isomorphism) is local on $T$ in the sense that any pullback of a local epimorphism (resp.\ local isomorphism) is a local epimorphism (resp.\ local isomorphism), and conversely if we pullback along a local epimorphism.
Note also that a morphism of presheaves $\mathcal T' \to \mathcal T$ is a local epimorphism (resp.\ isomorphism)  if and only if the corresponding morphism of fibered categories $\mathbf B_{/\mathcal T} \to \mathbf B_{/\mathcal T}$ is a local epimorphism (resp.\ isomorphism).

Finally, a morphism $T' \to T$ is a local epimorphism (resp.\ a local isomorphism) if and only if the corresponding  morphism $T'^+ \to T^+$ is so when $\widehat {\mathbf B}$ is endowed with its canonical topology (where coverings are local epimorphisms of presheaves).
More generally, one calls a family $\{T_i \to T\}_{i \in I}$ \emph{locally epimorphic} if $\coprod_{i \in I} T_i^+ \to T^+$ is a local epimorphism.
For example, the family $\{U \to T\}_{U \in T}$ of all objects of $T$ is always locally epimorphic.

\subsection{Effective Descent} \label{apdesc}

Fibered categories are useful for descent theory and we shall briefly describe how this works.
Let $T$ be a fibered category over a category $\mathbf B$.
A sieve $R$ of $U \in \mathbf B$ is said to satisfy \emph{effective descent} for $T$ if the canonical functor is an equivalence:
\[
\mathrm{Hom}_{\mathbb Fib(\mathbf B)}(U, T) \simeq \mathrm{Hom}_{\mathbb Fib(\mathbf B)}(R, T)
\]
(or equivalently $T(U) \simeq T(R)$).
When this functor is only fully faithful, one says that $R$ is a sieve of \emph{descent}, so that descent is \emph{effective} when moreover the functor is essentially surjective.
The sieve $R$ is said to be \emph{universally of (effective) descent} for $T$  if $R_V$ is a sieve of (effective) descent in $V$ for all $V \to U$ in $\mathbf B$.
Sieves of universally effective descent for $T$ define a topology on $\mathbf B$ called the \emph{topology of (effective) descent} for $T$.

A family $\{f_i : V_i \to U\}$ of morphisms in $\mathbf B$ is said to satisfy \emph{(effective) descent} for $T$ if the sieve generated by this family satisfies (effective) descent with respect to  $T$.
In the case of a single morphism $f : V \to U$, we will denote by $(V \to U)$ the sieve generated by $f$: this is the image of the morphism $\mathbf B_{/V} \to \mathbf B_{/U}$.
If $f$ admits self products, then  the category $T(V \to U)$ is equivalent to the category of descent data relative to $f$: a \emph{descent datum} is an object $\mathcal F \in T(V)$ endowed with an isomorphism $p_2^*\mathcal F \simeq p_1^*\mathcal F$ on $V \times_U V$ satisfying the cocycle condition on triple products.
Note that, a family $\{f_i : V_i \to U\}$ of morphisms in $\mathbf B$ satisfies (effective) descent if and only if the morphism $\coprod_{i\in I} V_i \to U$ satisfies (effective) descent when the direct sum is disjoint and universal (this can always be achieved by moving into $\widehat {\mathbf B}$).

If $\mathbf B$ is a site and effective descent topology for $T$ is finer than the topology of $\mathbf B$, then $T$ is called a \emph{stack} on $\mathbf B$.
If this only holds for the descent topology, then $T$ is said to be \emph{separated} and sometimes called a \emph{prestack}\footnote{By analogy with presheaves, it might be better to call prestack any fiberred category over a site.}.
Alternatively, it means that any covering sieve for the topology of $\mathbf B$ is universally of (effective) descent with respect to  $T$.
Note that $T$ is a (pre-) stack on $\mathbf B$ if and only if $T^+$ is a (pre-) stack on $\widehat {\mathbf B}$.
As a standard example, if $\mathbf B$ is any site, then the fibered category whose fiber over $U$ is the opposite to the category $\widetilde U$ of sheaves on $U$ is a stack over $\mathbf B$.
Also, if $\mathcal T$ is a presheaf of sets on $\mathbf B$, then the fibered category $\mathbf B_{/\mathcal T}$ is a stack (resp.\ a prestack) if and only $\mathcal T$ is a sheaf (resp.\ is separated).
Be careful however, that if $\mathcal T$ is a sheaf of \emph{categories}, then $\mathbf B_{/\mathcal T}$ is not necessarily a stack but only a prestack in general.
Any fibered category $T$ over a site $\mathbf B$ has a stackification $T^{\sharp}$ which is unique up to a unique $2$-isomorphism.
In the case of \emph{discrete} fibered categories (i.e.\ categories fibered in sets), this is compatible with the notion of associated sheaf.
A morphism $T' \to T$ is a local isomorphism if and only if the corresponding morphism of stacks $ T'^{\sharp} \to T^{\sharp}$ is an equivalence.

Let us finally mention that if $T''$ is a prestack (resp.\ a stack) and $T' \to T$ is a local epimorphism (resp.\ a local isomorphism), then the functor
\[
\mathrm{Hom}_{\mathbb Fib(\mathbf B)}(T, T'') \to \mathrm{Hom}_{\mathbb Fib(\mathbf B)}(T', T'')
\]
is fully faithful (resp.\ an equivalence).

\subsection{$D$-objects} \label{formaco}

We introduce some formalism that will be needed for cohomological descent.

We fix a base site $\mathbf B$ with fibered products\footnote{One may always replace $\mathbf B$ with $\widehat{\mathbf B}$ if necessary or else use hypercoverings.} and we consider the fibered category $\mathbf T$ over $\mathbf B$ whose fiber over any $U \in \mathbf B$ is the opposite to the category $\widetilde U$ of sheaves on $U$ (more generally, $\mathbf T$ could be any category bifibered in dual of topoi).
Given a (small) category $D$, a \emph{$D$-object of $\mathbf B$} is a contravariant functor $U : D \to \mathbf B$.
If we pull $\mathbf T$ back along the $D$-object $U$, then we obtain what is usually called a \emph{$D$-topos} $\mathbf T_U$ (i.e.\ a bifibered category in topoi over $D$).
The category of sections
\[
\widetilde {U} := \mathrm{Hom}_{D}(D,\mathbf T_{U})
\]
is then a topos and we shall call \emph{sheaf over $U$} an object of $\widetilde U$.
Concretely, a sheaf $\mathcal F$ over $U$ is simply a family of sheaves $\mathcal F(i)$ over $U(i)$ for $i \in D$ with compatible transition maps $\lambda^{-1}\mathcal F(j) \to \mathcal F(i)$ (or equivalently $\mathcal F(j) \to \mathcal \lambda_*F(i)$) for $\lambda : i \to j$ in $D$.
It is also convenient to consider the fibered category over $\mathbf{Cat}$ whose fiber over $D$ is the category $\mathrm{Hom}(D^{\mathrm{op}}, \mathbf B)$ of $D$-objects of $\mathbf B$ and pullback is given by $u^{-1}U = U \circ u^{\mathrm{op}}$ for $u : E \to D$.
Now, if $V$ is an $E$-object of $\mathbf B$, it makes sense to talk about a morphism $f : V \to U$ over $u : E \to D$.
Concretely, $f$ is given by a compatible family of morphisms $f(i) : V(i) \to U(u(i))$ in $\mathbf B$.
The morphism $f$ induces a morphism of topoi $f : \widetilde V \to \widetilde U$ whose inverse image is given by $f^{-1}(\mathcal F)(i) = f(i)^{-1}\mathcal F(u(i))$.

For our convenience, we introduce the following \emph{ad hoc} terminology:

\begin{dfn} \label{defstab}
Let $U$ (resp.\ $V$) be a $D$-object (resp.\ an $E$-object) of $\mathbf B$, $f : V \to U$ a morphism over some $u : E \to D$ and $\mathcal F$ a complex of abelian sheaves on $U$.
We then say that $f$ \emph{stabilizes} $\mathcal F$ if the adjunction map
\[
\mathcal F \to \mathrm Rf_*f^{-1}\mathcal F
\]
is an isomorphism.
\end{dfn}

The following elementary lemma will be useful later:

\begin{lem} \label{baslem}
If $g : W \to V$ stabilizes $f^{-1}\mathcal F$, then $f$ stabilizes $\mathcal F$ if and only if $f \circ g$ stabilizes $\mathcal F$.
\end{lem}

\begin{proof}
We have
\[
\mathrm R(f\circ g)_*(f\circ g)^{-1}\mathcal F = \mathrm Rf_*\mathrm Rg_*g^{-1}f^{-1}\mathcal F = \mathrm Rf_*f^{-1}\mathcal F . \qedhere
\]
\end{proof}

Let us also mention that, if we are given a distinguished triangle 
\[
\mathcal F' \to \mathcal F \to \mathcal F'' \to \cdots
\]
and $f$ stabilizes two of them, then it is also stabilizes the other. 

We shall mostly consider the case $D = \mathbf \Delta^k$ where $k=0,1,2$ and $\mathbf \Delta$ denotes the category of simplexes.
In the case $k=0$ so that $D = \{0\}$, we can identify a $D$-object of $\mathbf B$ with an object $U \in \mathbf B$ (and the corresponding topoi as well).
Let us now consider some morphism $f : V \to U$ in $\mathbf B$.
We can then build the $0$-coskeleton $V(\bullet)$ of $V$ over $U$ which is the simplicial complex given by
\[
V(i) := (V/U)(i) := \underbrace{V \times_U \cdots \times_U V}_{i+1}.
\]
There exists an \emph{augmentation} map $f_\epsilon : V(\bullet) \to U$ over the final functor $\epsilon : \Delta \to \{0\}$ simply given for each $i \in \mathbf\Delta$ by the projection $f(i) : V(i) \to U$.

If we are given another morphism $g: W \to U$ then, we can also consider the bisimplicial complex $V(\bullet) \times_U W(\bullet)$ and we will be interested in the commutative diagram
\begin{align} \label{simdiag}
\xymatrix{V(\bullet) \times_U W(\bullet)\ar[rr]^{g_{\epsilon}(\bullet)} \ar[d]^{f_\epsilon(\bullet)} &&  V(\bullet) \ar[d]^{f_{\epsilon}}  \\ W(\bullet) \ar[rr]^{g_\epsilon} && U}
\end{align}
over
\[
\xymatrix{\mathbf\Delta^2 \ar[rr]^{p_1} \ar[d]^{p_2} &&  \mathbf\Delta \ar[d]^{\epsilon}  \\ \mathbf\Delta \ar[rr]^{\epsilon} && \{0\}.}
\]
For each $i \in \mathbf\Delta$ and any small category $D$, we can consider the inclusion map
\[
i : D \hookrightarrow \mathbf \Delta \times D, \quad j \mapsto (i,j).
\]
Given a $(\mathbf \Delta \times D)$-objet $V$ of $\mathbf B$, we will then write $V(i) := i^{-1}V$ and $\mathcal F(i) := i^{-1}\mathcal F$ if $\mathcal F$ is a sheaf on $V$.
For example,
\[
(V(\bullet) \times_U W(\bullet))(i) = V(i) \times_U W(\bullet) \simeq (V(i) \times_U W/V(i))(\bullet).
\]

Then, we have the following generalization of proposition 3.1.5 of Exposé Vbis of \cite{SGA4}:

\begin{lem}[Conrad] \label{Conrad}
Let $f: V \to U$ be a morphism over $p_1 : \mathbf\Delta^2 \to \mathbf\Delta$ and $\mathcal F$ a complex of abelian sheaves on $U$.
The morphism $f$ stabilizes $\mathcal F$ if and only if $f(i) : V(i) \to U(i)$ stabilizes $\mathcal F(i)$ for all $i \in \mathbf\Delta$.
\end{lem}

\begin{proof}
We follow the proof given by Brian Conrad.
For an abelian sheaf $\mathcal F$ on $U$ and an abelian sheaf $\mathcal G$ on $V$, we have
\[
(f^{-1}\mathcal F)(i) \simeq f(i)^{-1}\mathcal F(i) \quad  \mathrm{and} \quad  (f_*\mathcal G)(i) \simeq f(i)_*\mathcal G(i).
\]
Actually, the first isomorphism is completely formal and the second one follows from  \cite[\href{https://stacks.math.columbia.edu/tag/0FN1}{Tag 0FN1}]{stacks-project}.
The important point now is lemma 7.6 of \cite{Conrad14} which shows that $\mathcal G(i)$ is injective when $\mathcal G$ is injective.
Our assertion then follows from the fact that the family $\{\mathcal F \mapsto \mathcal F(i)\}_{i \in \mathbf\Delta}$ is conservative.
\end{proof}

\subsection{Cohomological descent} \label{apcdes}

Cohomological descent is a powerful technic that was initiated by Bernard Saint-Donat in Exposé Vbis of \cite{SGA4} (see also \cite{Deligne74}, \cite{Laszlo03} and \cite{Conrad14}).
In \cite{ZureickBrown14}), David Zureick-Brown was able to derive from Conrad's approach that there actually exists a topology of cohomological descent with respect to  any fibered category of abelian sheaves (and not only for a thick bifibered category such as in \cite{SGA4}).

\begin{dfn}
A morphism $f : V \to U$ in $\mathbf B$ is said to satisfy \emph{cohomological descent} with respect to a complex $\mathcal F$ on $U$ if the augmentation map $f_\epsilon : V(\bullet) \to U$ stabilizes $\mathcal F$ (see definition \ref{defstab}).
It is called \emph{universally of cohomological descent} with respect to $\mathcal F$ if, given any morphism $\pi : U' \to U$, the inverse image $f' : V' \to U'$ of $f$ satisfies cohomological descent with respect to $\pi^{-1}\mathcal F$.
\end{dfn}

When direct sums are disjoint and universal in $\mathbf B$, a direct sum 
\[
\coprod f_i : \coprod_{i \in I} V_i \to \coprod_{i \in I} U_i
\]
is (universally) of cohomological descent with respect to a direct sum $\coprod_{i \in I} \mathcal F_i$ if and only if each $f_i : V_i \to U_i$ is (universally) of cohomological descent with respect to $\mathcal F_i$.
More generally, we will also say that a family $\{f_i : V_i \to U\}_{i\in I}$ is \emph{(universally) of cohomological descent} with respect to $\mathcal F$ when the morphism $\coprod_{i \in I} V_i \to U$ is (universally) of cohomological descent with respect to $\mathcal F$.
Note that our condition that direct sums are disjoint and universal might always be achieved after replacing $\mathbf B$ with $\widehat{\mathbf B}$.

For explicit computations, one can use the following description:
\[
\mathrm Rf_{\epsilon*}f_\epsilon^{-1}\mathcal F = [\mathrm Rf(0)_*f(0)^{-1} \mathcal F \overset {d_0} \to \cdots \overset {d_{i-1}} \to \mathrm Rf(i)_*f(i)^{-1} \mathcal F\overset {d_i} \to \cdots]
\]
where
\[
d^i = p_{\widehat 1}(i)^{-1} - p_{\widehat 2}(i)^{-1} + \cdots + (-1)^i p_{\widehat {i+2}}(i)^{-1}
\]
if $p_{\widehat k}(i)$ denotes the projection that forgets the $k$th component for $k=1, \ldots, i+2$.
In particular, we see that $f : V \to U$ satisfies cohomological descent with respect to $\mathcal F$ if and only if
\[
\left\{\begin{array}l
\mathcal F = \ker(f(0)_*f(0)^{-1} \mathcal F \to f(1)_*f(1)^{-1} \mathcal F)
\\
\\ \forall k > 0, \quad \mathrm R^kf_{\epsilon*}f_\epsilon^{-1}\mathcal F =0.
\end{array}\right.
\]
Then, there exist spectral sequences
\[
E_1^{i,j} := \mathrm R^jf(i)_*f(i)^{-1} \mathcal F \Rightarrow \mathcal F
\]
and
\[
E_1^{i,j} := \mathrm H^j(U, f(i)^{-1} \mathcal F) \Rightarrow \mathrm H^{i+j}(U, \mathcal F).
\]

From this direct description, one can easily derive the following:

\begin{prop} \label{loccd}
Any local epimorphism $f : V \to U$ is universally of cohomological descent with respect to  abelian sheaves.
\end{prop}

\begin{proof}
This is done for example in \cite{Olsson07}, lemma 1.4.24.
\end{proof}

Together with the next lemma \ref{cohdes}, proposition \ref{loccd} implies that being universally of cohomological descent is local downstairs.

\begin{lem} \label{cohdes}
Consider a cartesian diagram
\[
\xymatrix{V' \ar[r]^{\pi'} \ar[d]^{f'} &  V\ar[d]^f  \\ U' \ar[r]^{\pi} & U}
\]
in $\mathbf B$ and let $\mathcal F$ be a complex of abelian sheaves on $U$.
Suppose $\pi$ is universally of cohomological descent with respect to $\mathcal F$.
Then, $f$ is universally of cohomological descent with respect to $\mathcal F$ if and only if $f'$ is universally of cohomological descent with respect to $\pi^{-1}\mathcal F$.
\end{lem}

\begin{proof}
Only the converse implication needs a proof.
We consider the commutative diagram
\[
\xymatrix{V(\bullet) \times_U U'(\bullet)\ar[rr]^-{\pi_{\epsilon}(\bullet)} \ar[d]^{f_\epsilon(\bullet)} &&  V(\bullet) \ar[d]^{f_{\epsilon}}  \\ U'(\bullet) \ar[rr]^-{\pi_\epsilon} && U.}
\]
Since $f'$ is \emph{universally} of cohomological descent with respect to $\pi^{-1}\mathcal F$, we see that, for all $j \in \mathbf\Delta$, the morphism
\[
f_\epsilon(j) : V(\bullet) \times_U U'(j) = (V' \times_{U'} U'(j)/U'(j))(\bullet) \to U'(j)
\]
stabilizes $\pi(i)^{-1}\mathcal F$.
It follows from lemma \ref{Conrad} that the left vertical map stabilizes $\pi_\epsilon^{-1}\mathcal F$.
The same argument shows that the top horizontal map stabilizes $f_\epsilon^{-1}\mathcal F$.
Our assertion therefore follows from lemma \ref{baslem}.
\end{proof}

Note that the word ``universal'' is essential in the above lemma which is the key that opens the door to a very flexible theory of cohomological descent.

\begin{prop} \label{compcohd}
Let $f : V \to U$ and $g : W \to V$ be two morphism in $\mathbf B$ and $\mathcal F$ a complex of abelian sheaves on $U$.
If $f \circ g$ is universally of cohomological descent with respect to $\mathcal F$, then $f$ is universally of cohomological descent with respect to $\mathcal F$ and the converse holds if $g$ is universally of cohomological descent with respect to $f^{-1}\mathcal F$.
\end{prop}

\begin{proof}
We consider the cartesian diagram
\[
\xymatrix{V \times_U W \ar[r]^-{p_1} \ar[d]^{p_2} &  V\ar[d]^f  \\ W \ar[r]^{f \circ g} & U.}
\]
Since $p_2$ has a section $s = (g, \mathrm{Id})$, this is a local epimorphism and therefore a morphism of universally cohomological descent with respect to  $(f \circ g)^{-1}\mathcal F$.
Thus, if $f \circ g$ is universally of cohomological descent with respect to $\mathcal F$, it follows from lemma \ref{cohdes} that $f$ also is universally of cohomological descent with respect to $\mathcal F$.
Conversely, assume that $f$ is universally of cohomological descent with respect to $\mathcal F$ and $g$ is universally of cohomological descent with respect to $f^{-1}\mathcal F$.
Since $g = p_1 \circ s$, it follows from the first part that $p_1$ is universally of cohomological descent with respect to $f^{-1}\mathcal F$ and we can apply lemma \ref{cohdes}  again.
\end{proof}

Note that, conversely, one can recover lemma \ref{cohdes} from the proposition.

We also need to check that moving from $\mathbf B$ to $\widehat{\mathbf B}$ is harmless since this does not follow directly from the formalism:

\begin{lem}
A morphism $f : V \to U$ in $\mathbf B$ is universally of cohomological descent with respect to a complex $\mathcal F$ if and only if it becomes so in $\widehat{\mathbf B}$.
\end{lem}

\begin{proof}
The condition is clearly necessary and we work now over $\widehat{\mathbf B}$.
If we are given a morphism $F \to U$ with $F \in \widehat{\mathbf B}$, then there exists a locally epimorphic family  $\{U_i \to F\}_{\in I}$ with $U_i \in \mathbf B$ and we may consider the commutative diagram
\[
\xymatrix{U_i(\bullet) \times_U V(\bullet) \ar[r] \ar[d] & (V \times_U F)(\bullet) \ar[d] \\ U_i(\bullet) \ar[r] & F.}
\]
We can then use the same arguments as in the proof of lemma \ref{cohdes} in order to see that the right hand side map stabilizes $\mathcal F$.
\end{proof}

As a consequence of this lemma, in order to show that a morphism $G \to F$ in $\widehat{\mathbf B}$ is universally of cohomological descent, it is sufficient to check that it satisfies cohomological descent after any pullback $U \to F$ with $U \in \mathbf B$.

Assume now that $T$ is a category fibered in abelian sheaves on $\mathbf B$.
It means that each $T(U)$ is a subcategory\footnote{We do not require that it is a thick subcategory nor that it is bifibered.} of the category of all abelian sheaves on $U$ and that, given any morphism $f:V \to U$ in $\mathbf B$ and any $\mathcal F \in T(U)$, we have $f^{-1}\mathcal F \in T(V)$.
It follows from proposition \ref{compcohd} that morphisms that are universally of cohomological descent with respect to all $\mathcal F \in T$ define a Grothendieck topology on $\mathbf B$.
Moreover, proposition \ref{loccd} shows that this topology is finer than the original one.

\begin{dfn}
A morphism $f : V \to U$ in $\mathbf B$ is said to satisfy \emph{total descent} with respect to a fibered category of abelian sheaves $T$ if it is simultaneously of universal effective descent and of universal cohomological descent with respect to all $\mathcal F \in T$.
\end{dfn}

This is usually called \emph{universal effective cohomological descent} but we'd rather make it shorter.
Again, this defines a Grothendieck topology which is finer than the usual topology.

\addcontentsline{toc}{section}{References}
\printbibliography

\Addresses

\end{document}